\documentclass[12pt]{article}

\usepackage[utf8]{inputenc}

\usepackage[dvipsnames]{xcolor}

\usepackage{algorithm, algorithmic}
\usepackage{amsmath, amssymb, amsthm}
\usepackage{bbm}
\usepackage[giveninits=true, maxnames=10, style=alphabetic]{biblatex}
\usepackage{cancel}
\usepackage{caption}
\usepackage{enumitem}
\usepackage{euscript}
\usepackage{float}
\usepackage{geometry}
\usepackage{graphicx}
\usepackage{letltxmacro}
\usepackage{libertine}
\usepackage{mathrsfs}
\usepackage{mathtools}
\usepackage{mdframed}
\usepackage{microtype}
\usepackage[libertine]{newtxmath}
\usepackage{thmtools}
\usepackage{tikz}
\usepackage[Symbolsmallscale]{upgreek}

\usepackage{sinhomath}

\usepackage{hyperref}

\hypersetup{citecolor=YBlue, colorlinks=true, linkcolor=YBlueA}

\DeclareMathOperator{\thresh}{\msf{thresh}}

\allowdisplaybreaks{}

\title{Lectures on Optimization}
\author{Sinho Chewi}
\date{\today}

\mdfdefinestyle{mdcolored}{backgroundcolor=YBlue!10, leftmargin=0em, rightmargin=0em, hidealllines=true, innertopmargin=0em}
\mdfdefinestyle{mdcoloredtwo}{backgroundcolor=YBlueA!10, leftmargin=0em, rightmargin=0em, hidealllines=true, innertopmargin=0em}

\declaretheoremstyle[headfont=\color{YBlue}\normalfont\bfseries, spaceabove=\topsep{}, mdframed={style=mdcolored}, spacebelow=\topsep{}]{colored}
\declaretheoremstyle[headfont=\color{YBlueA}\normalfont\bfseries, spaceabove=\topsep{}, mdframed={style=mdcoloredtwo}, spacebelow=\topsep{}]{coloredtwo}
\declaretheoremstyle[headfont=\color{YBlue}\normalfont\bfseries, spaceabove=\topsep{}, spacebelow=\topsep{}]{question}

\declaretheorem[name=Theorem, numberwithin=section, style=colored]{thm}

\declaretheorem[name=Corollary, sibling=thm, style=colored]{cor}
\declaretheorem[name=Definition, sibling=thm, style=colored]{defn}
\declaretheorem[name=Example, sibling=thm, style=coloredtwo]{ex}
\declaretheorem[name=Exercise, numberwithin=section, style=question]{question}
\declaretheorem[name=Lemma, sibling=thm, style=colored]{lem}
\declaretheorem[name=Proposition, sibling=thm, style=colored]{prop}
\declaretheorem[name=Remark, sibling=thm, style=colored]{rmk}

\numberwithin{equation}{section}

\begin{document}

\maketitle

\tableofcontents{}

\newpage

\section{Introduction and basics of convex functions}

These lecture notes accompany S\&DS 4320/6320 (Advanced Optimization Techniques), taught at Yale University in Spring 2025 and Spring 2026.
They are not meant to be comprehensive.

The notes are primarily based on the books~\cite{Bub15CvxOpt, Beck17FirstOrder, Nes18CvxOpt}, as well as my personal understanding of the subject formed through discussions with many people over the years.
Please send me feedback via email.
I thank Linghai Liu, Leda Wang, Ruixiao Wang, Ilias Zadik, and Matthew S.\ Zhang for corrections.

\paragraph*{Audience.}
This course focuses on the theory of optimization.
In particular, the course is \textbf{mathematical} in nature and taught in a theorem{--}proof format.
The course assumes familiarity with basic proofs and logical reasoning, as well as linear algebra, multivariate calculus, and probability theory.

The reader should also be familiar with asymptotic notions (big-$O$ notation).
We use the shorthand notation $a \lesssim b$ (resp.\ $a \gtrsim b$) to mean that $a \le Cb$ (resp.\ $a \ge b/C$) for an absolute constant $C > 0$ (i.e., a constant that does not depend on other parameters of the problem), and $a \asymp b$ to mean that both $a \lesssim b$ and $a \gtrsim b$ hold.
We use $a = O(b)$ and $a \lesssim b$ interchangeably.

\subsection{Overview of the course}\label{ssec:overview}

The basic problem of optimization is to compute an approximate minimizer of a given function $f : \eu X \to \R$.
In this course, $\eu X$ is always taken to be a subset of $\R^d$, although generalizations are possible (e.g., to manifolds).

\paragraph*{Black-box optimization and the oracle model.}
What does it mean to ``compute''?
The answer depends on the representation of $f$ and our model of computation.
We start by studying \emph{black-box optimization}.
In this model, we presume that we can \emph{evaluate} $f$, and possibly its derivatives, at any chosen point $x \in \eu X$.

The advantage of the black-box model is that it applies very \emph{generally}: it is difficult to find situations in which we need to optimize a function but we cannot even evaluate it!
Consequently, algorithms developed in this model can be applied to the majority of problems encountered in practice\footnote{There is a caveat: in this course, we solely consider continuous optimization problems. Combinatorial optimization is an entirely different beast.}{---}witness the ubiquity of gradient descent.

The disadvantage is that by its very generality, it cannot take advantage of additional structural information about $f$ which can bring computational savings.
That is why, later in the course, we turn toward the study of \emph{structured} optimization problems.

It is easy, at least at an intuitive level, to describe algorithms which are valid in the black-box model. Namely, they are algorithms which only ``interact'' with $f$ through evaluations of $f$ and its derivatives.
The existence of an algorithm, together with a corresponding mathematical analysis of the number of iterations to reach an approximate minimizer contingent upon assumptions on $f$, provide an \emph{upper bound} on the complexity of the optimization task.
In this course, we are also interested in \emph{lower bounds}, which delineate fundamental limitations encountered by \emph{any} algorithm.
In order to prove such a lower bound, we need to formalize the notion of ``interaction'' alluded to above, and this leads to the important concept of an \emph{oracle}.

First, observe that it does not make sense to discuss the complexity of optimizing a \emph{single} function $f$.
For if $x_\star$ is the minimizer of $f$, we can consider the algorithm ``output $x_\star$'', which yields the correct answer in one iteration.
But this algorithm is silly, since it utterly fails at optimizing any other function whose minimizer does not happen to be $x_\star$.
Reflecting upon this situation, we do not consider an optimization algorithm to be sensible when it happens to succeed for one particular problem; rather, we expect it to succeed on many similar problems.
Hence, we talk about a \emph{class} of functions $\mc F$ of interest, and we require our algorithms to succeed on \emph{every} $f \in \mc F$.

The algorithm is designed to succeed on $\mc F$ and thus, in an anthropomorphic sense, it ``knows'' $\mc F$.
However, it does not know which particular $f \in \mc F$ it is trying to optimize.
(If it possessed knowledge of $f$, then we run into the issue from before, namely it could simply output the minimizer.)
The role of the oracle is to act as an intermediary between the algorithm and the function.
Namely, we assume that the algorithm is allowed to ask certain questions (``queries'') to the oracle for $f$, and this is the only means by which the algorithm can gather more information about $f$.
The allowable queries and responses determine the nature of the oracle, e.g.:
\begin{itemize}
    \item a \textbf{zeroth-order oracle} accepts a query point $x\in\R^d$ and outputs $f(x)$;
    \item a \textbf{first-order oracle} accepts a query point $x\in\R^d$ and outputs $(f(x),\nabla f(x))$.
\end{itemize}
Most of the course focuses on optimization with a first-order oracle, but other oracles are possible (e.g., \emph{linear optimization oracles} and \emph{proximal oracles}).
The zeroth-order and first-order oracles are easy to justify, as they correspond to the black-box model described above.
As the oracles become more exotic, it becomes necessary to show that they are reasonable, by describing important applications in which such access to $f$ is feasible.

The \emph{query complexity} of $\mc F$ for a particular choice of oracle, as a function of the prescribed tolerance $\varepsilon$, is then (informally) defined to be the minimum number $N$ such that there exists an algorithm which, for any $f \in \mc F$, makes $N$ queries to the oracle for $f$ and outputs a point $x$ with $f(x) - \min f \le \varepsilon$.

It is worth noting that query complexity is not the same as computational complexity.
Indeed, query complexity only counts the number of interactions with the oracle, and the algorithm is allowed to perform unlimited computations between interactions.
In principle, this could lead to a situation in which query complexity is wholly unrepresentative of the true computational cost of optimization{---}this would be the case if optimal algorithms in the oracle model were contrived and impractical.
Thankfully, this is not the case.
The oracle model is widely adopted as the standard model for optimization because it is the setting in which we can make precise claims about complexity, and because it generally aligns with optimization in practice.

This summarizes the conceptual framework for optimization theory{---}the ``identity cards of the field''~\cite{Nes18CvxOpt}, although a careful treatment of the framework only becomes necessary when discussing lower bounds (and hence we elaborate on the details then).
As a branch of mathematics, the theory of optimization could be defined as the quest to characterize the query complexity of various classes $\mc F$, under various oracle models, and thereby identify optimal algorithms.
This indeed remains a core element of the field, but as query complexity reaches maturity, research has shifted toward different types of questions, often inspired by practical developments.

\paragraph*{The role of convexity.}
In order to optimize efficiently, we need to place assumptions on $f$, ideally minimal ones.
For example, we can assume that $f$ is continuous.
In this course, however, we are interested in \emph{quantitative} rates of convergence for algorithms, and for this purpose, a \emph{qualitative} assumption such as continuity is not enough.
A quantitative form of continuity is to assume that $f$ is $L$-\emph{Lipschitz in the $\ell_\infty$ norm}:
\begin{align}\label{eq:linf_Lip}
    \abs{f(x) - f(y)} \le L\max_{i\in [d]}{\abs{x[i]-y[i]}} \qquad\text{for all}~x,y\in\eu X\,.
\end{align}
Also, for concreteness, let us take $\eu X$ to be the cube, $\eu X = {[0,1]}^d$.
In the language of the framework above, we consider the class
\begin{align}\label{eq:Lip_class}
    \mc F
    &= \{f : {[0,1]}^d \to\R \mid f~\text{satisfies}~\eqref{eq:linf_Lip}\}\,.
\end{align}
One can then prove the following negative result.

\begin{thm}
    For any $0 < \varepsilon < L/2$ and any deterministic algorithm, the complexity of $\varepsilon$-approximately minimizing functions in the class defined in~\eqref{eq:Lip_class} to within $\varepsilon$ using a zeroth-order oracle is at least $\lfloor \frac{L}{2\varepsilon} \rfloor^d$.
\end{thm}

Thus, for $\varepsilon < L/2$, the complexity grows \emph{exponentially} with the dimension.
The proof is not difficult; see, e.g.,~\cite[Theorem 1.1.2]{Nes18CvxOpt}.
It is also robust: variants of the result can be proven when the notion of Lipschitzness is w.r.t.\ the $\ell_2$ norm; when the oracle is taken to be a first-order oracle; when the algorithm is allowed to be randomized; etc.
The message is clear: in order for optimization to be tractable in the worst case, we must impose some structural assumptions.

The black-box oracles we have been considering are \emph{local} in nature: given a query point $x \in \R^d$, the oracle reveals some information about the behavior of $f$ in a local neighborhood of $x$.
Assumptions such as Lipschitzness effectively govern how large this local neighborhood is.
But ultimately, to render optimization tractable, we must ensure that local information yields global consequences.
As justified in the next subsection, a key assumption that makes this possible is \emph{convexity}.

Of course, not every problem is convex, and non-convex optimization often still succeeds.
But for the purpose of understanding the core principles underlying optimization, there is no better starting place.
It is important to remember that convex problems abound in every application domain; here, we give two classical examples from statistics.

\begin{ex}[logistic regression]\label{ex:logistic_regression}
    The data consists of $n$ pairs $(X_i,Y_i) \in \R^d\times \{0,1\}$, where $X_i$ is a vector of covariates and $Y_i$ is a binary response.
    The statistical model assumes that the pairs are independently drawn, the covariates are deterministic, and $Y_i$ has a Bernoulli distribution with parameter $\exp(\langle \theta, X_i \rangle)/\{1+\exp(\langle \theta, X_i\rangle)\}$.
    The goal is to infer the parameter $\theta$.

    The maximum likelihood estimator (MLE) for this model is the solution to the convex optimization problem
    \begin{align*}
        \widehat \theta_{\rm MLE}
        &\in \argmin_{\theta\in\R^d} \frac{1}{n} \sum_{i=1}^n \bigl(\log(1+\exp{\langle\theta, X_i\rangle}) - Y_i\,\langle \theta, X_i \rangle\bigr)\,.
    \end{align*}
\end{ex}

\begin{ex}[LASSO]\label{ex:LASSO}
    The data consists of $n$ pairs $(X_i,Y_i)\in\R^d\times \R$.
    The statistical model assumes that the pairs are independently drawn, and that $Y_i = \langle \theta, X_i \rangle + \xi_i$, where the $\xi_i$'s are i.i.d.\ noise variables independent of the $X_i$'s.
    When the parameter $\theta$ is assumed to be sparse, it is standard to use the LASSO estimator, which is the solution to the convex optimization problem
    \begin{align*}
        \widehat \theta_{\rm LASSO} = \argmin_{\theta\in\R^d}{\Bigl\{ \frac{1}{2n}\sum_{i=1}^n {(Y_i - \langle \theta, X_i \rangle)}^2 + \lambda\,\norm \theta_1\Bigr\}}\,.
    \end{align*}
    Here, $\lambda > 0$ is the regularization parameter and $\norm \cdot_1$ denotes the $\ell_1$ norm, defined via $\norm \theta_1 \deq \sum_{i=1}^d \abs{\theta[i]}$.
\end{ex}

In these examples, the estimator is defined as the solution to a convex problem which is not solvable in closed form, necessitating the use of numerical optimization.
Actually, it is not that most problems in the ``wild'' are convex and hence there was a need to develop convex optimization.
In fact, it often goes the other way around: convex optimization is such a powerful tool that problems are intentionally formulated to be convex.
This is the case for the LASSO estimator, which can be motivated as a convex relaxation of the (statistically superior) $\ell_0$-constrained least-squares estimator.

\paragraph*{First-order methods.}
This course largely focuses on first-order methods, namely, gradient descent and its variants.
This class of methods is natural from the perspective of the theory.
Equally importantly, first-order methods are lightweight and therefore scalable to large problem sizes, making them the method of choice even for highly non-convex settings which fall squarely outside of the theory.

\paragraph*{Beyond the black-box model.}
After developing results for the black-box model, we study structured problems which admit more efficient solutions.
The LASSO estimator of~\autoref{ex:LASSO} can be treated as a ``composite'' optimization problem (a sum of a smooth and a non-smooth function), and the estimators in both~\autoref{ex:logistic_regression} and~\autoref{ex:LASSO} (and empirical risk minimization more generally) are ``finite sum'' problems whose computation can be sped up via the use of stochastic gradients.
Other examples include the use of alternative geometries (mirror descent) and the use of coordinate-wise structure (alternating minimization/coordinate descent).

We also study interior-point methods, which are a practically effective suite of algorithms which solve linear programs (LPs) and semidefinite programs (SDPs) with polynomial iteration complexities.

Further topics are considered as time permits.

\subsection{Preliminaries on convexity and smoothness}\label{ssec:prelim}

We assume familiarity with the basic notion of convexity, and we briefly review it here.

\begin{defn}
    A subset $\eu C \subseteq \R^d$ is \textbf{convex} if for all $x,y\in\eu C$ and all $t\in [0,1]$, the point $(1-t)\,x+t\,y$ also lies in $\eu C$.
\end{defn}

\begin{defn}
    Let $\eu C$ be convex and let $\alpha \ge 0$.
    A function $f : \eu C\to\R$ is \textbf{$\alpha$-convex} if for all $x,y\in\eu C$ and all $t\in [0,1]$,
    \begin{align}\label{eq:str_cvx}
        f((1-t)\,x+t\,y)
        \le (1-t)\,f(x) + t\,f(y) - \frac{\alpha}{2}\,t\,(1-t)\,\norm{y-x}^2\,.
    \end{align}
\end{defn}

When $\alpha = 0$, this is just the usual definition of a convex function.
When $\alpha > 0$, we say that the function is \emph{strongly} convex.

The definition above has the advantage that it does not require $f$ to be differentiable.
However, for the purposes of checking and utilizing convexity, it is convenient to have the following equivalent reformulations, which should be committed to memory.
For simplicity, we focus on the case $\eu C = \R^d$.

\begin{prop}[convexity equivalences]\label{prop:cvx_equiv}
    Let $\eu C = \R^d$ and $\alpha\ge 0$.
    \begin{enumerate}
        \item If $f$ is continuously differentiable,~\eqref{eq:str_cvx} is equivalent to each of the following:
            \begin{equation}\label{eq:str_cvx_1}
                f(y) \ge f(x) + \langle \nabla f(x), y-x \rangle + \frac{\alpha}{2}\,\norm{y-x}^2 \qquad\text{for all}~x,y\in\R^d\,.
            \end{equation}
            \begin{equation}\label{eq:str_cvx_2}
                \langle \nabla f(y) - \nabla f(x), y-x\rangle \ge \alpha\,\norm{y-x}^2 \qquad\text{for all}~x,y\in\R^d\,.
            \end{equation}
        \item If $f$ is twice continuously differentiable,~\eqref{eq:str_cvx} is equivalent to
            \begin{align}\label{eq:str_cvx_3}
                \langle v, \nabla^2 f(x)\,v\rangle \ge \alpha\,\norm v^2 \qquad\text{for all}~v,x\in\R^d\,.
            \end{align}
    \end{enumerate}
\end{prop}
\begin{proof}
    Assume that $f$ is continuously differentiable.

    {\eqref{eq:str_cvx}} $\Rightarrow$~\eqref{eq:str_cvx_1}: Rearranging~\eqref{eq:str_cvx} yields, for $t > 0$,
    \begin{align*}
        f(y) \ge f(x) + \frac{f((1-t)\,x+t\,y) - f(x)}{t} + \frac{\alpha\,(1-t)}{2}\,\norm{y-x}^2\,.
    \end{align*}
    Sending $t\searrow 0$ yields~\eqref{eq:str_cvx_1}.

    {\eqref{eq:str_cvx_1}} $\Rightarrow$~\eqref{eq:str_cvx_2}: Swap $x$ and $y$ in~\eqref{eq:str_cvx_1} and add the resulting inequality back to~\eqref{eq:str_cvx_1}.

    {\eqref{eq:str_cvx_2}} $\Rightarrow$~\eqref{eq:str_cvx}: By the fundamental theorem of calculus, for $v \deq y-x$,
    \begin{align*}
        f(y)
        &= f(x) + \int_0^1 \langle \nabla f(x+sv), v\rangle\,\D s\,, \\
        f((1-t)\,x+t\,y)
        &= f(x) + \int_0^1 \langle \nabla f(x+stv), tv \rangle\,\D s\,.
    \end{align*}
    Hence,~\eqref{eq:str_cvx_2} yields
    \begin{align*}
        f((1-t)\,x+t\,y) - (1-t)\,f(x) - t\,f(y)
        &= -t\int_0^1 \langle \nabla f(x+sv) - \nabla f(x+stv), v \rangle\,\D s \\
        &\le -t\int_0^1 \alpha s\,(1-t)\,\norm v^2\,\D s
        = - \frac{\alpha}{2}\,t\,(1-t)\,\norm v^2\,.
    \end{align*}

    Finally, assume that $f$ is twice continuously differentiable.
    Letting $y = x + \varepsilon v$ in~\eqref{eq:str_cvx_2} and sending $\varepsilon \searrow 0$ establishes~\eqref{eq:str_cvx_3}.
    Conversely, the fundamental theorem of calculus shows that
    \begin{align*}
        \langle \nabla f(y) - \nabla f(x), y-x\rangle
        &= \int_0^1 \langle \nabla^2 f(x + t\,(y-x))\,(y-x),\, y-x \rangle\,\D t\,,
    \end{align*}
    and hence~\eqref{eq:str_cvx_3} implies~\eqref{eq:str_cvx_2}.
\end{proof}

The equivalent statements each have their own interpretation: for $\alpha = 0$,~\eqref{eq:str_cvx} states that $f$ lies below each of its secant lines between the intersection points;~\eqref{eq:str_cvx_1} states that $f$ globally lies above each of its tangent lines;~\eqref{eq:str_cvx_2} states that $\nabla f$ is a monotone vector field; and~\eqref{eq:str_cvx_3} is a statement about curvature.

As noted above, the key feature of convexity is that local information yields global conclusions.
Before describing this, let us first recall some basic facts about optimization.
For simplicity, we consider unconstrained optimization throughout.

\begin{lem}[existence of minimizer]\label{lem:existence_minimizer}
    Let $f : \R^d\to\R$ be continuous and its level sets be bounded.
    Then, there exists a global minimizer of $f$.
\end{lem}
\begin{proof}
    The proof uses some analysis.
    Let $x_0\in\R^d$ and let $\eu K \deq \{f \le f(x_0)\}$ denote the level set.
    By the continuity assumption, $\eu K$ is closed and bounded, thus compact.
    Let ${\{x_n\}}_{n\in\N}$ be a minimizing sequence, $f(x_n) \to \inf f$.
    By compactness, it admits a subsequence, still denoted ${\{x_n\}}_{n\in\N}$, which converges to some $x_\star\in\R^d$.
    By continuity, $f(x_\star) = \lim_{n\to\infty} f(x_n) = \inf f$.
\end{proof}

\begin{lem}[necessary conditions for optimality]\label{lem:necc_conditions}
    Let $f : \R^d\to\R$ be minimized at $x_\star$.
    \begin{enumerate}
        \item If $f$ is continuously differentiable, then $\nabla f(x_\star) = 0$.
        \item If $f$ is twice continuously differentiable, then $\nabla^2 f(x_\star) \succeq 0$.
    \end{enumerate}
\end{lem}
\begin{proof}
    Let $v \in \R^d$ and $\varepsilon > 0$; then, $f(x_\star + \varepsilon v) - f(x_\star) \ge 0$.
    If $f$ is continuously differentiable, this yields $\int_0^1 \langle \nabla f(x_\star + \varepsilon tv), v \rangle\,\D t \ge 0$.
    By continuity of $\nabla f$, sending $\varepsilon \searrow 0$ proves that $\langle \nabla f(x_\star), v \rangle \ge 0$ for all $v\in\R^d$, which entails $\nabla f(x_\star) = 0$.

    If $f$ is twice continuously differentiable, we can expand once more to obtain $0 \le \int_0^1 t \int_0^1 \langle \nabla^2 f(x_\star + \varepsilon st v) \,v, v\rangle\,\D s \, \D t$.
    By continuity of $\nabla^2 f$, sending $\varepsilon \searrow 0$ then proves that $\langle \nabla^2 f(x_\star)\,v, v \rangle \ge 0$ for all $v\in\R^d$.
\end{proof}

The conditions $\nabla f(x_\star) = 0$, $\nabla^2 f(x_\star) \succeq 0$ are necessary for optimality, but not sufficient in general.
The issue is that the proof of~\autoref{lem:necc_conditions} is entirely local, so the same conclusion holds even if $x_\star$ is only assumed to be a \emph{local} minimizer.
On the other hand, under the assumption of convexity, the first-order necessary condition becomes sufficient.

\begin{lem}[sufficient condition for optimality]
    Let $f : \R^d\to\R$ be \emph{convex} and continuously differentiable, and let $\nabla f(x_\star) = 0$.
    Then, $x_\star$ is a global minimizer of $f$.

    In particular, every local minimizer of $f$ is a global minimizer.
\end{lem}
\begin{proof}
    This easily follows from~\eqref{eq:str_cvx_1} with $x = x_\star$.
\end{proof}

Next, we note that the minimizer is unique if $f$ is strictly convex.

\begin{lem}[uniqueness of minimizer]\label{lem:unique_min}
    Assume that $f : \R^d\to\R$ is strictly convex, i.e., for all distinct $x,y\in\R^d$ and $t\in (0,1)$, $f((1-t)\,x+t\,y) < (1-t)\,f(x) + t\,f(y)$.
    Then, if $f$ admits a minimizer $x_\star$, it is unique.
\end{lem}
\begin{proof}
    If we had two distinct minimizers $x_\star$, $\tilde x_\star$, so that $f(x_\star) = f(\tilde x_\star)$, then strict convexity would imply $f(\frac{1}{2}\,x_\star + \frac{1}{2}\,\tilde x_\star) < f(x_\star)$, which is a contradiction.
\end{proof}

If $f$ is strongly convex, then it is strictly convex.
Also, from, e.g.,~\eqref{eq:str_cvx_1}, we see that $f$ grows at least quadratically at $\infty$, which implies that it has bounded level sets.
We can therefore conclude:

\begin{cor}
    Let $f : \R^d\to\R$ be strongly convex and continuously differentiable.
    Then, it admits a unique minimizer $x_\star$, which is characterized by $\nabla f(x_\star) = 0$.
\end{cor}

Finally, when discussing algorithms, we also need a dual condition{---}an \emph{upper} bound on the Hessian{---}which in this context is called \emph{smoothness}.\footnote{This is not to be confused with the mathematical usage of ``smoothness'' as ``infinitely differentiable''.}

\begin{defn}
    Let $\beta \ge 0$.
    We say that $f : \R^d\to\R$ is \textbf{$\beta$-smooth} if it is continuously differentiable and
    \begin{align}\label{eq:smooth}
        f(y) \le f(x) + \langle \nabla f(x), y-x \rangle + \frac{\beta}{2}\,\norm{y-x}^2\qquad\text{for all}~x,y\in\R^d\,.
    \end{align}
\end{defn}

The following proposition is established in the same way as~\autoref{prop:cvx_equiv}, so we omit the proof.

\begin{prop}[smoothness equivalences]
    Let $f : \R^d\to\R$ be continuously differentiable and $\beta\ge 0$.
    Then, $f$ is $\beta$-smooth if and only if
    \begin{align*}
        \langle \nabla f(y) - \nabla f(x), y - x\rangle
        \le \beta\,\norm{y-x}^2 \qquad\text{for all}~x,y\in\R^d\,.
    \end{align*}
    If $f$ is twice continuously differentiable, this is also equivalent to
    \begin{align*}
        \langle v, \nabla^2 f(x)\,v \rangle \le \beta\,\norm v^2 \qquad\text{for all}~v,x\in\R^d\,.
    \end{align*}
\end{prop}

If $f$ is convex, $\beta$-smooth, and twice continuously differentiable, then $0 \le \nabla^2 f \preceq \beta I$, which implies that the gradient $\nabla f$ is $\beta$-Lipschitz:
\begin{align}\label{eq:lip_grad}
    \norm{\nabla f(y) - \nabla f(x)} \le \beta\,\norm{y-x} \qquad\text{for all}~x,y\in\R^d\,.
\end{align}
This remains true even without assuming twice differentiability (\autoref{qu:smooth_implies_lipgrad}).

\subsection*{Bibliographical notes}

For further discussion on the oracle model, see~\cite[\S 1]{NemYud1983Complexity}.

\subsection*{Exercises}

\begin{question}\label{qu:quad_strcvx}
    Let $f = \frac{\alpha}{2}\,\norm \cdot^2$, where $\alpha \ge 0$.
    Show via direct computation that~\eqref{eq:str_cvx} holds with equality.
\end{question}

\begin{question}
    Show that for any twice continuously differentiable function $f : \R \to \R$ and any $t\in [0,1]$,
    \begin{align*}
        f(t) = (1-t) \,f(0) + t\,f(1) - \int_0^1 f''(s) \min\{s\,(1-t)\,,\,(1-s)\,t\}\,\D s\,.
    \end{align*}
\end{question}

\section{Gradient flow}\label{sec:grad_flow}

Before we turn toward our main first-order algorithm of interest, namely gradient descent, we first study the situation in continuous time via the gradient flow.
Throughout this section, we let ${(x_t)}_{t\ge 0}$ denote the gradient flow for $f$:
\begin{align}\label{eq:GF}\tag{$\msf{GF}$}
    \dot x_t = -\nabla f(x_t)\,.
\end{align}
This is an ordinary differential equation (ODE), and since the main purpose of this section is to develop intuition, we assume that $f$ is twice continuously differentiable and do not worry about showing that~\eqref{eq:GF} is well-posed.
We use the following notation throughout these notes:
\begin{align*}
    x_\star \in \argmin f\,, \qquad f_\star \deq \min f = f(x_\star)\,.
\end{align*}
Generally, we always assume that $f$ admits a minimizer.

The most basic property of~\ref{eq:GF} is that it always decreases the function value.

\begin{lem}[descent property of~\ref{eq:GF}]\label{lem:gf_descent}
    For any $f : \R^d\to\R$, the gradient flow ${(x_t)}_{t\ge 0}$ of $f$ satisfies
    \begin{align*}
        \partial_t f(x_t)
        &= -\norm{\nabla f(x_t)}^2
        \le 0\,.
    \end{align*}
\end{lem}
\begin{proof}
    By the chain rule, $\partial_t f(x_t) = \langle \nabla f(x_t), \dot x_t \rangle = -\norm{\nabla f(x_t)}^2$.
\end{proof}

To obtain quantitative convergence results, we now use the assumption of convexity.
Our first result shows that under strong convexity, the gradient flow \emph{contracts}.

\begin{thm}[contraction of~\ref{eq:GF}]\label{thm:gf_contraction}
    Let $f : \R^d\to\R$ be $\alpha$-convex.
    Let ${(y_t)}_{t\ge 0}$ be another gradient flow for $f$, i.e., $\dot y_t = -\nabla f(y_t)$.
    Then, for all $t\ge 0$,
    \begin{align*}
        \norm{y_t - x_t}
        &\le \exp(-\alpha t)\,\norm{y_0-x_0}\,.
    \end{align*}
\end{thm}
\begin{proof}
    We differentiate the squared distance between the two flows:
    \begin{align*}
        \partial_t(\norm{y_t - x_t}^2)
        &= 2\,\langle y_t - x_t, \dot y_t - \dot x_t \rangle
        = -2\,\langle y_t - x_t, \nabla f(y_t) - \nabla f(x_t)\rangle
        \le -2\alpha\,\norm{y_t - x_t}^2\,,
    \end{align*}
    where the last inequality is~\eqref{eq:str_cvx_2}.
    The proof is concluded by applying Gr\"onwall's lemma (see~\autoref{lem:gronwall}) below.
\end{proof}

The proof above arrives at what is called a \emph{differential inequality}, that is, an inequality which holds between a quantity and its derivative{(s)}.
This is a common strategy for analyzing ODEs/PDEs, and it can be loosely viewed as the continuous-time analogue of induction.
The following standard lemma is useful for handling such inequalities.

\begin{lem}[Gr\"onwall]\label{lem:gronwall}
    Suppose that $u : [0,T]\to\R$ is a continuously differentiable curve that satisfies the differential inequality
    \begin{align*}
        \dot u(t)
        &\le Au(t) + B(t)\,, \qquad t \in [0,T]\,.
    \end{align*}
    Then, it holds that
    \begin{align*}
        u(t)
        &\le u(0)\exp(At) + \int_0^t B(s) \exp(A\,(t-s))\,\D s\,, \qquad t\in [0,T]\,.
    \end{align*}
\end{lem}
\begin{proof}
    The idea is to differentiate $t \mapsto \exp(-At)\,u(t)$:
    \begin{align*}
        \partial_t[\exp(-At)\,u(t)]
        &= \exp(-At)\,\{-Au(t) + \dot u(t)\}
        \le B(t)\exp(-At)\,.
    \end{align*}
    By the fundamental theorem of calculus,
    \begin{align*}
        \exp(-At)\,u(t) - u(0)
        \le \int_0^t B(s) \exp(-As)\,\D s\,.
    \end{align*}
    Rearranging yields the result.
\end{proof}

There are many variants of Gr\"onwall's lemma that can be proven in similar ways, e.g., we can allow time-varying $A$ as well.

Returning to~\autoref{thm:gf_contraction}, we can apply~\autoref{lem:gronwall} with $A = -2\alpha$ and $B = 0$ to conclude that $\norm{y_t - x_t}^2 \le \exp(-2\alpha t)\,\norm{y_0 - x_0}^2$, which proves the theorem.
Note in particular that we can take $y_t = x_\star$ for all $t\ge 0$, so it yields the following statement about convergence to the minimizer: $\norm{x_t - x_\star} \le \exp(-\alpha t)\,\norm{x_0-x_\star}$.

The next result is about convergence in function value, and unlike~\autoref{thm:gf_contraction}, it yields convergence for the case $\alpha = 0$ as well.

\begin{thm}[convergence of~\ref{eq:GF} in function value]\label{thm:gf_fn_value}
    Let $f : \R^d\to\R$ be $\alpha$-convex, $\alpha \ge 0$.
    Then, for all $t\ge 0$,
    \begin{align*}
        f(x_t) - f_\star
        &\le \frac{\alpha}{2\,(\exp(\alpha t) - 1)}\,\norm{x_0 - x_\star}^2\,.
    \end{align*}
    When $\alpha = 0$, the right-hand side should be interpreted as its limiting value as $\alpha \to 0$, namely, $\frac{1}{2t}\,\norm{x_0 - x_\star}^2$.
\end{thm}
\begin{proof}
    We differentiate $t\mapsto \norm{x_t - x_\star}^2$, but this time we apply~\eqref{eq:str_cvx_1}:
    \begin{align*}
        \partial_t(\norm{x_t - x_\star}^2)
        &= -2\,\langle\nabla f(x_t), x_t - x_\star\rangle
        \le -\alpha\,\norm{x_t - x_\star}^2 - 2\,(f(x_t) - f_\star)\,.
    \end{align*}
    Applying Gr\"onwall's lemma (\autoref{lem:gronwall}) with $A = -\alpha$, $B(t) = -2\,(f(x_t) - f_\star)$,
    \begin{align*}
        0 \le \norm{x_t - x_\star}^2
        \le \exp(-\alpha t)\,\norm{x_0 - x_\star}^2 - 2\int_0^t \exp(-\alpha\,(t-s))\,(f(x_s) - f_\star) \, \D s\,.
    \end{align*}
    By the descent property (\autoref{lem:gf_descent}), $f(x_s) \ge f(x_t)$, so that
    \begin{align*}
        \int_0^t \exp(-\alpha\,(t-s))\,(f(x_s) - f_\star) \, \D s
        &\ge (f(x_t) - f_\star) \int_0^t \exp(-\alpha\,(t-s))\,\D s \\[0.25em]
        &= (f(x_t) - f_\star)\, \frac{1-\exp(-\alpha t)}{\alpha}\,.
    \end{align*}
    Rearranging yields the result.
\end{proof}

When $\alpha > 0$,~\autoref{thm:gf_fn_value} shows that $f(x_t) - f_\star = O(\exp(-\alpha t))$.
When $\alpha = 0$, the rate becomes $f(x_t) - f_\star = O(1/t)$.
Actually, the rate in~\autoref{thm:gf_fn_value} is not sharp (see~\autoref{qu:gf_sharp} and~\autoref{qu:gf_sharp_strcvx}).
However, the statement and proof are chosen because they form the basis of our approach in discrete time.

Next, we observe that convexity is not needed for convergence in function value.
Due to the descent property (\autoref{lem:gf_descent}), it suffices to have a lower bound on the norm of the gradient to ensure that we make sufficient progress.
For example, we can impose the following condition.

\begin{defn}
    Let $f : \R^d\to\R$ be continuously differentiable and $\alpha > 0$.
    We say that $f$ satisfies a \textbf{Polyak{--}\L{}ojasiewicz} (\textbf{P\L{}}) \textbf{inequality} with constant $\alpha$ if
    \begin{align}\label{eq:PL}\tag{$\msf{P\L{}}$}
        \norm{\nabla f(x)}^2
        &\ge 2\alpha\,(f(x) - f(x_\star)) \qquad\text{for all}~x\in\R^d\,.
    \end{align}
\end{defn}

The next statement is an immediate corollary of~\autoref{lem:gf_descent},~\eqref{eq:PL}, and Gr\"onwall's lemma (\autoref{lem:gronwall}).

\begin{cor}[convergence of~\ref{eq:GF} under~\ref{eq:PL}]\label{cor:gf_pl}
    Let $f : \R^d\to\R$ satisfy~\eqref{eq:PL} with constant $\alpha > 0$.
    Then, for all $t \ge 0$,
    \begin{align*}
        f(x_t) - f_\star
        &\le (f(x_0) - f_\star) \exp(-2\alpha t)\,.
    \end{align*}
\end{cor}

We present a few key properties of the P\L{} inequality.

\begin{prop}[strong convexity $\Rightarrow$~\ref{eq:PL} $\Rightarrow$ quadratic growth]
    Let $f : \R^d\to\R$ and $\alpha > 0$.
    The following implications hold.
    \begin{enumerate}
        \item If $f$ is $\alpha$-convex, then $f$ satisfies~\eqref{eq:PL} with constant $\alpha$.
        \item If $f$ satisfies~\eqref{eq:PL} with constant $\alpha$, then it satisfies the following \textbf{quadratic growth} property:
            \begin{align}\label{eq:qg}\tag{$\msf{QG}$}
                f(x) - f_\star
                &\ge \frac{\alpha}{2} \inf_{x_\star \in \eu X_\star} \norm{x-x_\star}^2\,, \qquad\text{for all}~x\in\R^d\,,
            \end{align}
            where $\eu X_\star$ denotes the set of minimizers of $f$.
    \end{enumerate}
\end{prop}
\begin{proof}\mbox{}
    \begin{enumerate}
        \item Setting $y = x_\star$ in~\eqref{eq:str_cvx_1}, we obtain
            \begin{align*}
                -(f(x) - f_\star)
                &\ge \langle \nabla f(x), x_\star - x \rangle + \frac{\alpha}{2} \,\norm{x-x_\star}^2 \\
                &\ge -\norm{\nabla f(x)}\, \norm{x_\star - x} + \frac{\alpha}{2} \,\norm{x-x_\star}^2
                \ge - \frac{1}{2\alpha}\,\norm{\nabla f(x)}^2\,,
            \end{align*}
            where the last inequality uses $ab \le \frac{\lambda}{2}\,a^2 + \frac{1}{2\lambda}\,b^2$ for all $\lambda > 0$.
        \item Let ${(x_t)}_{t\ge 0}$ denote the gradient flow for $f$ started at $x_0 = x$.
            For simplicity, we present a proof \emph{assuming} that the gradient flow converges to a point $x_\star$, although this assumption can be avoided (cf.~\cite{KarNutSch16PL}).
            By~\autoref{cor:gf_pl}, we see that $x_\star \in \eu X_\star$.

            We start by observing that
            \begin{align*}
                \partial_t(\norm{x_t - x_0}^2)
                &= -2\,\langle \nabla f(x_t), x_t - x_0\rangle
                \le 2\,\norm{\nabla f(x_t)}\,\norm{x_t - x_0}
            \end{align*}
            and hence
            \begin{align*}
                \partial_t \norm{x_t - x_0}
                &\le \norm{\nabla f(x_t)}\,.
            \end{align*}
            We differentiate the following quantity: $\ms L_t \deq \sqrt{\frac{\alpha}{2}}\,\norm{x_t - x_0} + \sqrt{f(x_t) - f_\star}$.
            \begin{align*}
                \dot{\ms L}_t
                &\le \sqrt{\frac{\alpha}{2}}\,\norm{\nabla f(x_t)} - \frac{\norm{\nabla f(x_t)}^2}{2\sqrt{f(x_t) - f_\star}}
                \le 0\,,
            \end{align*}
            where we applied~\eqref{eq:PL}.
            Since $\ms L_0 = \sqrt{f(x) - f_\star}$ and $\ms L_\infty = \sqrt{\frac{\alpha}{2}}\,\norm{x - x_\star}$, we deduce the result from $\ms L_0 \ge \ms L_\infty$.
    \end{enumerate}
\end{proof}

Hence, strong convexity implies~\eqref{eq:PL}, but is~\eqref{eq:PL} truly weaker than convexity?
Indeed, there are examples.
In particular, the P\L{} condition has been of interest in recent years because it holds for certain overparametrized models (\autoref{qu:PL_example}).

We conclude this section by studying the implication of~\autoref{lem:gf_descent} alone.
The fundamental theorem of calculus shows that
\begin{align*}
    \frac{1}{t} \int_0^t \norm{\nabla f(x_s)}^2\,\D s
    &\le \frac{f(x_0) - f(x_t)}{t}
    \le \frac{f(x_0) - f_\star}{t}\,.
\end{align*}
We therefore arrive at the following simple consequence.

\begin{cor}[convergence of~\ref{eq:GF} in gradient norm]
    For any $f : \R^d\to\R$,
    \begin{align*}
        \min_{s\in [0,t]}{\norm{\nabla f(x_s)}}
        &\le \sqrt{\frac{f(x_0) - f_\star}{t}}\,.
    \end{align*}
\end{cor}

(In contrast, note that if we additionally assume convexity, then~\autoref{qu:gf_sharp} shows that $\norm{\nabla f(x_t)} = O(1/t)$.)

This implies there exists a sequence of times ${\{t_n\}}_{n\in\N} \nearrow \infty$ such that $\norm{\nabla f(x_{t_n})} \to 0$.
(Indeed, $\min_{s\in [n, 2n]}{\norm{\nabla f(x_s)}} = O(1/n^{1/2})$, so we can choose $t_n \in [n, 2n]$.)
However, the gradient flow may not converge.
Famously, it is a result of~\cite{Loj1963Top} that for \emph{real analytic} $f$, if the gradient flow remains bounded, then it does converge, and hence necessarily to a stationary point.
Of course, such a stationary point may not be a global minimizer.

The idea of subsequent sections is to replicate the preceding analysis in discrete time.

\subsection*{Bibliographical notes}

My understanding of~\autoref{thm:gf_fn_value},~\autoref{qu:gf_sharp}, and~\autoref{qu:gf_sharp_strcvx} is based on extensive discussions with Jason M.\ Altschuler, Adil Salim, Andre Wibisono, and Ashia Wilson.
The proof in~\autoref{qu:gf_sharp} is taken from~\cite{OttVil01Comment}, and the extension in~\autoref{qu:gf_sharp_strcvx} to $\alpha > 0$ is recorded in~\cite[\S F]{LiaMitWib24IndepSamples}.
Both of these references pertain to the Langevin diffusion, but underneath the hood they make use of principles from optimization; see~\cite{Chewi26Book} for an introduction to this perspective.

The P\L{} inequality is attributed to~\cite{Loj1963Top, Pol1963Gradient} and it was popularized in~\cite{KarNutSch16PL}.
The proof that~\eqref{eq:PL} implies the quadratic growth inequality goes back at least to the celebrated work of~\cite{OttVil00LSI}.

\subsection*{Exercises}

\begin{question}\label{qu:gf_sharp}
    Let $f$ be convex.
    Show that the following quantity is decreasing, $\dot{\ms L_t} \le 0$:
    \begin{align*}
        \ms L_t
        &\deq t^2\,\norm{\nabla f(x_t)}^2 + 2t\,(f(x_t) - f_\star) + \norm{x_t - x_\star}^2\,.
    \end{align*}
    Deduce the following gradient bound:
    \begin{align*}
        \norm{\nabla f(x_t)}^2
        &\le \frac{1}{t^2}\, \norm{x_0 - x_\star}^2\,.
    \end{align*}
    Moreover, use~\eqref{eq:str_cvx_1} to argue that $2t\,(f(x_t) - f_\star) \le t^2\,\norm{\nabla f(x_t)}^2 + \norm{x_t - x_\star}^2$, hence
    \begin{align}\label{eq:gf_sharp}
        f(x_t) - f_\star
        &\le \frac{1}{4t}\,\norm{x_0 - x_\star}^2\,.
    \end{align}
    Note that this improves upon~\autoref{thm:gf_fn_value} by a factor of $2$.
    Furthermore, show that~\eqref{eq:gf_sharp} is sharp, as follows: for any $R, t > 0$, let $f : x\mapsto \frac{R}{2t}\max\{0, x\}$, $x_0 = R$, and show that~\eqref{eq:gf_sharp} holds with equality.
\end{question}

\begin{question}\label{qu:gf_sharp_strcvx}
    Extend~\autoref{qu:gf_sharp} to the case $\alpha > 0$.
    Toward this end, consider
    \begin{align*}
        \ms L_t
        &\deq A_t \,\norm{\nabla f(x_t)}^2 + 2B_t\,(f(x_t) - f_\star) + \norm{x_t - x_\star}^2\,.
    \end{align*}
    Choose $A_t$, $B_t$ carefully to ensure that $\dot{\ms L_t} \le -\alpha \ms L_t$, and thereby deduce the following sharp bounds:
    \begin{align*}
        \norm{\nabla f(x_t)}^2
        &\le \frac{\alpha^2\,\norm{x_0-x_\star}^2}{\exp(2\alpha t)\,{(1-\exp(-\alpha t))}^2}\,,
        \qquad f(x_t) - f_\star\le \frac{\alpha\,\norm{x_0-x_\star}^2}{2\,(\exp(2\alpha t) - 1)}\,.
    \end{align*}
\end{question}

\begin{question}\label{qu:PL_example}
    Let $f : \R^n \to \R$ be $\alpha$-convex with $\alpha > 0$, and let $g : \R^d\to \R^n$ with $d \ge n$.
    Assume that $g$ is surjective and that for all $x\in \R^d$, if $\nabla g(x)$ denotes the Jacobian at $x$ (interpreted as a $d\times n$ matrix), then ${\nabla g(x)}^\T\,\nabla g(x) \succeq \sigma I_n$.
    Show that the composition $f\circ g$ satisfies~\eqref{eq:PL} with constant $\alpha\sigma$.
    Note that for $d > n$, there are typically multiple minimizers of $f\circ g$.
\end{question}

\section{Gradient descent: smooth case}\label{sec:gd}

In this section, we study the \textbf{gradient descent} algorithm:
\begin{align}\label{eq:GD}\tag{$\msf{GD}$}
    x_{n+1}
    &\deq x_n - h\,\nabla f(x_n)\,.
\end{align}
From the perspective of numerical analysis, this is the \emph{Euler} or \emph{forward} discretization of~\eqref{eq:GF}.
Our aim is to show that if $f$ is smooth, and the step size is sufficiently small (as a function of the smoothness), then the conclusions for~\eqref{eq:GF} transfer to~\eqref{eq:GD}.
Throughout this section, we assume that $f$ is twice continuously differentiable and $\beta$-smooth.

Some of the results in this section pertain to a single step of~\eqref{eq:GD}, so we use the following notation:
\begin{align*}
    x^+
    &\deq x - h\,\nabla f(x)\,.
\end{align*}

The first step is to establish the descent property.

\begin{lem}[descent lemma]\label{lem:descent}
    For any $\beta$-smooth $f : \R^d\to\R$, if $h \le 1/\beta$, then
    \begin{align*}
        f(x^+) - f(x)
        &\le - \frac{h}{2}\,\norm{\nabla f(x)}^2\,.
    \end{align*}
\end{lem}
\begin{proof}
    By the smoothness inequality~\eqref{eq:smooth},
    \begin{align*}
        f(x^+)
        &\le f(x) + \langle \nabla f(x), x^+ - x \rangle + \frac{\beta}{2}\,\norm{x^+ - x}^2
        = f(x) - h\,\norm{\nabla f(x)}^2 + \frac{\beta h^2}{2}\,\norm{\nabla f(x)}^2\,.
    \end{align*}
    If $h \le 1/\beta$, then $-h\,(1-\beta h/2) \le -h/2$.
\end{proof}

It is natural to state the subsequent results in terms of the following parameter.

\begin{defn}\label{defn:cond_numb}
    Let $f$ be $\alpha$-convex and $\beta$-smooth.
    Then, the \textbf{condition number} of $f$ is defined to be the ratio $\kappa \deq \beta/\alpha \ge 1$.
\end{defn}

When $f$ is quadratic, $f(x) = \frac{1}{2}\,\langle x, A\,x\rangle$ with $A$ symmetric, then $\alpha$, $\beta$ correspond to the minimum and maximum eigenvalues of $A$ respectively, and the ratio $\beta/\alpha$ is known in numerical linear algebra as the condition number of the matrix $A$.
Thus,~\autoref{defn:cond_numb} provides a natural generalization of this notion.
With this definition in hand, we now arrive at our first convergence result for~\eqref{eq:GD}.

\begin{thm}[contraction of~\ref{eq:GD}]\label{thm:contraction_gd}
    Let $f$ be $\alpha$-convex and $\beta$-smooth.
    For all $x,y\in\R^d$ and step size $h \le 1/\beta$,
    \begin{align*}
        \norm{y^+ - x^+}
        &\le {(1-\alpha h)}^{1/2}\,\norm{y-x}\,.
    \end{align*}
\end{thm}
\begin{proof}
    Expanding the square,
    \begin{align*}
        \norm{y^+ - x^+}^2
        &= \norm{y-x}^2 - 2h\,\langle y-x, \nabla f(y) - \nabla f(x) \rangle + h^2\,\norm{\nabla f(y) - \nabla f(x)}^2\,.
    \end{align*}
    By~\eqref{eq:coercivity} in~\autoref{qu:smooth_implies_lipgrad} below, for $h \le 1/\beta$ and from~\eqref{eq:str_cvx_2} we have
    \begin{align*}
        \norm{y^+ - x^+}^2
        &\le \norm{y-x}^2 - h\,\langle \nabla y-x, \nabla f(y) - \nabla f(x) \rangle
        \le (1-\alpha h)\,\norm{y-x}^2\,. \qedhere
    \end{align*}
\end{proof}

In particular, if we take $y = x_\star$, $h=1/\beta$, and iterate, it yields
\begin{align*}
    \norm{x_N - x_\star}
    &\le \bigl(1 - \frac{1}{\kappa}\bigr){\bigsp}^{N/2}\,\norm{x_0 - x_\star}
    \le \exp\bigl( - \frac{N}{2\kappa}\bigr)\,\norm{x_0 - x_\star}\,.
\end{align*}
Thus, to obtain $\norm{x_N - x_\star} \le \varepsilon$, it suffices to take $N \ge 2\kappa\log(\norm{x_0 - x_\star}/\varepsilon)$.

The essence of these proofs is that the first-order term (scaling as $h$) replicates the continuous-time calculation, and we must apply smoothness in an appropriate way to control the second-order term (scaling as $h^2$).
In the above proof, note that if we na\"{\i}vely use Lipschitzness of the gradient~\eqref{eq:lip_grad} to control the second-order term, it leads to the suboptimal choice of step size $h = 1/(\beta\kappa)$, and a contraction factor of $(1-1/\kappa^2){}^{1/2}$.
To obtain $\norm{x_N - x_\star} \le \varepsilon$, we would then have the estimate $N \ge 2\kappa^2 \log(\norm{x_0 - x_\star}/\varepsilon)$, which is substantially worse.
In conclusion, a bit of finesse is necessary.
(In fact,~\autoref{thm:contraction_gd} can also be improved, and the sharp rate is derived in~\autoref{qu:gd_sharp_contraction}.)

Next, we turn toward the analogue of~\autoref{thm:gf_fn_value}.

\begin{thm}[convergence of~\ref{eq:GD} in function value]\label{thm:gd_fn_value}
    Let $f$ be $\alpha$-convex and $\beta$-smooth.
    For any step size $h \le 1/\beta$,
    \begin{align}\label{eq:gd_recursion}
        \norm{x^+ - x_\star}^2
        &\le (1-\alpha h)\,\norm{x-x_\star}^2 - 2h\,(f(x^+) - f_\star)\,.
    \end{align}
    Therefore,
    \begin{align}\label{eq:gd_fn_value}
        f(x_N) - f_\star
        &\le \frac{\alpha}{2\,\{{(1-\alpha h)}^{-N} - 1\}}\,\norm{x_0 - x_\star}^2\,.
    \end{align}
    When $\alpha = 0$, the right-hand side should be interpreted as its limiting value as $\alpha \to 0$, namely, $\frac{1}{2Nh}\,\norm{x_0 - x_\star}^2$.
\end{thm}
\begin{proof}
    Expanding the square and applying convexity via~\eqref{eq:str_cvx_1},
    \begin{align*}
        \norm{x^+ - x_\star}^2
        &= \norm{x-x_\star}^2 - 2h\,\langle \nabla f(x), x - x_\star \rangle + h^2\,\norm{\nabla f(x)}^2 \\
        &\le (1-\alpha h)\,\norm{x-x_\star}^2 - 2h\,(f(x) - f_\star) + h^2\,\norm{\nabla f(x)}^2\,.
    \end{align*}
    For $h\le 1/\beta$, the descent lemma (\autoref{lem:descent}) now implies~\eqref{eq:gd_recursion}.

    The proof of~\eqref{eq:gd_fn_value}, based on iterating the recursive inequality~\eqref{eq:gd_recursion}, is justified after~\autoref{lem:discrete_gronwall} below.
\end{proof}

We remark for later use that the proof of~\eqref{eq:gd_recursion} goes through even if we replace $x_\star$ with any other point $z\in\R^d$, i.e.,
\begin{align}\label{eq:evi}
    \norm{x^+ - z}^2
    &\le (1-\alpha h)\,\norm{x-z}^2 - 2h\,(f(x^+) - f(z))\,, \qquad\text{for all}~z\in\R^d\,.
\end{align}

Iterating~\eqref{eq:gd_recursion} is a matter of unrolling the recursion, but in order to maintain the analogy with continuous time, we refer to the lemma below as ``discrete Gr\"onwall''.

\begin{lem}[discrete Gr\"onwall]\label{lem:discrete_gronwall}
    Suppose that for some $A > 0$,
    \begin{align*}
        u_{n+1}
        &\le Au_n + B_n \qquad\text{for}~n = 0,1,\dotsc,N-1\,.
    \end{align*}
    Then,
    \begin{align*}
        u_N
        &\le A^N u_0 + \sum_{n=1}^N A^{N-n} B_{n-1}\,.
    \end{align*}
\end{lem}
\begin{proof}
    We multiply the given inequality by $A^{-(n+1)}$ to form a telescoping sum:
    \begin{align*}
        A^{-N}\,u_N - u_0
        &= \sum_{n=0}^{N-1} A^{-(n+1)} \,(u_{n+1} - Au_n)
        \le \sum_{n=0}^{N-1} A^{-(n+1)}\,B_n\,.
    \end{align*}
    Rearrange to obtain the result.
\end{proof}

To complete the proof of~\autoref{thm:gd_fn_value}, we apply~\autoref{lem:discrete_gronwall} with $u_n = \norm{x_n - x_\star}^2$, $A = 1-\alpha h$, and $B_n = -2h\,(f(x_{n+1}) - f_\star)$, yielding
\begin{align*}
    2h\sum_{n=1}^N {(1-\alpha h)}^{N-n}\, (f(x_n) - f_\star)
    &\le {(1-\alpha h)}^N\,\norm{x_0 - x_\star}^2\,.
\end{align*}
For $h\le 1/\beta$, the descent lemma (\autoref{lem:descent}) implies $f(x_n) - f_\star \ge f(x_N) - f_\star$, so
\begin{align*}
    f(x_N) - f_\star
    \le \frac{\norm{x_0 - x_\star}^2}{2h \sum_{n=1}^N {(1-\alpha h)}^{-n}}
    = \frac{\alpha\,\norm{x_0 - x_\star}^2}{2\,\{{(1-\alpha h)}^{-N} - 1\}}\,.
\end{align*}
In particular, let us set $h = 1/\beta$.
For $\alpha > 0$ it yields
\begin{align*}
    f(x_N) - f_\star
    &\le \frac{\alpha\,\norm{x_0 - x_\star}^2}{2\,\{{(1-1/\kappa)}^{-N} - 1\}}
\end{align*}
and for $\alpha = 0$, it yields
\begin{align*}
    f(x_N) - f_\star
    &\le \frac{\beta\,\norm{x_0 - x_\star}^2}{2N}\,.
\end{align*}

The proof of convergence under~\eqref{eq:PL} is strikingly easy.

\begin{thm}[convergence of~\ref{eq:GD} under~\ref{eq:PL}]
    Let $f$ be $\beta$-smooth and satisfy~\eqref{eq:PL} with constant $\alpha > 0$.
    Then, for all $h \le 1/\beta$,
    \begin{align*}
        f(x_N) - f_\star
        &\le {(1-\alpha h)}^N\,(f(x_0) - f_\star)\,.
    \end{align*}
\end{thm}
\begin{proof}
    By the descent lemma (\autoref{lem:descent}) and~\eqref{eq:PL},
    \begin{align*}
        f(x^+) - f_\star
        = f(x) - f_\star + f(x^+) - f(x)
        &\le f(x) - f_\star - \frac{h}{2}\,\norm{\nabla f(x)}^2 \\
        &\le (1-\alpha h)\,(f(x) - f_\star)\,. \qedhere
    \end{align*}
\end{proof}

Finally, we present the result for obtaining a stationary point.

\begin{thm}
    Let $f$ be $\beta$-smooth and $h\le 1/\beta$.
    Then,
    \begin{align*}
        \min_{n=0,1,\dotsc,N-1}{\norm{\nabla f(x_n)}}
        &\le \sqrt{\frac{2\,(f(x_0) - f_\star)}{Nh}}\,.
    \end{align*}
\end{thm}
\begin{proof}
    Telescope the descent lemma (\autoref{lem:descent}):
    \begin{align*}
        \frac{h}{2N} \sum_{n=0}^{N-1} \norm{\nabla f(x_n)}^2
        &\le \frac{1}{N} \sum_{n=0}^{N-1}\,(f(x_n) - f(x_{n+1}))
        \le \frac{f(x_0) - f_\star}{N}\,. \qedhere
    \end{align*}
\end{proof}

We summarize the results for~\ref{eq:GD} in~\autoref{tab:gd}.

\begin{table}[h]
    \centering
    \begin{tabular}{ccc}
        \textbf{Assumptions} & \textbf{Criterion} & \textbf{Iterations} \\
        $\alpha$\text{-convex}, $\beta$\text{-smooth} & $\norm{x_N - x_\star} \le \varepsilon$ & $O(\kappa \log(R/\varepsilon))$ \\
        $\alpha$\text{-convex}, $\beta$\text{-smooth} & $f(x_N) - f_\star \le \varepsilon$ & $O(\kappa \log(\alpha R^2/\varepsilon))$ \\
        \text{convex}, $\beta$\text{-smooth} & $f(x_N) - f_\star \le \varepsilon$ & $O(\beta R^2/\varepsilon)$ \\
        $\alpha$\text{-\eqref{eq:PL}}, $\beta$\text{-smooth} & $f(x_N) - f_\star \le \varepsilon$ & $O(\kappa \log(\Delta_0/\varepsilon))$ \\
        $\beta$\text{-smooth} & $\displaystyle\min_{n=0,1,\dotsc,N-1}{\norm{\nabla f(x_n)}} \le \varepsilon$ & $O(\beta \Delta_0/\varepsilon^2)$
    \end{tabular}
    \caption{Rates for~\ref{eq:GD} with step size $1/\beta$. Here, $R \deq \norm{x_0 - x_\star}$ and $\Delta_0 \deq f(x_0) - f_\star$.}\label{tab:gd}
\end{table}

\begin{ex}[logistic regression revisited]\label{ex:logistic_gd}
    For fun, let us revisit logistic regression (\autoref{ex:logistic_regression}) from a statistical lens.
    For concreteness, we consider Gaussian design, $X_i \simiid \normal(0, I)$, and assume that the data is generated from the model with a true parameter $\theta^\star$.
    Let $\widehat{\eu L}$ denote the MLE objective, let $\eu L \deq \E\widehat{\eu L}$ denote the population risk, and let $R \deq \norm{\theta^\star} \ge 1$.
    The state-of-the-art result~\cite{ChaLerMou24Logistic} shows that if $n \gtrsim Rd$ for a sufficiently large implied constant, $\widehat \theta_{\rm MLE}$ exists with probability $\ge 1-\exp(-d)$ and satisfies the optimal risk bound $\eu L(\widehat \theta_{\rm MLE}) - \eu L(\theta^\star) \lesssim d/n$.

    In practice, we cannot compute $\widehat \theta_{\rm MLE}$ exactly, so we use~\ref{eq:GD}.
    From~\cite{ChaLerMou24Logistic}, any estimator $\widehat\theta$ satisfying $\widehat{\eu L}(\widehat \theta) - \widehat{\eu L}(\widehat \theta_{\rm MLE}) \lesssim d/n$ satisfies the same statistical risk bound as $\widehat \theta_{\rm MLE}$, up to a universal constant.
    We take $\widehat \theta = \widehat \theta_{\rm GD}$ to be the output of~\ref{eq:GD} after $N$ steps, and check how large $N$ must be in order for this to hold.
    As justified in~\autoref{qu:logistic_gd}, we can expect an iteration complexity of $N \asymp R^2 n/d$.
\end{ex}

\subsection*{Bibliographical notes}

My understanding of~\autoref{thm:gd_fn_value} is again based on extensive discussions with Jason M.\ Altschuler, Adil Salim, Andre Wibisono, and Ashia Wilson.

\subsection*{Exercises}

\begin{question}\label{qu:smooth_implies_lipgrad}
    Let $f : \R^d\to\R$ be convex and $\beta$-smooth.
    Apply~\autoref{lem:descent} to the function $y\mapsto f(y) - \langle \nabla f(x), y \rangle$ and observe that this function is minimized at $x$ in order to prove
    \begin{align}\label{eq:pre_coercivity}
        f(y)
        &\ge f(x) + \langle \nabla f(x), y-x\rangle + \frac{1}{2\beta}\,\norm{\nabla f(y) - \nabla f(x)}^2\,.
    \end{align}
    From this, deduce that
    \begin{align}\label{eq:coercivity}
        \norm{\nabla f(y) - \nabla f(x)}^2
        &\le \beta\,\langle \nabla f(y) - \nabla f(x), y-x\rangle\,.
    \end{align}
    Finally, use the Cauchy{--}Schwarz inequality to show that $\nabla f$ is $\beta$-Lipschitz, i.e., that~\eqref{eq:lip_grad} holds.
    Note that this proof that convexity and $\beta$-smoothness together imply~\eqref{eq:lip_grad} does not require $f$ to be twice differentiable.
\end{question}

\begin{question}\label{qu:gd_sharp_contraction}
    Let $f$ be $\alpha$-convex and $\beta$-smooth.
    Let $T \deq {\id} - h\,\nabla f$ denote the one-step~\ref{eq:GD} mapping.
    By the fundamental theorem of calculus,
    \begin{align*}
        \norm{y^+ - x^+}
        = \norm{T(y) - T(x)}
        &= \Bigl\lVert \int_0^1 \nabla T((1-t)\,x + t\,y)\,(y-x)\, \D t \Bigr\rVert \\[0.25em]
        &\le \Bigl(\int_0^1 \norm{\nabla T((1-t)\,x + t y)}_{\rm op}\,\D t\Bigr)\,\norm{y-x}\,.
    \end{align*}
    For any $z \in \R^d$, bound the eigenvalues of $\nabla T(z)$ and show that the choice of step size $h$ which minimizes the bound on $\norm{\nabla T(z)}_{\rm op}$ is $h = 2/(\alpha + \beta)$.
    Deduce the sharp rate
    \begin{align*}
        \norm{y^+ - x^+}
        &\le \frac{\kappa-1}{\kappa+1}\,\norm{y-x}\,.
    \end{align*}
    Note that for large $\kappa$, the contraction factor is approximately $\exp(-2/\kappa)$, so this improves upon the iteration complexity implied by~\autoref{thm:contraction_gd} by a factor of nearly $4$.
\end{question}

\begin{question}
    Let $f(x) \deq \frac{1}{2}\,\langle x, A\,x\rangle$ where $A \succ 0$.
    Write out the iterates of GD explicitly, and check how sharp the results of this section are.
\end{question}

\begin{question}\label{qu:logistic_gd}
    What does~\autoref{thm:gd_fn_value} imply for logistic regression (\autoref{ex:logistic_regression})?
    In the setting of~\autoref{ex:logistic_gd}, use the fact that $\lambda_{\max}(\frac{1}{n} \sum_{i=1}^n X_i X_i^\T) \lesssim 1$ with high probability\footnote{This is a standard fact about the Wishart distribution; see, e.g.,~\cite[Theorem 4.4.5]{Ver18HighDimProb}.} to justify the claimed $R^2 n/d$ iteration complexity.
\end{question}

\begin{question}\label{qu:softmax_smoothing}
    Suppose that we wish to minimize $f : \R^d\to\R$ given by
    \begin{align*}
        f(x) \deq \max_{i\in [m]}{\{\langle a_i, x\rangle - b_i\}} + \frac{\lambda}{2}\,\norm x^2\,.
    \end{align*}
    Here, $\lambda > 0$.
    Now, consider smoothing the objective: for a parameter $\beta > 0$, let
    \begin{align*}
        f_\beta(x) \deq \frac{1}{\beta} \log \sum_{i=1}^m \exp(\beta\,\{\langle a_i, x \rangle - b_i\}) + \frac{\lambda}{2}\,\norm x^2\,.
    \end{align*}
    \begin{enumerate}
        \item Show that $f \le f_\beta \le f + \beta^{-1} \log m$.
        \item Show that $f_\beta$ is strongly convex and smooth, and compute the parameters in terms of $\beta$, $\lambda$, and $A$, where $A \in \R^{m\times d}$ has rows $a_i^\T$, $i\in [m]$.
        \item Suppose that $R_\beta \ge \norm{x_0 - x_{\beta,\star}}$ where $x_{\beta,\star} \in \argmin f_\beta$, and run~\ref{eq:GD} on $f_\beta$ with initialization $x_0$.
            Find choices for $\beta$ and the number of steps $N$ such that~\ref{eq:GD} obtains an $\varepsilon$-approximate minimizer for the original function $f$.
    \end{enumerate}
\end{question}


\section{Lower bounds for smooth optimization}

The goal of this section is to establish lower complexity bounds for convex smooth optimization.
Refer to \S\ref{ssec:overview} for a conceptual first discussion of the oracle model.

Before doing so, we present some reductions between the convex and strongly convex settings which save us some effort.

\subsection{Reductions between the convex and strongly convex settings}\label{ssec:reductions}

For brevity, let us say that an algorithm \emph{successfully optimizes} a function class $\ms F$ in $\phi(\ms F, R, \varepsilon)$ iterations if, given any $f \in \ms F$ and $x_0\in\R^d$ with $\norm{x_0 - x_\star} \le R$, it outputs $x$ with $f(x) - f_\star \le \varepsilon$ using no more than $\phi(\ms F, R, \varepsilon)$ queries to a first-order oracle for $f$.

\begin{lem}\label{lem:weakcvx_to_strcvx}
    Assume there is an algorithm which successfully optimizes the class of convex and $\beta$-smooth functions in $\phi(\beta R^2/\varepsilon)$ iterations.

    Then, there is an explicit algorithm which successfully optimizes the class of $\alpha$-convex and $\beta$-smooth functions in $O(\phi(8\kappa) \log(\alpha R^2/\varepsilon))$ iterations.
\end{lem}
\begin{proof}
    For any $\bar x\in\R^d$ and $\varepsilon > 0$, let $\eu A(\bar x, \bar\varepsilon)$ denote the output of the given algorithm, starting from $\bar x$ and with error tolerance $\bar \varepsilon$.
    By assumption, it outputs $\hat x$ with $f(\hat x) - f_\star \le \bar\varepsilon$ in $\phi(\beta \bar R^2/\bar \varepsilon)$ iterations, where $\bar R \deq \norm{\bar x - x_\star}$.

    Let $f$ be $\alpha$-strongly convex and $\beta$-smooth, and let $R \deq \norm{x_0 - x_\star}$.
    We define a sequence of points as follows.
    Let $x_1 \deq \eu A(x_0, \varepsilon_1)$.
    By~\eqref{eq:qg}, we have
    \begin{align*}
        \frac{\alpha}{2}\,\norm{x_1 - x_\star}^2 \le f(x_1) - f_\star \le \varepsilon_1\,.
    \end{align*}
    Set $\varepsilon_1 = \alpha R^2/8$, so that
    \begin{align}\label{eq:geom_decr}
        \norm{x_1-x_\star} \le R_1 \deq R/2 = \norm{x_0-x_\star}/2\,.
    \end{align}
    For $\kappa \deq \beta/\alpha$, this requires $\phi(8\kappa)$ iterations.

    In general, at round $k$, let $x_{k+1} \deq \eu A(x_k, \alpha R_k^2/8)$, where $R_k \deq R/2^k$.
    By the same reasoning as~\eqref{eq:geom_decr}, each round requires $\phi(8\kappa)$ iterations and yields a point $x_{k+1}$ with $\norm{x_{k+1} - x_\star} \le R_{k+1}$.
    Thus, if repeat this procedure for $O(\log(\alpha R^2/\varepsilon))$ rounds, we can reach a point $\tilde x$ satisfying $\tilde R \deq \norm{\tilde x-x_\star} \le \sqrt{\varepsilon/\alpha}$.

    Finally, let $x \deq \eu A(\tilde x, \varepsilon)$, so that $f(x) - f_\star \le \varepsilon$.
    The complexity of this final step is $\phi(\beta \tilde R^2/\varepsilon) = \phi(\kappa)$.
\end{proof}

For example, if we combine the $\alpha=0$ case of~\autoref{thm:gd_fn_value} with~\autoref{lem:weakcvx_to_strcvx}, taking $\phi(x) = O(x)$, we recover the $\alpha > 0$ case of~\autoref{thm:gd_fn_value}, up to constants.

\begin{lem}\label{lem:strcvx_to_weakcvx}
    Assume there is an algorithm which successfully optimizes the class of $\alpha$-convex and $\beta$-smooth functions in $\phi(\kappa) \log(\alpha R^2/\varepsilon)$ iterations.

    Then, there is an explicit algorithm which successfully optimizes the class of convex and $\beta$-smooth functions in $O(\phi(2\beta R^2/\varepsilon))$ iterations.
\end{lem}
\begin{proof}
    Let $f$ be convex and $\beta$-smooth.
    We apply the given algorithm to the regularized function $f_\delta \deq f + \frac{\delta}{2} \,\norm{\cdot - x_0}^2$, obtaining a point $x$ such that $f_\delta(x) \le \min f_\delta + \varepsilon/2$.
    If $x_{\delta,\star}$ denotes the minimizer of $f_\delta$, then
    \begin{align*}
        f(x)
        \le f_\delta(x)
        &\le f_\delta(x_{\delta,\star}) + \frac{\varepsilon}{2}
        \le f_\delta(x_\star) + \frac{\varepsilon}{2}
        = f_\star + \frac{\delta}{2}\,\norm{x_0-x_\star}^2 + \frac{\varepsilon}{2}\,.
    \end{align*}
    We now set $\delta = \varepsilon/R^2$, so that $f(x) - f_\star \le \varepsilon$.

    It remains to estimate the complexity.
    We first note that $f_\delta(x_{\delta,\star}) \le f_\delta(x_\star)$ implies $\norm{x_0 - x_{\star,\delta}} \le \norm{x_0 - x_\star}$, so the initial distance to the minimizer of $f_\delta$ is also bounded by $R$.
    We can assume that $\varepsilon \le \beta R^2$ (or else the minimization problem is trivial).
    Then, the smoothness of $f_\delta$ is bounded by $\beta + \delta \le 2\beta$, and the condition number of $f_\delta$ is bounded by $2\beta R^2/\varepsilon$.
    Substitute these quantities into the complexity of the given algorithm.
\end{proof}

Thus, the $\alpha > 0$ case of~\autoref{thm:gd_fn_value} and~\autoref{lem:strcvx_to_weakcvx} recover the $\alpha = 0$ case of~\autoref{thm:gd_fn_value} up to constants.

Taken together,~\autoref{lem:weakcvx_to_strcvx} and~\autoref{lem:strcvx_to_weakcvx} show that the $0$-convex and strongly convex settings are essentially equivalent to each other, in that an optimal method for one class yields an optimal method for the other class.
Thus, we now aim to address the following question: what is the smallest possible $\phi(\cdot)$?

\subsection{Lower bounds}

According to the discussion in \S\ref{ssec:overview}, establishing a lower complexity bound requires showing that \emph{any} algorithm which interacts with the first-order oracle using at most a prescribed number of queries cannot have performance better than the lower bound.
Actually, although this is possible (see~\cite{NemYud1983Complexity}), it is not especially easy.
It was shown by Nesterov in an earlier edition of~\cite{Nes18CvxOpt} that by imposing natural restrictions on the class of algorithms under consideration, it is possible to establish the lower bounds in a more transparent way.
Accordingly, his approach has become standard in the field, and it is the approach we adopt here as well.
It does, however, have the drawback of not applying to general query algorithms; for example, it does not apply against randomized algorithms.

The class of algorithms we consider is the following one.

\begin{defn}\label{def:gradient_span}
    An algorithm is called a \textbf{gradient span} algorithm if it deterministically generates a sequence of points ${\{x_n\}}_{n\in\N}$ such that for all $n\in\N$,
    \begin{align*}
        x_{n+1}
        &\in x_0 + \spn\{\nabla f(x_0),\dotsc,\nabla f(x_n)\}\,.
    \end{align*}
\end{defn}

For example,~\ref{eq:GD} is a gradient span algorithm.
On the basis of this assumption, we now establish the following result; recall the asymptotic notation $\gtrsim$, which only hides a universal constant.

\begin{thm}[lower bound for convex, smooth minimization]\label{thm:cvx_smooth_lb}
    For any $1 \le N \le \frac{d-1}{2}$, $\beta > 0$, and $x_0\in\R^d$, there exists a convex and $\beta$-smooth function $f : \R^d\to\R$ such that for any gradient span algorithm,
    \begin{align*}
        f(x_N) - f_\star
        &\gtrsim \frac{\beta\,\norm{x_0 - x_\star}^2}{N^2}\,.
    \end{align*}
    In other words, in order to obtain $f(x_N) - f_\star \le \varepsilon$, the number of iterations must satisfy
    \begin{align*}
        N
        &\gtrsim \sqrt{\frac{\beta \,\norm{x_0-x_\star}^2}{\varepsilon}}\,.
    \end{align*}
\end{thm}

Before proving this result, we observe that by applying~\autoref{lem:strcvx_to_weakcvx} with $\phi(x) \asymp \sqrt x$, it yields the following corollary.

\begin{thm}[lower bound for strongly convex, smooth minimization]\label{thm:strcvx_smooth_lb}
    For any $0 < \alpha < \beta$, any $\varepsilon > 0$, any $d$ sufficiently large, and any $x_0\in\R^d$, there exists an $\alpha$-convex and $\beta$-smooth function $f : \R^d\to\R$ such that for any gradient span algorithm, in order to obtain $f(x_N) - f_\star \le \varepsilon$, the number of iterations must satisfy
    \begin{align*}
        N
        &\gtrsim \sqrt\kappa\log \frac{\alpha\,\norm{x_0 - x_\star}^2}{\varepsilon}\,.
    \end{align*}
\end{thm}

\begin{proof}[Proof of~\autoref{thm:cvx_smooth_lb}]
    By translating the problem, we may assume $x_0 = 0$.
    The construction is based on the following function:
    \begin{align*}
        f_n : \R^d\to\R\,, \qquad
        f_n(x)
        &\deq \frac{\beta}{4}\,\Bigl\{ \frac{1}{2}\,\Bigl( {x[1]}^2 + \sum_{k=1}^{n-1} {(x[k] - x[k+1])}^2 + {x[n]}^2\Bigr) - x[1]\Bigr\}\,.
    \end{align*}
    For any $v\in\R^d$,
    \begin{align*}
        \langle v,\nabla^2 f_n(x)\,v\rangle
        &= \frac{\beta}{4} \,\Bigl({v[1]}^2 + \sum_{k=1}^{n-1} {(v[k] - v[k+1])}^2 + {v[n]}^2\Bigr)
        \le \beta\,\norm v^2\,,
    \end{align*}
    so each $f_n$ is convex and $\beta$-smooth.

    We prove by induction that when we apply a gradient span algorithm to $f_d$, the $n$-th iterate $x_n$ belongs to the subspace
    \begin{align*}
        \eu V_n
        &\deq \{x\in\R^d : x[k] = 0~\text{for all}~k=n+1,\dotsc,d\}\,.
    \end{align*}
    Clearly, $x_0 \in \eu V_0$.
    Inductively, suppose that $x_k \in \eu V_k$ for all $k \le n$.
    Then,
    \begin{align*}
        \nabla f_d(x_k)
        &= \frac{\beta}{4}\, \bigl(x_k[1]\,e_1 + \sum_{j=1}^k (x_k[j] - x_k[j+1])\,(e_j - e_{j+1})\bigr) - \frac{\beta}{4}\,e_1
        \in \eu V_{k+1}\,,
    \end{align*}
    hence
    \begin{align*}
        x_{n+1} \in \spn\{\nabla f_d(x_0),\dotsc,\nabla f_d(x_n)\} \subseteq \eu V_{n+1}\,.
    \end{align*}
    This completes the induction.
    Also, since $f_N = f_d$ on $\eu V_N$, it follows that
    \begin{align*}
        f_d(x_N)
        = f_N(x_N)
        \ge {(f_N)}_\star\,.
    \end{align*}

    The next step is to estimate ${(f_n)}_\star \deq \min f_n$ for all $n$.
    By setting the gradient to zero, $\nabla f_n(x_{n,\star}) = 0$, we obtain the following system of equations:
    \begin{align*}
        2x_{n,\star}[1] - x_{n,\star}[2]
        &= 1\,, \\
        x_{n,\star}[k-1] - 2x_{n,\star}[k] + x_{n,\star}[k+1]
        &= 0\,, \qquad\text{for}~k=2,\dotsc,n-1\,, \\
        -x_{n,\star}[n-1] + 2x_{n,\star}[n]
        &= 0\,.
    \end{align*}
    The solution is $x_{n,\star}[k] = 1-\frac{k}{n+1}$ for all $k \in [n]$.
    Writing $f_n(x) = \frac{\beta}{4} \,\{ \frac{1}{2}\,\langle x, A_n\,x\rangle - \langle e_1, x \rangle\}$, the system above reads $A_n x_{n,\star} = e_1$, hence
    \begin{align*}
        {(f_n)}_\star
        &= f_n(x_{n,\star})
        = - \frac{\beta}{8}\,\langle e_1, x_{n,\star}\rangle
        = - \frac{\beta}{8}\,\bigl(1 - \frac{1}{n+1}\bigr)\,.
    \end{align*}
    Moreover, $\norm{x_0 - x_{n,\star}}^2 = \norm{x_{n,\star}}^2 \le n$.
    Finally, it yields
    \begin{align*}
        f_d(x_N) - {(f_d)}_\star
        \ge {(f_N)}_\star - {(f_d)}_\star
        &= \frac{\beta}{8}\,\bigl( \frac{1}{N+1} - \frac{1}{d+1}\bigr) \\[0.25em]
        &\ge \frac{\beta\,\norm{x_0 - x_{d,\star}}^2}{8d}\, \bigl( \frac{1}{N+1} - \frac{1}{d+1}\bigr)\,.
    \end{align*}
    Choosing $d \asymp N$, e.g., $d = 2N+1$, yields the stated lower bound.
\end{proof}

Notably, the iteration complexity lower bounds~\autoref{thm:cvx_smooth_lb} and~\autoref{thm:strcvx_smooth_lb} are smaller than the bounds attained by~\ref{eq:GD} in~\autoref{thm:gd_fn_value} by a square root.
As developed in the next sections, in fact the lower bounds are tight and~\ref{eq:GD} is suboptimal.

We make two further remarks.
First, it is perhaps surprising that the lower bound construction is a \emph{quadratic} function; in some sense, quadratics are the hardest convex and smooth functions to optimize.
Second, the lower bound requires the ambient dimension to be larger than the iteration count; this is crucial for the proof technique, which relies on the algorithm discovering one new dimension per iteration.
This turns out to be fundamental because there are better methods in low dimension, for quadratics and even for general convex functions.

\subsection*{Exercises}

\begin{question}\label{qu:weakcvx_distance_lb}
    In the setting of~\autoref{thm:cvx_smooth_lb} and using the same construction as in the proof, show that $\norm{x_N - x_\star}^2 \gtrsim \norm{x_0 - x_\star}^2$.
    In other words, in the $0$-convex case, it is not possible to make progress in the sense of distance to the minimizer by more than a constant factor.
\end{question}

\begin{question}\label{qu:strcvx_lb}
    We used the reductions from \S\ref{ssec:reductions} to reduce the strongly convex lower bound to the $0$-convex lower bound for the sake of brevity, but it is of course possible to develop the strongly convex lower bound directly.
    Consider the function
    \begin{align*}
        f : \R^\infty \to\R\,, \qquad f(x)
        &\deq \frac{\beta - \alpha}{8} \,\Bigl\{{x[1]}^2 + \sum_{n=1}^\infty {(x[n] - x[n+1])}^2 - 2x[1]\Bigr\} + \frac{\alpha}{2}\,\norm x^2\,.
    \end{align*}
    By adapting the proof of~\autoref{thm:cvx_smooth_lb}, show that any gradient span algorithm satisfies
    \begin{align*}
        f(x_N) - f_\star
        \ge \frac{\alpha}{2}\,\norm{x_N - x_\star}^2
        \ge \frac{\alpha}{2}\,\Bigl( \frac{\sqrt\kappa-1}{\sqrt\kappa+1}\Bigr){\Bigsp}^{2N}\,\norm{x_0 - x_\star}^2\,.
    \end{align*}
\end{question}

\section{Acceleration}

We now show that the lower bounds of~\autoref{thm:cvx_smooth_lb} and~\autoref{thm:strcvx_smooth_lb} can be attained via algorithms which improve upon~\ref{eq:GD}.
This is known as the \emph{acceleration} phenomenon in optimization.
We begin with the quadratic case.

\subsection{Quadratic case: the conjugate gradient method}

In this section, the objective function is quadratic:
\begin{align*}
    f : \R^d\to\R\,, \qquad f(x) = \frac{1}{2}\,\langle x, A\,x\rangle - \langle b, x\rangle\,,
\end{align*}
where $A$ is a symmetric matrix, $A \succ 0$.
Note also that minimizing $f$ corresponds to solving the system of equations $Ax_\star = b$.
We now introduce the \emph{conjugate gradient} method~\cite{HesSti1952Conjugate}.

The method is succinctly described as follows:
\begin{align}\label{eq:CG}\tag{$\msf{CG}$}
    x_{n+1}
    &\deq \argmin\bigl\{f(x) \bigm\vert x \in x_0 + \spn\{\nabla f(x_0),\nabla f(x_1),\dotsc,\nabla f(x_n)\}\bigr\}\,.
\end{align}
This scheme is very natural in light of the definition of a gradient span algorithm (\autoref{def:gradient_span}) that we encountered for the lower bounds.
However, it is not yet clear that~\eqref{eq:CG} can be implemented cheaply.
Using the fact that $f$ is quadratic, our aim is to show that~\eqref{eq:CG} can be rewritten as a simple iteration that uses one gradient query per step.

As is usually the case in linear algebra, instead of working with the set of vectors $\{\nabla f(x_0),\nabla f(x_1),\dotsc,\nabla f(x_n)\}$, it is more convenient to work with an \emph{orthogonal} set $\{p_0,p_1,\dotsc,p_n\}$.
Here, orthogonality is with respect to the inner product $\langle \cdot, \cdot \rangle_A$, i.e., we will require $\langle p_i, A\,p_j \rangle = 0$ for all $i\ne j$.
We start with $p_0 \deq \nabla f(x_0)$, and we write $\eu K_n \deq \spn\{p_0,p_1,\dotsc,p_n\}$.
We must address the following two questions:
\begin{itemize}
    \item Given $\eu K_n$ and $x_n$, how can we compute $x_{n+1} = \argmin_{x_0+\eu K_n} f$?
    \item Given $\eu K_n$ and $\nabla f(x_{n+1})$, how can we compute $p_{n+1}$ and thus $\eu K_{n+1}$?
\end{itemize}

For the first question, we may assume inductively that $x_n = \argmin_{x_0 + \eu K_{n-1}} f$, which means that $\langle \nabla f(x_n), p_k\rangle = 0$ for all $k < n$.
The next point is taken to be $x_{n+1} = x_n + h_n p_n$, chosen so that $\langle \nabla f(x_{n+1}), p_k \rangle = 0$ for all $k \le n$.
Since $\nabla f$ is linear,
\begin{align*}
    \langle \nabla f(x_{n+1}), p_k \rangle
    &= \langle \nabla f(x_n) + h_n\,Ap_n, p_k \rangle\,.
\end{align*}
For $k < n$, this equals zero by the inductive hypothesis on $x_n$, and the orthogonality of $\{p_0,p_1,\dotsc,p_n\}$.
We choose $h_n$ to ensure that this equals zero for $k=n$ too:
\begin{align*}
    h_n
    &= -\frac{\langle \nabla f(x_n), p_n \rangle}{\norm{p_n}_A^2}\,.
\end{align*}

For the second question, we want to compute the Gram{--}Schmidt orthogonalization of $\nabla f(x_{n+1})$ w.r.t.\ $\{p_0,p_1,\dotsc,p_n\}$ in the $\langle \cdot, \cdot \rangle_A$ inner product.
We claim that $\nabla f(x_{n+1})$ is already $A$-orthogonal to $p_k$ for $k < n$, so that
\begin{align}\label{eq:CG_p}
    p_{n+1}
    &= \nabla f(x_{n+1}) - \langle \nabla f(x_{n+1}), p_n\rangle_A\, \frac{p_n}{\norm{p_n}_A^2}\,.
\end{align}
To justify this, we show that for $k < n$, $\boxed{Ap_k \in \eu K_{k+1}}$, hence
\begin{align*}
    \langle \nabla f(x_{n+1}), p_k \rangle_A
    &= \langle\nabla f(x_{n+1}), Ap_k \rangle
    = 0
\end{align*}
using the fact shown above that $\nabla f(x_{n+1})$ is orthogonal (in the usual inner product) to $\eu K_n$.
Finally, the boxed equation is shown through the following lemma.

\begin{lem}\label{lem:krylov}
    For all $n\in \N$,
    \begin{align*}
        \eu K_n
        = \spn\{p_0, Ap_0, \dotsc, A^n p_0\}\,.
    \end{align*}
\end{lem}
\begin{proof}
    We proceed via induction, where the case $n=0$ is obvious.
    Assuming it holds at iteration $n$, let us show that $p_{n+1} \in \widetilde{\eu K}_{n+1} \deq \spn\{p_0,Ap_0,\dotsc,A^{n+1} p_0\}$.
    By~\eqref{eq:CG_p}, it suffices to show that $\nabla f(x_{n+1}) \in \widetilde{\eu K}_{n+1}$.
    However, as discussed above, $\nabla f(x_{n+1}) = \nabla f(x_n) + h_n\,Ap_n = p_0 + h_0\,Ap_0 + \cdots + h_n \, Ap_n \in \widetilde{\eu K}_{n+1}$.

    Conversely, we must show that $A^{n+1} p_0 \in \eu K_{n+1}$.
    Since $A^n p_0 \in \eu K_n$, we can write $A^n p_0 = \sum_{k=0}^n c_k p_k$, thus $A^{n+1} p_0 = \sum_{k=0}^n c_k\,Ap_k$.
    By the inductive hypothesis, each $Ap_k$ for $k < n$ belongs to $\eu K_n$, so it suffices to have $Ap_n \in \eu K_{n+1}$.
    However, we can observe that $Ap_n = h_n^{-1}\,(\nabla f(x_{n+1}) - \nabla f(x_n)) \in \eu K_{n+1}$ by~\eqref{eq:CG_p}.
\end{proof}

\begin{defn}
    The subspaces ${\{\eu K_n\}}_{n\in\N}$ are called \textbf{Krylov subspaces}.
\end{defn}

Finally, let us write the iterations in a form which is convenient for implementation.
Note first that $\langle \nabla f(x_n), \nabla f(x_{n+1}) \rangle = 0$ (indeed, $\nabla f(x_{n+1})$ is orthogonal to all of $\eu K_n$).
So,
\begin{align*}
    \frac{\langle \nabla f(x_{n+1}), p_n \rangle_A}{\norm{p_n}_A^2} 
    &= \frac{\langle \nabla f(x_{n+1}), \nabla f(x_{n+1}) -\nabla f(x_n) \rangle}{h_n\,\norm{p_n}_A^2}
    = -\frac{\norm{\nabla f(x_{n+1})}^2}{\langle \nabla f(x_n), p_n\rangle}
\end{align*}
and $\norm{\nabla f(x_n)}^2 = \langle \nabla f(x_n), \nabla f(x_n) \rangle = \langle \nabla f(x_n), p_n \rangle$ using~\eqref{eq:CG_p} and the fact that $\nabla f(x_n)$ is orthogonal to $\eu K_{n-1}$.
This yields the following iteration, where we write $r_n \deq Ax_n - b = \nabla f(x_n)$ for the residual.
\begin{align*}
    x_{n+1}
    &= x_n - \frac{\norm{r_n}^2}{\langle p_n, A\,p_n\rangle}\,p_n\,, \quad r_{n+1} = r_n - \frac{\norm{r_n}^2}{\langle p_n, A\,p_n\rangle}\,Ap_n\,, \quad p_{n+1} = r_{n+1} + \frac{\norm{r_{n+1}}^2}{\norm{r_n}^2}\,p_n\,.
\end{align*}
This algorithm requires one matrix-vector multiplication per iteration, namely, the computation of $Ap_n$.

Note that if $p_{n+1} = 0$, then $\nabla f(x_{n+1}) \in \eu K_n$, yet $\nabla f(x_{n+1}) \perp \eu K_n$ and thus $\nabla f(x_{n+1}) = 0$, $x_{n+1} = x_\star$.
Since $p_d = 0$ (an orthogonal set in $\R^d$ cannot have more than $d$ non-zero elements), we arrive at the following conclusion.

\begin{thm}[termination of~\ref{eq:CG}]\label{thm:termination_CG}
    The~\ref{eq:CG} algorithm returns the exact minimizer in at most $d$ iterations.
\end{thm}

Let us now show that~\ref{eq:CG} can find an approximate minimizer at the accelerated rate.

\begin{thm}[accelerated convergence for~\ref{eq:CG}]\label{thm:CG_rate}
    Let $0 \prec \alpha I \preceq A \preceq \beta I$.
    Then,~\ref{eq:CG} outputs $x_N$ satisfying $f(x_N) - f_\star \le \varepsilon$ in $N = O(\sqrt \kappa \log \frac{f(x_0) - f_\star}{\varepsilon})$ iterations.
\end{thm}
\begin{proof}
    By the descent lemma (\autoref{lem:descent}) and the defining property of~\ref{eq:CG},
    \begin{align*}
        f(x_{n+1})
        &\le f\bigl(x_n - \frac{1}{\beta}\,\nabla f(x_n)\bigr)
        \le f(x_n) - \frac{1}{2\beta}\,\norm{\nabla f(x_n)}^2\,,
    \end{align*}
    so that
    \begin{align*}
        f(x_0) - f_\star
        &\ge \frac{1}{2\beta} \sum_{n=0}^{N-1} \norm{\nabla f(x_n)}^2\,.
    \end{align*}
    On the other hand, since $\nabla f(x_n) \perp x_{k+1} - x_k$ for $k < n$,
    \begin{align*}
        f_\star - f(x_n)
        &\ge \langle \nabla f(x_n), x_\star - x_n \rangle
        = \langle \nabla f(x_n), x_\star - x_0 \rangle\,.
    \end{align*}
    If we sum these inequalities and use orthogonality of the gradients,
    \begin{align*}
        N\,(f(x_N) - f_\star)
        &\le \sum_{n=0}^{N-1} (f(x_n) - f_\star)
        \le \Bigl\langle \sum_{n=0}^{N-1} \nabla f(x_n), x_0 - x_\star \Bigr\rangle
        \le \Bigl\lVert \sum_{n=0}^{N-1} \nabla f(x_n) \Bigr\rVert\,\norm{x_0 - x_\star} \\
        &\le \Bigl(\sum_{n=0}^{N-1} \norm{\nabla f(x_n)}^2\Bigr){\Bigsp}^{1/2} \,\sqrt{\frac{2\,(f(x_0) - f_\star)}{\alpha}}
        \le 2\sqrt{\kappa}\,(f(x_0) - f_\star)\,.
    \end{align*}
    Let $N$ be such that $f(x_N) - f_\star \ge (f(x_0) - f_\star)/2$.
    The inequality above then implies that $N \le 4\sqrt\kappa$.
    Thus, every $4\sqrt\kappa$ iterations, the objective gap decreases by a factor of $2$.
\end{proof}

Using the reduction in~\autoref{lem:strcvx_to_weakcvx}, this also yields an accelerated algorithm for minimizing smooth, weakly convex quadratics.
We now sketch an alternative proof in order to explain the classical link with polynomial approximation.

Due to~\autoref{lem:krylov}, $x_N - x_0 \in \eu K_{N-1}$ can be written in the form $x_N - x_0 = \sum_{n=0}^{N-1} c_n A^n p_0$, so $x_N - x_\star = x_0 - x_\star + \sum_{n=0}^{N-1} c_n A^{n+1}\,(x_0 - x_\star) = P_N(A)\,(x_0 - x_\star)$ where $P_N$ is a polynomial of degree at most $N$ satisfying $P_N(0) = 1$.
Conversely, if $Q_N$ is any other degree-$N$ polynomial with $Q_N(0) = 1$, then $\tilde x_N \deq x_0 + A^{-1}\,(Q_N(A) - I)\,p_0 \in x_0 + \eu K_{N-1}$ satisfies $\tilde x_N - x_\star = x_0 - x_\star + A^{-1}\,(Q_N(A) - I)\,p_0 = Q_N(A)\,(x_0 - x_\star)$.

This equivalence, together with the fact that the output $x_N$ of~\ref{eq:CG} minimizes $f$ over $x_0 + \eu K_{N-1}$, shows that
\begin{align*}
    f(x_N) - f_\star
    &= \frac{1}{2} \min\{\norm{Q_N(A)\,(x_0 - x_\star)}_A^2 : Q_N \in \R_{\le N}[X]\,,\; Q_N(0) = 1\}\,,
\end{align*}
where $\R_{\le N}[X]$ denotes the set of polynomials with real-valued coefficients and with degree at most $N$.
Furthermore, since $A$ and $Q_N(A)$ commute,
\begin{align*}
    \norm{Q_N(A)\,(x_0 - x_\star)}_A^2
    &\le \norm{Q_N(A)}_{\rm op}^2\,\norm{x_0 - x_\star}_A^2
    \le \bigl(\max_{[\lambda_{\min}(A),\, \lambda_{\max}(A)]}{\abs{Q_N}^2}\bigr)\,\norm{x_0 - x_\star}_A^2\,.
\end{align*}
We have arrived at the following result.

\begin{lem}[{\ref{eq:CG}} and polynomial approximation]\label{lem:CG_poly_approx}
    Assume that $0 \prec \alpha I \preceq A \preceq \beta I$.
    Then, the output $x_N$ of~\ref{eq:CG} satisfies
    \begin{align*}
        f(x_N) - f_\star
        &\le \min\bigl\{\max_{\lambda \in [\alpha,\beta]}\abs{Q_N(\lambda)}^2 : Q_N \in \R_{\le N}[X]\,, \; Q_N(0) = 1\bigr\}\,(f(x_0) - f_\star)\,.
    \end{align*}
\end{lem}

Informally, this result states that~\ref{eq:CG} performs as well as the best possible degree-$N$ polynomial in $A$.
To bound the rate of convergence of~\ref{eq:CG}, it therefore remains to exhibit a judicious polynomial $Q_N$.
This is accomplished by the family of Chebyshev polynomials, on which many volumes have been written.

\begin{defn}\label{defn:chebyshev}
    The degree-$n$ \textbf{Chebyshev polynomial} $T_n$ is defined so that $\cos(n\theta) = T_n(\cos\theta)$ for all $\theta \in \R$.
\end{defn}

It is not obvious at first glance that $T_n$ is indeed a degree-$n$ polynomial, but this can be established via trigonometric identities.
The use of the Chebyshev polynomials to establish a rate of convergence for~\ref{eq:CG} is explored in~\autoref{qu:CG_chebyshev}.

Here, we point out another interesting fact that arises from this connection.
Recall from the proof of~\autoref{lem:CG_poly_approx} that if we can compute $\tilde x_N \deq x_0 + A^{-1}\,(Q_N(A) - I)\,p_0$, then it incurs error at most $f(\tilde x_N) - f_\star \le (\max_{\lambda \in [\alpha,\beta]}{\abs{Q_N(\lambda)}}^2)\,(f(x_0) - f_\star)$.
In particular, rather than using~\ref{eq:CG}, we can try to compute the polynomial $x \mapsto (Q_N(x)-1)/x$ directly, where $Q_N$ is the polynomial in~\autoref{qu:CG_chebyshev} which witnesses the fast convergence of~\ref{eq:CG}.
Although we omit the details, it is worth noting that the family of Chebyshev polynomials satisfies a so-called three-term recurrence:
\begin{align*}
    T_{n+1}(x)
    &= 2x\,T_n(x) - T_{n-1}(x)\,, \qquad x \in \R\,.
\end{align*}
In fact, orthogonal families of polynomials usually do.\footnote{This arises in connection with second-order differential operators.}
From an algorithmic standpoint, it leads to an optimization algorithm of the form
\begin{align*}
    x_{n+1}
    &= c_0\,Ax_n + c_1\,x_{n-1} + c_2\,b\,,
\end{align*}
where $c_0, c_1, c_2 \in \R$ are fixed coefficients.
Note that unlike~\ref{eq:GD}, $x_{n+1}$ depends on the previous \emph{two} iterates.
This is often referred to as \emph{momentum}, and also forms the basis for acceleration for general convex functions.

\begin{rmk}[practicality of~\ref{eq:CG}]
    Solving the linear system $Ax=b$ via Gaussian elimination requires $O(d^3)$ operations and is numerically unstable, whereas for well-conditioned matrices $A$,~\ref{eq:CG} returns an approximate solution in $\widetilde O(\sqrt \kappa)$ iterations, each of which requires a matrix-vector multiplication.
    A matrix-vector multiplication requires $O(d^2)$ time in the worst case, but can be faster if $A$ is sparse.
    In practice,~\ref{eq:CG} is widely used, especially when combined with other strategies such as preconditioning.
\end{rmk}

\subsection{General case: continuous time}

Although it does not follow the historical development of events, we begin our treatment of acceleration for general convex smooth functions in continuous time.
As identified in~\cite{SuBoyCan16Nesterov}, the continuous-time ODE is
\begin{align}\label{eq:AGF}\tag{$\msf{AGF}$}
    \begin{aligned}
        \dot x_t
        &= p_t\,, \\
        \dot p_t
        &= -\nabla f(x_t)-\gamma_t p_t\,.
    \end{aligned}
\end{align}
We refer to~\eqref{eq:AGF} as the \emph{accelerated gradient flow}, and the variable $p_t$ admits the physical interpretation of momentum (for a particle with unit mass).
The dynamics consists of two parts: the equations
\begin{align*}
    \dot x_t
    &= p_t\,, \\
    \dot p_t
    &= -\nabla f(x_t)
\end{align*}
are known as Hamilton's equations, and they are the standard first-order reformulation of Newton's law of motion $\ddot x_t = -\nabla f(x_t)$ with potential energy $f$.
Hamilton's equations conserve the energy (or Hamiltonian) $H(x,p) \deq f(x) + \frac{1}{2}\,\norm p^2$, and this conservation property is perhaps undesirable for an optimization algorithm which seeks to minimize $f$.
Thus, the second part of the dynamics, $\dot p_t = -\gamma_t p_t$ adds a dissipative \emph{friction} force, where $\gamma_t \ge 0$ is a possibly time-varying coefficient of friction.

In the case where $f$ is merely assumed to be convex, it turns out that the right choice of friction coefficient is $\gamma_t = 3/t$.
This is mysterious at first sight and was obtained by taking the continuous-time limit of Nesterov's discrete algorithm in the next subsection.
We begin with a convergence analysis in this setting.
(Similar caveats as for \S\ref{sec:grad_flow} apply here; we assume that $f$ is smooth, that it admits a minimizer $x_\star$, and that~\eqref{eq:AGF} is well-posed.)

\begin{thm}[convergence of~\ref{eq:AGF} under convexity]\label{thm:AGF_cvxty}
    Let $f : \R^d\to\R$ be convex and let ${(x_t)}_{t\ge 0}$ evolve along~\ref{eq:AGF} with $\gamma_t = 3/t$ and $p_0 = 0$.
    Then, for all $t\ge 0$,
    \begin{align*}
        f(x_t) - f_\star
        &\le \frac{2\,\norm{x_0-x_\star}^2}{t^2}\,.
    \end{align*}
\end{thm}
\begin{proof}
    Consider the auxiliary point $z_t \deq x_t + \frac{t}{2}\,p_t$, and the Lyapunov function
    \begin{align*}
        \ms L_t
        &\deq \frac{t^2}{2}\,(f(x_t) - f_\star) + \norm{z_t - x_\star}^2\,.
    \end{align*}
    The computation below shows that $\dot{\ms L}_t \le 0$, which implies the result.
    The choice of Lyapunov function is mysterious, so we partially demystify it after the proof.

    Straightforward differentiation and convexity yield
    \begin{align*}
        \dot{\ms L}_t
        &= t\,(f(x_t) - f_\star) + \frac{t^2}{2}\,\langle \nabla f(x_t), p_t \rangle - t\,\langle \nabla f(x_t), z_t - x_\star\rangle \\
        &= t\,(f(x_t) - f_\star) - t\,\langle \nabla f(x_t), x_t - x_\star\rangle
        \le 0\,. \qedhere
    \end{align*}
\end{proof}

Although the Lyapunov function above appears fortuitous, it can be derived in a reasonably systematic manner; see~\autoref{qu:derive_su_boyd_candes}.
The strongly convex case is similar, and is left as~\autoref{qu:AGF_strcvx}.

\begin{thm}[convergence of~\ref{eq:AGF} under strong convexity]\label{thm:AGF_strcvx}
    Let $f : \R^d\to\R$ be $\alpha$-convex and let ${(x_t)}_{t\ge 0}$ evolve along~\ref{eq:AGF} with $\gamma_t = 2\sqrt \alpha$ and $p_0 = 0$.
    For all $t\ge 0$,
    \begin{align*}
        f(x_t) - f_\star
        &\le 2\exp(-\sqrt\alpha\,t)\,(f(x_0) - f_\star)\,.
    \end{align*}
\end{thm}

We plot the non-monotonic behavior of~\ref{eq:AGF} in~\autoref{fig:AGF}, which helps to explain the need for choosing a different Lyapunov function $\ms L_t$.

\begin{figure}[h]
    \centering
    \includegraphics[width=0.7\textwidth]{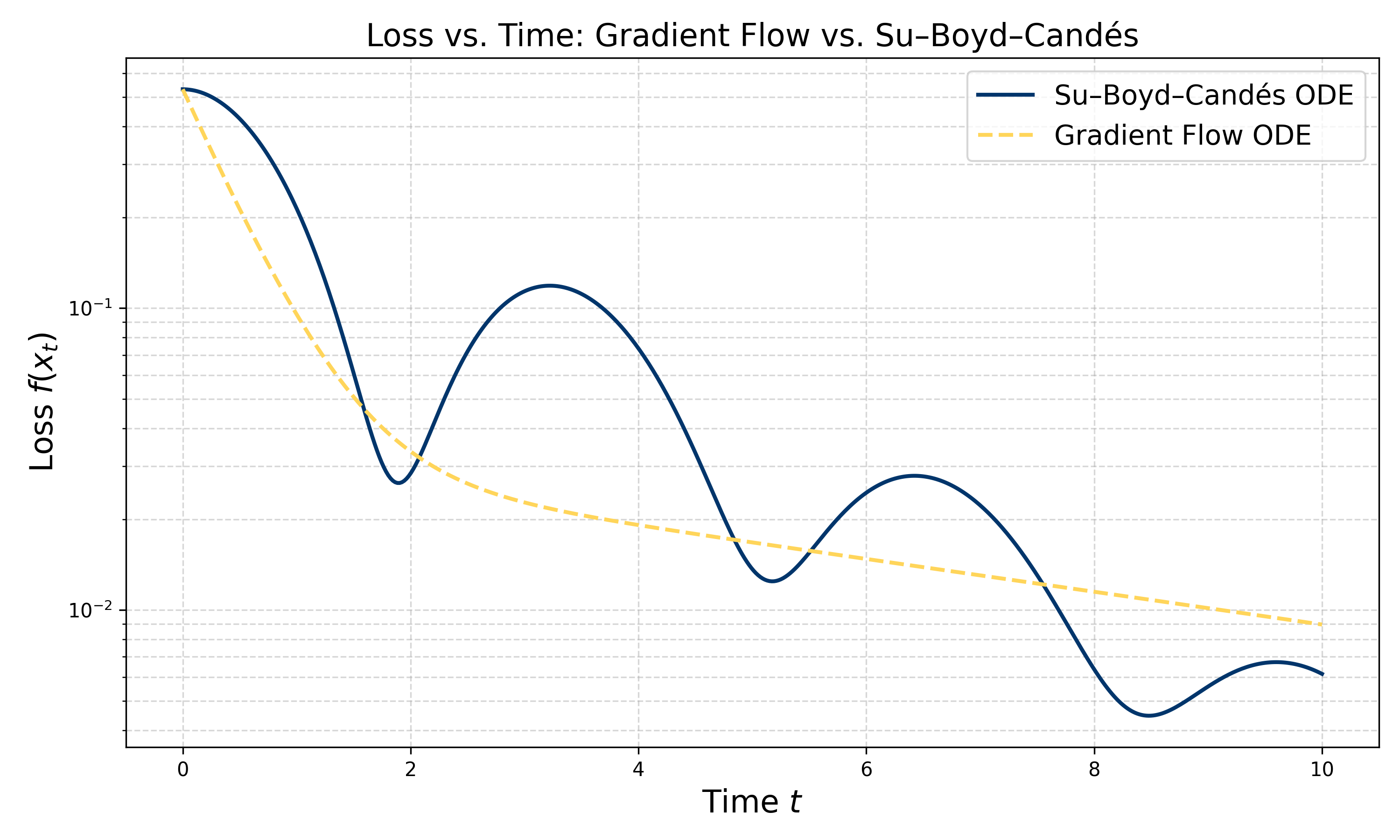}
    \caption{\footnotesize We compare the gradient flow with~\ref{eq:AGF} for $f(x)\deq \frac{1}{2}\,(\alpha\,{x[1]}^2 + {x[2]}^2)$ and $\gamma = 2\sqrt \alpha$, where $\alpha = 1/16$ and $x_0 = (1,1)$. Note that the loss is not monotonic over time, hence the need for a different Lyapunov function $\ms L_t$.}\label{fig:AGF}
\end{figure}

Recall that under convexity and $\alpha$-convexity, the objective gap $f(x_t) - f_\star$ for~\ref{eq:GF} converges at the rates $O(1/t)$ and $O(\exp(-2\alpha t))$ respectively.
On the other hand, for~\ref{eq:AGF}, the convergence happens at the rates $O(1/t^2)$ and $O(\exp(-\sqrt\alpha\,t))$ respectively.
This is strongly suggestive of the square root factor speed-up, that is, \emph{acceleration}.
However, we caution that it is dangerous to deduce conclusions from continuous-time analysis alone.
For example, we can run any ODE faster, which can make the continuous-time convergence rate arbitrarily fast; however, this does not translate into a better discrete-time algorithm, since speeding up time makes the ODE more unstable and therefore requires a smaller step size for discretization.

So how, then, can we discretize~\ref{eq:AGF}?
Part of the subtlety of acceleration is that not all discretizations work.
For example, we could consider
\begin{align*}
    x_{n+1}
    &\approx x_n +h\,p_{n+1}\,, \\
    p_{n+1}
    &\approx p_n - h\,\nabla f(x_n) - \gamma_n h\,p_n
\end{align*}
which is equivalent to the update
\begin{align*}
    x_{n+1}
    &= x_n - h^2\,\nabla f(x_n) + (1-\gamma_n h)\,(x_n - x_{n-1})\,.
\end{align*}
Or, if we do not presume to know the coefficients for the discrete-time scheme in advance, we could write the update as
\begin{align*}
    x_{n+1}
    &= x_n - \eta_n\,\nabla f(x_n) + \theta_n\,(x_n - x_{n-1})\,.
\end{align*}
In other words, we take a gradient step and then apply momentum.
This is known as Polyak's heavy ball method, and although it can be tuned to converge at the rate of~\ref{eq:CG} for quadratic objectives, this same tuning leads to divergence for general convex functions~\cite{LesRecPac16QuadConstraints}.
On the other hand, the optimal method in the next subsection can be written in the form
\begin{align*}
    x_{n+1}
    &= x_n + \theta_n\,(x_n - x_{n-1}) - \eta_n\,\nabla f\bigl(x_n + \theta_n\,(x_n - x_{n-1})\bigr)\,.
\end{align*}
In other words, we add momentum and then take a gradient step.

\subsection{General case: discrete time}

The acceleration phenomenon is undoubtedly one of the most elusive and fascinating aspects of optimization, so it is no surprise that it has been explored through many different angles over the course of countless research papers.
At this juncture, we must choose how to present the method and in what level of detail.

Having explored acceleration carefully in the quadratic case and in continuous time, here we follow the expedient route by giving perhaps the most direct and shortest proof, at the cost of generality and intuition.\footnote{Perhaps I will change my mind in a future edition of these notes.}

We analyze the following method, introduced in~\cite{Nes1983Accel}, with $x_{-1} = x_0$:
\begin{align}\label{eq:AGD}\tag{$\msf{AGD}$}
    x_{n+1}
    &\deq x_n + \theta_n\,(x_n - x_{n-1}) - \frac{1}{\beta}\,\nabla f\bigl(x_n + \theta_n\,(x_n - x_{n-1})\bigr)\,.
\end{align}

\begin{thm}[convergence of~\ref{eq:AGD}]\label{thm:AGD}
    Let $f$ be convex and $\beta$-smooth.
    Define the sequence: $\lambda_0 \deq 0$ and $\lambda_{n+1} \deq \frac{1}{2}\,(1+\sqrt{1+4\lambda_n^2})$ for $n\in\N$.
    Set $\theta_n \deq (\lambda_n-1)/\lambda_{n+1}$.
    Then,~\ref{eq:AGD} satisfies
    \begin{align*}
        f(x_N) - f_\star
        &\le \frac{2\beta\,\norm{x_0-x^\star}^2}{N^2}\,.
    \end{align*}
\end{thm}
\begin{proof}
    Let $y_n \deq x_n + \theta_n\,(x_n - x_{n-1})$, so that $x_{n+1} = y_n - \frac{1}{\beta}\,\nabla f(y_n)$.
    Recall from~\eqref{eq:evi} that for any $z\in\R^d$, it holds that
    \begin{align*}
        \norm{x_{n+1} - z}^2
        &\le \norm{y_n - z}^2 - \frac{2}{\beta}\,(f(x_{n+1}) - f(z))\,.
    \end{align*}
    Rearranging, it yields
    \begin{align*}
        f(x_{n+1}) - f(z)
        &\le \frac{\beta}{2}\,(\norm{y_n - z}^2 - \norm{x_{n+1} - z}^2)
        = - \frac{\beta}{2}\,\norm{x_{n+1} - y_n}^2 - \beta\,\langle x_{n+1} - y_n, y_n - z\rangle\,.
    \end{align*}
    We apply this inequality with two points, $z = x_n$ and $z = x_\star$.
    By multiplying the first inequality by $\lambda_{n+1}-1\ge 0$ and adding it to the second inequality, it implies
    \begin{align*}
        &(\lambda_{n+1}-1)\,(f(x_{n+1}) - f(x_n)) + f(x_{n+1}) - f_\star \\[0.25em]
        &\qquad \le - \frac{\beta\lambda_{n+1}}{2}\,\norm{x_{n+1}-y_n}^2 - \beta\,\langle x_{n+1} - y_n,\, \lambda_{n+1}\, y_n - (\lambda_{n+1} - 1)\,x_n - x_\star\rangle \\[0.25em]
        &\qquad = \frac{\beta}{2\lambda_{n+1}}\,\bigl(\norm{\lambda_{n+1}\,y_n - (\lambda_{n+1} - 1)\,x_n - x_\star}^2 - \norm{\lambda_{n+1}\,x_{n+1} - (\lambda_{n+1}-1)\,x_n - x_\star}^2\bigr)\,,
    \end{align*}
    where the last line uses the identity $\norm a^2 + 2\,\langle a, b \rangle = \norm{a+b}^2 - \norm b^2$.
    Our goal is to produce a telescoping sum, which is the case if we ensure that
    \begin{align*}
        \lambda_{n+1}\,x_{n+1} - (\lambda_{n+1} - 1)\,x_n = \lambda_{n+2}\,y_{n+1} - (\lambda_{n+2} - 1)\,x_{n+1}\,.
    \end{align*}
    By substituting in $y_{n+1} = x_{n+1} + \theta_{n+1}\,(x_{n+1} - x_n)$, some algebra shows that it suffices to take $\theta_{n+1} = (\lambda_{n+1} - 1)/\lambda_{n+2}$.

    After multiplying the above inequality by $\lambda_{n+1}$ and summing, we find that
    \begin{align*}
        \frac{\beta}{2} \,\norm{\lambda_1\, y_0 - (\lambda_1 - 1)\, x_0 - x_\star}^2
        &\ge \sum_{n=0}^{N-1} \{\lambda_{n+1}^2\,(f(x_{n+1}) - f_\star) - \lambda_{n+1}\,(\lambda_{n+1}-1)\,(f(x_n) - f_\star)\}\,.
    \end{align*}
    We also want the right-hand side to telescope, so we set $\lambda_{n+1}\,(\lambda_{n+1} - 1) = \lambda_n^2$, which yields the recursion $\lambda_{n+1} = \frac{1}{2}\,(1+\sqrt{1+4\lambda_n^2})$.
    With $\lambda_0 = 0$, it yields
    \begin{align*}
        f(x_N) - f_\star
        &\le \frac{\beta\,\norm{y_0 - x_\star}^2}{2\lambda_N^2}
        = \frac{\beta\,\norm{x_0 - x_\star}^2}{2\lambda_N^2}\,.
    \end{align*}
    Finally, it is straightforward to show by induction that $\lambda_N \ge N/2$.
\end{proof}

By applying the reduction in~\autoref{lem:weakcvx_to_strcvx}, it also yields an accelerated algorithm for the strongly convex case, i.e., an algorithm that achieves $f(x_N) - f_\star \le \varepsilon$ in $O(\sqrt \kappa \log \frac{\alpha R^2}{\varepsilon})$ iterations, where $R \deq \norm{x_0 - x_\star}$.

\begin{ex}
    If we apply the accelerated method to logistic regression (see~\autoref{ex:logistic_gd}), it improves the iteration complexity from $O(R^2 n/d)$ to $O(R\sqrt{n/d})$.
\end{ex}

\subsection*{Bibliographical notes}

The simple proof of~\autoref{thm:CG_rate} is taken from~\cite{NemYud1983Complexity}.
The discussion on Chebyshev polynomials follows~\cite{Vis12Lap}.

The literature on acceleration is too large to be surveyed here, but we mention a recent result in a somewhat different direction: what is the best rate of~\ref{eq:GD} just by changing the step sizes?
Thus, we consider the iteration $x_{n+1} = x_n - h_n\,\nabla f(x_n)$, with the only freedom being to choose the sequence ${\{h_n\}}_{n\in\N}$.
It turns out a constant step size schedule is not optimal, and as established in~\cite{AltPar24Hedging, AltPar24HedgingII}, the so-called silver step size schedule achieves the rates of~\autoref{lem:weakcvx_to_strcvx} and~\autoref{lem:strcvx_to_weakcvx} with $\phi(x) = x^{\log_\rho 2} \approx x^{0.786}$ with $\rho \deq 1+\sqrt 2$.
This is a rate intermediate between the unaccelerated rate of~\ref{eq:GD} and the accelerated rate of~\ref{eq:AGD}.

\subsection*{Exercises}

\begin{question}\label{qu:CG_chebyshev}
    Define the polynomial $Q_n(x) = T_n(\frac{\alpha+\beta-2x}{\beta-\alpha})/T_n(\frac{\alpha+\beta}{\beta-\alpha})$.
    Show that $Q_n(0) = 1$ and use~\autoref{defn:chebyshev} to establish the identity
    \begin{align*}
        T_n(x) = \frac{1}{2}\,\bigl((x-\sqrt{x^2 - 1}){}^n + (x+\sqrt{x^2-1}){}^n\bigr) \qquad\text{for}~\abs x \ge 1\,.
    \end{align*}
    One can show that this identity actually holds for all $x\in\R$.
    Use this to show that
    \begin{align*}
        \max_{x \in [\alpha,\beta]}{\abs{Q_n(x)}}
        &\le 2\,\Bigl( \frac{\sqrt\kappa-1}{\sqrt\kappa+1}\Bigr){\Bigsp}^n\,.
    \end{align*}
    Note that by combining this with~\autoref{lem:CG_poly_approx}, it yields an exponential rate of convergence for~\ref{eq:CG} matching the lower bound of~\autoref{qu:strcvx_lb}.
\end{question}

\begin{question}\label{qu:derive_su_boyd_candes}
    To better understand the proof of~\autoref{thm:AGF_cvxty}, consider a Lyapunov function of the form
    \begin{align*}
        \ms L_t
        &= \norm{x_t - x_\star}^2 + a_t\,\langle x_t - x_\star, p_t \rangle + b_t\,\norm{p_t}^2 + c_t\,(f(x_t) - f_\star)\,.
    \end{align*}
    Note that this is the most general Lyapunov function consisting of a combination of a quadratic function in $x_t - x_\star$ and $p_t$, as well as the objective gap; here, it is crucial that we include the mixed term $a_t\,\langle x_t - x_\star, p_t \rangle$.
    Our goal is to choose the coefficients $a_t$, $b_t$, $c_t$ so that $\dot{\ms L}_t \le 0$.

    Compute the derivative in time of $\ms L_t$ along~\ref{eq:AGF} with $\gamma_t = 3/t$, and apply convexity to the term $\langle \nabla f(x_t), x_t - x_\star\rangle$.
    In the resulting expression, since the terms $\langle x_t - x_\star, p_t \rangle$ and $\langle \nabla f(x_t), p_t \rangle$ do not have definite signs, ensure that the coefficients in front of these terms vanish through a suitable choice of $a_t$, $b_t$, $c_t$.
    Show that this leads to $a_t = t + \bar a t^3$ for some $\bar a \ge 0$.
    Next, from the remaining terms, obtain the condition $\dot b_t \le \min\{\frac{a_t}{2}, \,\frac{6b_t}{t} - a_t\}$, which implies $3\dot b_t \le 6b_t/t$, hence we consider $b_t = b_0 + \bar b t^2$ for some $b_0,\bar b \ge 0$.
    Furthermore, argue that we must take $\bar a = 0$ and $\bar b = \frac{1}{4}$.
    To ensure that $\ms L_0$ only depends on $\norm{x_0 - x_\star}$, we set $b_0 = c_0 = 0$.
    Finally, check that with these choices, we have $b_t \ge a_t^2/4$, which is necessary to ensure that $\ms L_t \ge c_t\,(f(x_t) - f_\star)$.

    Show that the Lyapunov function derived in this way coincides with the one used in~\autoref{thm:AGF_cvxty}.
\end{question}

\begin{question}\label{qu:AGF_strcvx}
    Prove~\autoref{thm:AGF_strcvx}.

    \emph{Hint}: Let $z_t \deq x_t + \frac{2}{\gamma}\,p_t$ and consider
    \begin{align*}
        \ms L_t
        &\deq f(x_t) - f_\star + \frac{\alpha}{2}\,\norm{z_t - x_\star}^2\,.
    \end{align*}
\end{question}

\section{Non-smooth convex optimization}

Thus far, we have considered the \emph{unconstrained} minimization of convex and \emph{smooth} functions $f$.
The next step is to consider a far more general class of problems by allowing for constraints and non-smoothness.

The two issues are related.
To minimize $f$ over a convex set $\eu C$, it is equivalent to minimize $f + \chi_{\eu C}$ over all of $\R^d$, where $\chi_{\eu C}$ is the convex indicator function for $\eu C$:
\begin{align}\label{eq:cvx_indicator}
    \chi_{\eu C}(x) \deq \begin{cases}
        0\,, & x \in \eu C\,, \\
        +\infty\,,&  x \notin \eu C\,.
    \end{cases}
\end{align}
In this reformulation, the objective function is allowed to take the value $+\infty$ and is certainly non-smooth.
Even if we do not reformulate the problem in this way, convex constraint sets often arise as the intersection of primitive constraints: $\eu C = \{f_i \le 0~\text{for all}~i\in [m]\}$.
This is equivalent to $\eu C = \{\max_{i\in [m]} f_i \le 0\}$, and the function $\max_{i\in [m]} f_i$ is non-smooth.

On the other hand, without strong convexity, it is not guaranteed that $f$ admits a minimizer over all of $\R^d$ (e.g., $f$ is a linear function, or consider the exponential function over $\R$).
It often makes sense to consider non-smooth minimization over bounded sets.
Thus, we tackle constraints and non-smoothness together.

Although we do not assume smoothness, we still need some minimal regularity for the function $f$.
As justified in~\autoref{lem:lip_cont}, convex functions are actually Lipschitz continuous in the interior of their domains, so it is natural to take as our new function class under consideration the class of convex and Lipschitz functions over bounded convex sets.

\subsection{Convex analysis}

We now work with convex functions $f : \R^d\to\R \cup \{\infty\}$.
The fact that $f$ can now take on the value $+\infty$ leads to some technical issues, but it allows us to seamlessly handle constraints.
Convexity can be defined in the usual way, but it is sometimes convenient to instead work with the epigraph.

\begin{defn}
    The \textbf{epigraph} of $f : \R^d\to\R\cup \{\infty\}$ is the following subset of $\R^d\times\R$:
    \begin{align*}
        \epi f
        &\deq \{(x,t) \in\R^d\times \R : f(x) \le t\}\,.
    \end{align*}
\end{defn}

\begin{defn}
    A function $f : \R^d\to\R \cup \{\infty\}$ is \textbf{convex} if for all $x,y\in\R^d$ and all $t\in [0,1]$, it holds that
    \begin{align*}
        f((1-t)\,x + t\,y) \le (1-t)\,f(x) + t\,f(y)\,.
    \end{align*}
    Equivalently, $f$ is convex if and only if $\epi f$ is a convex set.
\end{defn}

\begin{defn}
    The \textbf{domain} of a function $f : \R^d\to\R\cup\{\infty\}$ is the set
    \begin{align*}
        \dom f \deq \{x\in \R^d : f(x) < \infty\}\,.
    \end{align*}
\end{defn}

The first point to emphasize is that at this level of generality, $f$ can still be quite pathological.
Indeed, consider the following function:
\begin{align}\label{eq:pathology}
    f(x) \deq \begin{cases}
        0\,, & \norm x < 1\,, \\
        \phi(x)\,, & \norm x = 1\,, \\
        +\infty\,, & \norm x > 1\,,
    \end{cases}
\end{align}
where $\phi$ is an \emph{arbitrary} non-negative function defined on the sphere $\{\norm \cdot = 1\}$.
Then, one can check that $f$ is convex.
However, $\phi$ need not be continuous or be coherent in any way whatsoever.
To avoid these types of situations, the basic regularity property that we impose is that $f$ is lower semicontinuous.

\begin{defn}
    A function $f : \R^d\to\R \cup \{\infty\}$ is \textbf{lower semicontinuous} if for all sequences ${\{x_n\}}_{n\in\N}$ converging to a point $x\in\R^d$, it holds that
    \begin{align*}
        f(x) \le \liminf_{n\to\infty} f(x_n)\,.
    \end{align*}
\end{defn}

In other words, when we pass to the limit of a convergent sequence, the value of $f$ can only drop down.
One way to motivate the relevance of this condition for convex optimization is that we often consider suprema $f = \sup_{\omega\in\Omega} f_\omega$ where ${\{f_\omega\}}_{\omega\in\Omega}$ is a collection of continuous functions; in fact, in many cases, we consider suprema of affine functions.
When $\Omega$ is finite, we know that the maximum of finitely many continuous functions is continuous.
But when $\Omega$ is infinite, the suprema of infinitely many continuous functions need not be continuous.
The class of lower semicontinuous functions is the smallest class of functions which contains all continuous functions and is closed under taking arbitrary suprema.
Further properties are explored in~\autoref{qu:lsc}.

It follows from that exercise that $f$ is convex and lower semicontinuous if and only if its epigraph is closed and convex.
So, when it comes to functions, we impose convexity and lower semicontinuity; and when it comes to sets, we impose convexity and closedness.
For example, one can also check that the convex indicator $\chi_{\eu C}$ is lower semicontinuous if and only if $\eu C$ is closed.
We use the following terminology.\footnote{This is not standard terminology but it is convenient.}

\begin{defn}
    A convex function $f : \R^d\to\R\cup\{\infty\}$ is \textbf{regular} if: it is not identically equal to $+\infty$, it is lower semicontinuous, and its domain has non-empty interior.
\end{defn}

Note that the definition excludes one more pathological case, the function $f(x) = +\infty$ for all $x\in\R^d$, which is of no interest to us.
Since the domain of a convex function is a convex set, if it has empty interior then it must be contained in a lower-dimensional affine space, and when we restrict to that space, the domain then has a non-empty interior; this is usually summarized by saying that any non-empty convex set has a non-empty \emph{relative interior}.
We do not delve into the details here, but this is why we regard the condition that the domain has non-empty interior as ``without loss of generality''.

We also note that in the proof of existence of a minimizer, it is really only lower semicontinuity that matters.

\begin{lem}[existence of minimizer]\label{lem:existence_min_ii}
    Let $f : \R^d\to\R \cup \{\infty\}$ be lower semicontinuous and its level sets be bounded.
    Then, there exists a global minimizer of $f$.
\end{lem}
\begin{proof}
    The proof is the same as for~\autoref{lem:existence_minimizer}, except that lower semicontinuity substitutes for continuity.
\end{proof}

\paragraph*{Regularity.}
Our next order of business is to establish properties of regular convex functions which allow us to manipulate them in proofs.
In particular, we show that they are ``almost'' differentiable, even though we did not assume it a priori; the source of this regularity is the convexity condition.

\begin{lem}[Lipschitz continuity]\label{lem:lip_cont}
    Let $f : \R^d\to\R\cup \{\infty\}$ be convex and let $x_0 \in \interior \dom f$.
    Then, $f$ is locally Lipschitz continuous around $x_0$.
\end{lem}
\begin{proof}
    We may assume that $x_0 = 0$.
    Since $0$ belongs to the interior of $\dom f$, we can fit a cross-polytope centered at the origin inside the domain: namely, there exists $\varepsilon > 0$ such that $\eu C \deq \conv\{\pm \varepsilon e_k : k\in [d]\}$ belongs to $\dom f$.
    First, we show that $f$ is bounded on $\eu C$: the upper bound follows because $f(\pm \varepsilon e_k) < \infty$ for all $k\in [d]$ and the maximum of $f$ over $\eu C$ is attained at one of the vertices (why?).
    For the lower bound, by convexity we have $f(x) \ge 2f(0) - f(-x) \ge 2f(0) - \max_{\eu C} f$ for all $x\in\eu C$.

    Next, we show that $f$ is Lipschitz on the smaller set $\eu C' \deq \conv\{\pm \frac{\varepsilon}{2}\,e_k : k\in [d]\}$.
    The point is that there is a constant $c_{d,\varepsilon} > 0$ such that for all $x,y\in\eu C'$, there is a point $y^+ \in \eu C$ such that the line segment from $x$ to $y$ is contained in the line segment from $x$ to $y^+$, and the extension is not too short: $\norm{y^+-x} \ge c_{d,\varepsilon}$.
    Then, by convexity,
    \begin{align*}
        f(y)
        &= f\Bigl( \frac{\norm{y^+ - y}}{\norm{y^+ - x}}\, x + \frac{\norm{y-x}}{\norm{y^+ - x}}\,y^+\Bigr)
        \le \frac{\norm{y^+ - y}}{\norm{y^+ - x}}\,f(x) + \frac{\norm{y-x}}{\norm{y^+ - x}}\,f(y^+)\,,
    \end{align*}
    hence
    \begin{align*}
        f(y) - f(x)
        &\le \frac{\norm{y-x}}{\norm{y^+ - x}}\,(f(y^+) - f(x))
        \le \frac{\sup_{\eu C} f - \inf_{\eu C} f}{c_{d,\varepsilon}}\,\norm{y-x}\,.
    \end{align*}
    Interchanging $x$ and $y$ proves the Lipschitz bound.
\end{proof}

This lemma shows that locally near $x_0$, $f(x)$ grows at most linearly in the distance $\norm{x-x_0}$ (as opposed to, say, $\sqrt{\norm{x-x_0}}$).
This suggests that $f$ may be differentiable at $x_0$.
This is not quite right, because $f$ may have a kink at $x_0$, but nevertheless we can find an appropriate substitute for differentiability.

\begin{defn}
    Let $f : \R^d\to\R\cup\{\infty\}$ be convex.
    We say that $p \in \R^d$ is a \textbf{subgradient} of $f$ at $x$ if for all $y\in\R^d$, it holds that
    \begin{align}\label{eq:subgrad}
        f(y) \ge f(x) + \langle p, y-x\rangle\,.
    \end{align}
    We denote the set of subgradients of $f$ at $x$ as $\partial f(x)$, and we refer to this set as the \textbf{subdifferential} of $f$ at $x$.
    Also, we set
    \begin{align*}
        \partial f
        &\deq \{(x,p) \in \R^d \times \R^d : p \in \partial f(x)\}\,.
    \end{align*}
\end{defn}

Note that by definition, if $0 \in \partial f(x)$, then $x$ is a global minimizer of $f$.

If $f$ is differentiable at $x_0 \in \interior\dom f$, then $\partial f(x_0)$ is a singleton: $\partial f(x_0) = \{\nabla f(x_0)\}$ (\autoref{qu:subdiff_grad}).
However, the subdifferential can be multi-valued.
A key example is the absolute value function, $f : x\mapsto \abs x$, for which $\partial f(0) = [-1,1]$.

For the purpose of optimization, it is enough to have at least one subgradient, which is the content of the following theorem.

\begin{thm}[subdifferential]\label{thm:subdiff}
    Let $f : \R^d\to\R\cup\{\infty\}$ be a regular convex function.
    If $x_0 \in \interior \dom f$, then $\partial f(x_0)$ is \textbf{non-empty}, bounded, convex, and closed.
\end{thm}

We follow a traditional route of deducing the non-emptiness from a separation theorem.
The proof of the following result is deferred.

\begin{thm}[supporting hyperplane]\label{thm:supporting_hyperplane}
    Let $\eu C$ be a closed and convex set, and let $x \in \partial\eu C$.
    Then, there exists a non-zero $p\in\R^d$ such that
    \begin{align*}
        \langle p, x \rangle \le \inf_{\eu C}{\langle p, \cdot \rangle}\,.
    \end{align*}
\end{thm}

\begin{proof}[Proof of~\autoref{thm:subdiff}]
    Since $(x_0,f(x_0)) \in \partial \epi f$, and $\epi f$ is closed and convex (by regularity of $f$), there is a supporting hyperplane $(p,q)$:
    \begin{align*}
        \langle p, x_0 \rangle + qf(x_0) \le \inf_{(x,t)\in\epi f}\{\langle p, x \rangle + qt\}\,.
    \end{align*}
    We can normalize the coefficients so that $\norm p^2 + q^2 = 1$, and we note that $q\ge 0$.

    If $x$ is sufficiently close to $x_0$, then
    \begin{align*}
        \langle p, x_0 - x \rangle
        &\le q\,(f(x) - f(x_0))
        \le Lq\,\norm{x-x_0}\,,
    \end{align*}
    where $L$ is the Lipschitz constant of $f$ near $x_0$.
    Taking $x = x_0 - \varepsilon p$ for small $\varepsilon > 0$, we deduce that $\norm p \le Lq$, hence from the normalization condition, $q\ne 0$.
    Thus, for any $x\in\dom f$, we deduce that
    \begin{align*}
        f(x)
        &\ge f(x_0) - \frac{1}{q}\,\langle p, x - x_0 \rangle\,,
    \end{align*}
    thus, $-p/q \in \partial f(x_0)$.

    The set $\partial f(x_0)$ is closed and convex as an intersection of the constraints in~\eqref{eq:subgrad}.
    Boundedness follows from~\autoref{qu:subdiff_lip}.
\end{proof}

\paragraph*{Constraints.}
When the constraint set $\eu C$ is simple, it is reasonable to suppose that we can compute the projection onto $\eu C$.
We study some properties of this projection operator.

\begin{defn}
    Let $\eu C$ be closed and convex.
    The \textbf{projection onto $\eu C$} is the mapping $\Pi_{\eu C} : \R^d\to\eu C$ defined by
    \begin{align*}
        \Pi_{\eu C}(x) \deq \argmin_{y\in\eu C}{\norm{y-x}^2}\,.
    \end{align*}
\end{defn}

The ``$\argmin$'' is non-empty because $\eu C$ is closed, and the uniqueness of the minimizer follows from a strict convexity argument as in~\autoref{lem:unique_min}.
When $\eu C$ is a linear subspace, then $\Pi_{\eu C}$ coincides with the linear algebra definition of projection, and in this case $\Pi_{\eu C}$ is linear.
In general, however, $\Pi_{\eu C}$ is a \emph{non-linear} operator.

The following lemma characterizes the projection.

\begin{lem}[characterization of projection]\label{lem:proj_characterization}
    Let $\eu C$ be closed and convex, and let $x \notin \eu C$.
    Then, $\Pi_{\eu C}(x)$ is the unique point satisfying the following condition:
    \begin{align}\label{eq:projection_characterization}
        \langle \Pi_{\eu C}(x) - x, \,x' - \Pi_{\eu C}(x) \rangle \ge 0 \qquad\text{for all}~x' \in \eu C\,.
    \end{align}
\end{lem}
\begin{proof}
    As in the proof of~\autoref{lem:necc_conditions}, the first-order necessary condition for optimality reads $\langle \Pi_{\eu C}(x) - x, v \rangle \ge 0$.
    However, because the optimization problem is constrained to lie in $\eu C$, this time we do not have the inequality for all $v$, but only for $v$ of the form $x' - \Pi_{\eu C}(x)$ where $x'\in\eu C$.
\end{proof}

This lemma furnishes the following important property.

\begin{lem}[convex projections are non-expansive]\label{lem:projection_contraction}
    Let $\eu C$ be closed and convex.
    Then, for all $x,y\in\R^d$,
    \begin{align*}
        \norm{\Pi_{\eu C}(y) - \Pi_{\eu C}(x)} \le \norm{y-x}\,.
    \end{align*}
\end{lem}
\begin{proof}
    By~\eqref{eq:projection_characterization},
    \begin{align*}
        \langle \Pi_{\eu C}(x) - x,\, \Pi_{\eu C}(y) - \Pi_{\eu C}(x) \rangle
        &\ge 0\,, \\
        \langle \Pi_{\eu C}(y) - y,\, \Pi_{\eu C}(x) - \Pi_{\eu C}(y) \rangle
        &\ge 0\,.
    \end{align*}
    Adding these inequalities yields
    \begin{align*}
        \norm{\Pi_{\eu C}(y) - \Pi_{\eu C}(x)}^2
        &\le \langle \Pi_{\eu C}(y) - \Pi_{\eu C}(x),\, y-x\rangle \le \norm{\Pi_{\eu C}(y) - \Pi_{\eu C}(x)}\,\norm{y-x}\,. \qedhere
    \end{align*}
\end{proof}

Actually, we can now return to prove the supporting hyperplane theorem.

\begin{proof}[Proof of~\autoref{thm:supporting_hyperplane}]
    First, we show that if $\eu C$ is a closed convex set and $x\notin \eu C$, then we can separate $\eu C$ from $x$.
    Namely, by~\eqref{eq:projection_characterization}, the vector $p \deq \Pi_{\eu C}(x) - x$ is non-zero and satisfies
    \begin{align*}
        \inf_{x'\in\eu C}{\langle p, x'\rangle}
        &\ge \langle p, \Pi_{\eu C}(x)\rangle
        = \norm{\Pi_{\eu C}(x) - x}^2 + \langle p, x\rangle
        \ge \langle p, x\rangle\,.
    \end{align*}

    To prove the supporting hyperplane theorem, note that since $x \in \partial \eu C$, there is a sequence of points ${\{x_n\}}_{n\in\N}$ which lies outside of $\eu C$, such that $x_n\to x$.
    For each $n$, let $p_n$ be a hyperplane that separates $\eu C$ from $x_n$, and by normalizing we may assume that $\norm{p_n} = 1$.
    Since ${\{p_n\}}_{n\in\N}$ is a bounded sequence, it contains a subsequence which converges to some unit vector $p$.
    By taking limits, it is easy to see that $p$ is a supporting hyperplane.
\end{proof}

\subsection{Projected subgradient methods}\label{ssec:projected_subgradient}

Methods for constrained optimization differ based on what they assume about the constraint set.
The first method we study assumes access to the projection mapping $\Pi_{\eu C}$ for the set $\eu C$.
This assumption is appropriate when the set $\eu C$ is particularly ``simple'', e.g., $\eu C$ is the ball $\eu C = \{\norm \cdot \le R\}$, in which case the projection can be computed in closed form.
When $\eu C$ is more complex, e.g., $\eu C$ is a polytope, we need more sophisticated methods.

Projected subgradient descent is the following method:
\begin{align}\label{eq:PSD}\tag{$\msf{PSD}$}
    x_{n+1}
    &\deq \Pi_{\eu C}\bigl(x_n - h\,\frac{p_n}{\norm{p_n}}\bigr)\,, \qquad p_n \in \partial f(x_n)\,.
\end{align}
Note that we use the normalized subgradient $p_n/\norm{p_n}$.
If we think about the example of the absolute value function $\abs \cdot$ with subdifferential $[-1,1]$ at the origin, we see that the magnitude of an arbitrary element of the subdifferential need not be informative.
Instead, the intuition behind non-smooth optimization is to use the subgradients as separating directions: in particular, by convexity, $f(x) - f(x_n) \ge \langle p_n, x-x_n\rangle$, so any minimizer must lie on one side of the hyperplane defined by $p_n$.

We let $x_\star$ denote a minimizer of $f$ over the closed convex set $\eu C$, and $f_\star \deq f(x_\star)$.

\begin{thm}[convergence of~\ref{eq:PSD}]\label{thm:psd}
    Let $f$ be convex and $L$-Lipschitz continuous on the closed convex set $\eu C$.
    Then,~\ref{eq:PSD} satisfies
    \begin{align*}
        f\bigl( \frac{1}{N} \sum_{n=0}^{N-1} x_n\bigr) - f_\star
        &\le \frac{1}{N} \sum_{n=0}^{N-1} (f(x_n) - f_\star)
        \le \frac{L}{2Nh}\,\norm{x_0 - x_\star}^2 + \frac{Lh}{2}\,.
    \end{align*}
    In particular, by setting $h = R/\sqrt{N}$, where $R$ is an upper bound on $\norm{x_0 - x_\star}$, it yields the convergence rate
    \begin{align*}
        f\bigl( \frac{1}{N} \sum_{n=0}^{N-1} x_n\bigr) - f_\star
        \le \frac{LR}{\sqrt N}\,.
    \end{align*}
\end{thm}
\begin{proof}
    The first inequality holds by convexity, so we focus on the second.
    The idea is similar to the proof of~\autoref{thm:gd_fn_value}, except that instead of using smoothness to handle the error term, we use Lipschitzness.
    By expanding the squared distance to the minimizer,
    \begin{align*}
        \norm{x_{n+1} - x_\star}^2
        &= \bigl\lVert \Pi_{\eu C}\bigl(x_n - h\, \frac{p_n}{\norm{p_n}}\bigr) - \Pi_{\eu C}(x_\star)\bigr\rVert^2
        \le \bigl\lVert x_n - h\, \frac{p_n}{\norm{p_n}} - x_\star\bigr\rVert^2 \\[0.25em]
        &= \norm{x_n - x_\star}^2 - \frac{2h}{\norm{p_n}}\,\langle p_n, x_n - x_\star \rangle + h^2 \\[0.25em]
        &\le \norm{x_n - x_\star}^2 - \frac{2h}{\norm{p_n}}\,(f(x_n) - f_\star) + h^2\,,
    \end{align*}
    where we used~\autoref{lem:projection_contraction}.
    Since $\norm{p_n} \le L$ for all $n$ (\autoref{qu:subdiff_lip}), we sum the inequalities:
    \begin{align*}
        \frac{1}{N} \sum_{n=0}^{N-1} (f(x_n) - f_\star)
        &\le \frac{L}{2Nh}\, \norm{x_0 - x_\star}^2 + \frac{Lh}{2}\,.\qedhere
    \end{align*}
\end{proof}

Thus, the averaged iterate $\bar x_N$ satisfies $f(\bar x_N) - f_\star \le \varepsilon$ provided $N \ge L^2 R^2/\varepsilon^2$.
Note that this convergence rate is substantially worse than the one for the smooth case (\autoref{thm:gd_fn_value}).
Another difference is that the descent lemma (\autoref{lem:descent}) is available in the smooth case which implies monotonic decrease of the objective value; here, there is no descent lemma, so the guarantee only holds for the averaged iterate, and averaging is crucial (\autoref{fig:nonsmooth}).
The analysis can also be performed under strong convexity, see~\autoref{qu:psd_str_cvx}.

Interestingly, if we only assume that $f$ is $L$-Lipschitz continuous over $\msf B(x_\star, R)$, rather than on all of $\eu C$, it is still possible to show that $\min_{n=0,\dotsc,N-1} f(x_n) - f_\star \le LR/\sqrt N$, although the proof becomes more involved~\cite[\S 3.2.3]{Nes18CvxOpt}.

\begin{figure}[h]
    \centering
    \includegraphics[width=0.85\textwidth]{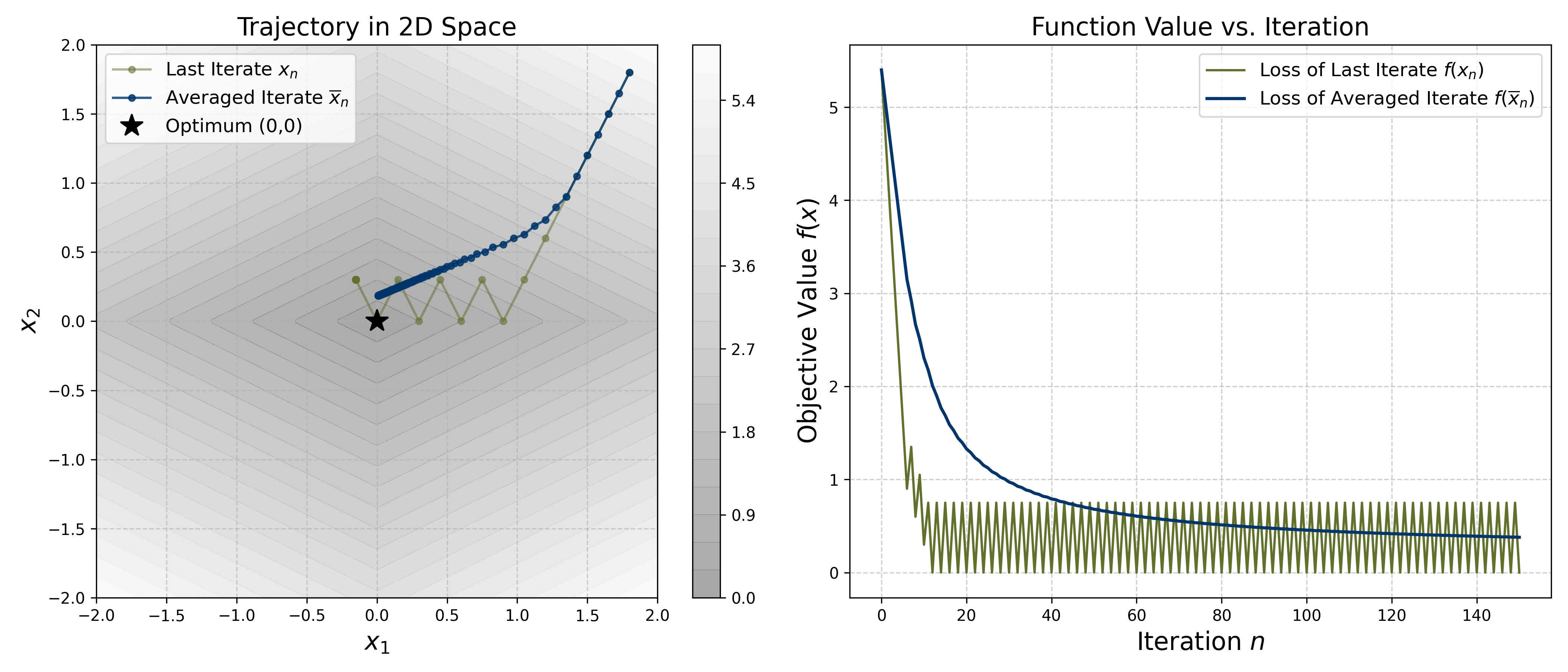}
    \caption{\footnotesize We plot the iterates of~\ref{eq:PSD}, with and without iterate averaging, for the function $f(x) \deq \abs{x[1]} + 2\,\abs{x[2]}$ with step size $h= 0.15$.}\label{fig:nonsmooth}
\end{figure}

The analysis above shows that when the projection operator is cheap to compute, optimization under constraints is straightforward provided that we interleave the gradient steps with projection steps.
We next tackle a more general setting in which we separate out the constraints into a ``simple'' set $\eu C$ for which we can compute the projection operator, and additional functional constraints $\{f_i \le 0~\text{for all}~i\in [m]\}$.
Thus, we consider
\begin{align*}
    \min\{f(x) \mid x \in \eu C\,,\; f_i(x) \le 0~\text{for all}~i\in [m]\}\,.
\end{align*}
We assume that $f$, $f_1,\dotsc,f_m$ are all regular convex functions, and write $f_{\max} \deq \max_{i\in [m]} f_i$.
The next algorithm is known as the projected subgradient method with functional constraints.
For $n=0,1,\dotsc,N-1$:
\begin{itemize}
    \item If $f_{\max}(x_n) \le \varepsilon$, set
        \begin{align*}
            x_{n+1}
            &\deq \Pi_{\eu C}\Bigl(x_n - \frac{\varepsilon}{\norm{p_n}^2}\,p_n\Bigr)\,, \qquad p_n \in \partial f(x_n)\,.
        \end{align*}
    \item Otherwise, set
        \begin{align*}
            x_{n+1}
            &\deq \Pi_{\eu C}\Bigl(x_n - \frac{f_{\max}(x_n)}{\norm{p_n}^2}\,p_n\Bigr)\,, \qquad p_n \in \partial f_{\max}(x_n)\,.
        \end{align*}
\end{itemize}

The algorithm requires computing elements of the subdifferential for the function $\max_{i\in [m]} f_i$.
We therefore first identify this subdifferential.

\begin{lem}[subdifferential of a maximum]\label{lem:subdiff_max}
    Let $f_1,\dotsc,f_m$ be regular convex functions.
    Then, for all $x\in\bigcap_{i\in [m]} \interior \dom f_i$,
    \begin{align*}
        \partial\bigl(\max_{i\in [m]} f_i\bigr)(x)
        &= \conv\bigl\{\partial f_i(x) \bigm\vert i\in [m]\,, \; f_i(x) = \max_{j\in [m]} f_j(x)\bigr\}\,.
    \end{align*}
\end{lem}
\begin{proof}
    ($\supseteq$)
    Let $f_{\max} \deq \max_{i\in [m]} f_i$ and $I_\star(x) \deq \{i\in [m] : f_i(x) = f_{\max}(x)\}$.
    If $\lambda$ is a probability vector and $p_i \in \partial f_i(x)$ for all $i \in I_\star(x)$, then
    \begin{align*}
        f_{\max}(y)
        \ge \sum_{i\in I_\star(x)} \lambda_i f_i(y)
        \ge \sum_{i\in I_\star(x)}\lambda_i\,(f_i(x) + \langle p_i, y - x\rangle)
        = f_{\max}(x) + \Bigl\langle \sum_{i\in I_\star(x)} p_i, \, y-x\Bigr\rangle\,.
    \end{align*}
    Hence, $\sum_{i\in I_\star(x)} \lambda_i p_i \in \partial f_{\max}(x)$.

    ($\subseteq$)
    Since the purpose of this lemma from the perspective of these notes is simply to compute an element of $\partial f_{\max}(x)$, we omit the proof of this direction.
    It can be proven, e.g., via Lagrangian duality or via more subdifferential theory.
\end{proof}

The next theorem provides the convergence rate for the method.

\begin{thm}[convergence of~\ref{eq:PSD} under functional constraints]\label{thm:PSD_functional}
    Let $f, f_1,\dotsc,f_m$ be convex and $L$-Lipschitz on the closed convex set $\eu C$.
    Then,~\ref{eq:PSD} under functional constraints satisfies
    \begin{align}\label{eq:psd_functional_success}
        \min\{f(x_n) \mid n=0,1,\dotsc,N-1\,, \; f_{\max}(x_n) \le \varepsilon\} - f_\star \le \varepsilon
    \end{align}
    provided that
    \begin{align*}
        N \ge \frac{L^2 \,\norm{x_0-x_\star}^2}{\varepsilon^2}\,.
    \end{align*}
\end{thm}

The theorem says that after $N$ iterations, we can find a point $\hat x_N$ which almost satisfies the functional constraints, in the sense that $f_{\max}(\hat x_N) \le \varepsilon$, and moreover $f(\hat x_N) - f_\star \le \varepsilon$.
The number of iterations is no more than the case without functional constraints.

\begin{proof}[Proof of~\autoref{thm:PSD_functional}]
    There are two cases for the algorithm.
    If the iteration $n$ belongs to the first case, then as we saw in the proof of~\autoref{thm:psd},
    \begin{align*}
        \norm{x_{n+1} - x_\star}^2
        &\le \norm{x_n - x_\star}^2 - \frac{2\varepsilon}{\norm{p_n}^2}\,(f(x_n) - f_\star) + \frac{\varepsilon^2}{\norm{p_n}^2}\,.
    \end{align*}
    If $f(x_n) - f_\star \le \varepsilon$, then since $f_{\max}(x_n) \le \varepsilon$ (by the definition of the first case), we have met the success condition~\eqref{eq:psd_functional_success}.
    Otherwise, $f(x_n) - f_\star > \varepsilon$, and the inequality above implies
    \begin{align*}
        \norm{x_{n+1}-x_\star}^2
        &< \norm{x_n - x_\star}^2 - \frac{\varepsilon^2}{\norm{p_n}^2}
        \le \norm{x_n - x_\star}^2 - \frac{\varepsilon^2}{L^2}\,.
    \end{align*}

    What happens in the second case?
    Here, we \emph{also} show that $\norm{x_{n+1} - x_\star} < \norm{x_n - x_\star}$: since $x_\star$ satisfies the functional constraints and $x_n$ does not, the subgradient $p_n \in \partial f_{\max}(x_n)$ still acts as a separating hyperplane.
    Indeed,
    \begin{align*}
        \norm{x_{n+1} - x_\star}^2
        &= \bigl\lVert \Pi_{\eu C}\bigl(x_n - \frac{f_{\max}(x_n)}{\norm{p_n}^2}\,p_n\bigr) - \Pi_{\eu C}(x_\star)\bigr\rVert^2
        \le \bigl\lVert x_n - \frac{f_{\max}(x_n)}{\norm{p_n}^2}\,p_n - x_\star\bigr\rVert^2 \\[0.5em]
        &= \norm{x_n - x_\star}^2 - \frac{2f_{\max}(x_n)}{\norm{p_n}^2}\,\langle p_n, x_n - x_\star\rangle + \frac{{f_{\max}(x_n)}^2}{\norm{p_n}^2} \\[0.5em]
        &\le \norm{x_n - x_\star}^2 - \frac{2f_{\max}(x_n)}{\norm{p_n}^2}\,f_{\max}(x_n) + \frac{{f_{\max}(x_n)}^2}{\norm{p_n}^2}
        < \norm{x_n - x_\star}^2 - \frac{\varepsilon^2}{L^2}\,.
    \end{align*}

    Summing these inequalities across the iterations yields
    \begin{align*}
        \norm{x_N - x_\star}^2
        &< \norm{x_0 - x_\star}^2 - \frac{N\varepsilon^2}{L^2}\,.
    \end{align*}
    For $N \ge L^2\,\norm{x_0 - x_\star}^2/\varepsilon^2$, this is not possible unless we reach the success condition~\eqref{eq:psd_functional_success} by iteration $N$.
\end{proof}

\begin{ex}[soft-margin SVM]
    An example of a problem that can be tackled via projected subgradient methods is soft-margin support vector machine (SVM) classification.
    Suppose that we have a dataset ${\{(x_i, y_i)\}}_{i\in [n]}$, where $x_i \in \R^d$ and $y_i \in \{\pm 1\}$.
    The output of the soft-margin SVM is the classifier $x \mapsto \sgn(\langle \theta^\star, x \rangle)$, where $\theta^\star$ minimizes
    \begin{align*}
        \theta \mapsto \frac{1}{n}\sum_{i=1}^n \ell_{\msf{hinge}}(y_i, \langle \theta, x_i\rangle) + \frac{\lambda}{2}\,\norm \theta^2\,.
    \end{align*}
    Here, $\ell_{\msf{hinge}}(y,\hat y) \deq \max\{0,1-y\hat y\}$ is the hinge loss, $\lambda > 0$ is a regularization parameter, and we have omitted the bias term (which can be handled by augmenting the feature vector $x$ as usual).
    This objective is strongly convex and Lipschitz over bounded sets, so we can apply projected subgradient descent (projecting onto, e.g., a Euclidean ball).
\end{ex}

\subsection{Cutting plane methods}

Non-smooth optimization uses subgradient directions in order to ``localize'' the solution set.
Pursuing this line of reasoning further leads to the family of cutting plane methods.

Suppose that we wish to minimize $f$ over a bounded, closed, convex set $\eu C$.
Let $\eu C_\star$ denote the set of minimizers.
The idea is to construct a sequence of convex sets $\eu C = \eu C_0, \eu C_1, \eu C_2,\dotsc$, which shrink toward $\eu C_\star$.
The set $\eu C_n$ represents possible candidates for the solution to the problem at iteration $n$.

If $x_n \in \eu C_n$ and $p_n \in \partial f(x_n)$, then the subgradient inequality reads
\begin{align*}
    0 \ge f(x_\star) - f(x_n) \ge \langle p_n, x_\star - x_n\rangle \qquad\text{for all}~x_\star \in \eu C_\star\,.
\end{align*}
Thus,
\begin{align*}
    \eu C_\star \subseteq \eu C_n \cap \{x \in \R^d : \langle p_n, x \rangle \le \langle p_n, x_n \rangle\}\,.
\end{align*}
We can take $\eu C_{n+1}$ to be any superset of the right-hand side above.

To finish specifying the scheme, we need a rule for choosing the points $x_n$ and the sets $\eu C_n$, with the goal of $\eu C_n$ shrinking as fast as possible.
The key is the following lemma from convex geometry, which we do not prove.

\begin{lem}[Gr\"unbaum]
    Let $\eu C \subseteq \R^d$ be a convex body (i.e., a compact convex set with non-empty interior) and let $x_{\eu C}$ denote the \emph{centroid} of $\eu C$: $x_{\eu C} \deq {(\vol\eu C)}^{-1}\int_{\eu C} x\,\D x$.
    Then, for any half-space $\eu H$ containing $x_{\eu C}$,
    \begin{align*}
        \frac{\vol(\eu C \cap \eu H)}{\vol(\eu C)}
        &\ge \bigl( \frac{d}{d+1}\bigr){\bigsp}^d
        \ge \frac{1}{\e}\,,
    \end{align*}
    where $\e \approx 2.72$ is a numerical constant.
\end{lem}

Consequently, if we choose $x_n$ to be the centroid of $\eu C_n$ and set
\begin{align}\label{eq:CoGM}\tag{$\msf{CoGM}$}
    \eu C_{n+1} = \eu C_n \cap \{\langle p_n, \cdot \rangle \le \langle p_n, x_n \rangle\}\,, \qquad x_n = x_{\eu C_n}\,,
\end{align}
then Gr\"unbaum's inequality shows that $\vol(\eu C_n \setminus \eu C_{n+1})/\vol(\eu C_n) \le 1/\e$, or
\begin{align*}
    \frac{\vol(\eu C_{n+1})}{\vol(\eu C_n)} \le 1-\frac{1}{\e}\,.
\end{align*}
Thus, we cut away a constant fraction of the volume at each iteration.
This is known as the \emph{center of gravity} method.

As stated,~\ref{eq:CoGM} is not a practical method.
The feasible set $\eu C_n$ at iteration $n$ can be quite complicated, making it prohibitively expensive to compute its centroid.
Centroids can be computed via Markov chain Monte Carlo (MCMC) methods for numerical integration, with guarantees available due to recent advances in log-concave sampling, but it is generally understood that this is a more difficult computational problem than the original convex optimization problem we set out to solve.
Nevertheless,~\ref{eq:CoGM} achieves the optimal complexity bound in the oracle model, so let us analyze its efficiency.

\begin{thm}[center of gravity]\label{thm:CoGM}
    Let $D \deq \diam \eu C$ and let $f : \R^d\to\R$ be convex and $L$-Lipschitz on $\eu C$.
    Then,~\ref{eq:CoGM} satisfies
    \begin{align*}
        \min_{k=0,1,\dotsc,N-1} f(x_k) - f_\star
        &\le DL\,\bigl(1 - \frac{1}{\e}\bigr){\bigsp}^{N/d}\,.
    \end{align*}
\end{thm}
\begin{proof}
    By the argument above, at iteration $N$, $\vol(\eu C_N)/\vol(\eu C) \le \lambda^N$, where we can take $\lambda = 1-1/\e$.
    Now consider the set $\hat{\eu C} \deq (1-t)\,x_\star + t\,\eu C$, where we choose $t$ so that $\vol(\hat{\eu C}) > \vol(\eu C_N)$; since $\vol(\hat{\eu C}) = t^d \vol(\eu C)$, we can take any $t > \lambda^{N/d}$.
    With this choice, there exists $\hat x \in \hat{\eu C} \setminus \eu C_N$.
    Since $\hat x \notin \eu C_N$, there exists $k < N$ for which
    \begin{align*}
        f(x_k) - f_\star
        &\le f(\hat x) - f_\star
        \le t\,\bigl(\sup_{\eu C} f - f_\star\bigr)
        \le tDL\,.
    \end{align*}
    The result follows by letting $t \searrow \lambda^{N/d}$.
\end{proof}

Thus, in principle, we can achieve $f(\hat x) - f_\star \le \varepsilon$ in $O(d\log(DL/\varepsilon))$ iterations.
Compared to~\autoref{thm:psd}, this result incurs only a logarithmic dependence on the ratio $DL/\varepsilon$, i.e., we can output a high-accuracy solution even for poorly conditioned convex sets.
On the other hand, it incurs dependence on the dimension.

Recall that the lower bound for convex smooth optimization (\autoref{thm:cvx_smooth_lb}) only applies in dimension $d \gtrsim \sqrt{\beta R^2/\varepsilon}$.
The center of gravity method explains why: a $\beta$-smooth function over a ball of radius $R$ is also $\beta R$-Lipschitz, so~\autoref{thm:CoGM} yields an oracle complexity of $O(d\log(\beta R^2/\varepsilon))$ in this case.
This is smaller than the lower bound of $\Omega(\sqrt{\beta R^2/\varepsilon})$ in~\autoref{thm:cvx_smooth_lb} when $d \ll \sqrt{\beta R^2/\varepsilon}/\log(\beta R^2/\varepsilon)$, so a lower bound construction cannot exist in any smaller dimension.\footnote{This discussion is not entirely correct since~\autoref{thm:cvx_smooth_lb} only applies to gradient span algorithms, which does not cover~\ref{eq:CoGM}. However, the moral of the discussion is true for bona fide oracle lower bounds.}
Note also that for convex quadratic minimization, there are methods which find the minimizer in $d$ queries (e.g.,~\autoref{thm:termination_CG} for~\ref{eq:CG}); the center of gravity method almost achieves this guarantee for general convex optimization.

Toward making cutting plane methods more practical, a famous example is the \emph{ellipsoid method}.
In this scheme, we take each set $\eu C_n$ to be an ellipsoid,
\begin{align}\label{eq:ellipsoid}
    \eu C_n
    &= \{x\in\R^d : \langle x-x_n, \Sigma_n^{-1}\,(x-x_n)\rangle\le 1\}\,.
\end{align}
At the next iteration, we must find a new ellipsoid $\eu C_{n+1}$ such that
\begin{align}\label{eq:new_ellipsoid}
    \eu C_{n+1} \supseteq \eu C_n \cap \{x\in\R^d : \langle p_n, x \rangle \le \langle p_n, x_n \rangle\}\,.
\end{align}
Here, we use the following geometric lemma (\autoref{qu:ellipsoid}).

\begin{lem}[ellipsoid]\label{lem:ellipsoid}
    Let $\eu C_n$ be the ellipsoid~\eqref{eq:ellipsoid} and let $p_n \in \R^d$ be a non-zero vector.
    Define $\eu C_{n+1} \deq \{x\in\R^d : \langle x-x_{n+1}, \Sigma_{n+1}^{-1}\,(x-x_{n+1})\rangle \le 1\}$, where
    \begin{align*}
        x_{n+1}
        &\deq x_n - \frac{1}{d+1}\, \frac{\Sigma_n p_n}{\sqrt{\langle p_n, \Sigma_n\,p_n\rangle}}\,, \\
        \Sigma_{n+1}
        &\deq \frac{d^2}{d^2 - 1}\,\Bigl( \Sigma_n - \frac{2}{d+1}\, \frac{\Sigma_n p_n p_n^\T \Sigma_n}{\langle p_n, \Sigma_n\,p_n\rangle}\Bigr)\,.
    \end{align*}
    Then, for $d > 1$, $\eu C_{n+1}$ satisfies~\eqref{eq:new_ellipsoid} and
    \begin{align*}
        \frac{\vol(\eu C_{n+1})}{\vol(\eu C_n)}
        &= \sqrt{\frac{d-1}{d+1}\, \bigl(\frac{d^2}{d^2-1}\bigr){\bigsp}^d}
        = 1 - \Omega\bigl( \frac{1}{d}\bigr)\,.
    \end{align*}
\end{lem}

By following the proof of~\autoref{thm:CoGM}, replacing $\lambda$ by $1-\Omega(1/d)$, one obtains the same guarantee as for~\ref{eq:CoGM} but with iteration count $O(d^2 \log(LD/\varepsilon))$.
(See~\autoref{qu:cutting_plane} for details.)
Thus, the cost of obtaining an implementable version of the center of gravity method is a larger query complexity.
Naturally, there have been numerous follow-up works in the field which aim at achieving the best of both worlds.

\subsection{Lower bounds}

In this section, we study lower bounds for convex non-smooth optimization.

\begin{thm}[lower bound for convex, non-smooth minimization]\label{thm:cvx_nonsmooth_lb}
    For any $x_0\in\R^d$, $d > N$, and $L, R > 0$, there exists a convex and $L$-Lipschitz function $f$ over $\msf B(x_\star,R)$ such that $x_0 \in \msf B(x_\star,R)$ and for any gradient span algorithm,
    \begin{align*}
        f(x_N) - f_\star
        &\gtrsim \frac{LR}{\sqrt N}\,.
    \end{align*}
\end{thm}
\begin{proof}
    Assume $x_0 = 0$ and define the function $f : \R^d \to \R$ by
    \begin{align*}
        f(x) \deq \gamma \max_{i\in [d]} x[i] + \frac{\alpha}{2}\,\norm x^2\,,
    \end{align*}
    where $\alpha, \gamma > 0$ are to be chosen.
    Note that this function is Lipschitz with constant $\gamma + \alpha\,(\norm{x_\star} + R)$.
    Also, if $I_\star(x) \deq \{i\in [d] : x[i] = \max_{j\in [d]} x[j]\}$, then from~\autoref{lem:subdiff_max},
    \begin{align*}
        \partial f(x) = \alpha x + \gamma \conv\{e_i : i\in I_\star(x)\}\,.
    \end{align*}
    The optimal point is $x_\star[k] = -\gamma/(\alpha d)$ for $k\in [d]$, by checking that $0 \in \partial f(x_\star)$.
    Thus, $\norm{x_\star} = \gamma/(\alpha\sqrt d)$ and the Lipschitz constant is at most $2\gamma + \alpha R$.

    We take a subgradient oracle which, given a point $x$, outputs $\alpha x + \gamma e_i \in \partial f(x)$, where $i = \min I_\star(x)$ is the \emph{first} coordinate of $x$ that achieves the maximum.
    From this property, it is straightforward to show via induction that $x_n \in \eu V_n$ for all $n$, where $\eu V_n$ is the subspace from the proof of~\autoref{thm:cvx_smooth_lb}.

    Since $d > N$, it follows that $f(x_N) \ge 0$.
    On the other hand,
    \begin{align*}
        f_\star
        &= f(x_\star)
        = - \frac{\gamma^2}{\alpha d} + \frac{\gamma^2}{2\alpha d}
        = - \frac{\gamma^2}{2\alpha d}\,.
    \end{align*}
    We set $d=N+1$, $\gamma = L/4$, $\alpha = \gamma/(R\sqrt d)$ (to ensure that $\norm{x_0 - x_\star} \le R$), which leads to a Lipschitz constant of $L/2 + L/(4\sqrt d) \le L$.
    It yields
    \begin{align*}
        f(x_N) - f_\star
        &\ge -f(x_\star)
        \gtrsim \frac{LR}{\sqrt N}\,. \qedhere
    \end{align*}
\end{proof}

Note that this matches the guarantee of~\ref{eq:PSD} (\autoref{thm:psd}), so projected subgradient descent is \emph{optimal} in the non-smooth setting.
In other words, without smoothness, there is no acceleration phenomenon.

There is a version of~\autoref{thm:cvx_nonsmooth_lb} in the strongly convex case (\autoref{qu:strcvx_nonsmooth_lb}).

\begin{thm}[lower bound for strongly convex, non-smooth minimization]\label{thm:strcvx_nonsmooth_lb}
    For any $x_0\in\R^d$, $d > N$, and $\alpha, L > 0$, there exists $R > 0$ and an $\alpha$-convex and $L$-Lipschitz function $f$ over $\msf B(x_\star,R)$ such that $x_0 \in \msf B(x_\star,R)$ and for any gradient span algorithm,
    \begin{align*}
        f(x_N) - f_\star
        &\gtrsim \frac{L^2}{\alpha N}\,.
    \end{align*}
\end{thm}

Next, in the low-dimensional setting, the following lower bound holds.

\begin{thm}[lower bound for convex, non-smooth minimization II]\label{thm:cutting_plane_lb}
    The oracle complexity of minimizing convex, $L$-Lipschitz functions over ${[-R,R]}^d$ to accuracy $\varepsilon$ is at least $\Omega(d \log(LR/\varepsilon))$.
\end{thm}

This shows that~\ref{eq:CoGM} is optimal as well.
Actually, we do not prove~\autoref{thm:cutting_plane_lb}; instead, we focus on the related but harder task of feasibility.

\begin{defn}
    Let $0 < \delta < R$.
    Let $\eu C \subseteq {[-R,R]}^d$ be a closed convex set such that there exists a ball $\msf B(x_\star, \delta) \subseteq \eu C$.
    The \textbf{feasibility problem} with parameters $(\delta,R)$ is the problem of outputting a point in $\interior \eu C$, given access to a separation oracle.
    Namely, given a point $x\in\R^d$, the separation oracle either reports that $x \in \eu C$, or it outputs a non-zero vector $p\in\R^d$ such that $\sup_{\eu C}{\langle p, \cdot \rangle} \le \langle p, x \rangle$.
\end{defn}

If one can solve the feasibility problem, then one can solve the convex Lipschitz minimization problem.
Indeed, given a convex, $L$-Lipschitz function $f$ over ${[-R,R]}^d$, suppose for the sake of argument that we know the optimal value $f_\star$.
Consider the feasibility problem for set $\eu C \deq \{f - f_\star \le \varepsilon\}$.
For $x_\star \deq \argmin_{{[-R,R]}^d} f$, we claim that $\msf B(x_\star, \varepsilon/L) \subseteq \eu C$; indeed this follows from $L$-Lipschitzness.\footnote{Actually this is not exactly true because $x_\star$ could lie near the boundary of ${[-R,R]}^d$. To fix this, one could instead look for a minimizer of $f$ over $\eu C' \deq {[-R+\delta, R-\delta]}^d$, i.e., define $x_{\delta,\star}$ to be a minimizer over this smaller cube and set $\eu C\deq \eu C' \cap \{f - f(x_{\delta,\star})\le \varepsilon\}$. If we take $\delta = \varepsilon/(L\sqrt d)$, then by $L$-Lipschitzness we see that any point in $\eu C$ is a $2\varepsilon$-minimizer of $f$ over ${[-R,R]}^d$, and now $\msf B(x_{\delta,\star}, \delta) \subseteq \eu C$. This does not really change the argument.}
Also, the subgradient oracle for $f$ yields a separation oracle for $\eu C$.
Thus, solving the feasibility problem for $\eu C$ with parameters $(\varepsilon/L, R)$ yields an $\varepsilon$-solution to the problem of minimizing $f$.

Since the feasibility problem is harder, the following theorem is weaker than~\autoref{thm:cutting_plane_lb}.
However, it is easier to prove, and it contains most of the main ideas.

\begin{thm}[lower bound for feasibility]
    For any deterministic algorithm, the feasibility problem with parameters $(\varepsilon,R)$ requires $\Omega(d\log(R/\varepsilon))$ queries.
\end{thm}
\begin{proof}
    We play a game with the algorithm.
    Suppose that the algorithm has chosen query points $x_1,\dotsc,x_n$ thus far.
    Our goal is to choose a vector $p_n${---}which is supposed to correspond to the output of a separation oracle{---}and we provide the algorithm with this vector, which it then uses to produce a new point $x_{n+1}$ and so on.
    Simultaneously, we also maintain a sequence of convex bodies (actually, boxes) $\eu C_0,\eu C_1,\dotsc,\eu C_N$.

    At the end of the game, the algorithm has produced points $x_1,\dotsc,x_N$, and we have produced vectors $p_1,\dotsc,p_N$.
    By itself, this is not yet meaningful; the algorithm is not designed to produce useful results, \emph{unless} $p_1,\dotsc,p_N$ are valid outputs from a separation oracle corresponding to a convex body $\eu C$ satisfying the assumptions of the feasibility problem.
    So, we aim to choose $p_1,\dotsc,p_N$ so that this holds with $\eu C = \eu C_N$.
    Now, we can use the following post hoc reasoning: \emph{had we} run the algorithm with the separation oracle for $\eu C_N$ from the outset, then the algorithm would have output the same sequence of points $x_1,\dotsc,x_N$, because it is deterministic, so this construction yields a valid lower bound (i.e., it requires more than $N$ iterations to solve the feasibility problem).
    This proof technique is known as the method of \emph{resisting oracles}, and its main drawback is that it does not apply to randomized algorithms.\footnote{Lower bounds for randomized algorithms require the use of information theory.}

    Let us instantiate the resisting oracle for the feasibility problem.
    At each iteration $n$, the convex body $\eu C_n$ is the box $\{x\in\R^d : a_n \le x \le b_n\}$; here, $a_n,b_n\in\R^d$ and the inequality is interpreted pointwise.
    We start with $a_0 = -R\mb 1_d$, $b_0 = +R\mb 1_d$, where $\mb 1_d$ is the all-ones vector; thus, $\eu C_0 = {[-R,R]}^d$.

    When the algorithm makes the first query $x_1$, we update the box by cutting it in half, based on the first coordinate of $x_1$.
    Namely, if $x_1[1] \le 0$, we set $a_1[1] = 0$, and $a_1[k] = a_0[k]$ for all $k > 1$; we output the separating vector $-e_1$.
    If $x_1[1] \ge 0$, we set $b_1[1] = 0$ and $b_1[k] = b_0[k]$ for all $k > 1$; we output the separating vector $+e_1$.
    In either case, $\vol(\eu C_1) = \frac{1}{2}\vol(\eu C_0)$ and $x_1 \notin \interior \eu C_1$.

    When the algorithm makes the second query $x_2$, we repeat this procedure except that we cut the box in half along the second coordinate.
    We continue in this fashion, cycling through the coordinates.

    Let $c_n$ denote the center of $\eu C_n$.
    We now claim that for each $n$, $\msf B(c_n, r_n) \subseteq \eu C_n$, where $r_n = (R/2)\,{(1/2)}^{n/d}$.
    Indeed, this is true for $n=0$.
    Also, for $n=ad$ for integer $a$, each side of the box has length $R\,{(1/2)}^a$, so the result is true in this case too.
    Finally, for $n=ad+b$, we have $\msf B(c_{(a+1)d}, R/2^{a+1}) \subseteq \eu C_{(a+1)d} \subseteq \eu C_n$ hence $\msf B(c_n, R/2^{a+1}) \subseteq \eu C_n$, and we note that $R/2^{a+1} \ge (R/2)\,{(1/2)}^{n/d}$.

    The resisting oracle construction succeeds up to iteration $N$ provided that $\eu C_N$ contains a ball of radius $\varepsilon$.
    It therefore suffices to have $(R/2)\,{(1/2)}^{N/d} \ge \varepsilon$, i.e., $N \gtrsim d\log(R/\varepsilon)$.
\end{proof}

\subsection*{Exercises}

\begin{question}\label{qu:lsc}\mbox{}
    \begin{enumerate}
        \item Prove that a function $f$ is lower semicontinuous if and only if for all $c\in\R$, the level set $\{f \le c\}$ is closed.
        \item Prove that a supremum of lower semicontinuous functions is lower semicontinuous.
        \item Show that the function defined in~\eqref{eq:pathology} is lower semicontinuous if and only if $\phi = 0$.
    \end{enumerate}
\end{question}

\begin{question}\label{qu:subdiff_grad}
    Prove that if $f$ is differentiable at $x_0 \in \interior \dom f$, then $\partial f(x_0) = \{\nabla f(x_0)\}$.
\end{question}

\begin{question}\label{qu:subdiff_lip}
    Let $f : \R^d\to\R$ be continuous and convex on a convex set $\eu C$.
    Prove that $f$ is Lipschitz continuous over $\eu C$ with constant $L$ if and only if for every $x_0 \in \interior \eu C$ and every $p \in \partial f(x_0)$, we have $\norm p \le L$.
\end{question}

\begin{question}\label{qu:subdiff_calcs}\mbox{}
    \begin{enumerate}
        \item Compute the subdifferential of the Euclidean norm $\norm \cdot$.
        \item Let $\lambda_{\max} : \mb S^d \to \R$ be the maximum eigenvalue function, and let $A \in \mb S^d$.
            Show that if $v$ is a unit eigenvector corresponding to the largest eigenvalue of $A$, then $vv^\T \in \partial \lambda_{\max}(A)$.
    \end{enumerate}
\end{question}

\begin{question}\label{qu:psd_str_cvx}
    Assume that $f$ is $\alpha$-strongly convex and $L$-Lipschitz continuous over the closed convex set $\eu C$.
    Prove that for~\ref{eq:PSD},
    \begin{align*}
        f(\bar x_N) - f_\star
        &\le \frac{\alpha}{2\,\{{(1-\alpha h/L)}^{-N}-1\}}\,\norm{x_0 - x_\star}^2 + \frac{Lh}{2}\,,
    \end{align*}
    where $\bar x_N$ is a suitable averaged iterate.
    Deduce that by setting $h = \varepsilon/L$, one can achieve $f(\bar x_N) - f_\star \le \varepsilon$ in $O(\frac{L^2}{\alpha\varepsilon} \log(\frac{\alpha R^2}{\varepsilon}))$ iterations (compared with $O(L^2 R^2/\varepsilon^2)$ iterations, as implied by~\autoref{thm:psd}).

    Also, show that under these assumptions, $\norm{x_0 - x_\star} \le 2L/\alpha$.
\end{question}

\begin{question}\label{qu:psd_str_cvx_sharp}
    This exercise shaves off a logarithmic factor from the previous exercise.
    Assume that $f$ is $\alpha$-strongly convex and $L$-Lipschitz continuous over the closed convex set $\eu C$.
    Consider the iteration
    \begin{align*}
        x_{n+1} = \Pi_{\eu C}(x_n - h_n p_n)\,, \qquad p_n \in \partial f(x_n)\,,
    \end{align*}
    where $h_n \deq 2/(\alpha\,(n+1))$.
    Show that
    \begin{align*}
        f(\bar x_N) - f_\star
        &\le \frac{2L^2}{\alpha\,(N+1)}\,.
    \end{align*}
    where $\bar x_N$ is a suitably averaged iterate.
    Thus, we can achieve $f(\bar x_N) - f_\star \le \varepsilon$ in $O(\frac{L^2}{\alpha\varepsilon})$ iterations, without any extraneous logarithmic factors.
\end{question}



\begin{question}\label{qu:cutting_plane}
    The analysis of the ellipsoid method (and general cutting plane schemes) presents an additional difficulty: since the next set $\eu C_{n+1}$ is only chosen to be a superset of $\eu C_n \cap \{\langle p_n, \cdot \rangle \le \langle p_n, x_n \rangle\}$, it is not guaranteed that $\eu C \subseteq \eu C_n$ for all $n$; in particular, the chosen point $x_n$ may lie outside of $\eu C$.

    Assume that we have access to a separation oracle for $\eu C$: given a point $x\notin\eu C$, the oracle outputs a non-zero vector $p\in\R^d$ such that $\sup_{\eu C}{\langle p, \cdot\rangle} \le \langle p, x \rangle$.
    Modify the cutting plane method as follows: if a chosen point $x_n$ does not lie in $\eu C$, then let $p_n$ be vector that separates $x_n$ from $\eu C$ and instead update $\eu C_{n+1}$ to be a superset of $\eu C_n \cap \{\langle p_n, \cdot \rangle \le \langle p_n, x_n \rangle\}$.
    We also allow $\eu C_0 \supseteq \eu C$, so that $x_0$ is not necessarily feasible either.
    Prove that if the sets are chosen so that $\vol(\eu C_{n+1})/\vol(\eu C_n)\le \lambda < 1$ for all $n$, then the following assertions hold.
    \begin{enumerate}
        \item If $\vol(\eu C_N) < \vol(\eu C)$, then there exists $n < N$ with $x_n \in \eu C$.
        \item If $\vol(\eu C_N) < \vol(\eu C)$, then there exists $n < N$ with $x_n \in \eu C$ \emph{and}
            \begin{align*}
                f(x_n) - f_\star \le DL\lambda^{N/d}\,\bigl( \frac{\vol\eu C_0}{\vol\eu C}\bigr){\bigsp}^{1/d}\,.
            \end{align*}
            \emph{Hint:} Define a sequence of sets $\eu C_0',\eu C_1',\eu C_2',\dotsc$ as follows.
            Start with $\eu C_0' = \eu C$ and $n_{-1} \deq 0$.
            For each $k \in\N$, let $n_k$ denote the first integer greater than $n_{k-1}$ for which $x_{n_k} \in \eu C$ and set $\eu C_{k+1}' \deq \eu C_k' \cap \{\langle p_{n_k}, \cdot \rangle \le \langle p_{n_k}, x_{n_k} \rangle\}$.
            Prove via induction that if $k(N)$ is the largest integer such that $n_{k(N)} \le N$, then $\eu C_{n_{k(N)}}' \subseteq \eu C_N$.
    \end{enumerate}
\end{question}

\begin{question}\label{qu:ellipsoid}
    Prove~\autoref{lem:ellipsoid}.
\end{question}


\begin{question}\label{qu:strcvx_nonsmooth_lb}
    Prove~\autoref{thm:strcvx_nonsmooth_lb}.
    (Use the same construction as in the proof of~\autoref{thm:cvx_nonsmooth_lb}, but choose the  parameters $\alpha$ and $\gamma$ differently.)
\end{question}

\section{Frank{--}Wolfe}

In order to overcome the lower bounds in the black-box setting, we must take advantage of additional structure in the problem.
The first method we study in this vein is the \emph{Frank{--}Wolfe} or \emph{conditional gradient} method~\cite{FraWol1956}.
Instead of assuming access to a projection oracle for the constraint set $\eu C$, it instead assumes access to a \emph{linear optimization oracle} (LOO) over the set $\eu C$:
\begin{align}\label{eq:LOO}\tag{$\msf{LOO}$}
    \text{Given}~p\in\R^d\,,~\text{output}~\argmin_{\eu C}{\langle p, \cdot\rangle}\,.
\end{align}
Here, we assume that $\eu C$ is compact (bounded and closed).

The oracle equivalently maximizes the convex function $-\langle p, \cdot\rangle$ over $\eu C$, so the $\argmin$ is attained at a vertex of $\eu C$.
Let us define these concepts properly.

\begin{defn}
    A point $x\in\eu C$ is called an \textbf{extreme point} or a \textbf{vertex} of $\eu C$ if there do not exist $x_0, x_1 \in\eu C$ and $t\in (0,1)$ such that $x = (1-t)\,x_0 + t\,x_1$.
\end{defn}

\begin{thm}
    Every compact convex set is the convex hull of its extreme points.
\end{thm}

For example, the set of vertices of the closed unit ball $\overline{\msf B(0,1)}$ is the sphere $\partial \msf B(0,1)$.
It follows that to implement~\eqref{eq:LOO}, it suffices to solve $\argmin_{\text{vertices of}~\eu C}{\langle p, \cdot \rangle}$.

We now present the Frank{--}Wolfe method for minimizing $f$ over $\eu C$:
\begin{align}\label{eq:FW}\tag{$\msf{FW}$}
    x_{n+1}
    &\deq (1-h_n)\,x_n + h_n\,\msf{LOO}(\nabla f(x_n))\,.
\end{align}

\begin{thm}[convergence of~\ref{eq:FW}]\label{thm:FW}
    Let $f$ be convex and $\beta$-smooth over $\eu C$.
    Let $D \deq \diam \eu C$ and $h_n = 2/(n+2)$.
    Then, for any $N\ge 1$,~\ref{eq:FW} satisfies
    \begin{align*}
        f(x_N) - f_\star
        &\le \frac{2\beta D^2}{N+1}\,.
    \end{align*}
\end{thm}
\begin{proof}
    Let $y_n \deq \msf{LOO}(\nabla f(x_n))$.
    Using $\beta$-smoothness,
    \begin{align*}
        f(x_{n+1}) - f(x_n)
        &\le \langle \nabla f(x_n), x_{n+1} - x_n \rangle + \frac{\beta}{2}\,\norm{x_{n+1}-x_n}^2 \\[0.25em]
        &\le h_n\,\langle \nabla f(x_n), y_n - x_n \rangle + \frac{\beta D^2 h_n^2}{2}
        \le h_n\,\langle \nabla f(x_n), x_\star - x_n \rangle + \frac{\beta D^2 h_n^2}{2} \\[0.25em]
        &\le -h_n\,(f(x_n) - f_\star) + \frac{\beta D^2 h_n^2}{2}\,.
    \end{align*}
    Rearranging,
    \begin{align*}
        f(x_{n+1}) - f_\star
        &\le (1-h_n)\,(f(x_n) - f_\star) + \frac{\beta D^2 h_n^2}{2}\,.
    \end{align*}
    For $h_n = 2/(n+2)$, we now prove the error bound by induction on $n$, where the base case $n=0$ follows from the inequality above.
    If the error bound holds at iteration $n$, then
    \begin{align*}
        f(x_{n+1}) - f_\star
        &\le \frac{n}{n+2}\, \frac{2\beta D^2}{n+1} + \frac{2\beta D^2}{{(n+2)}^2}
        \le \frac{2\beta D^2}{n+2}\,. \qedhere
    \end{align*}
\end{proof}

The analysis above is actually not the most natural one, since it fails to capture the affine invariance of the Frank{--}Wolfe algorithm (\autoref{qu:fw_affine_inv}).

Besides positing different oracle access than projected gradient methods, the Frank{--}Wolfe method has the appealing property of producing sparse solutions.
This connects with results known as approximate Carath\'eodory theorems.
First, let us recall the classical statement of Carath\'eodory's theorem.

\begin{thm}[Carath\'eodory]\label{thm:caratheodory}
    Let $\eu C \subseteq \R^d$ be a compact convex set and let $x\in\eu C$.
    Then, $x$ can be written as a convex combination of $d+1$ vertices of $\eu C$.
\end{thm}

Caution: in this theorem, the choice of $d+1$ vertices of course depends on $x$ itself.
If every point in $\eu C$ could be written as a convex combination of the \emph{same} $d+1$ vertices, this would say that $\eu C$ only has $d+1$ vertices at all.

Carath\'eodory's theorem says that even if a convex body has exponentially many vertices, such as the cube ${[-1,1]}^d$, any given point has a succinct representation using only $d+1$ vertices.
However, the size of the representation grows with the ambient dimension.
What happens if we relax the requirement that the representation is exact?
The following simple argument, often attributed to B.\ Maurey, shows that the size of the representation is \emph{dimension-free}, and the convex combination even uses equal weights.

\begin{thm}[approximate Carath\'eodory]
    Let $\eu C \subseteq \R^d$ be a compact convex set with diameter $D$, let $0 < \varepsilon < 1$, and let $x \in\eu C$.
    Then, there exist vertices $y_1,\dotsc,y_N \in\eu C$ with
    \begin{align*}
        \Bigl\lVert x - \frac{1}{N}\sum_{i=1}^N y_i \Bigr\rVert \le \varepsilon D\,, \qquad N \le \frac{1}{\varepsilon^2}\,.
    \end{align*}
\end{thm}
\begin{proof}
    By~\autoref{thm:caratheodory}, there exist vertices $\bar y_1,\dotsc,\bar y_{d+1} \in \eu C$ and a probability distribution $\lambda$ over $[d+1]$ such that $x = \sum_{j=1}^{d+1} \lambda_j \bar y_j$.
    Now consider the distribution $\mu = \sum_{j=1}^{d+1} \lambda_j \delta_{\bar y_j}$ and sample points $Y_1,\dotsc,Y_N\simiid \mu$.
    Note that each $Y_i$ is a vertex of $\eu C$.
    Then, since the mean of $\mu$ is $x$, the usual variance calculation shows that
    \begin{align*}
        \E\Bigl[\Bigl\lVert x - \frac{1}{N} \sum_{i=1}^N Y_i \Bigr\rVert^2\Bigr]
        &= \frac{\sum_{j=1}^{d+1} \lambda_j\,\norm{x-\bar y_j}^2}{N}
        \le \frac{D^2}{N}\,.
    \end{align*}
    Choose $N$ to make the right-hand side at most $\varepsilon^2 D^2$.
\end{proof}

The approximate Carathe\'odory theorem has implications, e.g., for controlling the covering numbers of polytopes.
But more broadly, the proof technique is quite influential and is at the root of other important developments, e.g., the existence of neural networks of small width which approximate functions in the Barron class~\cite{Bar1993Approx, Bach17BreakingCurse}.

Now comes the punchline: Frank{--}Wolfe renders the approximate Carath\'eodory theorem constructive.
Indeed, suppose that the~\ref{eq:LOO} always outputs a vertex.
After $N-1$ iterations of~\ref{eq:FW} starting from a vertex, the iterate $x_{N-1}$ is a convex combination of at most $N$ vertices.
At the same time, if we apply~\autoref{thm:FW} to the $2$-smooth function $f : z\mapsto \norm{x-z}^2$, where $x \in \eu C$ and $f_\star = 0$, we see that $\norm{x_{N-1} - x}^2 \le 4D^2/N$.

The full statement of~\autoref{thm:FW} can therefore be seen as a generalization of the approximate Carath\'eodory principle: the iterate of~\ref{eq:FW} is a sparse combination of vertices which is approximately optimal.
We next demonstrate an example in which this sparsity property is crucial.

%
%

\begin{ex}[low-rank estimation]
    Consider the nuclear norm ball
    \begin{align*}
        \eu C = \Bigl\{X \in\R^{d\times d} : \norm X_* = \sum_{i=1}^d \sigma_i(X) \le 1\Bigr\}\,.
    \end{align*}
    This constraint set often arises in low-rank matrix recovery as a convex relaxation of a rank constraint.
    Projection onto the set $\eu C$ requires projecting the singular values onto the simplex; this requires computing a full SVD, which uses $O(d^3)$ arithmetic operations.
    On the other hand, since
    \begin{align*}
        \eu C = \conv\{uv^\T : u,v \in\R^d\,, \; \norm u = \norm v = 1\}\,,
    \end{align*}
    the~\ref{eq:LOO} for $\eu C$ involves solving, for any $P \in \R^{d\times d}$,
    \begin{align*}
        \argmin_{X\in \eu C}{\langle P, X \rangle}
        &= \argmin\{\langle P, uv^\T \rangle : u,v\in\R^d\,, \; \norm u = \norm v = 1\}\,.
    \end{align*}
    Solving this amounts to computing the top singular vector of $P$, which is often implemented via power iteration at cost $O(d^2)$ per step.
    Moreover,~\ref{eq:FW} yields an $\varepsilon$-accurate solution with rank $O(1/\varepsilon)$.
\end{ex}

\subsection*{Exercises}

\begin{question}\label{qu:fw_affine_inv}
    Show that~\ref{eq:FW} is affine-invariant in the following sense.
    Let $A \in \R^{d\times d}$ be an invertible matrix.
    Show that the iterates ${\{\hat x_n\}}_{n\in\N}$ of~\ref{eq:FW} applied to the problem of minimizing $\hat x \mapsto f(A\hat x)$ over the set $A^{-1}\eu C$ are related to the iterates ${\{x_n\}}_{n\in\N}$ of~\ref{eq:FW} on the original problem via $x_n = A\hat x_n$.
\end{question}


\section{Proximal methods}\label{sec:prox}

Can we solve non-smooth problems at the same rate as smooth problems?
The black-box lower bounds say \emph{no} in general, but if the non-smooth part is ``simple'' in the sense that it admits an implementable proximal oracle, the answer becomes yes.

\begin{defn}
    Let $f : \R^d\to\R\cup\{\infty\}$.
    The \textbf{proximal oracle} for $f$ is the mapping $\prox_f : \R^d\to\R^d$ given by
    \begin{align*}
        \prox_f(y)
        \deq \argmin_{x\in\R^d}{\bigl\{ f(x) + \frac{1}{2}\,\norm{y-x}^2\bigr\}}\,.
    \end{align*}
\end{defn}

If $f$ is a regular convex function, then the optimization problem defining the proximal oracle is strongly convex, so it admits a unique minimizer by~\autoref{lem:unique_min} and~\autoref{lem:existence_min_ii}.
Note also that
\begin{align*}
    \prox_{hf}(y)
    = \argmin_{x\in\R^d}{\bigl\{ hf(x) + \frac{1}{2}\,\norm{y-x}^2\bigr\}}
    = \argmin_{x\in\R^d}{\bigl\{ f(x) + \frac{1}{2h}\,\norm{y-x}^2\bigr\}}\,,
\end{align*}
where $h > 0$ plays the role of a step size.

The value of the optimization problem defining $\prox_f$ also has a name.

\begin{defn}\label{defn:moreau_yosida}
    Let $f : \R^d\to\R\cup \{\infty\}$.
    The \textbf{Moreau{--}Yosida envelope} of $f$ with parameter $h > 0$ is the mapping $f_h : \R^d\to\R \cup \{\infty\}$ given by
    \begin{align*}
        f_h(y) \deq \inf_{x\in\R^d}{\bigl\{f(x) + \frac{1}{2h}\,\norm{y-x}^2\bigr\}}\,.
    \end{align*}
\end{defn}

\subsection{Algorithms and examples}

The proximal oracle is a regularized version of the original optimization problem.
Assuming for the moment that we can compute the proximal oracle easily, let us explore its uses for algorithm design.

The simplest algorithm is to repeatedly iterate the proximal mapping.
This is known as the \emph{proximal point method}~\cite{Roc1976PPM}.
\begin{align}\label{eq:PPM}\tag{$\msf{PPM}$}
    x_{n+1} \deq \prox_{hf}(x_n)\,.
\end{align}
Assume for the moment that $f$ is smooth and that the next point $x_{n+1}$ can be obtained from the first-order optimality condition for $\prox_{hf}$.
This leads to
\begin{align*}
    0 = \nabla f(x_{n+1}) + \frac{1}{h}\,(x_{n+1}-x_n) \iff x_{n+1} = x_n - h\,\nabla f(x_{n+1})\,.
\end{align*}
Note that this is similar to the~\ref{eq:GD} update, except that the gradient is evaluated at the subsequent point $x_{n+1}$.
In numerical analysis, we say that~\ref{eq:GD} is an \emph{explicit} discretization of the gradient flow, whereas~\ref{eq:PPM} is an \emph{implicit} discretization.
The advantage of an explicit method is ease of implementation; it does not require solving a (non-linear) system in order to perform an update.
The advantage of an implicit method is stability.

Recall that the results in \S\ref{sec:grad_flow} for~\ref{eq:GF} do not require smoothness of $f$, whereas the results in \S\ref{sec:gd} for~\ref{eq:GD} do.
(We studied the non-smooth case for~\ref{eq:GD} in \S\ref{ssec:projected_subgradient}, but it requires decreasing step sizes and averaging.)
Shortly, we shall see that~\ref{eq:PPM} is similar to~\ref{eq:GF}, in that it also does not require smoothness.

The most powerful results using the proximal oracle, however, are for the problem of \emph{composite optimization}.
Here, the goal is to minimize a sum of functions:
\begin{align*}
    \minimize\qquad F
    &\deq f + g\,.
\end{align*}
We assume that $f$ is smooth and that $g$ is non-smooth.

\begin{ex}[LASSO as composite optimization]\label{ex:LASSO_comp}
    The computation of the LASSO estimator from~\autoref{ex:LASSO} is the canonical example of composite optimization, where
    \begin{align*}
        f : \theta \mapsto \frac{1}{2n} \sum_{i=1}^n {(Y_i - \langle \theta, X_i \rangle)}^2\,, \qquad g : \theta \mapsto \lambda\,\norm \theta_1\,.
    \end{align*}

    In this example, the non-smooth part is particularly simple, so we can compute its proximal oracle in closed form.
    First, note that it is coordinate-wise decomposable:
    \begin{align*}
        \prox_{\lambda\,\norm\cdot_1}(y)
        &= \argmin_{x\in\R^d}{\bigl\{\lambda\,\norm x_1 + \frac{1}{2} \,\norm{y-x}^2\bigr\}} \\
        &= \sum_{i=1}^d \Bigl(\argmin_{x[i] \in \R}{\bigl\{\lambda\,\abs{x[i]} + \frac{1}{2}\,{(y[i] - x[i])}^2\bigr\}} \Bigr)\,e_i\,.
    \end{align*}
    Therefore, it suffices to solve the problem in dimension one.
    A direct computation (see~\autoref{qu:soft_thresh}) then yields
    \begin{align*}
        \prox_{\lambda\,\abs\cdot}(y)
        &= {(\abs y - \lambda)}_+\sgn y
        \eqqcolon \thresh_\lambda(y)
    \end{align*}
    where ${(\cdot)}_+ \deq \max\{0, \cdot\}$ denotes the positive part.
    The operator $\thresh_\lambda$, known as the \emph{soft thresholding operator} (\autoref{fig:soft_thresh}), reduces the magnitude of its input by $\lambda$, or to $0$ if the original magnitude is less than $\lambda$.
    The proximal operator for $\lambda\,\norm \cdot_1$ simply applies $\thresh_\lambda$ to each coordinate.
\end{ex}

\begin{figure}
    \centering
    \includegraphics[width=0.6\textwidth]{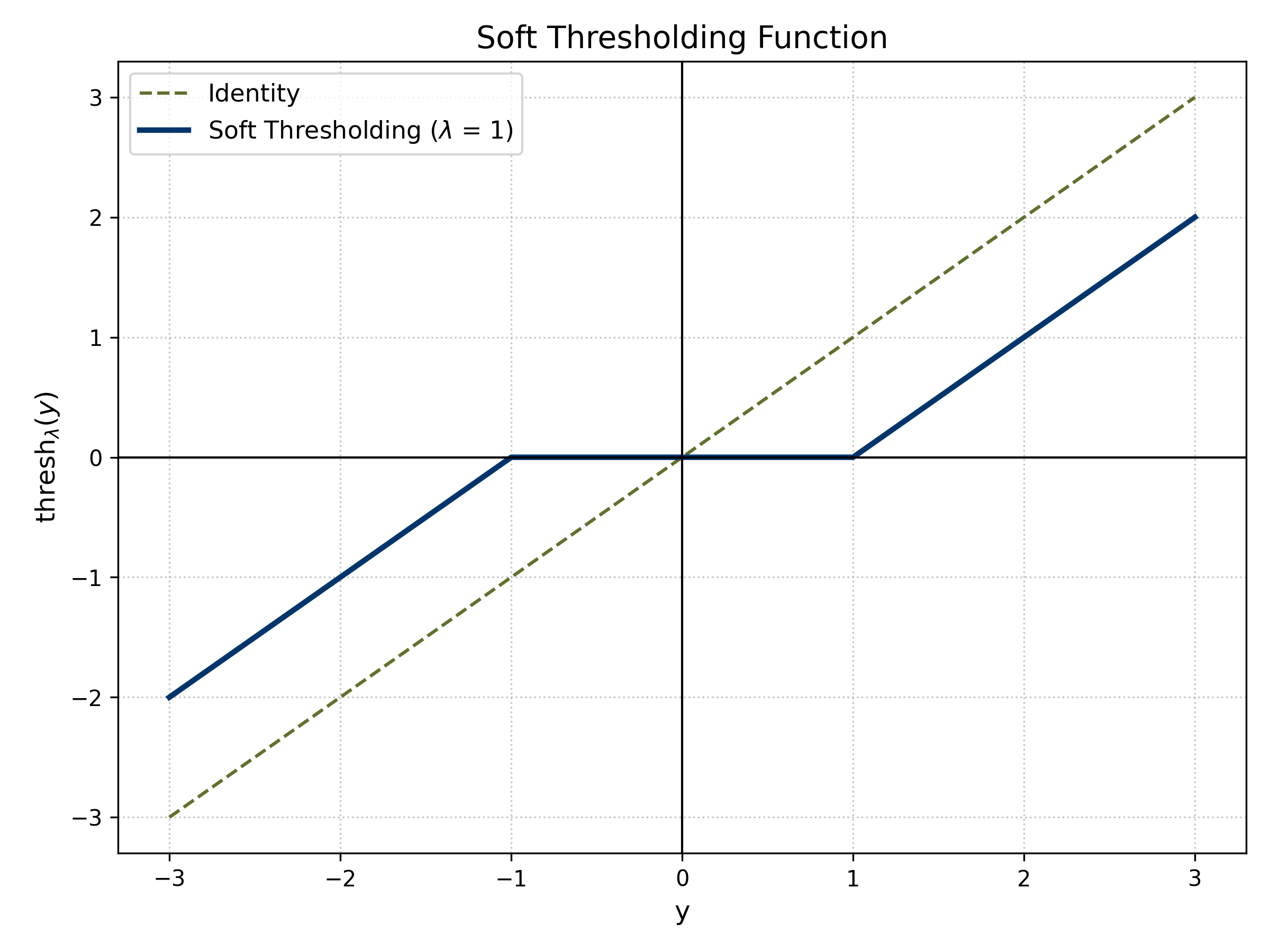}
    \caption{\footnotesize The soft thresholding operator.}\label{fig:soft_thresh}
\end{figure}

\begin{ex}[constrained optimization as composite optimization]
    Consider the problem of minimizing a smooth function $f$ over a closed convex set $\eu C$.
    We can also treat this as composite optimization with
    \begin{align*}
        g = \chi_{\eu C}\,.
    \end{align*}
    (Recall the convex indicator defined in~\eqref{eq:cvx_indicator}.)
    In this case, the proximal oracle for $g$ is
    \begin{align*}
        \prox_{h\chi_{\eu C}}(y)
        &= \argmin_{x\in\R^d}{\bigl\{\chi_{\eu C}(x) + \frac{1}{2h}\,\norm{y-x}^2\bigr\}}
        = \argmin_{x\in\eu C}{\bigl\{ \frac{1}{2h}\,\norm{y-x}^2\bigr\}}
        = \Pi_{\eu C}(y)\,.
    \end{align*}
    So, the proximal oracle for $\chi_{\eu C}$ is the projection oracle for $\eu C$.
\end{ex}

The above examples motivate the assumption that we have access to the proximal oracle for the non-smooth part $g$.
Further examples of computable proximal oracles can be found on the website \href{https://proximity-operator.net}{\textsf{proximity-operator.net}}.

The algorithm we consider in this context is known as \emph{proximal gradient descent}.
\begin{align}\label{eq:PGD}\tag{$\msf{PGD}$}
    x_{n+1}
    &\deq \argmin_{x\in\R^d}{\bigl\{f(x_n) + \langle \nabla f(x_n), x-x_n \rangle + g(x) + \frac{1}{2h}\,\norm{x-x_n}^2\bigr\}}\,.
\end{align}
In other words, we take the objective function $F = f+g$ and linearize only the smooth part.
The update can be rewritten as follows.
By completing the square,
\begin{align*}
    x_{n+1}
    &= \argmin_{x\in\R^d}{\bigl\{g(x) + \frac{1}{2h}\, \norm{x-x_n + h\,\nabla f(x_n)}^2\bigr\}}
    = \prox_{hg}(x_n - h\,\nabla f(x_n))\,.
\end{align*}
This corresponds to taking an explicit step on $f$, followed by an implicit step on $g$.
It is not obvious that this algorithm converges to $x_\star$, the minimizer of $F = f+g$.
However, note that if $g$ is differentiable, then
\begin{align*}
    x_{n+1}
    &= x_n - h\,\nabla f(x_n) - h\,\nabla g(x_{n+1})\,.
\end{align*}
If $x_n = x_\star$, then $x_{n+1} = x_\star$ is the solution since $0 = \nabla F(x_\star) = \nabla f(x_\star) + \nabla g(x_\star)$.
Thus, provided that $f$ and $g$ are convex and differentiable, $x_\star$ is the unique fixed point.

For the LASSO problem, the iteration reads
\begin{align*}
    x_{n+1}
    &= \thresh_{\lambda h}(x_n - h\,\nabla f(x_n))\,.
\end{align*}
In the literature, this is known as the \emph{iterative shrinking-thresholding algorithm} (ISTA).
For constrained optimization, proximal gradient descent is projected gradient descent.

\subsection{Convergence analysis}

We study the convergence of~\ref{eq:PGD}, since it includes~\ref{eq:PPM} as a special case (take $f=0$).

\begin{thm}[convergence of~\ref{eq:PGD}]\label{thm:PGD}
    Let $f$ be $\alpha_f$-convex and $\beta_f$-smooth, and let $g$ be $\alpha_g$-convex.
    Let the step size $h$ satisfy $h \le 1/\beta_f$, let $x^+$ denote the next iterate of~\ref{eq:PGD} started from $x$, and let $y\in\R^d$.
    Then,
    \begin{align}\label{eq:PGD_one_step}
        (1+\alpha_g h)\,\norm{y-x^+}^2
        &\le (1-\alpha_f h)\,\norm{y-x}^2 - 2h\,(F(x^+) - F(y))\,.
    \end{align}
    In particular, if we set $y = x_\star$ and iterate, it yields
    \begin{align*}
        F(x_N) - F_\star
        &\le \frac{\alpha_f + \alpha_g}{2\,(\lambda_h^{-N}-1)}\,\norm{x_0 - x_\star}^2\,,
    \end{align*}
    where $\lambda_h \deq (1-\alpha_f h)/(1+\alpha_g h)$.
\end{thm}
\begin{proof}
    Let $\psi_x$ denote the objective function in the definition of~\ref{eq:PGD}.
    Then, $\psi_x$ is $(\alpha_g + 1/h)$-strongly convex with minimizer $x^+$, so by the quadratic growth inequality,
    \begin{align*}
        \psi_x(y)
        &\ge \psi_x(x^+) + \frac{\alpha_g + 1/h}{2}\,\norm{y-x^+}^2\,.
    \end{align*}
    On one hand, by $\alpha_f$-convexity,
    \begin{align*}
        \psi_x(y)
        &= f(x) + \langle \nabla f(x), y-x\rangle + g(y) + \frac{1}{2h}\,\norm{y-x}^2
        \le F(y) + \frac{1/h-\alpha_f}{2}\,\norm{y-x}^2\,.
    \end{align*}
    On the other hand, by $\beta_f$-smoothness,
    \begin{align*}
        \psi_x(x^+)
        &= f(x) + \langle \nabla f(x), x^+ - x \rangle + g(x^+) + \frac{1}{2h}\,\norm{x^+-x}^2 \\
        &\ge F(x^+) + \frac{1/h - \beta_f}{2}\,\norm{x^+-x}^2
        \ge F(x^+)\,.
    \end{align*}
    Combining these inequalities and rearranging,
    \begin{align*}
        (1+\alpha_g h)\,\norm{y-x^+}^2
        &\le (1-\alpha_f h)\,\norm{y-x}^2 -2h\,(F(x^+) - F(y))\,.
    \end{align*}

    Note that by taking $y = x$, it yields the descent property
    \begin{align*}
        F(x^+) - F(x)
        &\le -\frac{1+\alpha_g h}{2h}\,\norm{x-x^+}^2
        \le 0\,.
    \end{align*}
    The final bound follows from~\autoref{lem:discrete_gronwall} and algebra.
\end{proof}

Refined analyses of~\ref{eq:PPM} are presented in~\autoref{qu:ppm_sharp},~\autoref{cor:contract_ppm}, and~\autoref{qu:ppm_pl}.

The key feature of~\autoref{thm:PGD} is that it essentially recovers the \emph{smooth} rate for~\ref{eq:GD} despite the presence of non-smoothness in the objective.
Thus, for the LASSO problem (\autoref{ex:LASSO_comp}), we can solve it as quickly as if it were a smooth problem via ISTA\@.

Moreover, the one-step inequality~\eqref{eq:PGD_one_step} is the~\ref{eq:PGD} analogue of the inequality~\eqref{eq:evi} which holds for~\ref{eq:GD}, and in turn,~\eqref{eq:evi} is the only property of~\ref{eq:GD} which plays a role in the proof of Nesterov acceleration (\autoref{thm:AGD}); the remainder of the proof is purely algebraic.
This naturally leads to an accelerated algorithm for composite optimization.

Starting with $x_{-1} = x_0$, consider
\begin{align}\label{eq:APGD}\tag{$\msf{APGD}$}
    x_{n+1}
    &\deq \msf{PGD}_{F,1/\beta}(x_n + \theta_n\,(x_n - x_{n-1}))\,,
\end{align}
where $\msf{PGD}_{F,1/\beta}$ denotes one step of~\ref{eq:PGD} on $F = f+g$ with step size $h=1/\beta$.

\begin{thm}[convergence of~\ref{eq:APGD}]
    Let $f$ be convex and $\beta$-smooth, and let $g$ be convex.
    Define the sequence: $\lambda_0 \deq 0$ and $\lambda_{n+1} \deq \frac{1}{2}\,(1+\sqrt{1+4\lambda_n^2})$ for $n\in\N$.
    Set $\theta_n \deq (\lambda_n - 1)/\lambda_{n+1}$.
    Then,~\ref{eq:APGD} satisfies
    \begin{align*}
        F(x_N) - F_\star
        &\le \frac{2\beta\,\norm{x_0 - x_\star}^2}{N^2}\,.
    \end{align*}
\end{thm}

When applied to LASSO\@, this algorithm is known as \emph{fast ISTA} or \emph{FISTA}~\cite{BecTeb09FISTA}.
Rates in the strongly convex setting can be obtained from the reduction in~\autoref{lem:weakcvx_to_strcvx}.

\subsection*{Exercises}

\begin{question}\label{qu:soft_thresh}
    Verify the computation of $\prox_{\lambda\,\abs\cdot}$ in~\autoref{ex:LASSO_comp}.
\end{question}

\begin{question}\label{qu:ppm_noncvx}
    Prove that even for non-convex $f$, as long as $x^+ \deq \prox_{hf}(x)$ is well-defined,
    \begin{align*}
        f(x^+) - f_\star
        &\le \frac{1}{2h}\,\norm{x-x_\star}^2\,.
    \end{align*}
    Thus, if we can implement~\ref{eq:PPM} for arbitrarily large step sizes $h > 0$, we can solve non-convex optimization.
\end{question}

\begin{question}\label{qu:ppm_sharp}
    To avoid technical difficulties, assume that $f$ is convex and differentiable everywhere.
    Show that for the~\ref{eq:PPM},~\eqref{eq:PGD_one_step} can be refined as follows: for all $y\in\R^d$,
    \begin{align*}
        \norm{x^+ - y}^2 \le \norm{x-y}^2 - 2h\,(f(x^+) - f(y)) - h^2\,\norm{\nabla f(x^+)}^2\,.
    \end{align*}
    Next, in analogy to~\autoref{qu:gf_sharp}, define the Lyapunov function
    \begin{align*}
        \ms L_n
        &\deq n^2 h^2\,\norm{\nabla f(x_n)}^2 + 2nh\,(f(x_n) - f_\star) + \norm{x_n - x_\star}^2\,,
    \end{align*}
    where ${\{x_n\}}_{n\in\N}$ are the iterates of~\ref{eq:PPM}, and show that $\ms L_{n+1} \le \ms L_n$.
    (Use the fact that~\ref{eq:PPM} is contractive, see~\autoref{cor:contract_ppm}.)
    Deduce the bounds
    \begin{align*}
        \norm{\nabla f(x_N)}
        &\le \frac{\norm{x_0 - x_\star}}{Nh}\,, \qquad f(x_N) - f_\star \le \frac{\norm{x_0 - x_\star}^2}{4Nh}\,.
    \end{align*}
    Observe that if $h\searrow 0$ while $Nh\to t$, it recovers the results of~\autoref{qu:gf_sharp}.
\end{question}

\section{Fenchel duality}

In this section, we study a notion of duality for convex functions.

\begin{defn}
    Let $f : \R^d\to\R \cup \{\infty\}$ be proper ($\dom f \ne \varnothing$).
    The \textbf{convex conjugate} or \textbf{Fenchel{--}Legendre conjugate} of $f$ is the function $f^* : \R^d\to\R \cup \{\infty\}$ defined by
    \begin{align*}
        f^*(y)
        &\deq \sup_{x\in\R^d}{\{\langle x, y \rangle - f(x)\}}\,.
    \end{align*}
\end{defn}

For any proper function $f$, the conjugate $f^*$ is always \emph{convex} and \emph{lower semicontinuous}, since it is a supremum of affine functions.
Conversely, if $f$ is a regular convex function (thus: proper, convex, and lower semicontinuous), then $f = f^{**}$ (\autoref{thm:double_conj}).

\begin{ex}\label{ex:cvx_conj}
    The verification of these examples is left as~\autoref{qu:cvx_conj_ex}.
    \begin{enumerate}
        \item If $f(x) = \frac{1}{2} \,\langle x, A\,x\rangle$ where $A \succ 0$, then $f^*(y) = \frac{1}{2}\,\langle y, A^{-1}\,y\rangle$.
        \item If $f(x) = \abs x^p/p$ for $p > 1$ and $x \in \R$, then $f^*(y) = \abs y^q/q$ where $1/p + 1/q = 1$.
        \item Let $\opnorm \cdot$ denote a norm over $\R^d$ (not necessarily Euclidean), and let $\opnorm \cdot_*$ denote the dual norm: $\opnorm y_* \deq \sup\{\langle x, y \rangle : x\in\R^d\,, \; \opnorm x \le 1\}$.

            If $f(x) = \opnorm x$, then $f^*(y) = \chi_{\eu C}(y)$ where $\eu C \deq \{y\in\R^d : \opnorm y_* \le 1\}$ is the closed unit ball in the dual norm.
    \end{enumerate}
\end{ex}

Before formally establishing further properties of this duality, we take a detour to explain the origin of this concept in classical mechanics.

\subsection{(Optional) Connection with classical mechanics}

\textbf{Disclaimer:} The material in this subsection is not necessarily the most relevant for optimization, and it is included for the sake of broader historical context.
We make no attempt to be rigorous: assume all functions are smooth, etc.

Newton's law of motion states that the trajectory ${(x_t)}_{t\ge 0}$ of a particle of mass $m$ obeys the differential equation $m\ddot x_t = F(x_t)$, where $F$ is the force.
The force is typically given as the gradient of a potential: $F = -\nabla \phi$.

In 1662, Pierre de Fermat proposed an explanation for the law of refraction via his principle of least action: light takes the path which minimizes the total travel time.
Is there such a principle for classical mechanics as well?
In 1760, Joseph-Louis Lagrange found such a variational principle: let $L(x, v) \deq \frac{1}{2}\,m\,\norm v^2 - \phi(x)$ denote the \emph{Lagrangian}, where $v$ denotes the velocity of the particle.
Note that the Lagrangian is the difference of the kinetic energy and the potential energy.
The action functional is
\begin{align*}
    \ms A({(x_t)}_{t\in [0,T]})
    &\deq \int_0^T L(x_t, \dot x_t)\,\D t\,.
\end{align*}
Lagrangian mechanics states that if a particle starts at $x_0$ at time $0$, and ends at $x_T$ at time $T$, then the path it takes in between is a stationary point of the action functional subject to the endpoint constraints.

We solve for the path using calculus of variations.
Let $x_{[0,T]} \deq {(x_t)}_{t\in [0,T]}$ be a shorthand for the path.
If $x_{[0,T]}$ is a stationary point, it means that for any perturbation $\delta x_{[0,T]}$, the difference $\ms A(x_{[0,T]} + \delta x_{[0,T]}) - \ms A(x_{[0,T]})$ should vanish to first order in $\delta x_{[0,T]}$.
The endpoint constraints require that $\delta x_0 = \delta x_T = 0$.
Thus,
\begin{align*}
    &\ms A(x_{[0,T]} + \delta x_{[0,T]}) - \ms A(x_{[0,T]})
    = \int_0^T \{L(x_t + \delta x_t, \dot x_t + \delta \dot x_t) - L(x_t, \dot x_t)\} \,\D t \\
    &\qquad = \int_0^T \{\langle \nabla_x L(x_t, \dot x_t), \delta x_t \rangle + \langle \nabla_v L(x_t,\dot x_t), \delta \dot x_t\rangle\}\,\D t + o(\norm{\delta x}) \\
    &\qquad = \int_0^T \langle \nabla_x L(x_t, \dot x_t) - \partial_t \nabla_v L(x_t,\dot x_t), \delta x_t \rangle \,\D t + o(\norm{\delta x})\,.
\end{align*}
The stationary point therefore satisfies the Euler{--}Lagrange equation
\begin{align*}
    \partial_t \nabla_v L(x_t,\dot x_t)
    &= \nabla_x L(x_t,\dot x_t)\,.
\end{align*}
For $L(x,v) = \frac{1}{2}\,m\,\norm v^2 - \phi(x)$, it recovers Newton's equation.

We now introduce the Legendre transform.
Define the \emph{Hamiltonian} $H$ to be the convex conjugate of $L$ with respect to the $v$-variable, i.e.,
\begin{align*}
    H(x,p)
    &\deq \sup_{v\in\R^d}{\{\langle p, v \rangle - L(x,v)\}}\,.
\end{align*}
The first-order condition reveals that
\begin{align*}
    p = \nabla_v L(x, v)\,.
\end{align*}
Instead of working with the variables $(x,v)$, we now work with the variables $(x,p)$.
The inverse of the transformation is given by
\begin{align}\label{eq:inverse_Legendre}
    v = \nabla_p H(x,p)\,.
\end{align}
Indeed, we will argue that a regular convex function $f$ satisfies $f = f^{**}$ (\autoref{thm:double_conj}).
Assuming that $v\mapsto L(x,v)$ is regular convex, it yields the dual representation
\begin{align*}
    L(x,v) = \sup_{p\in\R^d}{\{\langle p, v \rangle - H(x,p)\}}\,,
\end{align*}
and the first-order condition for this problem yields~\eqref{eq:inverse_Legendre}.

Thus, if we define $p_t \deq \nabla_v L(x_t,\dot x_t)$, we can reformulate the Euler{--}Lagrange equation as follows.
First, $\dot x_t = v_t = \nabla_p H(x_t, p_t)$ by~\eqref{eq:inverse_Legendre}.
Also, $\nabla_x H(x,p) = -\nabla_x L(x, v)$ by the envelope theorem, so $\dot p_t = \partial_t \nabla_v L(x_t, \dot x_t) = \nabla_x L(x_t,\dot x_t) = -\nabla_x H(x_t,p_t)$.
In summary,
\begin{align*}
    \dot x_t
    &= \nabla_p H(x_t, p_t)\,,
    \qquad \dot p_t = -\nabla_x H(x_t,p_t)\,.
\end{align*}
These are known as \emph{Hamilton's equations}, and it is easy to verify that they conserve the Hamiltonian: $\partial_t H(x_t, p_t) = 0$.
Compared to Newton's law, which is a second-order differential equation for the trajectory, Hamilton's equations are a system of coupled first-order differential equations evolving in phase space.

For our running example, $p = mv$ is interpreted as the \emph{momentum}, and
\begin{align*}
    H(x,p) = \bigl\langle p, \frac{p}{m}\bigr\rangle - \frac{1}{2}\,m\,\bigl\lVert \frac{p}{m} \bigr\rVert^2 + \phi(x)
    = \frac{1}{2m}\,\norm p^2 + \phi(x)\,,
\end{align*}
which is the total energy (kinetic plus potential).
Hamilton's equations read
\begin{align*}
    m\dot x_t = p_t\,, \qquad p_t = -\nabla \phi(x_t)\,.
\end{align*}

What does duality say about the action functional?
Surprisingly, it relates back to other concepts we have already seen.
Specialize now to the case where the Lagrangian only depends on $v$ ($\phi = 0$; no external potential, so we expect particles to move in straight lines).
Define the following function of space and time:
\begin{align*}
    u(t,x)
    &\deq \inf\Bigl\{\int_0^t L(\dot x_s)\,\D s + f(x_0) \Bigm\vert x : [0,t]\to\R^d\,,\; x_t = x\Bigr\}\,.
\end{align*}
In words, we minimize the action functional up to time $t$, subject to the constraint that we hit $x$ at time $t$.
We also add an initial cost $f(x_0)$.
The function $u$ resembles the notion of the \emph{value function} or \emph{cost-to-go function} in dynamic programming, and indeed it satisfies a dynamic programming principle: for $0 \le s < t$,
\begin{align}\label{eq:hjb_dp}
    u(t,y)
    &= \inf_{x\in\R^d}{\bigl\{ (t-s)\,L\bigl( \frac{y-x}{t-s}\bigr) + u(s,x)\bigr\}}\,.
\end{align}
The heuristic derivation of this identity is as follows: consider a potential candidate $x$ for the value of the path at time $s$.
Given $x$, the best possible value of $\int_0^s L(\dot x_r)\,\D r + f(x_0)$ is $u(s,x)$.
For the remaining part, by convexity,
\begin{align*}
    \int_s^t L(\dot x_r)\,\D r
    &\ge (t-s)\, L\Bigl(\frac{1}{t-s}\int_s^t \dot x_r\,\D r\Bigr)
    = (t-s)\,L\bigl( \frac{y-x}{t-s}\bigr)\,.
\end{align*}
The lower bound is achieved if $\dot x_r$ is constant for $r \in [s,t]$, i.e., $x_{[s,t]}$ is a straight line.

In particular, since $u(0,\cdot) = f$, we see that
\begin{align}\label{eq:hopf_lax}
    u(t,y)
    &\deq \inf_{x\in\R^d}{\bigl\{tL\bigl( \frac{y-x}{t}\bigr) + f(x)\bigr\}}\,.
\end{align}

\begin{defn}\label{defn:hopf_lax}
    The \textbf{Hopf{--}Lax semigroup} ${(Q_t)}_{t\ge 0}$ is a family of operators which maps functions to functions, such that $Q_t f(y)$ is defined to be the right-hand side of~\eqref{eq:hopf_lax}.
\end{defn}

The dynamic programming principle~\eqref{eq:hjb_dp} shows that $Q_t f = Q_{t-s}(Q_s f)$.
Thus, we have the properties $Q_0 = {\id}$, $Q_{s+t} = Q_s Q_t = Q_t Q_s$ for all $s,t\ge 0$, which are the defining properties of a semigroup.

These concepts are fundamental, so it is unsurprising that they have been rediscovered in different contexts.
In the context of convex analysis, the corresponding operation is known as the infimal convolution.

\begin{defn}
    Let $f, g : \R^d\to\R\cup\{\infty\}$.
    The \textbf{infimal convolution} of $f$ and $g$, denoted $f\infc g$, is the function defined by
    \begin{align*}
        (f\infc g)(y)
        &\deq \inf_{x\in\R^d}{\{f(x) + g(y-x)\}}\,.
    \end{align*}
\end{defn}

In this notation, $Q_t f = tL(\cdot/t) \infc f$.
Interestingly, the operation of convex conjugation turns addition into infimal convolution and vice versa.

\begin{thm}[convex conjugation and infimal convolution]
    Let $f$, $g$ be regular convex functions.
    Then,
    \begin{align*}
        {(f\infc g)}^*
        = f^* + g^*\,.
    \end{align*}
    Conversely, if $\interior\dom f \cap \interior \dom g \ne \varnothing$, then
    \begin{align*}
        {(f+g)}^*
        &= f^* \infc g^*\,.
    \end{align*}
\end{thm}
\begin{proof}
    For the first statement, note that
    \begin{align*}
        {(f\infc g)}^*(y)
        &= \sup_{x\in\R^d}{\{\langle x, y \rangle - (f\infc g)(x)\}}
        = \sup_{x\in\R^d}{\bigl\{\langle x, y \rangle - \inf_{z\in\R^d}{\{f(z) + g(x-z)\}}\bigr\}} \\
        &= \sup_{x,z\in\R^d}{\{\langle z, y \rangle - f(z) + \langle x-z, y \rangle - g(x-z)\}}
        = f^*(y) + g^*(y)\,.
    \end{align*}
    The first statement also implies that ${(f^* \infc g^*)}^* = f^{**} + g^{**} = f+g$ by~\autoref{thm:double_conj}.
    By applying convex conjugation to both sides, ${(f^* \infc g^*)}^{**} = {(f+g)}^*$, which implies the second statement if $f^* \infc g^*$ equals its double conjugate.
    For this, we need to know that $f^* \infc g^*$ is regular convex, which follows from the condition on the domains (see~\cite[Theorem 16.4]{Roc1997CvxAnalysis}).
\end{proof}

There is a surprising analogy with the Fourier transform, which transforms convolution into multiplication.
Recall that for $f, g  : \R^d\to\C$, the Fourier transform is given by $\ms Ff(\xi) \deq \int f(x)\exp(-2\uppi \mb i\,\langle \xi, x \rangle)\,\D x$, the convolution is given by $(f*g)(y) \deq \int f(x)\,g(y-x)\,\D x$, and we have the key property $\ms F(f*g) = \ms Ff\,\ms Fg$.

To see a connection more precisely, note that we usually work with the algebra $(+, \cdot)$ with its familiar properties: there is an additive identity $0$ such that $x + 0 = 0 + x = x$ for all $x$; every $x$ has an additive inverse $-x$ satisfying $x + (-x) = 0$; multiplication distributes over addition; etc.
Now introduce a new structure, consisting of the operations $(\min, +)$.
This shares some properties with the usual algebra: the identity element for $\min$ is $+\infty$, and $+$ distributes over $\min$, i.e., $x + \min(y,z) = \min(x+y,x+z)$.
However, we also lose some properties: e.g., not every element has an inverse for the $\min$ operation.
This is sometimes known as the \emph{min-plus algebra} despite the fact that it is not technically an algebra; more accurately, it is called the \emph{tropical semiring}.

If we think of integrals as continuous summations, then convolution is a sum of products; infimal convolution is a min of sums.
Hence, infimal convolution is the tropical analogue of convolution.
The following table summarizes further analogies.

\begin{center}
    \begin{tabular}{cc}
        \centering
        $(+,\times)$ & $(\min, +)$ \\
        convolution & infimal convolution \\
        Fourier transform & convex conjugate \\
        Gaussians & convex quadratics \\
        diffusion processes & gradient flow \\
        heat equation & Hamilton{--}Jacobi equation \\
        heat semigroup & Hopf{--}Lax semigroup
    \end{tabular}
\end{center}

We conclude this discussion by using this perspective to show that the Hopf{--}Lax semigroup solves the following PDE\@, known as the \emph{Hamilton{--}Jacobi equation}:
\begin{align}\label{eq:HJ}
    \partial_t u + H(\nabla_x u) = 0\,.
\end{align}

The proof is patterned on the following derivation of the solution to the heat equation $\partial_t u = \Delta u$ with initial condition $u(0,\cdot) = f$; here $\Delta u = \sum_{i=1}^d \partial_i^2 u$ is the Laplacian.
If we take the Fourier transform of both sides of the equation, then $\partial_t \ms Fu = \ms F\Delta u = -4\uppi^2\,\norm\cdot^2\,\ms Fu$, where the last equality follows from differentiating the Fourier transform under the integral.
This implies that $\partial_t \log \ms Fu = -4\uppi^2 \,\norm \cdot^2$, or $\ms Fu(t,\cdot) = \ms F f \exp(-4\uppi^2 t\,\norm \cdot^2)$.
Using the fact that the inverse Fourier transform transforms multiplication into convolution, one can then show that $u(t,\cdot) = f * \ms F\exp(-4\uppi^2 t\,\norm \cdot^2) = f * \normal(0, 4tI)$.

In the same way, we start with~\eqref{eq:HJ} and take the convex conjugate of both sides.
Using the shorthand notation $f_t \deq u(t,\cdot)$, and since $f_t^*(p) = \sup_{v\in\R^d}{\{\langle p, v \rangle - f_t(v)\}}$ with the supremum attained at $v = \nabla f_t^*(p)$,
\begin{align*}
    \partial_t f_t^*(p)
    &= -\partial_t f_t(\nabla f_t^*(p))
    = H\bigl(\nabla f_t(\nabla f_t^*(p))\bigr)
    = H(p)\,.
\end{align*}
Hence, $f_t^* = tH + f^*$ and $f_t = {(tH)}^* \infc f = tL(\cdot/t) \infc f$.
Thus, the solution to~\eqref{eq:HJ} is given by the Hopf{--}Lax semigroup as claimed.

When $L = H = \frac{1}{2}\,\norm \cdot^2$,~\eqref{eq:HJ} becomes $\partial_t u + \frac{1}{2}\,\norm{\nabla_x u}^2 = 0$ and the Hopf{--}Lax semigroup $Q_t f$ coincides with the Moreau{--}Yosida envelope (\autoref{defn:moreau_yosida}).
This yields an unexpected connection between the Hamilton{--}Jacobi equation and the~\ref{eq:PPM}.

\subsection{Duality correspondences}

\begin{thm}[Fenchel{--}Young]\label{thm:fenchel_young}
    Let $f : \R^d\to\R\cup\{\infty\}$ be regular and convex.
    Then,
    \begin{align*}
        f(x) + f^*(p) \ge \langle p, x \rangle \qquad\text{for all}~p,x\in\R^d\,.
    \end{align*}
    Moreover, equality holds if and only if $p \in \partial f(x)$, if and only if $x \in \partial f^*(p)$.
\end{thm}
\begin{proof}
    The inequality is trivial from the definition of $f^*$.
    If equality holds, then for any $p', x'\in\R^d$,
    \begin{align*}
        f(x')
        &\ge \langle p, x' \rangle - f^*(p)
        = f(x) + \langle p, x'-x\rangle\,, \\
        f^*(p')
        &\ge \langle p', x \rangle - f(x)
        = f^*(p) + \langle x, p'-p\rangle\,,
    \end{align*}
    i.e., $p\in\partial f(x)$ and $x \in \partial f^*(p)$.
    Conversely, if $p\in \partial f(x)$, then
    \begin{align*}
        f^*(p)
        &= \sup_{x'\in\R^d}{\{\langle p, x' \rangle - f(x')\}}
        \le \langle p, x \rangle - f(x)\,. \qedhere
    \end{align*}
\end{proof}

\begin{thm}[double conjugation]\label{thm:double_conj}
    Let $f : \R^d\to\R\cup \{\infty\}$.
    Then, $f \ge f^{**}$.

    Moreover, if $f$ is regular and convex, then equality holds: $f = f^{**}$.
\end{thm}
\begin{proof}
    For the first statement,
    \begin{align}\label{eq:double_conj_identity}
        f^{**}(z)
        &= \sup_{y\in\R^d}{\bigl\{\langle y, z\rangle - \sup_{x\in\R^d}{\{\langle x, y\rangle - f(x)\}}\bigr\}}
        = \sup_{y\in\R^d}\inf_{x\in\R^d}{\{\langle y, z -x\rangle + f(x)\}}
        \le f(z)
    \end{align}
    by choosing $x=z$.

    Now assume that $f$ is regular and convex.
    If $z \in \interior \dom f$, then by~\autoref{thm:subdiff} there exists $p \in \partial f(z)$, so that $f(x) \ge f(z) + \langle p, x-z \rangle$ for all $x\in\R^d$.
    By taking $y=p$,
    \begin{align*}
        f^{**}(z)
        \ge \inf_{x\in\R^d}{\{\langle p, z-x \rangle + f(x)\}}
        \ge f(z)\,,
    \end{align*}
    which proves the equality for such $z$.
    For brevity, we omit the proof for $z\notin \interior \dom f$ (see, e.g.,~\cite[Theorem 12.2]{Roc1997CvxAnalysis}).
\end{proof}

This result implies that in general, if $f_*$ is the largest convex and lower semicontinuous function which is smaller than $f$, then $f_* = f^{**}$.
Indeed, $f_* \ge f^{**}$ by definition, whereas $f \ge f_*$ implies $f^{**} \ge {(f_*)}^{**} = f_*$.
The proof above also shows that whenever $\partial f(x) \ne \varnothing$, then $f(x) = f^{**}(x)$.
In particular, if $x_\star$ is a minimizer of $f$, then $0 \in \partial f(x_\star)$ and $f(x_\star) = f^{**}(x_\star)$; moreover, by taking $y = 0$ in~\eqref{eq:double_conj_identity} we see that $\inf f = \inf f^{**}$.
Thus, we can start with a non-convex function $f$ and ``convexify'' it by replacing it with $f^{**}$ while preserving the optimal value, although this is seldom useful in practice.

Properties of $f$ are often reflected as ``dual'' properties for $f^*$.
For example, if $f$ is regular convex, the following assertions hold (see~\cite{Roc1997CvxAnalysis}):
\begin{itemize}
    \item $f$ is Lipschitz if and only if $\dom f^*$ is bounded.
    \item $\epi f$ contains no non-vertical half-lines if and only if $\dom f^* = \R^d$.
    \item $f$ has no lines along which it is affine if and only if $\interior \dom f^* \ne \varnothing$.
    \item $f$ has bounded level sets if and only if $0 \in \interior \dom f^*$.
    \item $f$ is differentiable at $x$ with $\nabla f(x) = p$ if and only if $(p, f^*(p))$ is an exposed point of $\epi f^*$.
        (An \emph{exposed point} of a convex set is a point at which some linear function attains its strict maximum over the convex set.)
\end{itemize}
For our purposes, we are most interested in conditions under which $\nabla f$ is a well-defined bijection from an open convex set $\eu C$ to its image $\nabla f(\eu C)$, with inverse given by ${(\nabla f)}^{-1} = \nabla f^*$.
In this case, the correspondence between $f$ and $f^*$ is known as the Legendre transformation and we informally discussed it in the previous subsection.
We accept the results in the following discussion without proof; see~\cite[\S26]{Roc1997CvxAnalysis} for details.

\begin{defn}
    Let $f : \R^d\to\R\cup\{\infty\}$ be regular convex.
    \begin{itemize}
        \item We say that $f$ is \textbf{essentially smooth} if $f$ is differentiable on $\eu C \deq \interior \dom f$ and $\lim_{n\to\infty}{\norm{\nabla f(x_n)}} \to \infty$ whenever ${\{x_n\}}_{n\in\N}$ is a sequence in $\eu C$ converging to $\partial \eu C$.
        \item We say that $f$ is \textbf{essentially strictly convex} if $f$ is strictly convex on every convex subset of $\dom \partial f \deq \{x\in\R^d : \partial f(x) \ne\varnothing\}$.
    \end{itemize}
\end{defn}

\begin{lem}\label{lem:ess_smooth}
    A regular convex function $f$ is essentially smooth if and only if $f^*$ is essentially strictly convex.
\end{lem}

\begin{thm}\label{thm:legendre_duality}
    Let $f$ be regular, strictly convex, and essentially smooth over $\eu C = \interior \dom f$.
    Then, $f^*$ is regular, strictly convex, and essentially smooth over $\eu C^* \deq \interior \dom f^*$.
    Moreover, $\nabla f : \eu C\to\eu C^*$ is a continuous bijection with ${(\nabla f)}^{-1} = \nabla f^*$.
\end{thm}

\begin{defn}
    We say that a function $f :\R^d\to\R\cup\{\infty\}$ is of \textbf{Legendre type} if it satisfies the assumptions of~\autoref{thm:legendre_duality}.
\end{defn}

To summarize, the condition that $f$ is regular convex ensures duality at the level of $f = f^{**}$.
The condition that $f$ is of Legendre type ensures duality at the level of ${(\nabla f)}^{-1} = \nabla f^*$.
Note also that if $f$, $f^*$ are sufficiently smooth, then by differentiating the equality $\nabla f(\nabla f^*) = \id$ one obtains the identity
\begin{align*}
    \nabla^2 f \circ \nabla f^*
    &= {[\nabla^2 f^*]}^{-1}\,.
\end{align*}
In particular, $\nabla^2 f \succeq \alpha I$ is equivalent to $[\nabla^2 f^*]{}^{-1} \preceq \alpha^{-1} I$, i.e., there is a duality between the properties of strong convexity and smoothness.
Let us prove this last fact without assuming differentiability.

\begin{lem}[convexity{--}smoothness duality]\label{lem:cvxty_smoothness_duality}
    Let $f : \R^d\to \R\cup\{\infty\}$ be regular and $\alpha$-convex for some $\alpha > 0$.
    Then, $f^*$ is $\alpha^{-1}$-smooth.
\end{lem}
\begin{proof}
    By the duality correspondences (including~\autoref{lem:ess_smooth}), $\dom f^* = \R^d$ and $f^*$ is differentiable everywhere.
    For two points $y,y'\in\R^d$, let $x,x'\in\R^d$ achieve the suprema in the definitions of $f^*(y)$, $f^*(y')$ respectively.
    By~\autoref{thm:fenchel_young}, $x = \nabla f^*(y)$ and $x' = \nabla f^*(y')$.
    Then, by strong convexity of $f - \langle \cdot, y\rangle$,
    \begin{align*}
        f(x') - \langle x', y\rangle
        &\ge f(x) - \langle x, y\rangle + \frac{\alpha}{2}\,\norm{x'-x}^2\,.
    \end{align*}
    Adding this to the analogous inequality with $x$ and $x'$ swapped,
    \begin{align*}
        \alpha\,\norm{\nabla f^*(y') - \nabla f^*(y)}^2
        = \alpha\,\norm{x'-x}^2
        &\le \langle x, y \rangle + \langle x', y' \rangle - \langle x', y \rangle - \langle x, y'\rangle \\
        &= \langle x'-x, y'-y\rangle\,.
    \end{align*}
    Rearranging this after Cauchy{--}Schwarz proves $\norm{\nabla f^*(y') - \nabla f^*(y)} \le \alpha^{-1}\,\norm{y'-y}$.
\end{proof}

\begin{cor}[contractivity of the proximal operator]\label{cor:contract_ppm}
    Let $f : \R^d\to\R\cup\{\infty\}$ be $\alpha$-convex.
    Then, $\prox_f$ is $1/(1+\alpha)$-Lipschitz.
\end{cor}
\begin{proof}
    We can write
    \begin{align*}
        \prox_f(y)
        &\deq \argmin_{x\in\R^d}{\bigl\{ f(x) + \frac{1}{2}\,\norm{y-x}^2\bigr\}}
        = -\argmax_{x\in\R^d}{\bigl\{ \langle x, y \rangle - f(x) - \frac{1}{2}\,\norm{x}^2\bigr\}}\,.
    \end{align*}
    This shows that $-{\prox_f}$ is the gradient of the convex conjugate of the function $f + \frac{1}{2}\,\norm \cdot^2$, which is $(1+\alpha)$-convex.
\end{proof}

For a closed convex set $\eu C$, $\Pi_{\eu C} = \prox_{\chi_{\eu C}}$, so this corollary recovers~\autoref{lem:projection_contraction}.
It also shows that~\ref{eq:PPM} with an $\alpha$-convex function contracts with rate $1/(1+\alpha h)$.

\subsection*{Bibliographical notes}

The treatment of classical mechanics is based on~\cite[\S3.3]{Evans10PDE}.
Variational principles also lead to an influential perspective on acceleration, see~\cite{WibWilJor16Accel}.

The problem of minimizing an action functional can be generalized to the problem of optimal control, and in that context, the corresponding Hamilton{--}Jacobi equation is known as the \emph{Hamilton{--}Jacobi{--}Bellman equation}.
The analogy between dynamic programming and diffusions can be pushed even further to obtain ``laws of large numbers'' and ``central limit theorems'' for the former, see~\cite[\S9.4]{Bac+1992Synchro}.

The result of~\autoref{qu:ppm_pl} is from~\cite{Chen+22ProxSampler}.

\subsection*{Exercises}

\begin{question}\label{qu:cvx_conj_ex}
    Verify the assertions in~\autoref{ex:cvx_conj}.
\end{question}

\begin{question}\label{qu:ppm_pl}
    To avoid technical difficulties, assume that $f$ is differentiable everywhere and satisfies~\eqref{eq:PL} with constant $\alpha > 0$.
    The goal of this exercise is to derive the sharp rate of convergence of the~\ref{eq:PPM} in this setting (which turns out to be non-trivial).

    For any $x\in\R^d$, let ${(Q_t)}_{t\ge 0}$ denote the Hopf{--}Lax semigroup and $x_t \deq \prox_{tf}(x)$.
    From the Hamilton{--}Jacobi equation~\eqref{eq:HJ} or via direct computation, show that $\partial_t Q_t f(x) = -\norm{x_t - x}^2/(2t^2)$.
    Use this to deduce that
    \begin{align}\label{eq:HJ_decay}
        \partial_t \{Q_t f(x) - f_\star\}
        &\le - \frac{\alpha}{1+\alpha t}\,\{Q_t f(x) - f_\star\}\,.
    \end{align}
    Finally, by adapting Gr\"onwall's lemma (\autoref{lem:gronwall}), use this to prove the \emph{sharp} rate
    \begin{align*}
        f(x_h) - f_\star
        \le \frac{f(x) - f_\star}{{(1+\alpha h)}^2}\,.
    \end{align*}
\end{question}

\begin{question}\label{qu:HJ_decay}
    The inequality~\eqref{eq:HJ_decay} implies that
    \begin{align}\label{eq:HJ_decay_2}
        Q_h f(x) - f_\star \le \frac{1}{1+\alpha h}\,(f(x) - f_\star)\,.
    \end{align}
    In this exercise, we give an alternative proof of this fact under the assumption that $f$ is $\alpha$-convex.
    As a consequence, we also obtain a result for $\alpha = 0$.

    Write out the definition of $Q_h f(x)$ as an infimum, and choose as a test point the interpolant $(1-t)\,x_\star + t\,x$.
    Pick $t > 0$ to derive~\eqref{eq:HJ_decay_2}.
    Also, in the case $\alpha = 0$, show that if $f(x) - f_\star \le \norm{x-x_\star}^2/h$,
    \begin{align*}
        Q_h f(x) - f_\star
        &\le \Bigl(1 - \frac{f(x)-f_\star}{2\,\norm{x-x_\star}^2}\,h\Bigr)\,(f(x) - f_\star)\,.
    \end{align*}
\end{question}


\section{Mirror methods}

Consider the following situation.

\begin{ex}[optimization in a different norm]\label{ex:different_norm}
    Suppose that $f : \R^d \to \R\cup\{\infty\}$ is convex and that we wish to minimize it over the simplex $\Delta_d \deq \{x\in\R_+^d : \sum_{i=1}^d x[i] = 1\}$.
    If $f$ is Lipschitz, then we can apply~\ref{eq:PSD} and obtain an $\varepsilon$-approximate solution in $O(L^2 R^2/\varepsilon^2)$ iterations.
    Here, $L$ is the Lipschitz constant, and $R \le 2$ is the radius.

    For example, suppose that we have $d$ actions and that the loss of the $i$-th action is $\ell[i]$, where the losses are bounded: $\abs{\ell[i]} \le 1$.
    If we choose an action randomly according to a probability distribution $x\in\Delta_d$, the expected loss is $\langle \ell, x\rangle \eqqcolon f(x)$.
    We then seek to minimize the expected loss over $\Delta_d$.\renewcommand{\thempfootnote}{$\dagger$}\footnote{Trivially, the solution to the optimization problem is given by the distribution which puts all of its mass on $\argmin_{i\in [d]} \ell[i]$. This problem is simply meant to illustrate the pitfalls of na\"{\i}vely using the Euclidean norm, but it also forms the basis for the more interesting setting of~\autoref{ex:experts}.}
    The Lipschitz constant of $f$ is $\norm\ell$, which could be as large as $\sqrt d$ in the worst case; the resulting complexity estimate of $O(d/\varepsilon^2)$ is poor in high dimension.

    Implicit in this discussion, however, is that we are measuring the Lipschitz constant and the radius with respect to the usual Euclidean norm.
    In this setting, however, it may make more sense to use the $\ell_1$ norm, in which case the Lipschitz constant is $\norm\ell_\infty \le 1$.
\end{ex}

Until now, we have identified points $x$ and gradients $\nabla f(x)$ as part of the same space $\R^d$, but this is because of the self-dual nature of the Euclidean norm.
Suppose now that $(\eu X, \opnorm \cdot)$ is a general (finite-dimensional) normed vector space and $f : \eu X\to\R\cup\{\infty\}$.
The dual space is $(\eu X^*, \opnorm \cdot_*)$, where $\eu X^*$ is the space of linear functionals $\ell : \eu X\to\R$, equipped with the dual norm $\opnorm\ell_* \deq \sup\{\abs{\ell(x)} : \opnorm x \le 1\}$.
The derivative of $f$ at $x$ is defined to be the linearization at $x$: if there exists an element $\ell \in \eu X^*$ such that
\begin{align*}
    \abs{f(x+v)-f(x)-\ell(v)} = o(\opnorm v) \qquad\text{as}~v\to 0\,,
\end{align*}
we say that $f$ is differentiable\footnote{Strictly speaking, this is the Fr\'echet derivative.} at $x$ and we write $Df(x)$ for the functional $\ell$.
Note that in this formalism, the derivative $Df(x)$ is an element of the dual space.

Above, we wrote $Df(x)$ instead of $\nabla f(x)$ to emphasize that in this context, we should no longer think of $Df(x)$ as belonging to the original space $\eu X$.
However, when $\eu X = \R^d$, it is still convenient to identify $Df(x)$ as a vector in $\R^d$, and we therefore continue to use the notation $\nabla f(x)$.
This is fine as long as we remember the following two points:
\begin{itemize}
    \item It does not make sense to add a point $x\in\eu X$ to a gradient $\nabla f(x) \in \eu X^*$.
    \item The size of $\nabla f(x)$ should be measured in the dual norm $\opnorm \cdot_*$.
\end{itemize}
In the context of~\autoref{ex:different_norm}, the dual norm for $\norm\cdot_1$ is the $\ell_\infty$ norm $\norm \cdot_\infty$.

Immediately, the first point above rules out~\ref{eq:GD} and~\ref{eq:PSD} as sensible algorithms.
A first attempt to remedy this issue is to somehow develop analogues of these algorithms in different norms, but this is seriously complicated by the fact that non-Euclidean norms lack crucial properties, e.g., we cannot ``expand the square'' as we did in previous proofs.

Instead, the idea of~\cite{NemYud1983Complexity} is to use the Fenchel{--}Legendre duality.

\begin{center}
    Throughout this section, $\phi : \R^d\to\R\cup\{\infty\}$ is a convex function of Legendre type.
    We refer to it as the \emph{mirror map}.
\end{center}

The idea is to use the auxiliary function $\phi$ to map the iterate $x_n$ into the dual space via $x_n^* = \nabla \phi(x_n)$.
Now that we are in the dual space, it makes sense to take a gradient step: $x_{n+1}^* = x_n^* - h\,\nabla f(x_n)$.
Then, we use $\nabla \phi^*$ to return: $x_{n+1} = \nabla \phi^*(x_{n+1})$.
The goal of this section is to formalize this idea and its analysis.

\subsection{Bregman divergences and relative convexity/smoothness}

We introduce a key definition which substitutes for the squared Euclidean norm~\cite{Bre1967}.

\begin{defn}
    Given a function $\phi : \R^d\to\R\cup\{\infty\}$ of Legendre type over $\eu C_\phi$, the corresponding \textbf{Bregman divergence} associated with $\phi$ is the map $D_\phi : \R^d \times \eu C_\phi \to \R \cup \{\infty\}$ defined by
    \begin{align*}
        D_\phi(x,y)
        &\deq \phi(x) - \phi(y) - \langle\nabla \phi(y), x-y\rangle\,.
    \end{align*}
\end{defn}

In words, $D_\phi(\cdot, y)$ is defined by subtracting from $\phi$ its linearization at $y$.
We can observe the following properties:
\begin{itemize}
    \item $D_\phi\ge 0$: this is equivalent to the convexity of $\phi$.
    \item $D_\phi$ is convex with respect to its first argument.
    \item If $\phi$ is twice continuously differentiable, then
        \begin{align}\label{eq:bregman_locally_quadratic}
            D_\phi(x,y)
            \sim \frac{1}{2}\,\langle x-y, \nabla^2 \phi(y)\,(x-y)\rangle \qquad\text{as}~x\to y\,.
        \end{align}
\end{itemize}
The last property indicates that $D_\phi$ should behave as a squared distance between $x$ and $y$.
In some respects this is true, e.g., $D_\phi$ satisfies a Pythagorean inequality (\autoref{qu:bregman_properties}).
However,~\eqref{eq:bregman_locally_quadratic} is a purely local statement, and a priori there does not seem to be a reason for $D_\phi$ to have useful global properties.
For example, $D_\phi$ is \emph{asymmetric}, and $\sqrt{D_\phi}$ does not in general satisfy a triangle inequality.
Nevertheless, it turns out that $D_\phi$ is a powerful global measure of progress, which is arguably the greatest surprise of mirror methods.

\begin{ex}[mirror maps]\mbox{}
    \begin{enumerate}
        \item Let $\phi(x) = \frac{1}{2}\,\norm x^2$.
            Then, $\nabla \phi$ is the identity mapping and $D_\phi$ is one-half times the squared Euclidean distance.
            So, our study of mirror methods subsumes the preceding Euclidean methods.
        \item Let $\phi(x) = \sum_{i=1}^d \{x[i] \log x[i] - x[i]\}$ for $x\in\R_+^d$.
            Then, $\nabla \phi(x) = \log x$, where $\log$ is applied coordinate-wise.
            The associated Bregman divergence is the Kullback{--}Leibler divergence $D_\phi(x,y) = \sum_{i=1}^d \{x[i] \log(x[i]/y[i]) - x[i] + y[i]\}$.
        \item Let $\phi(X) = \tr(X\log X - X)$ for $X \succ 0$; this is known as the von Neumann entropy.
            The associated Bregman divergence is the quantum relative entropy $D_\phi(X,Y) = \tr(X\,(\log X - \log Y) - X + Y)$.
    \end{enumerate}
\end{ex}

Let us define notions of convexity and smoothness \emph{relative to} $\phi$.

\begin{defn}
    Let $f : \R^d\to\R\cup\{\infty\}$ be differentiable on $\interior \dom f \subseteq \eu C_\phi$.
    \begin{itemize}
        \item We say that $f$ is \textbf{$\alpha$-convex relative to $\phi$} if $D_f \ge \alpha D_\phi$.
        \item We say that $f$ is \textbf{$\beta$-smooth relative to $\phi$} if $D_f \le \beta D_\phi$.
    \end{itemize}
\end{defn}

Similarly to \S\ref{ssec:prelim}, there are equivalent reformulations of these definitions; see~\cite[Proposition 1.1]{LuFreNes18RelSmooth} for details.

\begin{prop}[relative convexity]
    For any $\alpha \ge 0$, the following are equivalent.
    \begin{itemize}
        \item $f$ is $\alpha$-convex relative to $\phi$.
        \item $f-\alpha\phi$ is convex.
        \item $\langle \nabla f(y) - \nabla f(x), y-x\rangle \ge \alpha\,\langle \nabla \phi(y)-\nabla \phi(x), y-x\rangle$ for all $x,y\in\interior \dom f$.
    \end{itemize}
    If $f$ is twice continuously differentiable on $\interior \dom f$, the above are also equivalent to:
    \begin{itemize}
        \item $\nabla^2 f \succeq \alpha\, \nabla^2 \phi$ on $\interior \dom f$.
    \end{itemize}
\end{prop}

\begin{prop}[relative smoothness]
    For any $\beta \ge 0$, the following are equivalent.
    \begin{itemize}
        \item $f$ is $\beta$-smooth relative to $\phi$.
        \item $\beta\phi - f$ is convex.
        \item $\langle \nabla f(y) - \nabla f(x), y-x\rangle \le \beta\,\langle \nabla \phi(y)-\nabla \phi(x), y-x\rangle$ for all $x,y\in\interior \dom f$.
    \end{itemize}
    If $f$ is twice continuously differentiable on $\interior \dom f$, the above are also equivalent to:
    \begin{itemize}
        \item $\nabla^2 f \preceq \beta\, \nabla^2 \phi$ on $\interior \dom f$.
    \end{itemize}
\end{prop}

For the case of $\phi = \frac{1}{2}\,\norm \cdot^2$, we recover the usual notions of convexity and smoothness described in \S\ref{ssec:prelim}.
These relative definitions satisfy similar properties as convexity/smoothness, e.g., if $f_1$, $f_2$ are respectively $\alpha_1$- and $\alpha_2$-convex relative to $\phi$ and $\lambda_1,\lambda_2 > 0$, then $\lambda_1 f_1+ \lambda_2 f_2$ is $(\lambda_1 \alpha_1+ \lambda_2\alpha_2)$-convex relative to $\phi$.
Also, we have a growth bound.

\begin{lem}[relative growth]\label{lem:bregman_growth}
    Suppose that $f$ is $\alpha$-convex relative to $\phi$ for some $\alpha > 0$, and that $f$ is minimized at an interior point $x_\star$ of its domain.
    Then, for all $x\in\R^d$,
    \begin{align*}
        f(x) - f_\star
        &\ge \alpha\,D_\phi(x,x_\star)\,.
    \end{align*}
\end{lem}
\begin{proof}
    The left-hand side is $D_f(x,x_\star)$.
\end{proof}

Other useful properties of Bregman divergences are explored in~\autoref{qu:bregman_properties}.

\subsection{Algorithms and convergence analysis}

Briefly, let us first consider the continuous-time picture.
Since we add the gradient of $f$ in the dual space, the dynamics we consider evolve according to
\begin{align}\label{eq:mirror_flow_dual}
    \partial_t \nabla \phi(x_t) = -\nabla f(x_t)\,.
\end{align}
By the chain rule, this is equivalent to the following evolution in the primal space:
\begin{align}\label{eq:mirror_flow}
    \dot x_t
    = -{[\nabla^2 \phi(x_t)]}^{-1}\,\nabla f(x_t)\,.
\end{align}
This can be interpreted as a preconditioned gradient flow.

Despite the fact that~\eqref{eq:mirror_flow_dual} and~\eqref{eq:mirror_flow} are equivalent in continuous time, they lead to different discretizations.
The discretization of~\eqref{eq:mirror_flow} is usually called natural gradient descent and it is related to the subject of information geometry~\cite{AmaNag00InfoGeo}.
In fact, one can view the use of the mirror map $\phi$ as equipping the space $\eu C_\phi$ with a Riemannian metric, namely, a local inner product $\langle u, v \rangle_x \deq \langle u,\nabla^2 \phi(x)\,v\rangle$.
This turns $\eu C_\phi$ into a so-called Hessian manifold.
From this perspective, the natural objects are geometric in nature: geodesics, length, curvature, etc.

However, this is \textbf{not} what we consider; mirror descent is obtained from discretization of~\eqref{eq:mirror_flow_dual} in the dual.
This conceptual point is so important that we isolate it into a remark.

\begin{rmk}
    The key distinguishing feature of mirror methods from preconditioned or Riemannian gradient methods is the existence of the \emph{global} progress measure given by the Bregman divergence $D_\phi$.
    In contrast, preconditioned/Riemannian gradient methods are purely \emph{local} in nature.
\end{rmk}

Now that we have emphasized the conceptual underpinnings of the methods, let us now turn to concrete algorithms.
We begin with the smooth case, and we consider the following \emph{mirror proximal gradient descent} method:
\begin{align}\label{eq:MPGD}\tag{$\msf{MPGD}$}
    x_{n+1}
    &\deq \argmin_{x\in\R^d}{\bigl\{ f(x_n) + \langle \nabla f(x_n), x-x_n \rangle + g(x) + \frac{1}{h} \, D_\phi(x, x_n) \bigr\}}\,.
\end{align}
Note that this incorporates the proximal splitting considered in \S\ref{sec:prox}, except that we replace $\frac{1}{2}\,\norm{x-x_n}^2$ with the more general $D_\phi(x,x_n)$.
We consider this iteration for the sake of generality, since it encompasses the following algorithms.
\begin{itemize}
    \item When $g = 0$, since $\nabla_1 D_\phi(x,x_n) = \nabla \phi(x) - \nabla \phi(x_n)$, the first-order optimality condition reads
        \begin{align*}
            \nabla \phi(x_{n+1})
            &= \nabla \phi(x_n) - h\,\nabla f(x_n)\,.
        \end{align*}
        This is known as \emph{mirror descent}.
    \item When $f = 0$, we obtain
        \begin{align*}
            x_{n+1}
            = \argmin_{x\in\R^d}{\bigl\{g(x) + \frac{1}{h}\,D_\phi(x,x_n)\bigr\}} \eqqcolon \prox_{hg}^\phi(x_n)\,,
        \end{align*}
        which is the \emph{mirror proximal point method}.
    \item When $g=\chi_{\eu C}$, where $\eu C \subseteq \eu C_\phi$ is a closed convex set,
        \begin{align*}
            x_{n+1}
            &= \argmin_{x\in\eu C}{\bigl\{\langle \nabla f(x_n), x-x_n\rangle + \frac{1}{h}\,\bigl(\phi(x) - \langle \nabla \phi(x_n), x-x_n \rangle\bigr)\bigr\}} \\
            &= \argmin_{x\in\eu C}{\{\phi(x) - \langle \nabla \phi(x_n) - h\,\nabla f(x_n),\, x-x_n \rangle\}} \\
            &= \Pi_{\eu C}^\phi\bigl( \nabla \phi^*(\nabla \phi(x_n) - h\,\nabla f(x_n))\bigr)\,,
        \end{align*}
        where $\Pi_{\eu C}^\phi$ is the Bregman projection (see~\autoref{qu:bregman_properties}).
        This is the mirror analogue of projected gradient descent.
\end{itemize}

\begin{thm}[convergence of~\ref{eq:MPGD}]\label{thm:MPGD}
    Let $f$ be $\alpha_f$-convex and $\beta_f$-smooth, and let $g$ be $\alpha_g$-convex, all relative to $\phi$.
    Let the step size $h$ satisfy $h \le 1/\beta_f$, let $x^+$ denote the next iterate of~\ref{eq:MPGD} started from $x$, and let $y\in\R^d$.
    Then,
    \begin{align*}
        (1+\alpha_g h)\,D_\phi(y, x^+)
        &\le (1-\alpha_f h)\,D_\phi(y,x) - h\,(F(x^+) - F(y))\,.
    \end{align*}
    In particular, if we set $y=x_\star$ and iterate, it yields
    \begin{align*}
        F(x_N) - F_\star
        &\le \frac{\alpha_f + \alpha_g}{\lambda_h^{-N}-1}\,D_\phi(x_\star,x_0)\,,
    \end{align*}
    where $\lambda_h \deq (1-\alpha_f h)/(1+\alpha_g h)$.
\end{thm}
\begin{proof}
    The proof is patterned upon the proof of~\autoref{thm:PGD}.
    Let $\psi_x$ denote the objective in~\eqref{eq:MPGD} starting from $x$ (rather than $x_n$).
    Then, $\psi_x$ is $(\alpha_g+1/h)$-convex relative to $\phi$ with minimizer $x^+$, so by the growth inequality (\autoref{lem:bregman_growth}),
    \begin{align*}
        \psi_x(y)
        &\ge \psi_x(x^+) + \bigl(\alpha_g + \frac{1}{h}\bigr)\,D_\phi(y,x^+)\,.
    \end{align*}
    On one hand, by $\alpha_f$-convexity,
    \begin{align*}
        \psi_x(y)
        &= f(x) + \langle \nabla f(x), y-x\rangle + g(y) + \frac{1}{h}\,D_\phi(y,x) \\
        &= f(y) - D_f(y,x) + g(y) + \frac{1}{h}\,D_\phi(y,x)
        \le F(y) + \bigl( \frac{1}{h} - \alpha_f\bigr)\,D_\phi(y,x)\,.
    \end{align*}
    On the other hand, by $\beta_f$-smoothness,
    \begin{align*}
        \psi_x(x^+)
        &= f(x) + \langle \nabla f(x), x^+ - x \rangle + g(x^+) + \frac{1}{h}\,D_\phi(x^+, x) \\
        &= f(x^+) - D_f(x^+, x) + g(x^+) + \frac{1}{h}\,D_\phi(x^+, x)
        \ge F(x^+) + \bigl(\frac{1}{h} - \beta_f\bigr)\,D_\phi(x^+, x)\,.
    \end{align*}
    Drop the term $(1/h-\beta_f)\,D_\phi(x^+,x)$ and combine the inequalities to prove the one-step bound.
    If we set $y=x$ in the one-step bound, it yields the descent lemma $F(x^+) - F(x) \le -h^{-1}\,(1+\alpha_g h)\,D_\phi(x,x^+) \le 0$, so we can iterate the one-step bound using the discrete Gr\"onwall lemma (\autoref{lem:discrete_gronwall}).
\end{proof}

Although this result is the analogue of the smooth convergence rate for~\ref{eq:GD} (\autoref{thm:gd_fn_value}), since $\nabla\phi$ necessarily blows up at the boundary $\partial \eu C_\phi$, so can $\nabla f$.
Therefore, this theorem actually covers examples in which $f$ is not at all smooth in the usual sense.

To relate this back to~\autoref{ex:different_norm}, consider convexity/smoothness relative to a norm.

\begin{defn}
    A function $f$ is \textbf{$\alpha$-convex} (resp.\ \textbf{$\beta$-smooth}) \textbf{relative to a norm $\opnorm \cdot$} if for all $x,y\in\interior \dom f$,
    \begin{align*}
        D_f(x,y)
        &\ge \frac{\alpha}{2}\,\opnorm{y-x}^2 \qquad \bigl(\text{resp.}~D_f \le \frac{\beta}{2}\,\opnorm{y-x}^2\bigr)\,.
    \end{align*}
\end{defn}

Suppose that $\phi$ is strongly convex relative to a norm $\opnorm \cdot$.
Then, to check that $f$ is smooth relative to $\phi$, it suffices to check that $f$ is smooth relative to $\opnorm \cdot$, so the norm can act as a useful intermediary.
Moreover, whereas the Bregman structure is crucial for carrying out the iterative analysis of~\ref{eq:MPGD}, the norm structure is often convenient too, e.g., for the use of tools such as Cauchy{--}Schwarz.

To illustrate this, we now consider the non-smooth case.
Here, we assume that $f$ is Lipschitz with respect to $\opnorm \cdot$:
\begin{align*}
    \abs{f(x) - f(y)} \le L\,\opnorm{x-y} \qquad\text{for all}~x,y\in\eu C_\phi\,.
\end{align*}
We again consider~\ref{eq:MPGD}, except that $\nabla f(x_n)$ should be interpreted as a subgradient; we leave the notation unchanged because it should not cause confusion.
The Lipschitz condition is then equivalent to the subgradient bound
\begin{align*}
    \opnorm{\nabla f(x)}_* \le L \qquad\text{for all}~x\in\eu C_\phi\,.
\end{align*}

\begin{thm}[convergence of~\ref{eq:MPGD}, non-smooth case]\label{thm:MPGD_non_smooth}
    Let $f$ and $g$ be convex, and let $f$ be $L$-Lipschitz with respect to a norm $\opnorm \cdot$.
    Let $\phi$ be $\alpha_\phi$-convex relative to $\opnorm \cdot$.
    Then, for~\ref{eq:MPGD}, it holds that
    \begin{align*}
        F\bigl( \frac{1}{N} \sum_{n=1}^N x_n\bigr) - F_\star
        &\le \frac{1}{N} \sum_{n=1}^N (F(x_n) - F_\star)
        \le \frac{D_\phi(x_\star, x_0)}{Nh} + \frac{2L^2 h}{\alpha_\phi}\,.
    \end{align*}
    In particular, if $R_\phi^2 \ge D_\phi(x_\star,x_0)$ and we choose step size $h^2 = \alpha_\phi R_\phi^2/(2L^2 N)$, then
    \begin{align*}
        F\bigl( \frac{1}{N} \sum_{n=1}^N x_n\bigr) - F_\star
        &\le LR_\phi \sqrt{\frac{8}{\alpha_\phi N}}\,.
    \end{align*}
\end{thm}
\begin{proof}
    Following the proof of~\autoref{thm:MPGD}, we still have
    \begin{align*}
        \psi_x(x^+) + \frac{1}{h}\,D_\phi(x_\star,x^+)
        &\le \psi_x(x_\star)
        \le F(x_\star) + \frac{1}{h}\,D_\phi(x_\star,x)\,.
    \end{align*}
    In the lower bound for $\psi_x(x^+)$, we originally used smoothness to upper bound $D_f(x^+,x)$, which is no longer available to us.
    Instead, by Cauchy{--}Schwarz,
    \begin{align*}
        D_f(x^+,x)
        = f(x^+) - f(x) - \langle \nabla f(x), x^+ - x \rangle
        &\le L\,\opnorm{x^+ - x} + \opnorm{\nabla f(x)}_*\,\opnorm{x^+ - x} \\
        &\le 2L\,\opnorm{x^+ - x}\,.
    \end{align*}
    Thus,
    \begin{align*}
        \psi_x(x^+)
        = F(x^+) - D_f(x^+, x) + \frac{1}{h}\,D_\phi(x^+, x)
        &\ge F(x^+) - 2L\,\opnorm{x^+ - x} + \frac{\alpha_\phi}{2h} \,\opnorm{x^+ - x}^2 \\
        &\ge F(x^+) - \frac{2L^2 h}{\alpha_\phi}\,.
    \end{align*}
    This leads to the one-step bound
    \begin{align*}
        D_\phi(x_\star,x^+)
        &\le D_\phi(x_\star,x) - h\,(F(x^+) - F_\star) + \frac{2L^2 h^2}{\alpha_\phi}\,.
    \end{align*}
    Iterating this inequality finishes the proof.
\end{proof}

\begin{ex}[optimization over the simplex]\label{ex:opt_simplex}
    We return to~\autoref{ex:different_norm}, and we use the entropic mirror map $\phi(x) = \sum_{i=1}^d \{x[i] \log x[i] - x[i]\}$.
    Then, $\phi$ is $1$-convex relative to the $\ell_1$-norm $\norm \cdot_1$ over the probability simplex $\Delta_d$; this is known as Pinsker's inequality (\autoref{qu:pinsker}).

    To minimize $f : \R^d\to\R\cup\{\infty\}$ over $\Delta_d$, we apply~\ref{eq:MPGD} with $g = \chi_{\Delta_d}$.
    Then,
    \begin{align*}
        \nabla \phi^*(\nabla \phi(x_n) - h\,\nabla f(x_n))
        = x_n \odot \exp(-h\,\nabla f(x_n))\,,
    \end{align*}
    where $\exp$ is applied pointwise and $\odot$ is the Hadamard (or pointwise) product.
    Also, one can check that $\Pi^\phi_{\Delta_d}(x) = x/\norm x_1$ simply normalizes the vector (\autoref{qu:bregman_proj_simplex}).
    Hence, the algorithm reads
    \begin{align*}
        x_{n+1}
        &= \frac{x_n \odot \exp(-h\,\nabla f(x_n))}{\norm{x_n \odot \exp(-h\,\nabla f(x_n))}_1}\,.
    \end{align*}

    Consider initializing at the uniform distribution $x_0 = \mb 1_d/d$.
    Then, for any $x_\star \in \Delta_d$,
    \begin{align*}
        D_\phi(x_\star,x_0)
        &= \KL(x_\star \mmid x_0)
        = \log d - \sum_{i=1}^d x_\star[i] \log \frac{1}{x_\star[i]}
        \le \log d\,,
    \end{align*}
    by Jensen's inequality.
    Consequently, we can take $R_\phi = \sqrt{\log d}$, and
    \begin{align*}
        f\bigl( \frac{1}{N} \sum_{n=1}^N x_n\bigr) - f_\star
        &\le L_1\sqrt{\frac{8\log d}{N}}\,,
    \end{align*}
    where $L_1$ is the Lipschitz constant of $f$ in the $\ell_1$ norm.
    This estimate is far better than the one described in~\autoref{ex:different_norm} for the Euclidean norm; we only pay an overhead which is logarithmic in the dimension.
\end{ex}

\subsection{Online algorithms and multiplicative weights}

Let us examine the proof of~\autoref{thm:MPGD_non_smooth} once more.
In that proof, we start with
\begin{align*}
    \psi_x(x^+) + \frac{1}{h}\,D_\phi(y,x^+)
    \le \psi_x(y)\,,
\end{align*}
which holds for all $y \in\R^d$.
If we expand out the terms, this is equivalent to
\begin{align*}
    &\langle \nabla f(x), x^+ - x\rangle + g(x^+) + \frac{1}{h}\,D_\phi(x^+,x) + \frac{1}{h}\,D_\phi(y,x^+) \\
    &\qquad \le \langle \nabla f(x), y-x \rangle + g(y) + \frac{1}{h}\,D_\phi(y,x)\,.
\end{align*}
On the left-hand side, if we apply Lipschitzness,
\begin{align*}
    \langle \nabla f(x), x^+ - x\rangle + \frac{1}{h}\,D_\phi(x^+,x)
    &\ge -\opnorm{\nabla f(x)}_*\, \opnorm{x^+ - x} + \frac{\alpha_\phi}{2h}\,\opnorm{x^+ - x}^2
    \ge - \frac{L^2 h}{2\alpha_\phi}\,.
\end{align*}
If we now specialize to the case $g = \chi_{\eu C}$, then for any $y\in\eu C$,
\begin{align*}
    \langle \nabla f(x), x-y\rangle
    &\le \frac{1}{h}\,\bigl(D_\phi(y,x) - D_\phi(y,x^+)\bigr) + \frac{L^2 h}{2\alpha_\phi}\,.
\end{align*}

Normally, we apply convexity to further lower bound the left-hand side, but let us now refrain from doing so.
We make the observation that in the derivation thus far, \emph{we have not used any property of $f$}; in fact, the same inequality holds even if $\nabla f(x)$ is replaced by an \emph{arbitrary} vector $p\in\R^d$ with bounded dual norm, provided that we redefine the update in~\ref{eq:MPGD} accordingly.

We now define the online version of mirror descent.
Let $\eu C \subseteq \eu C_\phi$ be a closed convex set, and let ${\{p_n\}}_{n\in\N}$ be an arbitrary sequence of vectors.
Define the updates
\begin{align}\label{eq:OMD}\tag{$\msf{OMD}$}
    x_{n+1}
    &\deq \argmin_{x\in \eu C}{\bigl\{\langle p_n, x-x_n\rangle + \frac{1}{h}\,D_\phi(x,x_n)\bigr\}}
    = \Pi_{\eu C}^\phi\bigl(\nabla \phi^*(\nabla \phi(x_n) - hp_n)\bigr)\,.
\end{align}
We immediately obtain the following theorem.

\begin{thm}[regret guarantee for~\ref{eq:OMD}]\label{thm:OMD}
    Let $\eu C \subseteq \eu C_\phi$ be a closed convex set, let $\phi$ be $\alpha_\phi$-convex relative to a norm $\opnorm \cdot$, and suppose that ${\{p_n\}}_{n=0}^{N-1}$ are bounded in dual norm by $L$, i.e., $\opnorm{p_n}_* \le L$ for all $n$.
    Then,~\ref{eq:OMD} satisfies
    \begin{align*}
        \sum_{n=0}^{T-1} \langle p_n, x_n\rangle
        &\le \inf_{y\in\eu C}{\Bigl\{\sum_{n=0}^{T-1} \langle p_n, y\rangle + \frac{D_\phi(y,x_0)}{h}\Bigr\}} + \frac{L^2 Th}{2\alpha_\phi}\,.
    \end{align*}
    In particular, if $R_\phi^2 \ge \sup_{y\in\eu C} D_\phi(y,x_0)$ and $h = R_\phi\sqrt{2\alpha_\phi}/(L\sqrt T)$, then
    \begin{align*}
        \sum_{n=0}^{T-1} \langle p_n, x_n\rangle
        &\le \inf_{y\in\eu C}\;\sum_{n=0}^{T-1} \langle p_n, y\rangle + LR_\phi\sqrt{2T/\alpha_\phi}\,.
    \end{align*}
\end{thm}

In the setting of online learning (with full information feedback), at each round $n$, the player must play an action $x_n$ belonging to some set $\eu C$ of actions.
An adversary then chooses a loss function $\ell_n$ belonging to some class of losses, the player incurs the loss $\ell_n(x_n)$, and the function $\ell_n(\cdot)$ is revealed to the player.
Thus, the total loss incurred by the player after $T$ rounds is $\sum_{n=0}^{T-1} \ell_n(x_n)$.
Since the losses are chosen in an adversarial fashion, one cannot hope to compete with a changing benchmark, so the measure of progress is to compare against the best \emph{fixed} point that one could have played in hindsight, which incurs loss $\inf_{y\in\eu C}\sum_{n=0}^{T-1} \ell_n(y)$.
The difference $\sum_{n=0}^{T-1} \ell_n(x_n) - \inf_{y\in\eu C}\sum_{n=0}^{T-1} \ell_n(y)$ is called the \emph{regret}, and the goal is to minimize it.
In particular, regret bounds that scale linearly with $T$ are often considered ``trivial'', whereas regret bounds that scale as $o(T)$ indicate that the algorithm has learned from its past mistakes.

With this context in mind,~\autoref{thm:OMD} is a regret guarantee for the~\ref{eq:OMD} algorithm for the \emph{linear bandit} problem in which the loss functions are linear, $\ell_n(\cdot) = \langle p_n, \cdot\rangle$, and the vectors belong to the dual norm ball $\{\opnorm \cdot_* \le L\}$.
This result is already interesting in the Euclidean case $\phi = \frac{1}{2}\,\norm \cdot^2$, but the simplex setting is of particular interest to its connection with a well-established algorithm.

\begin{ex}[learning from expert advice]\label{ex:experts}
    On each day $n$, an investor seeks to predict the price of a stock.
    There are $d$ so-called ``experts'' who give daily predictions.
    On the following day, the investor compares their predictions with reality and assigns them losses $\ell_n[1],\dotsc,\ell_n[d] \in [-1,1]$.
    (For example, we could set $\ell_n[i] = +1$ if expert $i$ incorrectly predicted the direction of change of the stock price on day $n$, and $\ell_n[i] = -1$ otherwise.)
    Not all of the experts are necessarily reliable, but some might be.
    Can we aggregate the expert forecasts and compete with the best of them in hindsight, i.e., incur small regret?

    The algorithm maintains a vector $x_n \in \Delta_d$ in the probability simplex.
    On each day $n$, the algorithm picks an expert $i_n \sim x_n$ and trusts the advice of the $i_n$-th expert.
    Note that the expected loss incurred by the algorithm is $\E_{i_n\sim x_n} \ell_n[i_n] = \langle \ell_n, x_n \rangle$, where $\ell_n\in\R^d$ is the vector of losses.
    (This is the online version of~\autoref{ex:different_norm}.)
    The regret is
    \begin{align*}
        \sum_{n=0}^{T-1} \langle \ell_n, x_n\rangle - \inf_{x\in\Delta_d} \sum_{n=0}^{T-1} \langle\ell_n, x\rangle
        &= \sum_{n=0}^{T-1} \langle \ell_n, x_n\rangle - \min_{i\in [d]} \sum_{n=0}^{T-1} \ell_n[i]\,.
    \end{align*}
    We update the vector $x_n$ using~\ref{eq:OMD} with $p_n = \ell_n$ and the entropic mirror map $\phi$.
    Note that $\opnorm{\ell_n}_* = \norm{\ell_n}_\infty \le 1$, and by~\autoref{ex:opt_simplex}, we can take $x_0 = \mb 1_d/d$ for which $R_\phi \le \sqrt{\log d}$.
    Therefore,~\autoref{thm:OMD} implies
    \begin{align*}
        \msf{Regret}_T(\text{\ref{eq:OMD}})
        &\le \sqrt{2T\log d}\,.
    \end{align*}
    The corresponding algorithm, with updates
    \begin{align*}
        x_{n+1}
        = \frac{x_n \odot \exp(-h\ell_n)}{\norm{x_n \odot \exp(-h\ell_n)}_1}\,,
    \end{align*}
    is known as the \emph{multiplicative weights} algorithm.
\end{ex}

\subsection*{Bibliographical notes}

The definitions and usage of relative convexity and smoothness are from~\cite{BauBolTeb17Descent, LuFreNes18RelSmooth}.
An interesting discussion of the various ways to discretize~\eqref{eq:mirror_flow_dual} and~\eqref{eq:mirror_flow} can be found in the paper~\cite{GunWooSre21Mirrorless}.
The example in~\autoref{qu:poisson_inverse} is taken from~\cite{BauBolTeb17Descent}.

The multiplicative weights algorithm is classical and has found numerous applications; see~\cite{AroHazKal12MultWeights} for a survey.

This section provides an introduction to online learning, although it should be noted that many of the interesting questions revolve around the more challenging setting of \emph{bandit feedback}, i.e., after each round $n$, the player only receives the value $\ell_n(x_n)$ of the incurred loss rather than the full loss function $\ell_n(\cdot)$.
Tackling this setting requires significant new ideas; see, e.g.,~\cite{BubCes12Bandits} for an exposition.

\subsection*{Exercises}

\begin{question}\mbox{}\label{qu:bregman_properties}
    \begin{enumerate}
        \item Prove that for all $x,x'\in\eu C_\phi$, $D_\phi(x,x') = D_{\phi^*}(\nabla \phi(x'), \nabla \phi(x))$.
        \item Let $\eu C \subseteq \eu C_\phi$ be a closed convex set and let $\Pi_{\eu C}^\phi : \eu C_\phi \to \eu C$ denote the Bregman projection operator:
            \begin{align*}
                \Pi_{\eu C}^\phi(x)
                \deq \argmin_{\eu C \cap \eu C_\phi} D_\phi(\cdot, x)\,.
            \end{align*}
            Show that $\langle \nabla \phi(\Pi_{\eu C}^\phi(x)) - \nabla \phi(x), \Pi_{\eu C}^\phi(x) - z \rangle \le 0$ for all $z\in \eu C$.
            Use this to justify the Pythagorean inequality
            \begin{align*}
                D_\phi(z, x)
                &\ge D_\phi(z, \Pi_{\eu C}^\phi(x)) + D_\phi(\Pi_{\eu C}^\phi(x), x)\,.
            \end{align*}
        \item Let $X$ be a random variable with $\E\abs{\phi(X)} < \infty$.
            For any $v\in\eu C_\phi$, establish the identity $\E D_\phi(X,v) - \E D_\phi(X,\E X) = D_\phi(\E X, v)$.
            From this, deduce that the Bregman barycenter coincides with the usual mean: $\argmin_{v\in\eu C_\phi} \E D_\phi(X, v) = \E X$.
    \end{enumerate}
\end{question}

\begin{question}\label{qu:strcvx_wrt_norm}
    Show that if $\phi$ is $\alpha$-convex relative to a norm $\opnorm \cdot$, then $\phi^*$ is $\alpha^{-1}$-smooth relative to the dual norm $\opnorm \cdot_*$.
\end{question}

\begin{question}\label{qu:pinsker}
    Prove that the entropic mirror map is $1$-convex relative to $\norm \cdot_1$ over the probability simplex $\Delta_d$.
\end{question}

\begin{question}\label{qu:bregman_proj_simplex}
    For the entropic mirror map $\phi$, prove that the Bregman projection $\Pi^\phi_{\Delta_d}$ onto the probability simplex simply normalizes the vector: $x\mapsto x/\norm x_1$.
\end{question}

\begin{question}\label{qu:mpgd_update}\mbox{}
    \begin{enumerate}
        \item More generally, show that~\ref{eq:MPGD} can be rewritten as the update
            \begin{align*}
                x_{n+1}
                &= \prox_{hg}^\phi\bigl(\nabla \phi^*(\nabla \phi(x_n) - h\,\nabla f(x_n))\bigr)\,.
            \end{align*}
        \item Consider the mirror map $\phi : x \mapsto -\sum_{i=1}^d \log x[i]$, defined over $\R_+^d$.
            Compute $\prox_{h\,\norm \cdot_1}^\phi$, the Bregman proximal operator for $\norm \cdot_1$.
    \end{enumerate}
\end{question}

\begin{question}\label{qu:mirror_holder}
    Let $\phi$ be a mirror map and assume that $F = f+g$, where $f$ is $\alpha_f$-convex and $g$ is $\alpha_g$-convex, $\alpha_f, \alpha_g \ge 0$.
    Instead of assuming that $f$ is relatively smooth, however, we instead assume that $D_f \le \beta_f D_\phi^{(1+s)/2}$ for some $0 \le s < 1$.

    Note that when $\phi = \norm \cdot^2/2$, this corresponds to
    \begin{align*}
        f(y) - f(x) - \langle \nabla f(x), y-x\rangle
        &\le \beta_f\,\bigl( \frac{\norm{y-x}}{2}\bigr){\bigsp}^{1+s}\,, \qquad \text{for all}~x,y\in\R^d\,.
    \end{align*}
    The case $s=0$ corresponds to Lipschitz $f$, whereas the case $s\nearrow 1$ corresponds to smooth $f$.
    The case $0 \le s < 1$ corresponds to partial smoothness (H\"older continuity of $\nabla f$).

    We consider the update
    \begin{align*}
        x^+
        &\deq \argmin_{y\in\R^d}{\bigl\{ f(x) + \langle \nabla f(x), y-x \rangle + g(y) + \frac{1}{h}\,D_\phi(y,x)\bigr\}}\,.
    \end{align*}
    \begin{enumerate}
        \item Prove that there is a constant $C_s > 0$ depending only on $s$ (which you do not need to specify) such that
            \begin{align*}
                &(1+\alpha_g h)\,D_\phi(y, x^+)
                \le (1-\alpha_f h)\,D_\phi(y,x) - h\,(F(x^+) - F(y)) + C_s\, {(\beta_f h)}^{2/(1-s)}\,.
            \end{align*}
        \item Suppose that $\alpha_f = \alpha_g = 0$.
            What iteration complexity does this imply to reach $F(\bar x_N) - F_\star \le \varepsilon$, as a function of $\beta_f$, $R \deq \sqrt{D_\phi(x_\star,x_0)}$, $s$, and $\varepsilon$?
            (Ignore numerical constants depending on $s$.)
    \end{enumerate}
\end{question}

\begin{question}\label{qu:poisson_inverse}
    Consider the problem of recovering an image $x\in\R_{++}^d$ from a noisy observation $y \approx Ax$, where $y\in\R_+^n$ and $A \in \R_{++}^{n\times d}$ is a matrix with positive entries.
    To solve this problem, we can set up the problem of minimizing 
    \begin{align*}
        x\mapsto D_{\phi_{\msf{ent}}}(y, Ax) + \lambda\,\norm x_1\,,
    \end{align*}
    where $\phi_{\msf{ent}}$ is the entropic mirror map.
    We apply~\ref{eq:MPGD}, using the negative logarithm as a mirror map, i.e., $\phi : x\mapsto -\sum_{i=1}^d \log x[i]$.
    Show that the first term in the objective is relatively convex and smooth, with smoothness constant bounded by $\norm y_1$.
    Deduce that we can obtain an $\varepsilon$-approximate solution in $O(\norm y_1\,D_\phi(x_\star, x_0)/\varepsilon)$ iterations.
    Also, write out the algorithm iterates explicitly.
\end{question}

\begin{question}\label{qu:zero_sum_games}
    Consider a two-player zero-sum game: the set of actions for player $1$ is $[m]$, and the set of actions for player $2$ is $[n]$.
    Let $A$ be the cost matrix for player $1$, i.e., $A[i,j]$ is the cost incurred by player $1$ (thus, the reward gained by player $2$) when player $1$ chooses action $i$ and player $2$ chooses action $j$.

    A \emph{mixed strategy} for player $1$ is a probability distribution over the action set $[m]${---}that is, it is an element of the simplex $\Delta_m$.
    Similarly, a mixed strategy for player $2$ is an element of $\Delta_n$.
    Note that if the players play according to mixed strategies $x$ and $y$ respectively, the expected cost is $\langle x,A\,y\rangle$.

    A \emph{Nash equilibrium} is a pair of mixed strategies $x_\star \in \Delta_m$, $y_\star \in \Delta_n$ such that
    \begin{align*}
        \langle x_\star, A\,y_\star\rangle = \min_{x\in\Delta_m} \max_{y\in\Delta_n}{\langle x, A\,y\rangle} = \max_{y\in\Delta_n} \min_{x\in\Delta_m}{\langle x, A\,y\rangle}\,.
    \end{align*}
    In words: even if player $1$ knows that player $2$ will use the mixed strategy $y_\star$, player $1$ has no incentive to deviate from the strategy $x_\star$, and vice versa.
    A priori, it is not obvious that a Nash equilibrium exists, but we will show how to compute it via multiplicative weights (or mirror descent).

    Indeed, consider running multiplicative weights where at each round $t$, the loss vector $\ell_t$ is the column $i_t$ of $A$, where $i_t \in \argmax_{i\in [n]} (A^\T x_t)[i]$ (this corresponds to player $2$ choosing the best action $i_t$ in response to the mixed strategy $x_t$).
    Show that, after $O((\norm A_{\max}^2 \log m)/\varepsilon^2)$ rounds, we obtain a mixed strategy $\hat x$ such that
    \begin{align*}
        \max_{y\in\Delta_m}{\langle \hat x, A\,y \rangle} \le \max_{y\in \Delta_n}\min_{x\in\Delta_m}{\langle x, A\,y\rangle} + \varepsilon\,.
    \end{align*}
    Here, $\norm A_{\max} \deq \max_{i\in [m]} \max_{j\in [n]}{\abs{A[i,j]}}$.
    Explain why this computes the value of the game, $\langle x_\star, A\,y_\star\rangle$, up to error $\varepsilon$.

    \emph{Remark}. By letting $\varepsilon\searrow 0$ and carrying out a convergence argument, this implies the existence of $x_\star$ (and by a symmetric argument, the existence of $y_\star$).
    This is known as von Neumann's minimax theorem.
\end{question}


\section{Alternating minimization}

In this section, we study the method of alternating minimization.
The goal is to minimize a function $f$ by decomposing the optimization variable $x$ into $D$ variables $x^1,\dotsc,x^D$.
In this decomposition, the individual variables do not have to be one-dimensional, so we let $x^i \in \R^{d_i}$.
The method is defined as follows:
\begin{align*}
    x_{n+1}^i
    &\deq \argmin_{x^i\in\R^{d_i}} f(x_{n+1}^1,\dotsc,x_{n+1}^{i-1}, x^i, x_n^{i+1},\dotsc,x_n^D)\,.
\end{align*}
In other words, we iterate through the variables cyclically and minimize $f$ over the $i$-th variable $x^i$, holding the other variables fixed.
The decomposition is chosen so that it is cheap to compute the minimizer over each individual variable.

\begin{ex}[low-rank matrix recovery]\label{ex:burer_monteiro}
    Suppose that we want to recover an unknown matrix $X_\star \in \R^{p_1\times p_2}$ which is observed through noisy observations $y_i \approx \langle A_i, X_\star\rangle$; here, the matrices $A_i \in \R^{p_1\times p_2}$ are part of the design and are known.
    If we further posit that $X_\star$ is low-rank, say of rank at most $r$, then we aim to solve
    \begin{align*}
        \minimize_{X \in \R^{p_1\times p_2}}\quad\sum_{i=1}^n {(y_i - \langle A_i, X \rangle)}^2 \qquad\text{such that}\qquad\rank X \le r\,.
    \end{align*}
    The rank constraint is difficult to deal with, so we instead factorize the matrix as $X = UV^\T$ where $U \in \R^{p_1\times r}$ and $V \in \R^{p_2\times r}$.
    This factorization is known as the \emph{Burer{--}Monteiro factorization}, after~\cite{BurMon03LowRank, BurMon05LocalMin}.
    The problem becomes
    \begin{align*}
        \minimize_{U\in\R^{p_1\times r},\,V\in\R^{p_2\times r}} \quad \sum_{i=1}^n {(y_i - \langle A_i, UV^\T\rangle)}^2\,.
    \end{align*}
    This is a non-convex problem, but at least it is now amenable to gradient-based methods.
    Alternatively, we can apply alternating minimization.
    In words, we minimize over $U$ while holding $V$ fixed, and then minimize over $V$ while holding $U$ fixed, and so on.
    Each iteration corresponds to solving an unconstrained least-squares problem and admits a closed-form solution.
\end{ex}

Although we present~\autoref{ex:burer_monteiro} as motivation for the design of alternating minimization methods in practice, as usual in these notes, we focus on guarantees in the convex case.
Nevertheless, the analysis of the convex case still applies to relevant problems; see the bibliographical notes for examples.

\subsection{Alternating projections}

We can use alternating minimization to find a point in the intersection of two closed convex sets $\eu C_1$ and $\eu C_2$.
In this case, we take
\begin{align*}
    f(x,y) = \chi_{\eu C_1}(x) + \chi_{\eu C_2}(y) + \norm{y-x}^2\,.
\end{align*}
If there exists $x_\star \in \eu C_1 \cap \eu C_2$, then $(x_\star,x_\star)$ is a minimizer for $f$, and the alternating minimization algorithm reads
\begin{align*}
    x_{n+1}
    &\deq \argmin_{x\in\R^d} f(x,y_n)
    = \Pi_{\eu C_1}(y_n)\,, \\
    y_{n+1}
    &\deq \argmin_{y\in\R^d} f(x_{n+1},y)
    = \Pi_{\eu C_2}(x_{n+1})\,.
\end{align*}
Thus, we alternate projecting onto $\eu C_1$ and onto $\eu C_2$.
This method is quite useful when projections onto $\eu C_1$, $\eu C_2$ individually are cheap, but the projection onto $\eu C_1\cap \eu C_2$ is expensive.
The method easily generalizes to the intersection of more than two convex sets.

As a ``warm up'', we first study the method of alternating projections.
Actually, we consider a generalization to alternating Bregman projections, which is needed for \S\ref{ssec:eot}:
\begin{align}\label{eq:ABP}\tag{$\msf{ABP}$}
    x_{n+1}
    &\deq \Pi_{\eu C_1}^\phi(y_n)\,, \qquad y_{n+1} \deq \Pi_{\eu C_2}^\phi(x_{n+1})\,,
\end{align}
where $\Pi_{\eu C}^\phi$ is the Bregman projection from~\autoref{qu:bregman_properties}.
We assume that $\eu C_1\cap\eu C_2\ne\varnothing$ and that $\eu C_1,\eu C_2 \subseteq \eu C_\phi$.

\begin{lem}\label{lem:ABP}
    For any $x_\star \in\eu C_1\cap \eu C_2$, the iterates of~\ref{eq:ABP} satisfy
    \begin{align*}
        \sum_{n=0}^\infty \{D_\phi(x_n, y_{n-1}) + D_\phi(y_n, x_n)\} \le D_\phi(x_\star, y_0)\,.
    \end{align*}
    Also, monotonicity holds:
    \begin{align*}
        D_\phi(x_\star, y_0)
        &\ge D_\phi(x_\star, x_1)
        \ge D_\phi(x_\star, y_1)
        \ge \cdots\,.
    \end{align*}
\end{lem}
\begin{proof}
    By the Pythagorean inequality (\autoref{qu:bregman_properties}),
    \begin{align*}
        D_\phi(x_\star, y_n) \ge D_\phi(x_\star, \Pi_{\eu C_1}^\phi(y_n)) + D_\phi(\Pi_{\eu C_1}^\phi(y_n), y_n)
        = D_\phi(x_\star, x_{n+1}) + D_\phi(x_{n+1},y_n)\,.
    \end{align*}
    By adding this to a similar inequality for the other projection and summing,
    \begin{align*}
        D_\phi(x_\star, y_0)
        &\ge \sum_{n=1}^\infty \{D_\phi(x_n, y_{n-1}) + D_\phi(y_n, x_n)\}\,. \qedhere
    \end{align*}
\end{proof}

We can use the preceding lemma to prove a convergence result for~\ref{eq:ABP}.
The following corollary relies on two additional technical assumptions for $\phi$ which must be checked, but note that they hold for the Euclidean case $\phi = \frac{\norm \cdot^2}{2}$.

\begin{cor}[convergence of~\ref{eq:ABP}]
    Assume that the following conditions hold:
    \begin{enumerate}
        \item For any $x\in\eu C_\phi$, the sublevel sets of $D_\phi(x,\cdot)$ are compact.
        \item If ${\{z_n\}}_{n\in\N}, {\{z_n'\}}_{n\in\N} \subseteq \eu C_\phi$ are such that $D_\phi(z_n,z_n') \to 0$, then $z_n - z_n' \to 0$.
    \end{enumerate}
    Then, the iterates of~\ref{eq:ABP} satisfy $x_n\to x_\star$ and $y_n\to x_\star$ for some $x_\star \in\eu C_1\cap\eu C_2$.
\end{cor}
\begin{proof}
    The first assumption ensures that there is a convergent subsequence ${\{x_{n_k}\}}_{k\in\N}$ that converges to some $x_\star \in \eu C_\phi$.
    Since $x_n \in \eu C_1$ for all $n$ and $\eu C_1$ is closed, then $x_\star \in \eu C_1$.
    Moreover, by~\autoref{lem:ABP}, $D_\phi(\Pi_{\eu C_2}^\phi(x_n), x_n) = D_\phi(y_n, x_n) \to 0$, so the second property shows that $\Pi_{\eu C_2}^\phi(x_n) - x_n\to 0$.
    Since $\eu C_2$ is closed, $x_\star \in \eu C_2$ as well.

    To upgrade the subsequential convergence to full convergence, we observe that $D_\phi(x_\star, x_{n_k})\to 0$, whence the monotonicity statement in~\autoref{lem:ABP} implies $D_\phi(x_\star, x_n) \to 0$ and $D_\phi(x_\star, y_n) \to 0$.
    By the second assumption, $x_n \to x_\star$ and $y_n\to x_\star$.
\end{proof}

Furthermore,~\autoref{lem:ABP} implies
\begin{align}\label{eq:ABP_guarantee}
    \min_{n=1,\dotsc,N}\{D_\phi(x_n, y_{n-1}) + D_\phi(y_n,x_n)\} \le \frac{D_\phi(x_\star,y_0)}{N}\,.
\end{align}
This does not, however, imply a rate of convergence for $x_n$ to $x_\star$.
In the Euclidean case, we can derive a rate of convergence by interpreting alternating projection as an instance of~\ref{eq:PSD}, see~\autoref{qu:alt_proj_as_psd}.

\subsection{Convergence analysis for alternating minimization}

We now return to the alternating minimization method.
We use the shorthand $x^S$ to denote the components in $S$, $x^S \deq {\{x^i\}}_{i\in S}$, where we abbreviate consecutive indices $\{i,\dotsc,j\}$ as $i{:}j$.
We perform an analysis in the smooth case.
However, similarly to how gradient-based methods do not suffer from non-smoothness provided that one has access to a proximal oracle, it turns out that coordinate-based methods do not suffer from non-smoothness provided that the non-smooth part respects the coordinate decomposition.
Hence, we consider the slightly more general problem of minimizing
\begin{align*}
    F : \R^{d_1\times\cdots\times d_D} \to \R\,, \qquad
    F(x^{1:D}) \deq f(x^{1:D}) + \sum_{i=1}^D g_i(x^i)\,,
\end{align*}
where $f$ is convex and smooth, and each $g_i$ is convex.
For shorthand, write $g \deq \bigoplus_{i=1}^D g_i$, that is, $g(x^{1:D}) \deq \sum_{i=1}^D g_i(x^i)$.
The algorithm reads
\begin{align}\label{eq:AM}\tag{$\msf{AM}$}
    x_{n+1}^i
    &\in \argmin_{x^i \in \R^{d_i}}{\{f(x_{n+1}^{1:i-1}, x^i, x_n^{i+1:D}) + g_i(x^i)\}}\,.
\end{align}

\begin{thm}[convergence of~\ref{eq:AM}]\label{thm:AM}
    Assume that $f$ is convex and $\beta$-smooth, and that each $g_i$ is convex.
    Then,~\ref{eq:AM} achieves $F(x_N^{1:D}) - F_\star \le \varepsilon$ if
    \begin{align*}
        N
        &\ge \Bigl(\log_2 \frac{F(x_0^{1:D}) - F_\star}{4\beta D^2 R^2}\Bigr){\Bigsp}_+ + \frac{8\beta D^2 R^2}{\varepsilon}\,,
    \end{align*}
    where $R \deq \sup_{n\in\N}{\norm{x_n^{1:D} - x_\star^{1:D}}}$.
\end{thm}
\begin{proof}
    By~\eqref{eq:pre_coercivity},
    \begin{align*}
        f(x_n^{1:D})
        &\ge f(x_{n+1}^1, x_n^{2:D}) + \langle \nabla_1 f(x_{n+1}^1, x_n^{2:D}), x_n^1 - x_{n+1}^1 \rangle \\
        &\qquad{} + \frac{1}{2\beta}\,\norm{\nabla f(x_{n+1}^1, x_n^{2:D}) - \nabla f(x_n^{1:D})}^2\,.
    \end{align*}
    On the other hand, since $\nabla_1 f(x_{n+1}^1, x_n^{2:D}) \in -\partial g_1(x_{n+1}^1)$,
    \begin{align*}
        g_1(x_n^1) + \langle \nabla_1 f(x_{n+1}^1, x_n^{2:D}), x_n^1 - x_{n+1}^1\rangle
        &\ge g_1(x_{n+1}^1)\,.
    \end{align*}
    By summing these two inequalities together with the corresponding ones for the other coordinates, we obtain a ``descent lemma''
    \begin{align*}
        F(x_n^{1:D}) - F(x_{n+1}^{1:D})
        &\ge \frac{1}{2\beta}\sum_{i=1}^D \norm{\nabla f(x_{n+1}^{1:i}, x_n^{i+1:D}) - \nabla f(x_{n+1}^{1:i-1}, x_n^{i:D})}^2\,.
    \end{align*}

    Next,
    \begin{align*}
        F(x_{n+1}^{1:D}) - F_\star
        &\le \langle \nabla f(x_{n+1}^{1:D}), x_{n+1}^{1:D} - x_\star^{1:D} \rangle + g(x_{n+1}^{1:D}) - g(x_\star^{1:D}) \\
        &= \sum_{i=1}^D \{\langle \nabla_i f(x_{n+1}^{1:D}), x_{n+1}^i - x_\star^i\rangle + g_i(x_{n+1}^i) - g_i(x_\star^i)\} \\
        &\le \sum_{i=1}^D \langle \nabla_i f(x_{n+1}^{1:D}) - \nabla_i f(x_{n+1}^{1:i}, x_n^{i+1:D}), x_{n+1}^i - x_\star^i\rangle \\
        &\le \sum_{i=1}^D \norm{\nabla f(x_{n+1}^{1:D}) - \nabla f(x_{n+1}^{1:i}, x_n^{i+1:D})}\,\norm{x_{n+1}^i - x_\star^i} \\
        &\le \sum_{i=1}^D \Bigl( \sum_{j=i}^{D-1} \norm{\nabla f(x_{n+1}^{1:j+1}, x_n^{j+2:D}) - \nabla f(x_{n+1}^{1:j}, x_n^{j+1:D})}\Bigr)\,\norm{x_{n+1}^i - x_\star^i} \\
        &\le \Bigl(\sum_{i=0}^{D-1} \norm{\nabla f(x_{n+1}^{1:i+1}, x_n^{i+2:D}) - \nabla f(x_{n+1}^{1:i}, x_n^{n+1:D})}\Bigr)\,\Bigl(\sum_{i=1}^D \norm{x_{n+1}^i - x_\star^i}\Bigr) \\
        &\le D\sqrt{\Bigl(\sum_{i=0}^{D-1} \norm{\nabla f(x_{n+1}^{1:i+1}, x_n^{i+2:D}) - \nabla f(x_{n+1}^{1:i}, x_n^{n+1:D})}^2\Bigr)\,\Bigl(\sum_{i=1}^D \norm{x_{n+1}^i - x_\star^i}^2\Bigr)} \\
        &\le DR\sqrt{\sum_{i=0}^{D-1} \norm{\nabla f(x_{n+1}^{1:i+1}, x_n^{i+2:D}) - \nabla f(x_{n+1}^{1:i}, x_n^{n+1:D})}^2}\,.
    \end{align*}
    Combined with the previous inequality, it yields, for $\Delta_n \deq F(x_n^{1:D}) - F_\star$,
    \begin{align*}
        \Delta_{n+1} - \Delta_n
        &\le - \frac{1}{2\beta}\sum_{i=1}^D \norm{\nabla f(x_{n+1}^{1:i}, x_n^{i+1:D}) - \nabla f(x_{n+1}^{1:i-1}, x_n^{i:D})}^2
        \le - \frac{1}{2\beta D^2 R^2}\,\Delta_{n+1}^2\,.
    \end{align*}
    If $\Delta_{n+1} \ge \Delta_n/2$, this yields $\Delta_{n+1} - \Delta_n \le -\Delta_n^2/(8\beta D^2 R^2)$, so in general
    \begin{align*}
        \Delta_{n+1} \le \max\bigl\{ \frac{1}{2}\,,\, \bigl(1- \frac{\Delta_n}{8\beta D^2 R^2}\bigr)\bigr\}\,\Delta_n\,.
    \end{align*}
    This implies that the error decays exponentially fast until iteration $n_0$ which satisfies $\Delta_{n_0} \le 4\beta D^2 R^2$.
    Thereafter,
    \begin{align*}
        \frac{1}{\Delta_n} - \frac{1}{\Delta_{n+1}}
        &= \frac{\Delta_{n+1}-\Delta_n}{\Delta_n \Delta_{n+1}}
        \le - \frac{1}{8\beta D^2 R^2}\,,
    \end{align*}
    which yields
    \begin{align*}
        \Delta_N
        &\le \frac{8\beta D^2 R^2\,\Delta_{n_0}}{8\beta D^2 R^2 + (N-n_0)\,\Delta_{n_0}}
        \le \frac{8\beta D^2 R^2}{N-n_0}\,. \qedhere
    \end{align*}
\end{proof}

Although~\autoref{thm:AM} provides a convergence guarantee for~\ref{eq:AM}, it incurs a poor dependence on the number of blocks $D${---}the complexity scales as $D^3$.
It turns out that this can be alleviated by randomly choosing a block at each iteration.
More precisely, define the following randomized alternating minimization algorithm:
\begin{align}\label{eq:RAM}\tag{$\msf{RAM}$}
    x_{n+1}
    &\deq \argmin_{x^{i(n)} \in \R^{d_{i(n)}}}{\{f(x_n^{1:i(n)-1}, x^{i(n)}, x_n^{i(n)+1:D}) + g_{i(n)}(x^{i(n)})\}}\,, \qquad i(n) \sim \unif([D])\,.
\end{align}
The analysis below also handles anisotropic smoothness: we assume that for each $i$,
\begin{align*}
    f(x^{1:i-1}, \cdot, x^{i+1:D})~\text{is}~\beta_i\text{-smooth for each}~x^{1:D} \in \R^{d_1\times \cdots \times d_D}\,.
\end{align*}
We refer to this condition succinctly by saying that $f$ is $\bs\beta$-smooth, where $\bs\beta = (\beta_1,\dotsc,\beta_D)$.
It induces the norm
\begin{align*}
    \norm{x^{1:D}}_{\bs\beta}
    &\deq \Bigl( \sum_{i=1}^D \beta_i\,\norm{x^i}^2 \Bigr){\Bigsp}^{1/2}\,.
\end{align*}

\begin{thm}[convergence of~\ref{eq:RAM}]\label{thm:RAM}
    Assume that $f$ is $\bs\beta$-smooth and $\alpha_f$-convex w.r.t.\ $\norm \cdot_{\bs \beta}$.
    Also, assume that $g$ is $\alpha_g$-convex w.r.t.\ $\norm \cdot_{\bs \beta}$.
    Then, the iterates of~\ref{eq:RAM} satisfy the following bounds.
    If $\alpha_f + \alpha_g > 0$, then
    \begin{align*}
        \E F(x_N^{1:D}) - F_\star
        \le \Bigl(1 - \frac{\alpha_f + \alpha_g}{(1+\alpha_g)\,D}\Bigr){\Bigsp}^N\,(F(x_0^{1:D}) - F_\star)\,.
    \end{align*}
    Otherwise, if $\alpha_f + \alpha_g = 0$, then
    \begin{align*}
        \E F(x_N^{1:D}) - F_\star
        &\le \frac{2DR_{\bs\beta}^2}{N}\,,
    \end{align*}
    where $R_{\bs\beta} \ge \sup_{n\in\N}\max\{\sqrt{F(x_n^{1:D}) - F_\star},\, \norm{x_n^{1:D} - x_\star^{1:D}}_{\bs \beta}\}$ almost surely.
\end{thm}

Before proving the theorem, let us compare the computational costs implied by these various results in the weakly convex case.
Suppose, for the sake of argument, that computing a full gradient $\nabla f$ is $O(D)$.
If the proximal oracle for $g$ is available, we can run~\ref{eq:PGD} to obtain an $\varepsilon$-solution at cost $O(\beta DR^2/\varepsilon)$.
On the other hand, each iteration of~\ref{eq:AM} requires minimization with respect to one of the variables, leading to a total cost of roughly $O(\beta D^3 R^2/\varepsilon)$, assuming that minimization over one variable is comparable in cost to computing a partial gradient.

For~\ref{eq:RAM}, let $\beta_{\max} \deq \max_{i\in [D]}\beta_i$.
The overall smoothness of $f$ satisfies $\beta_{\max} \le \beta \le \sum_{i=1}^D \beta_i \le D\beta_{\max}$ and, ignoring the first term in the definition of $R_{\bs\beta}$, $R_{\bs\beta}^2 \le \beta_{\max} R^2$.\footnote{The result for~\ref{eq:RAM} requires an almost sure bound on the iterates, but let us ignore this detail for the sake of this discussion.}
The cost for~\ref{eq:RAM} is therefore $O(\beta_{\max} DR^2/\varepsilon)$, which is ``never worse'' than the rate for~\ref{eq:PGD}, but could be substantially better when $\beta$ is closer to the upper bound $D\beta_{\max}$.
This is the case when the directions of high smoothness are not aligned with the coordinate directions (e.g., imagine that the Hessian matrix looks like the all-ones matrix).
Thus,~\ref{eq:RAM} can ``adapt'' to directional smoothness, which is particularly appealing since~\ref{eq:RAM} has no tuning parameters (not even a step size!).

\begin{proof}[Proof of~\autoref{thm:RAM}]
    Let $y^{1:D} \in \R^{d_1\times\cdots\times d_D}$ and let $\E_n$ denote the expectation over $i(n)$ only.
    Then,
    \begin{align*}
        \E_n F(x_{n+1}^{1:D})
        &\le \E_n F(x_n^{1:i(n)-1}, y^{i(n)}, x_n^{i(n)+1:D}) \\
        &\le \E_n\bigl[f(x_n^{1:D}) + \langle \nabla_{i(n)} f(x_n^{1:D}), y^{i(n)} - x_n^{i(n)} \rangle + \frac{\beta_{i(n)}}{2}\,\norm{y^{i(n)} - x_n^{i(n)}}^2 \bigr] \\
        &\qquad{} + \E_n\Bigl[\sum_{i\ne i(n)} g_i(x_n^i) + g_{i(n)}(y^{i(n)})\Bigr] \\[0.25em]
        &= f(x_n^{1:D}) + \frac{1}{D} \,\langle \nabla f(x_n^{1:D}), y^{1:D} - x_n^{1:D}\rangle + \frac{1}{2D}\,\norm{y^{1:D} - x_n^{1:D}}_{\bs \beta}^2 \\[0.25em]
        &\qquad{} + \frac{D-1}{D}\,g(x_n^{1:D}) + \frac{1}{D}\,g(y^{1:D}) \\[0.25em]
        &\le \frac{D-1}{D}\, f(x_n^{1:D}) + \frac{1}{D}\, f(y^{1:D}) + \frac{1-\alpha_f}{2D}\,\norm{y^{1:D} - x_n^{1:D}}_{\bs \beta}^2 \\[0.25em]
        &\qquad{} + \frac{D-1}{D}\,g(x_n^{1:D}) + \frac{1}{D}\,g(y^{1:D})\,.
    \end{align*}
    By taking the infimum over $y^{1:D}$, we have shown that
    \begin{align*}
        \E_n F(x_{n+1}^{1:D})
        &\le \frac{D-1}{D}\,F(x_n^{1:D}) + \frac{1}{D}\,Q_{1/(1-\alpha_f)} F(x_n^{1:D})\,,
    \end{align*}
    where ${(Q_t)}_{t\ge 0}$ denotes the Hopf{--}Lax semigroup (\autoref{defn:hopf_lax}) with respect to $\norm \cdot_{\bs \beta}$.
    By~\autoref{qu:HJ_decay}, we have
    \begin{align*}
        \frac{Q_{1/(1-\alpha_f)} F(x_n^{1:D}) - F_\star}{F(x_n^{1:D}) - F_\star}
        \le \begin{cases}
            (1-\alpha_f)/(1+\alpha_g)\,, &\text{if}~\alpha_f + \alpha_g > 0\,, \\
            1 - (F(x_n^{1:D})-F_\star)/(2R_{\bs\beta}^2)\,, &\text{if}~\alpha_f+\alpha_g = 0\,.
        \end{cases}
    \end{align*}
    By taking the expectation again, in the first case it yields
    \begin{align*}
        \E F(x_{n+1}^{1:D}) - F_\star
        &\le \Bigl(1 - \frac{\alpha_f+\alpha_g}{(1+\alpha_g)\,D}\Bigr)\,\{\E F(x_n^{1:D}) - F_\star\}\,,
    \end{align*}
    and in the second case, by Jensen's inequality,
    \begin{align*}
        \E F(x_{n+1}^{1:D}) - F_\star
        &\le \E\Bigl[\Bigl(1 - \frac{F(x_n^{1:D})-F_\star}{2DR_{\bs\beta}^2}\Bigr)\,\{F(x_n^{1:D}) - F_\star\}\Bigr] \\[0.25em]
        &\le \Bigl(1 - \frac{\E F(x_n^{1:D})-F_\star}{2DR_{\bs\beta}^2}\Bigr)\,\{\E F(x_n^{1:D}) - F_\star\}\,.
    \end{align*}
    The results follow by iterating these inequalities.
\end{proof}

\subsection{Case study: entropic optimal transport}\label{ssec:eot}

As a case study, we apply these ideas to a concrete problem of modern interest.
Namely, over the past decade, there has been considerable interest in applications of optimal transport to machine learning, among other domains.
In this problem, we are given two probability distributions $\mu$, $\nu$ over spaces $\eu X$ and $\eu Y$ respectively, as well as a cost function $c : \eu X\times\eu Y \to\R$.
In this section, we focus on the case where $\eu X$, $\eu Y$ are finite sets, although the ideas presented here readily generalize.
The optimal transport cost between $\mu$ and $\nu$ with cost $c$ is
\begin{align*}
    \msf{OT}(\mu,\nu)
    &\deq \inf\{\E c(X,Y) : X \sim \mu\,, \, Y\sim \nu\}\,,
\end{align*}
where the infimum is taken over all \emph{couplings} $(X,Y)$, i.e., jointly defined random variables with marginal laws $\mu$ and $\nu$ respectively.
A particularly common choice is the Euclidean cost, $c(x,y) = \norm{y-x}^2$, but other choices are common too.
Since the structure of the cost does not play any role here, we leave it general.

Focus on the case where $\eu X$, $\eu Y$ are finite sets.
If we write $\gamma$ for the joint distribution of $(X,Y)$, the optimal transport problem can be written
\begin{align*}
    \minimize_{\gamma \in \R_+^{\eu X\times\eu Y}}\qquad\sum_{x\in\eu X,\,y\in\eu Y} c_{x,y}\gamma_{x,y}\qquad\text{such that}\qquad \begin{cases}
        \sum_{y\in\eu Y} \gamma_{x,y} = \mu_x~\text{for all}~x\in\eu X\,, \\
        \sum_{x\in\eu X} \gamma_{x,y} = \nu_y~\text{for all}~y\in\eu Y\,.
    \end{cases}
\end{align*}
More compactly, if we write $C = {(c_{x,y})}_{x\in\eu X,\,y\in\eu Y}$ for the cost matrix and $\mb 1_{\eu X}$, $\mb 1_{\eu Y}$ for the all-ones vectors on $\eu X$ and on $\eu Y$ respectively, this can be written
\begin{align*}
    \minimize_{\gamma \in \R_+^{\eu X\times\eu Y}}\qquad \langle C, \gamma\rangle\qquad\text{such that}\qquad \begin{cases}
        \gamma \mb 1_{\eu Y} = \mu\,, \\
        \gamma^\T \mb 1_{\eu X} = \nu\,.
    \end{cases}
\end{align*}
This is readily recognized as a linear program, but solving it as such is expensive.
There are specialized combinatorial algorithms{---}see~\cite{PeyCut19OT}{---}but the computational cost scales at least as $d^3$ if $d = \abs{\eu X} = \abs{\eu Y}$ (for simplicity, the input dimensions are the same).

On the other hand, the size of the input matrix $C$ is $d^2$, and optimistically we ask if we can solve the problem in $\widetilde O(d^2)$ time{---}that is, nearly linear time in the size of the input.
We shall see that this is the case, provided that we add some entropic regularization to the problem, as popularized in machine learning by Cuturi~\cite{Cut13Sinkhorn}.

Given a regularization parameter $\varepsilon_{\msf{reg}} > 0$, the goal is to instead solve
\begin{align}\label{eq:EOT}
    \minimize_{\gamma \in \R_+^{\eu X\times\eu Y}}\qquad \langle C, \gamma\rangle + \varepsilon_{\msf{reg}}\KL(\gamma \mmid \mu\otimes \nu)\qquad\text{such that}\qquad \begin{cases}
        \gamma \mb 1_{\eu Y} = \mu\,, \\
        \gamma^\T \mb 1_{\eu X} = \nu\,.
    \end{cases}
\end{align}
Call the value of this problem $\msf{OT}_{\varepsilon_{\msf{reg}}}(\mu,\nu)$.
Why does this make the problem so much easier to solve?
The answer is that by Kantorovich duality, the dual to the unregularized problem turns out to be
\begin{align*}
    \maximize_{f\in\R^{\eu X},\,g\in\R^{\eu Y}}\qquad\sum_{x\in\eu X} f_x \mu_x + \sum_{y\in\eu Y} g_y \nu_y\qquad\text{such that}\qquad f_x + g_y \le c_{x,y}~\text{for all}~x\in\eu X\,,\,y\in\eu Y\,,
\end{align*}
which is still a constrained problem.
On the other hand, the dual to the regularized problem is unconstrained.

\begin{thm}[EOT duality]\label{thm:eot_duality}
    Consider the following dual problem:
    \begin{align}\label{eq:eot_dual}
        \maximize_{f\in\R^{\eu X},\,g\in\R^{\eu Y}}\quad\sum_{x\in\eu X} f_x \mu_x + \sum_{y\in\eu Y} g_y \nu_y - \varepsilon_{\msf{reg}}\,\biggl(\sum_{x\in\eu X,\,y\in\eu Y} \exp\Bigl( \frac{f_x + g_y - c_{x,y}}{\varepsilon_{\msf{reg}}}\Bigr)\,\mu_x\nu_y - 1\biggr)
    \end{align}
    Let $f^\star$, $g^\star$ solve the dual problem.
    Then, $\gamma^\star$ defined by
    \begin{align}\label{eq:eot_plan}
        \gamma^\star_{x,y}
        &\deq \exp\Bigl( \frac{f^\star_x + g^\star_y - c_{x,y}}{\varepsilon_{\msf{reg}}}\Bigr)\,\mu_x \nu_y
    \end{align}
    is the unique solution to the entropic optimal transport problem.

    Moreover, $\gamma^\star$ is characterized as the unique distribution of the form~\eqref{eq:eot_plan} for some $f^\star$, $g^\star$ with the correct marginals.
\end{thm}

For a proof, see, e.g.,~\cite[Proposition 4.3]{CheNilRig25OT}, although in this discrete setting it can be proven via Lagrange multipliers (\autoref{qu:EOT_duality}).

By replacing $c_{x,y}$ with $c_{x,y}/\varepsilon_{\msf{reg}}$ and rescaling $f_x$, $g_y$ accordingly, \underline{we may set $\varepsilon_{\msf{reg}}$ = 1} without loss of generality, so we adopt this convention henceforth.

Let $\eu D(f,g)$ denote the dual functional, i.e., the objective of~\eqref{eq:eot_dual}.
Since the dual is unconstrained, we propose to solve it by alternating maximization.
Namely, given iterates $f^n$, $g^n$, we set
\begin{align*}
    f^{n+1} \deq \argmax_{f\in\R^{\eu X}} \eu D(f, g^n)\,, \qquad g^{n+1} \deq \argmax_{g\in\R^{\eu Y}} \eu D(f^{n+1}, g)\,.
\end{align*}
By solving for the first-order conditions, the updates can be written explicitly:
\begin{align}\label{eq:sinkhorn}
    f^{n+1}_x
    &= -\log \sum_{y\in\eu Y} \exp(g^n_y - c_{x,y})\,\nu_y\,, \qquad
    g_y^{n+1}
    = -\log\sum_{x\in\eu X} \exp(f^{n+1}_x-c_{x,y})\,\mu_x\,.
\end{align}
At this point, one can try to apply~\autoref{thm:AM}, but the correct geometry for this problem is not Euclidean.

Instead, consider what happens to the matrices
\begin{align*}
    \gamma^n_{x,y} \deq \exp(f_x^n + g_y^n - c_{x,y})\,\mu_x\nu_y\,, \qquad \gamma^{n+1/2}_{x,y} \deq \exp(f_x^{n+1} + g_y^n - c_{x,y})\,\mu_x \nu_y\,.
\end{align*}
Performing the update $f^n \mapsto f^{n+1}$ implicitly performs the update $\gamma^n \mapsto \gamma^{n+1/2}$.
The $\eu X$-marginal of $\gamma^{n+1/2}$ is computed as follows:
\begin{align*}
    \sum_{y\in\eu Y} \gamma_{x,y}^{n+1/2}
    &= \mu_x \sum_{y\in\eu Y} \exp(f_x^{n+1} + g_y^n - c_{x,y})\,\nu_y
    = \mu_x\,,
\end{align*}
by~\eqref{eq:sinkhorn}.
Thus, $\gamma^{n+1/2}$ has the correct $\eu X$-marginal $\mu$, although its $\eu Y$-marginal may not be correct.
In fact, one can see that $\gamma^{n+1/2}$ is obtained from $\gamma^n$ by normalizing its rows to fix its $\eu X$-marginal.
Similarly, the update $g^n \mapsto g^{n+1}$ implicitly performs the update $\gamma^{n+1/2} \mapsto \gamma^{n+1}$, and $\gamma^{n+1}$ has the correct $\eu Y$-marginal $\nu$.
In this form, this is known as \emph{Sinkhorn's matrix scaling algorithm}~\cite{Sinkhorn1964}.
In words, Sinkhorn's algorithm iteratively ``fixes the rows, and then fixes the columns, and then fixes the rows\ldots''.

We can therefore identify Sinkhorn's algorithm as an instance of alternating Bregman projections.
Indeed, we define the constraint sets
\begin{align*}
    \eu C^\mu
    &\deq \{\gamma \in \R_+^{\eu X\times \eu Y} : \gamma \mb 1_{\eu Y} = \mu\}\,, \qquad \eu C_\nu \deq \{\gamma \in\R_+^{\eu X\times \eu Y} : \gamma^\T \mb 1_{\eu X} = \nu\}\,,
\end{align*}
and we let $\phi : \gamma \mapsto \sum_{x\in\eu X,\,y\in\eu Y} (\gamma_{x,y} \log \gamma_{x,y} - \gamma_{x,y})$ denote the entropic mirror map, then similarly to~\autoref{qu:bregman_proj_simplex} one can show that the Bregman projections onto $\eu C^\mu$ and $\eu C_\nu$ normalize the rows and columns respectively.
This yields
\begin{align*}
    \gamma^{n+1/2} = \Pi_{\eu C^\mu}^\phi(\gamma^n)\,, \qquad \gamma^{n+1} = \Pi_{\eu C_\nu}^\phi(\gamma^{n+1/2})\,.
\end{align*}

From this perspective, Sinkhorn's algorithm aims to find a point in the intersection $\eu C^\mu \cap \eu C_\nu$.
Does this mean that the intersection is a singleton, which is the solution to the entropic optimal transport problem?
No!
Note that the intersection $\eu C^\mu \cap \eu C_\nu$ only encodes the \emph{constraints} of the original problem, not the objective which depends on the cost function $c$.
In fact, by~\autoref{thm:eot_duality}, different choices of $c$ give rise to different entropic optimal couplings, so $\eu C^\mu \cap \eu C_\nu$ is certainly not a singleton.

What is true, however, is that if $\Gamma_c$ denotes the set of joint distributions of the form~\eqref{eq:eot_plan}, then the unique element of $\eu C^\mu \cap \eu C_\nu \cap \Gamma_c$ solves the entropic optimal transport problem.
This is the last statement of~\autoref{thm:eot_duality}.
Moreover, Sinkhorn's algorithm maintains the property that if we initialize at an element of $\Gamma_c$, e.g., by taking $f^0 = g^0 = 0$, then the algorithm iterates all remain in $\Gamma_c$.
So, remarkably, the alternating Bregman projections do indeed solve our problem.

Let us now see what~\autoref{lem:ABP} implies for Sinkhorn's algorithm.
In the following, we assume that we initialize at a probability distribution $\gamma^0 \in \Gamma_c$, e.g., we can take
\begin{align*}
    \gamma^0 = \frac{\exp(-c)\,(\mu\otimes \nu)}{\norm{\exp(-c)\,(\mu\otimes \nu)}_1}\,.
\end{align*}

\begin{thm}[convergence of Sinkhorn's algorithm]\label{thm:sinkhorn}
    Initialize Sinkhorn's algorithm at a probability distribution $\gamma^0 \in \Gamma_c$.
    Suppose that the number of iterations $N$ satisfies
    \begin{align*}
        N \ge \frac{2\KL(\gamma^\star \mmid \gamma^0)}{\varepsilon^2}\,.
    \end{align*}
    Then, there exists an iteration $n \in \{0,1,\dotsc,N-1\}$ and $\hat\gamma \in \{\gamma^n, \gamma^{n+1/2}\}$ such that if $\hat\mu$, $\hat\nu$ denote the marginals of $\hat\gamma$, then
    \begin{align*}
        \norm{\hat\mu-\mu}_1 + \norm{\hat\nu-\nu}_1 \le \varepsilon\,.
    \end{align*}
\end{thm}
\begin{proof}
    By~\eqref{eq:ABP_guarantee}, we know that
    \begin{align*}
        \min_{n=0,1,\dotsc,N-1}{\{\KL(\gamma^{n+1/2} \mmid \gamma^n) + \KL(\gamma^{n+1} \mmid \gamma^{n+1/2})\}} \le \frac{\KL(\gamma^\star \mmid \gamma^0)}{N}\,.
    \end{align*}
    Let $\mu^n$ denote the $\eu X$-marginal of $\gamma^n$:
    \begin{align*}
        \mu^n_x
        &= \sum_{y\in\eu Y} \gamma^n_{x,y}
        = \mu_x \exp(f^n_x) \sum_{y\in\eu Y} \exp(g^n_y - c_{x,y})\,\nu_y
        = \mu_x \exp(f^n_x - f^{n+1}_x)\,.
    \end{align*}
    Therefore, since the $\eu X$-marginal of $\gamma^{n+1/2}$ is $\mu$,
    \begin{align*}
        \KL(\gamma^{n+1/2} \mmid \gamma^n)
        &= \sum_{x\in\eu X,\,y\in\eu Y} \gamma^{n+1/2}_{x,y} \log \frac{\exp(f^{n+1}_x + g^n_y - c_{x,y})}{\exp(f^n_x + g^n_y - c_{x,y})}
        = \sum_{x\in\eu X,\,y\in\eu Y} \gamma^{n+1/2}_{x,y} \,(f_x^{n+1} - f_x^n) \\
        &= \sum_{x\in\eu X} \mu_x \,(f_x^{n+1} - f_x^n)
        = \sum_{x\in\eu X} \mu_x \log \frac{\mu_x}{\mu^n_x}
        = \KL(\mu \mmid \mu^n)
        \ge \frac{1}{2}\,\norm{\mu^n - \mu}_1^2\,,
    \end{align*}
    where the last inequality is Pinsker's inequality.
    The result follows.
\end{proof}

However, this is not the last word on Sinkhorn's algorithm.
For instance, it does not provide convergence of the last iterate.
It turns out that Sinkhorn's algorithm admits a third interpretation: as an instantiation of mirror descent (\autoref{qu:sinkhorn_mirror}).
Using this, one can prove the following theorem.

\begin{thm}[convergence of Sinkhorn's algorithm, II]\label{thm:sinkhorn_ii}
    Initialize Sinkhorn's algorithm at a probability distribution $\gamma^0 \in \Gamma_c$.
    Then, if $\mu^N$ denotes the $\eu X$-marginal of $\gamma^N$,
    \begin{align*}
        \KL(\mu^N \mmid \mu)
        &\le \frac{\KL(\gamma^\star \mmid \gamma^0)}{N}\,.
    \end{align*}
\end{thm}

\subsection*{Bibliographical notes}

The analysis of alternating minimization has recently inspired analyses of the coordinate ascent variational inference (CAVI) algorithm~\cite{ArnLac24CAVI, LavZan24CAVI}, the expectation maximization (EM) algorithm~\cite{CapJoh25EM}, and Gibbs sampling~\cite{AscLavZan24Gibbs}.
The proof of~\autoref{thm:RAM} is also taken from~\cite{LavZan24CAVI}.

For an introduction to optimal transport for statisticians, see~\cite{CheNilRig25OT}.
Other treatments of optimal transport, aimed at more mathematical audiences, include~\cite{Vil03Topics, Vil09OT, San15OT}.
The literature on (entropic) optimal transport is vast, so we only mention a few relevant references: the proof of~\autoref{thm:sinkhorn} is similar in spirit to~\cite{AltNilRig17Sinkhorn}; Sinkhorn's algorithm as interpreted as mirror descent in~\cite{Leg21Sinkhorn}; and the interpretation in~\autoref{qu:sinkhorn_mirror} is from~\cite{AubKorLeg22Sinkhorn}.

\subsection*{Exercises}

\begin{question}\label{qu:alt_proj_as_psd}\mbox{}
    \begin{enumerate}
        \item Consider the problem of finding a point in the intersection of two closed convex sets $\eu C_1$ and $\eu C_2$.
            Show that the alternating (Euclidean) projection method can be interpreted as subgradient descent on the function $\max_{i=1,2}\dist(\cdot, \eu C_i)$ where $\dist(x,\eu C) \deq \inf_{y\in\eu C}{\norm{y-x}}$.
            Identify the step sizes.
        \item Generalize this to the problem of finding a point in $\bigcap_{i=1}^m \eu C_i$, and show that the resulting subgradient method can be interpreted as a ``greedy'' alternating projection method.
            Use the analysis of~\autoref{thm:psd} to obtain a rate of convergence.
    \end{enumerate}
\end{question}

\begin{question}\label{qu:am_pl}
    Here, we present a simple convergence analysis of alternating minimization in the case where there are only two blocks, $f$ satisfies~\eqref{eq:PL} with constant $\alpha$ and is $\beta$-smooth, and $g = 0$.

    For any $h > 0$, by definition of alternating minimization,
    \begin{align*}
        f(x_{n+1}^1, x_n^2)
        &\le f(x_n^1 - h\,\nabla_1 f(x_n^1, x_n^2),\, x_n^2)\,.
    \end{align*}
    Apply the descent lemma for~\ref{eq:GD} (\autoref{lem:descent}), the fact that $\nabla_2 f(x_n^1, x_n^2) = 0$ (why?), and the P\L{} inequality to deduce a one-step inequality $f(x_{n+1}^1, x_n^2) - f_\star \le (1-\alpha/\beta)\,(f(x_n^1,x_n^2) - f_\star)$.
    Deduce that
    \begin{align*}
        f(x_N^1, x_N^2) - f_\star
        &\le \bigl(1 - \frac{\alpha}{\beta}\bigr){\bigsp}^{2N}\,(f(x_0^1, x_0^2) - f_\star)\,.
    \end{align*}
\end{question}

\begin{question}\label{qu:gauss_southwell}
    Consider coordinate descent with the Gauss{--}Southwell rule:
    \begin{align*}
        x_{n+1} \deq x_n - h\,\nabla_{i_n} f(x_n)\,e_{i_n}\,, \qquad i_n = \argmax_{i\in [d]}{\abs{\nabla_i f(x_n)}}\,.
    \end{align*}
    This is a ``greedy'' version of coordinate descent.
    \begin{enumerate}
        \item Show that
            \begin{align*}
                x_{n+1}
                &= \argmin_{x\in\R^d}{\bigl\{f(x_n) + \langle \nabla f(x_n), x-x_n\rangle + \frac{1}{2h}\,\norm{x-x_n}_1^2\bigr\}}\,.
            \end{align*}
        \item Therefore, argue that if $f$ is smooth in the $\ell_1$-norm,
            \begin{align*}
                f(y) \le f(x) + \langle \nabla f(x), y-x\rangle + \frac{\beta}{2}\,\norm{y-x}_1^2\,,
            \end{align*}
            then for $h = 1/\beta$ we have the descent lemma
            \begin{align*}
                f(x_{n+1})
                &\le f(x_n) - \frac{1}{2\beta}\,\norm{\nabla f(x_n)}_\infty^2\,.
            \end{align*}
        \item If $f$ satisfies a~\ref{eq:PL} inequality in the $\ell^1$-norm,
            \begin{align*}
                \norm{\nabla f(x)}_\infty^2
                &\ge 2\alpha\,(f(x) - f_\star)\,, \qquad\text{for all}~x\in\R^d\,,
            \end{align*}
            then for $\kappa = \beta/\alpha$,
            \begin{align*}
                f(x_N) - f_\star
                &\le \bigl(1 - \frac{1}{\kappa}\bigr){\bigsp}^N\,(f(x_0) - f_\star)\,.
            \end{align*}
    \end{enumerate}
    The moral of the story is that coordinate methods operate in the $\ell_1$ geometry.
\end{question}

\begin{question}\mbox{}\label{qu:EOT_duality}
    \begin{enumerate}
        \item Setting $\varepsilon_{\msf{reg}} = 1$ in~\eqref{eq:EOT}, show that the objective is equivalent to minimizing $\gamma\mapsto \KL(\gamma \mmid \gamma^0)$, where $\gamma^0_{x,y} = \mu_x \nu_y \exp(-c_{x,y})$.
        \item Introduce two Lagrange multipliers $\kappa,\lambda\in\R$ and argue that~\eqref{eq:EOT} is equivalent to
            \begin{align*}
                \min_{\gamma\in\R_+^{\eu X\times\eu Y}} \;\max_{\kappa,\lambda\in\R}\; \bigl\{\KL(\gamma \mmid \gamma^0) + \kappa\,(\gamma \mb 1_{\eu Y} - \mu) + \lambda\,(\gamma^\T \mb 1_{\eu X} - \nu)\bigr\}\,.
            \end{align*}
            Without justification, assume that we can switch the $\min$ and $\max$, so that the above problem is equivalent to
            \begin{align*}
                \max_{\kappa,\lambda\in\R}\;\min_{\gamma\in\R_+^{\eu X\times\eu Y}} \; \bigl\{\KL(\gamma \mmid \gamma^0) + \kappa\,(\gamma \mb 1_{\eu Y} - \mu) + \lambda\,(\gamma^\T \mb 1_{\eu X} - \nu)\bigr\}\,.
            \end{align*}
            Argue that the solution $\gamma$ to this problem is of the form $\gamma_{x,y} = \exp(f_x + g_y - c_{x,y})\,\mu_x \nu_y$ for some $f \in \R^{\eu X}$, $g\in\R^{\eu Y}$.
    \end{enumerate}
\end{question}

\begin{question}\label{qu:sinkhorn_mirror}
    For a joint distribution $\gamma$, let $(\Pi_{\eu X}){}_\# \gamma$ denote its $\eu X$-marginal.
    Consider the objective functional $\eu F : \gamma \mapsto \KL((\Pi_{\eu X}){}_\# \gamma \mmid \mu)$.
    Show that the iteration $\gamma^n \mapsto \gamma^{n+1}$ of Sinkhorn's algorithm can be viewed as one step of mirror descent on $\eu F$ with the entropic mirror map $\phi$, constraint set $\eu C_\nu$, and step size $1$.
    By checking relative convexity and smoothness, prove~\autoref{thm:sinkhorn_ii}.
\end{question}

\section{Stochastic optimization}

Our next topic is optimization with stochastic gradients, which dates back to the field of stochastic approximation~\cite{RobMon1951}.
Besides its relevance in situations where the gradient cannot be computed exactly, stochastic optimization is particularly important for machine learning and statistics for at least two reasons.
First, it can be viewed as a method for directly minimizing the population risk, and we can establish generalization bounds provided that we perform a single pass over our data.
Second, it is routinely used to minimize empirical risk functions by approximating the full gradient by mini-batches over the data. 
Our treatment therefore centers around these applications.

\subsection{Stochastic mirror proximal gradient descent}

We start with the fundamental convergence result.
Suppose that we wish to minimize $F = f+g$, where we only have access to stochastic gradients for $f$.
More precisely, we assume that at each $x\in\interior \dom f$, we can compute a random vector $\hat\nabla f(x)$ which is unbiased: $\E \hat\nabla f(x) = \nabla f(x)$.
Actually, we can also let $f$ be non-smooth, in which case we require that $\E\hat \nabla f(x) \in \partial f(x)$.
The analysis below can also handle the case where the stochastic gradient is biased, at the expense of an additional error term.

Let $\phi : \R^d\to\R\cup \{\infty\}$ be a mirror map.
We consider the following iteration:
\begin{align}\label{eq:SMPGD}\tag{$\msf{SMPGD}$}
    x_{n+1}
    &\deq \argmin_{x\in\R^d}{\bigl\{ f(x_n) + \langle \hat\nabla f(x_n), x-x_n \rangle + g(x) + \frac{1}{h}\,D_\phi(x,x_n) \bigr\}}\,.
\end{align}
This is the \emph{stochastic mirror proximal gradient descent} algorithm.\footnote{Yes, the name is comically long.}
For most applications, we do not need all of these aspects (stochastic, mirror, proximal) simultaneously, but we may as well include them to emphasize that a unified proof is possible.
Anyway, it is helpful to include a proximal term since it allows for projections, and the use of a Bregman divergence is a bonus since it covers stochastic mirror descent.

\begin{thm}[convergence of~\ref{eq:SMPGD}]\label{thm:SMPGD}
    Let $\phi$ be a mirror map which is $\alpha_\phi$-convex relative to a norm $\opnorm \cdot$.
    We assume that $f$ is $\alpha_f$-convex and $g$ is $\alpha_g$-convex, relative to $\phi$; we let $\lambda_h \deq (1-\alpha_f h)/(1+\alpha_g h)$.
    We assume that the stochastic gradient is unbiased.
    \begin{itemize}
        \item (smooth case) Assume that $f$ is $\beta_f$-smooth relative to $\phi$, that $h \le 1/(2\beta_f)$, and that we have a variance bound for the stochastic gradient:
            \begin{align}\label{eq:sgd_var_bd}
                \E[\opnorm{\hat\nabla f(x) -\nabla f(x)}_*^2] \le \sigma^2 d \qquad\text{for all}~x\in \eu C_\phi\,.
            \end{align}
            Then, for a suitably averaged iterate $\bar x_N$,
            \begin{align*}
                \E F(\bar x_N) - F_\star
                &\le \frac{\alpha_f+\alpha_g}{\lambda_h^{-N}-1}\,D_\phi(x_\star,x_0) + \frac{(1+\alpha_g h)\,\sigma^2 dh}{\alpha_\phi}\,.
            \end{align*}
        \item (non-smooth case) Assume that the stochastic gradients are $L^2$-bounded,
            \begin{align}\label{eq:sgd_grad_bd}
                \E[\opnorm{\hat\nabla f(x)}_*^2] \le L^2 \qquad\text{for all}~x\in\eu C_\phi\,.
            \end{align}
            Then, for a suitably averaged iterate $\bar x_N$,
            \begin{align*}
                \E F(\bar x_N) - F_\star
                &\le \frac{\alpha_f+\alpha_g}{\lambda_h^{-N}-1}\,D_\phi(x_\star,x_0) + \frac{2\,(1+\alpha_g h)\,L^2 h}{\alpha_\phi}\,.
            \end{align*}
    \end{itemize}
\end{thm}
\begin{proof}
    We prove a one-step inequality for
    \begin{align*}
        x^+ \deq \argmin \psi_x\,, \qquad \psi_x(y) \deq f(x) + \langle \hat\nabla f(x), y-x\rangle + g(y) + \frac{1}{h}\,D_\phi(y,x)\,.
    \end{align*}
    By the relative growth inequality (\autoref{lem:bregman_growth}), for any $y \in \eu C_\phi$,
    \begin{align*}
        \bigl(\alpha_g + \frac{1}{h}\bigr)\,D_\phi(y,x^+) + \psi_x(x^+)
        &\le \psi_x(y)\,.
    \end{align*}
    For the right-hand side, in both cases,
    \begin{align*}
        \E \psi_x(y)
        &= f(x) + \langle \nabla f(x), y-x\rangle + g(y) + \frac{1}{h}\,D_\phi(y,x)
        = F(y) - D_f(y,x) + \frac{1}{h}\,D_\phi(y,x) \\
        &\le F(y) + \frac{1-\alpha_f h}{h}\,D_\phi(y,x)\,.
    \end{align*}

    \textbf{Smooth case.}
    Here, we lower bound $\E \psi_x(x^+)$ as follows: since $h\le 1/(2\beta_f)$,
    \begin{align*}
        \E \psi_x(x^+)
        &= \E\bigl[f(x) + \langle \hat\nabla f(x), x^+ - x\rangle + g(x^+) + \frac{1}{h}\,D_\phi(x^+, x)\bigr] \\
        &= \E\bigl[f(x) + \langle \nabla f(x), x^+ - x\rangle + g(x^+) + \frac{1}{h}\,D_\phi(x^+, x) + \langle \hat\nabla f(x) - \nabla f(x), x^+ - x\rangle\bigr] \\
        &= \E\bigl[F(x^+) - D_f(x^+, x) + \frac{1}{h}\,D_\phi(x^+, x) + \langle \hat\nabla f(x) - \nabla f(x), x^+ - x\rangle\bigr] \\
        &\ge \E\bigl[F(x^+) + \frac{1}{2h}\,D_\phi(x^+, x) - \opnorm{\hat\nabla f(x) - \nabla f(x)}_*\,\opnorm{x^+ - x}\bigr] \\[0.5em]
        &\ge \E\bigl[F(x^+) + \frac{\alpha_\phi}{4h}\,\opnorm{x^+ - x}^2\bigr] - \sqrt{\E[\opnorm{\hat\nabla f(x) - \nabla f(x)}_*^2]\E[\opnorm{x^+ - x}^2]} \\[0.5em]
        &\ge \E\bigl[F(x^+) + \frac{\alpha_\phi}{4h}\,\opnorm{x^+ - x}^2\bigr] - \sqrt{\sigma^2 d \E[\opnorm{x^+ - x}^2]}
        \ge \E F(x^+) - \frac{\sigma^2 dh}{\alpha_\phi}\,.
    \end{align*}
    This leads to the one-step bound
    \begin{align*}
        (1+\alpha_g h)\E D_\phi(y, x^+)
        &\le (1-\alpha_f h)\,D_\phi(y,x) - h\,(\E F(x^+) - F(y)) + \frac{\sigma^2 dh^2}{\alpha_\phi}\,.
    \end{align*}

    \textbf{Non-smooth case.}
    In this case, note that
    \begin{align*}
        \opnorm{\nabla f(x)}_*
        &= \opnorm{\E\hat\nabla f(x)}_*
        \le \sqrt{\E[\opnorm{\hat\nabla f(x)}_*^2]}
        \le L\,,
    \end{align*}
    so that $f$ is $L$-Lipschitz with respect to $\opnorm \cdot$.
    Then,
    \begin{align*}
        \E \psi_x(x^+)
        &= \E\bigl[f(x) + \langle \hat\nabla f(x), x^+ - x\rangle + g(x^+) + \frac{1}{h}\,D_\phi(x^+, x)\bigr] \\[0.25em]
        &= \E\bigl[F(x^+) + f(x) - f(x^+) + \langle \hat\nabla f(x), x^+ - x\rangle + \frac{1}{h}\,D_\phi(x^+, x)\bigr] \\[0.25em]
        &\ge \E\bigl[F(x^+) -L\,\opnorm{x^+ - x} - \opnorm{\hat\nabla f(x)}_*\,\opnorm{x^+ - x} + \frac{1}{h}\,D_\phi(x^+, x)\bigr] \\[0.5em]
        &\ge \E\bigl[F(x^+) + \frac{\alpha_\phi}{2h}\,\opnorm{x^+ - x}^2\bigr] - 2L\sqrt{\E[\opnorm{x^+ - x}^2]}
        \ge \E F(x^+) - \frac{2L^2 h}{\alpha_\phi}\,.
    \end{align*}
    This leads to the one-step bound
    \begin{align*}
        (1+\alpha_g h)\E D_\phi(y, x^+)
        &\le (1-\alpha_f h)\,D_\phi(y,x) - h\,(\E F(x^+) - F(y)) + \frac{2L^2 h^2}{\alpha_\phi}\,.
    \end{align*}

    \textbf{Completing the proof.}
    Observe that the one-step bounds have exactly the same form in both cases.
    We take $y = x_\star$ and iterate as usual.
    Let $E = \sigma^2 dh^2/\alpha_\phi$ in the first case, and $E = 2L^2 h^2/\alpha_\phi$ in the second case.
    Then, the discrete Gr\"onwall lemma (\autoref{lem:discrete_gronwall}) and some computations yield
    \begin{align*}
        \sum_{n=1}^N \frac{\lambda_h^{N-n}}{\sum_{k=1}^N \lambda_h^k}\,\{\E F(x_n) - F_\star\}
        &\le \frac{\alpha_f + \alpha_g}{\lambda_h^{-N}-1}\,D_\phi(x_\star,x_0) + \frac{1+\alpha_g h}{h}\,E\,.
    \end{align*}
    This yields a convergence rate for a suitably averaged iterate.
\end{proof}

For simplicity, let us assume that $\alpha_\phi = 1$ and $R^2 \ge D_\phi(x_\star, x_0)$.
We let $\alpha \deq \alpha_f + \alpha_g$.
To obtain $\varepsilon$ error,~\autoref{thm:SMPGD} implies the rates in~\autoref{tab:sgd}.

\begin{table}[h]
    \centering
    \begin{tabular}{cc}
        \textbf{Assumptions} & \textbf{Iterations} \\
        convex, smooth & $O(\sigma^2 dR^2/\varepsilon^2)$ \\
        strongly convex, smooth & $O(\sigma^2 d \log(\alpha R^2/\varepsilon)/(\alpha \varepsilon))$ \\
        convex, non-smooth & $O(L^2 R^2/\varepsilon^2)$ \\
        strongly convex, non-smooth & $O(L^2 \log(\alpha R^2/\varepsilon)/(\alpha \varepsilon))$
    \end{tabular}
    \caption{Rates for~\ref{eq:SMPGD} with an appropriate step size and averaging.}\label{tab:sgd}
\end{table}

A few remarks are in order.
\begin{enumerate}
    \item The effect of the stochasticity of the gradients, at least under our assumptions, is qualitatively the same as the effect of non-smoothness.
        Indeed, the rates of $O(1/\varepsilon^2)$ under convexity and $O(1/\varepsilon)$ under strong convexity reflect the rates for~\ref{eq:PSD} in \S\ref{ssec:projected_subgradient}.
        Similarly, in this setting we do not have a descent lemma, and hence we should average the iterates.
    \item In the non-smooth case, stochasticity of the gradients does not hurt the rate at all, provided that the stochastic gradients satisfy an appropriate $L^2$ bound.
    \item By Jensen's inequality, it is not hard to see that~\eqref{eq:sgd_grad_bd} implies~\eqref{eq:sgd_var_bd} with $\sigma^2 d \le 4L^2$.
        Hence, the variance bound~\eqref{eq:sgd_var_bd} is indeed a weaker assumption, although the analysis requires a stronger assumption{---}smoothness{---}for $f$.
        At first glance, it may seem that~\eqref{eq:sgd_var_bd} and~\eqref{eq:sgd_grad_bd} are similar, but the former is a variance bound and the latter is an $L^2$ bound.
        Suppose, for instance, that $\opnorm \cdot$ is the Euclidean norm, and that $\hat\nabla f$ is computed on the basis of a mini-batch of $B$ samples.
        Then, we expect~\eqref{eq:sgd_var_bd} to decay as $1/B$, whereas~\eqref{eq:sgd_grad_bd} does not.
    \item Since stochastic optimization behaves like non-smooth optimization, we also expect that it is not possible to ``accelerate''.
        Indeed, many of the rates in~\autoref{thm:SMPGD} can be shown to be optimal up to logarithmic terms~\cite{RagRak11IBC, Aga+12StochasticOpt}.
\end{enumerate}

\subsection{Implications for statistical generalization}\label{ssec:sgd_generalization}

Suppose that we have a dataset $\{Z_i : i\in [n]\} \subseteq \eu Z$ of i.i.d.\ samples, a parameter space $\Theta \subseteq \R^d$, and a loss $\ell : \Theta \times \eu Z \to \R$.
Define the \emph{empirical risk} and the \emph{population risk}:
\begin{align*}
    \eu R_n(\theta)
    &\deq \frac{1}{n} \sum_{i=1}^n \ell(\theta; Z_i)\,, \qquad \eu R(\theta) \deq \E \eu R_n(\theta)\,.
\end{align*}
We further assume that $\theta \mapsto \ell(\theta; z)$ is $L$-Lipschitz for all $z\in\eu Z$, and that $\Theta$ is is the ball $\msf B(0,R)$ of radius $R$.

\begin{ex}[regression]
    Suppose that $Z_i = (X_i, Y_i)$ with $X_i\in\R^d$ and $Y_i \in [-1,1]$.
    Assume that for each $\theta\in\Theta$, we have a predictor $f_\theta : \R^d\to [-1,1]$.
    Then, we can consider the squared loss
    \begin{align*}
        \ell(\theta; z)
        = \ell(\theta; (x,y))
        = {(y - f_\theta(x))}^2\,,
    \end{align*}
    If we further consider linear regression, $f_\theta : x \mapsto \langle \theta, x \rangle$, then $f_\theta$ is $4r$-Lipschitz for all $z = (x,y)$ with $\norm x \le r$ and $\abs y \le 1$.
    Linear regression is not as restrictive as it may seem, since we can imagine that each $X_i$ is actually the output of a high-dimensional feature map, in which case we obtain kernel regression.
\end{ex}

The setting we have described is actually the standard one for statistical learning theory, and it covers much more than regression (e.g., classification and density estimation), but regression is a good representative example.
Also, since our conclusions below will not involve the dimension $d$, we can think of the function class as infinite-dimensional.

Define the minimizers
\begin{align*}
    \widehat \theta_n \in \argmin_{\theta\in\Theta} \eu R_n(\theta)\,, \qquad \theta^\star \in \argmin_{\theta \in \Theta} \eu R(\theta)\,.
\end{align*}
We view $\theta^\star$ as the ground truth value of the parameter that we wish to recover.
Since we do not have access to the population risk $\eu R$, we must base our procedures on the samples ${\{Z_i\}}_{i\in [n]}$, so it is natural to use $\widehat\theta_n$, the \emph{empirical risk minimizer} (ERM), as our estimator.
This is the starting point for statistical theory, but it says nothing about how we can actually compute $\widehat \theta_n$.

The power of stochastic gradient descent (SGD) is that we can view it as a method to \emph{directly} minimize the population risk.
We consider the iteration
\begin{align}\label{eq:SGD}\tag{$\msf{SGD}$}
    \theta_{k+1}
    &\deq \theta_k - h\,\nabla_\theta \ell(\theta_k; Z_{k+1})\,.
\end{align}
If we denote $\hat \nabla \eu R(\theta) = \nabla_\theta\ell(\theta; Z)$, then this is indeed an unbiased stochastic gradient: $\E\hat\nabla \eu R(\theta) = \nabla \eu R(\theta)$.
Due to our Lipschitz assumption, we also know that
\begin{align*}
    \E[\norm{\hat\nabla \eu R(\theta)}^2] \le L^2\,.
\end{align*}
However, the implicit assumption in~\autoref{thm:SMPGD} is that the randomness is fresh at each iteration, so we are not allowed to reuse any samples.
This limits the total number of iterations of~\ref{eq:SGD} to the sample size $n$, and we refer to this as \emph{one-pass} SGD\@.

From~\autoref{thm:SMPGD}, if we further assume that \emph{$\theta \mapsto \ell(\theta; z)$ is convex for every $z\in\eu Z$}, then~\ref{eq:SGD} with an optimized step size and averaging satisfies
\begin{align*}
    \E \eu R(\bar \theta_n) - \eu R(\theta^\star)
    &\lesssim \frac{LR}{\sqrt n}\,.
\end{align*}
Here, the guarantee is in terms of the difference between the expected risk of our estimator, the averaged iterate of~\ref{eq:SGD}, and the best possible risk.
This is known as the \emph{excess risk} and it is the best we can hope for, since $\theta^\star$ may not be identifiable (unique).

How does this compare to the performance of ERM\@?
Analysis of the ERM estimator starts with the following decomposition:
\begin{align*}
    \eu R(\widehat \theta_n) - \eu R(\theta^\star)
    &= \eu R(\widehat \theta_n) - \eu R_n(\widehat \theta_n) + {\underbrace{\eu R_n(\widehat \theta_n) - \eu R_n(\theta^\star)}_{\le 0}} + \eu R_n(\theta^\star) - \eu R(\theta^\star) \\
    &\le 2\sup_{\theta\in\Theta}{\abs{\eu R_n(\theta) - \eu R(\theta)}}\,.
\end{align*}
We therefore want to show that $\sup_{\theta\in\Theta}{\abs{\eu R_n(\theta) - \eu R(\theta)}}$ tends to zero at a certain rate, which is known as a uniform convergence argument.
This is a type of stochastic process (since $\eu R_n$ is random) known as an empirical process, and sophisticated tools have been developed for its study.
After applying a number of them (symmetrization, contraction principle, control of the Rademacher complexity), one can show that
\begin{align}\label{eq:ERM_rate}
    \E \eu R(\widehat \theta_n) - \eu R(\theta^\star)
    &\lesssim \frac{LR}{\sqrt n}\,,
\end{align}
just as for~\ref{eq:SGD}.

Actually, it is worth remarking that the bounds from empirical process theory depend on various notions of the complexity of the class $\{\ell(\theta; \cdot) : \theta \in \Theta\}$.
A traditional approach measures this complexity essentially by counting the number of free parameters in the class (Vapnik{--}Chervonenkis or VC theory), and would not match the dimension-free rate attained by SGD\@.
In order to do so, one needs to turn toward ``size-based'' measures of complexity that take into account the fact that $\Theta$ lies in a ball, hence the need to control the Rademacher complexity directly.

Anyway, we can show that the ERM estimator satisfies~\eqref{eq:ERM_rate}, and moreover, this argument does not require convexity of the loss.
However, when we discuss how to compute the ERM estimator, we need to assume convexity anyway.
Let us suppose that we compute it by running~\ref{eq:GD} (or more specifically,~\ref{eq:PSD}) on the empirical risk $\eu R_n$.
Since the statistical error is already $LR/\sqrt n$, we only need to optimize to this level of accuracy.
Applying~\autoref{thm:psd}, we see that the number of iterations of~\ref{eq:PSD} is roughly $n$.
This is the same number of iterations as one-pass~\ref{eq:SGD}, except that each iteration of~\ref{eq:SGD} is roughly $n$ times cheaper than the corresponding one for~\ref{eq:PSD}.

In conclusion, one-pass~\ref{eq:SGD} produces an estimator which has comparable statistical performance to the ERM estimator, with a computational cost roughly $n$ times cheaper than minimizing the empirical risk directly using~\ref{eq:PSD}.

The $n^{-1/2}$ dependence on the sample size $n$ is known as a ``slow rate''.
It can be improved when we additionally assume that for every $z\in\eu Z$, $\theta \mapsto \ell(\theta; z)$ is $\alpha$-convex for some $\alpha > 0$.
In this case,~\autoref{thm:SMPGD} yields
\begin{align*}
    \E \eu R(\bar \theta_n) - \eu R(\theta^\star)
    &\lesssim \frac{L^2}{\alpha n}\,.
\end{align*}
(Actually,~\autoref{thm:SMPGD} yields a slightly worse result by a logarithmic factor, but this can be fixed with a time-varying choice of step sizes.)
Due to strong convexity, it also implies parameter recovery:
\begin{align*}
    \sqrt{\E[\norm{\bar \theta_n - \theta^\star}^2]}
    &\lesssim \frac{L}{\alpha\sqrt n}\,.
\end{align*}

What about for the ERM estimator?
This time, one needs a refined argument based on \emph{localized} Rademacher complexities, and it again eventually yields
\begin{align*}
    \E \eu R(\widehat \theta_n) - \eu R(\theta^\star)
    &\lesssim \frac{L^2}{\alpha n}\,.
\end{align*}
As for computation, by applying~\ref{eq:PSD} and the result of~\autoref{qu:psd_str_cvx} (again, omitting the logarithmic factor which can be removed with better step sizes), the conclusion is the same: the number of iterations is the same as for~\ref{eq:SGD}, but each iteration is roughly $n$ times more expensive.

This discussion suggests that, at least when the risk is convex, averaged one-pass~\ref{eq:SGD} is expected to be an excellent estimator, both computationally and statistically.
The next subsection will further reinforce this point.

\subsection{Central limit theorem for Polyak{--}Ruppert averaging}

\textbf{Disclaimer:} This subsection is somewhat technical, so rather than tracing through all of the details, you are encouraged to follow the high-level ideas.

We now turn to a celebrated result in stochastic optimization: namely, that the iterates of SGD with Polyak{--}Ruppert averaging obey a central limit theorem.
Let
\begin{align}\label{eq:ASGD}\tag{$\msf{ASGD}$}
    \theta_{n+1}
    &\deq \theta_n - h_{n+1}\,\hat\nabla f(\theta_n)\,,
    \qquad \bar \theta_n
    \deq \frac{1}{n} \sum_{j=0}^{n-1} \theta_j\,.
\end{align}
Our goal is to show that $\sqrt n\,(\bar \theta_n - \theta^\star) \to \normal(0, \Sigma)$ for a certain covariance matrix $\Sigma$.
Throughout, we write
\begin{align*}
    \hat\nabla f(\theta_n)
    &= \nabla f(\theta_n) + \xi_{n+1}\,,
\end{align*}
where conditionally on $\theta_n$, the random vector $\xi_{n+1}$ has zero mean.
Our condition on the step sizes is
\begin{align}\label{eq:clt_step_size}
    h_n = n^{-\gamma}\,, \qquad \text{for some}~\frac{1}{2} < \gamma < 1\,.
\end{align}

We recall some preliminaries on convergence in distribution.
\begin{itemize}
    \item We say that a sequence ${\{X_n\}}_{n\in\N}$ of random vectors converges in distribution (or converges weakly in law) to a probability distribution $\mu$, denoted $X_n\todist{} \mu$, if $\E f(X_n) \to \int f\,\D \mu$ for every bounded, continuous function $f : \R^d\to\R$.
        This is the same notion of convergence as in the classical central limit theorem (CLT).
    \item We say that ${\{X_n\}}_{n\in\N}$ tends to $0$ in probability if for all $\varepsilon > 0$, $\Pr(\norm{X_n} \ge \varepsilon) \to 0$.
        If $\E\norm{X_n} \to 0$, then $X_n\to 0$ in probability; this follows from Markov's inequality.
    \item If $X_n = Y_n + Z_n$, and $Z_n \to 0$ in probability, then ${\{X_n\}}_{n\in\N}$ and ${\{Y_n\}}_{n\in\N}$ have the same distributional limit.
\end{itemize}

\paragraph*{Quadratic case.}
We begin with the quadratic case with
\begin{align*}
    f : \R^d\to\R\,,\qquad f(\theta) = \frac{1}{2}\,\langle \theta-\theta^\star, A\,(\theta-\theta^\star)\rangle\,.
\end{align*}
We assume $A \succ 0$.
We do not treat this case separately merely as a ``warm-up''; the analysis here is crucial for the general case as well.

In the following, we write $\sup_{n\ge 1} \E[\norm{\xi_n}^{2+}]< \infty$ to mean $\sup_{n\ge 1} \E[\norm{\xi_n}^{2+\delta}] < \infty$ for some $\delta > 0$, i.e., we have slightly more than two moments.

\begin{thm}[CLT for~\ref{eq:ASGD}, quadratic case]\label{thm:clt_quadratic}
    Assume that $f$ is a strongly convex quadratic function, and that the step sizes satisfy the condition~\eqref{eq:clt_step_size}.
    Assume that conditionally on $\theta_n$, each $\xi_{n+1}$ is mean zero and has covariance $\Xi_{n+1}$, such that $n^{-1} \sum_{k=1}^n \Xi_k \to S_\infty$ in probability and $\sup_{n\ge 1}{\E[\norm{\xi_n}^{2+}]} < \infty$.
    Then,
    \begin{align*}
        \sqrt n\,(\bar \theta_n - \theta^\star)
        \todist{} \normal(0, A^{-1} S_\infty A^{-1})\,.
    \end{align*}
\end{thm}

Before turning toward the proof, let us consider a simple example.

\begin{ex}[estimating the mean of a Gaussian]\label{ex:gaussian_clt}
    Suppose that we have samples ${\{X_n\}}_{n\in\N} \simiid \normal(\theta^\star, A^{-1})$ for a known covariance matrix $A \succ 0$, and we wish to estimate the mean $\theta^\star$.
    We consider
    \begin{align*}
        \theta_{n+1}
        &= \theta_n - h_{n+1}\,A\,(\theta_n - X_{n+1})\,, \qquad \bar \theta_n = \frac{1}{n} \sum_{j=0}^{n-1} \theta_j\,.
    \end{align*}
    This corresponds to the quadratic loss function with $\xi_n = A\,(\theta^\star - X_n)$.
    In this case, the ${\{\xi_n\}}_{n\in\N}$ are i.i.d., with mean zero and covariance matrix $A$.

    If we use the sample mean, then
    \begin{align*}
        \sqrt n\,\Bigl( \frac{1}{n} \sum_{j=1}^n X_j - \theta^\star\Bigr)
        &\sim \normal(0, A^{-1})\,.
    \end{align*}
    On the other hand, the CLT for~\ref{eq:ASGD} with $S_\infty = A$ shows that
    \begin{align*}
        \sqrt n\,(\bar \theta_n - \theta^\star)
        &\to \normal(0, A^{-1})\,.
    \end{align*}
\end{ex}

The first step is to write out the iterates.
For $\delta_n \deq \theta_n - \theta^\star$,
\begin{align*}
    \delta_{n+1}
    &= \delta_n - h_{n+1}\,A\delta_n - h_{n+1} \xi_{n+1}
    = (I - h_{n+1} A)\,\delta_n - h_{n+1} \xi_{n+1}\,.
\end{align*}
Unrolling,
\begin{align*}
    \delta_n
    &= \Bigl[\prod_{j=1}^n (I-h_j A)\Bigr]\,\delta_0 - \sum_{j=1}^n \Bigl[h_j \prod_{k=j+1}^n (I-h_k A)\Bigr]\,\xi_j\,.
\end{align*}
Defining $\bar \delta_n \deq n^{-1} \sum_{j=0}^{n-1} \delta_j$, it yields
\begin{align}
    \bar \delta_n
    &= \frac{1}{n} \sum_{j=0}^{n-1} \Bigl[\prod_{k=1}^{j} (I-h_k A)\Bigr]\,\delta_0 - \frac{1}{n} \sum_{j=0}^{n-1} \sum_{k=1}^{j} \Bigl[h_k \prod_{\ell=k+1}^j (I-h_\ell A)\Bigr] \,\xi_k \nonumber\\[0.25em]
    &= \frac{1}{n} \sum_{j=0}^{n-1} \Bigl[\prod_{k=1}^{j} (I-h_k A)\Bigr]\,\delta_0 - \frac{1}{n} \sum_{k=1}^{n-1} \underbrace{h_k \sum_{j=k}^{n-1} \Bigl[\prod_{\ell=k+1}^j (I-h_\ell A)\Bigr]}_{\eqqcolon M_k^n} \,\xi_k \nonumber\\[0.25em]
    &= \frac{1}{n} \,M_0^n\delta_0 - \frac{1}{n} \sum_{k=1}^{n-1} A^{-1} \xi_k - \frac{1}{n} \sum_{k=1}^{n-1} (M_k^n - A^{-1})\, \xi_k\,,\label{eq:clt_iterate}
\end{align}
where we set $h_0 \deq 1$.
The intuition is that if all of the $h_\ell$ were the same, then $M_k^n$ would equal $h \sum_{j=k}^{n-1} {(I-hA)}^{j-k} \to A^{-1}$ as $n\to\infty$, so we hope that the last term tends to zero.

\begin{lem}\label{lem:clt_technical}
    The $M_k^n$ matrices are uniformly bounded in operator norm.
    Also,
    \begin{align*}
        \frac{1}{n} \sum_{k=1}^{n-1} {\norm{M_k^n - A^{-1}}_{\rm op}} \to 0\qquad\text{as}~n\to\infty\,.
    \end{align*}
\end{lem}
\begin{proof}[Proof sketch]
    For intuition, suppose that $S = \sum_{n=0}^\infty {(I-hA)}^n$ for $h < 1/\norm A_{\rm op}$.
    To show that $hS = A^{-1}$, one can argue that $(I-hA)\,S = S - I$, which leads to $hS = A^{-1}$.
    We want to replicate this type of proof, but it is significantly complicated due to the time-varying step sizes.
    Let $B_j^k \deq \prod_{\ell=j}^k (I-h_\ell A)$, so that $M_k^n = h_k\sum_{j=k}^{n-1} B_{k+1}^j$.
    We start with an equation for the $B$'s: since
    \begin{align*}
        B_j^{k+1}
        &= B_j^k - h_{k+1} AB_j^k
        = \cdots
        = I - A\sum_{\ell=j-1}^k h_{\ell+1} B_j^\ell\,,
    \end{align*}
    we can write
    \begin{align*}
        M_k^n
        &= h_k \sum_{j=k}^{n-1} B_{k+1}^j
        = \sum_{j=k}^{n-1} (h_k - h_{j+1} + h_{j+1})\, B_{k+1}^j
        = \sum_{j=k}^{n-1} (h_k - h_{j+1})\,B_{k+1}^j + A^{-1}\,(I-B_{k+1}^n)\,.
    \end{align*}
    Therefore,
    \begin{align*}
        AM_k^n - I
        &= A\sum_{j=k}^{n-1} (h_k - h_{j+1})\,B_{k+1}^j -B_{k+1}^n\,.
    \end{align*}
    Since $A$ is bounded above and below,
    \begin{align*}
        \frac{1}{n} \sum_{k=1}^{n-1} {\norm{M_k^n - A^{-1}}_{\rm op}}
        &\lesssim \frac{1}{n} \sum_{k=1}^{n-1} \sum_{j=k}^{n-1} {\abs{h_k - h_{j+1}}}\,\norm{B_{k+1}^j}_{\rm op} + \frac{1}{n} \sum_{k=2}^n {\norm{B_k^n}_{\rm op}}\,.
    \end{align*}
    Also, it is easy to see that for sufficiently small step sizes (which is all that matters in the asymptotic regime),
    \begin{align*}
        \norm{B_k^n}_{\rm op}
        &\le \prod_{\ell=k}^n (I-\alpha h_\ell)
        \le \exp\Bigl(-\alpha \sum_{\ell=k}^n h_\ell \Bigr)\,.
    \end{align*}

    So, for the second term,
    \begin{align*}
        \frac{1}{n} \sum_{k=2}^n {\norm{B_k^n}_{\rm op}}
        &\le \frac{1}{n} \sum_{k=1}^n \exp\Bigl(-\alpha \sum_{\ell=k}^n h_\ell\Bigr)\,.
    \end{align*}
    Let $\tau_k \deq \sum_{\ell=1}^k h_\ell \approx {(1-\gamma)}^{-1}\, k^{1-\gamma}$.
    If, for any $t > 0$, we define $\tau(t) = t^{1-\gamma}$, the summation is roughly equivalent to the following integral (up to replacing $\alpha$ by $\alpha/(1-\gamma)$):
    \begin{align*}
        I
        &\deq \int_1^n \exp(-\alpha\,(\tau(n) - \tau(k)))\,\D k\,.
    \end{align*}
    For ease of presentation, we focus on bounding the integral instead.
    By change of variables, $t = \tau(k)$,
    \begin{align*}
        I
        &\asymp \int_1^{n^{1-\gamma}} t^{\gamma/(1-\gamma)} \exp(-\alpha\,(n^{1-\gamma} - t))\,\D t
        \lesssim n^\gamma \int_1^{n^{1-\gamma}} \exp(-\alpha\,(n^{1-\gamma} - t))\,\D t
        \lesssim n^\gamma\,.
    \end{align*}
    Since $\gamma < 1$, the second term tends to zero.

    As for the first term, we follow a similar strategy and approximate it via
    \begin{align*}
        &\frac{1}{n} \iint_{1 \le k \le j \le n} \bigl( \frac{1}{k^\gamma} - \frac{1}{j^\gamma}\bigr)\,\exp(-\alpha\,(\tau(j) - \tau(k)))\,\D j\,\D k \\[0.25em]
        &\qquad \asymp \frac{1}{n} \iint_{1\le s \le t \le n^{1-\gamma}} \underbrace{(t^{\gamma/(1-\gamma)} - s^{\gamma/(1-\gamma)})}_{\text{apply Taylor expansion}}\exp(-\alpha\,(t-s))\,\D s\,\D t \\
        &\qquad \lesssim \frac{n^{2\gamma-1}}{n} \iint_{1\le s \le t \le n^{1-\gamma}} (t-s)\exp(-\alpha\,(t-s))\,\D s \,\D t
        \lesssim \frac{n^\gamma}{n} \int_1^{n^{1-\gamma}} t\exp(-\alpha t)\,\D t
        \to 0\,.
    \end{align*}
    This completes the heuristic proof of the convergence.
    The first statement, about the boundedness of the $M_k^n$, can be proved via similar arguments.
\end{proof}

Returning to~\eqref{eq:clt_iterate}, note that
\begin{align*}
    \sqrt n\,\bar \delta_n
    &= \frac{1}{n^{1/2}} \,M_0^n\delta_0 - \frac{1}{n^{1/2}} \sum_{k=1}^{n-1} A^{-1} \xi_k - \frac{1}{n^{1/2}} \sum_{k=1}^{n-1} (M_k^n - A^{-1})\, \xi_k\,,
\end{align*}
where $\E\norm{M_0^n \delta_0}/\sqrt n \to 0$.
For the last term, we use the fact that the $\xi_k$'s are orthogonal: for $k < \ell$, by conditioning on $\theta_{1:\ell-1} \deq (\theta_1,\dotsc,\theta_{\ell-1})$,
\begin{align*}
    \E\langle (M_k^n - A^{-1})\,\xi_k,\, (M_\ell^n - A^{-1})\,\xi_\ell\rangle
    &= \E\langle (M_k^n - A^{-1})\,\xi_k,\, (M_\ell^n - A^{-1})\E[\xi_\ell \mid \theta_{1:\ell-1}]\rangle
    = 0\,.
\end{align*}
Therefore,
\begin{align*}
    \frac{1}{n} \E\Bigl[\Bigl\lVert \sum_{k=1}^{n-1} (M_k^n - A^{-1})\,\xi_k \Bigr\rVert^2\Bigr]
    &= \frac{1}{n} \sum_{k=1}^{n-1} \E[\norm{(M_k^n - A^{-1})\,\xi_k}^2]
    = \frac{1}{n} \sum_{k=1}^{n-1} \langle {(M_k^n - A^{-1})}^2, \E\Xi_k\rangle \\
    &\lesssim \frac{1}{n} \sum_{k=1}^n {\norm{M_k^n - A^{-1}}_{\rm op}}
    \to 0\,.
\end{align*}
Hence, to obtain a distributional limit for $\sqrt n\,\bar \delta_n$, it suffices to obtain one for the second term above.
This will be accomplished via martingale theory.

\paragraph{Martingale CLT.}
In the context of stochastic optimization, the noise sequence ${\{\xi_n\}}_{n\in\N}$ is not i.i.d.; indeed, we want to consider
\begin{align*}
    \xi_n
    &= \hat\nabla f(\theta_n) - \nabla f(\theta_n)\,,
\end{align*}
and since the iterates ${\{\theta_n\}}_{n\in\N}$ are random and depend on the noise sequence, it leads to a complicated dependence structure for the noise sequence.
Nevertheless, it falls within the framework of martingale theory.

\begin{defn}
    An increasing sequence of $\sigma$-algebras ${\{\ms F_n\}}_{n\in\N}$ is called a \textbf{filtration}.
    We think of $\ms F_n$ as the information available to an observer up to iteration $n$.

    A sequence of random vectors ${\{X_n\}}_{n\in\N}$ is called a \textbf{martingale} if for all $n$, $X_n$ is $\ms F_n$-measurable, $\E\norm{X_n} < \infty$, and
    \begin{align*}
        \E[X_{n+1} \mid \ms F_n] = X_n\,.
    \end{align*}
\end{defn}

In other words, the difference $X_{n+1}-X_n$ is conditionally unbiased given the information $\ms F_n$ available at iteration $n$.
If we set $X_n \deq \sum_{k=1}^n \xi_k$, then ${\{X_n\}}_{n\in\N}$ is a martingale; we sometimes refer to the noise sequence ${\{\xi_n\}}_{n\in\N}$ as a \emph{martingale difference sequence}.

Our next goal is to establish the following theorem.

\begin{thm}[martingale CLT]
    Let ${\{\xi_n\}}_{n\in\N}$ be a martingale difference sequence and write $\Xi_{n+1} \deq \cov(\xi_{n+1} \mid \ms F_n)$.
    Assume that $\sup_{n\in\N}\E[\norm{\xi_n}^{2+}] < \infty$ and that $n^{-1} \sum_{k=1}^n \Xi_k \to S_\infty$ in probability.
    Then,
    \begin{align*}
        \frac{1}{\sqrt n} \sum_{k=1}^n \xi_k \todist{} \normal(0, S_\infty) \qquad\text{as}~n\to\infty\,.
    \end{align*}
\end{thm}

This is a special case of a more general theorem on triangular arrays of Lindeberg{--}Feller type.
For simplicity, we prove it under the stronger hypothesis that ${\{\xi_n\}}_{n\in\N}$ is \textbf{uniformly bounded}.

\begin{proof}
    Let $X_k \deq n^{-1/2} \sum_{\ell=1}^k \xi_\ell$; note that this depends on $n$ but we suppress it from the notation.
    Consider the characteristic function: for $\mb i \deq \sqrt{-1}$ and $\lambda \in \R^d$,
    \begin{align*}
        \phi_n(\lambda)
        &\deq \E\exp(\mb i\,\langle \lambda, X_n\rangle)\,.
    \end{align*}
    Due to standard results on Fourier inversion, it suffices to prove that the characteristic function $\phi_n$ converges pointwise to the characteristic function of the Gaussian,
    \begin{align*}
        \phi_\infty(\lambda)
        &\deq \E\exp(\mb i\,\langle \lambda, Z \rangle)
        = \exp\bigl( - \frac{1}{2}\,\langle \lambda, S_\infty\,\lambda\rangle\bigr)\,, \qquad Z \sim \normal(0, S_\infty)\,.
    \end{align*}
    Let $S_k \deq n^{-1} \sum_{\ell=1}^k \Xi_\ell$.
    We start by writing
    \begin{align*}
        \abs{\phi_n(\lambda) - \phi_\infty(\lambda)}
        &\le \bigl\lvert \E\bigl[\exp(\mb i\,\langle \lambda, X_n\rangle)\bigl(1- \exp\bigl( \frac{1}{2}\,\langle \lambda, S_n\,\lambda \rangle - \frac{1}{2}\,\langle \lambda, S_\infty\,\lambda\rangle\bigr)\bigr)\bigr]\bigr\rvert \\[0.25em]
        &\qquad{} + \bigl\lvert \exp\bigl( - \frac{1}{2}\,\langle \lambda, S_\infty\,\lambda \rangle\bigr)\,\bigl( \E\exp\bigl(\mb i\,\langle \lambda, X_n\rangle + \frac{1}{2}\,\langle \lambda, S_n\,\lambda \rangle\bigr) - 1\bigr)\bigr\rvert \\[0.25em]
        &\le \E\bigl\lvert 1- \exp\bigl( \frac{1}{2}\,\langle \lambda, S_n\,\lambda \rangle - \frac{1}{2}\,\langle \lambda, S_\infty\,\lambda\rangle\bigr) \bigr\rvert\\[0.25em]
        &\qquad{} + \bigl\lvert \E\exp\bigl(\mb i\,\langle \lambda, X_n\rangle + \frac{1}{2}\,\langle \lambda, S_n\,\lambda \rangle\bigr) - 1 \bigr\rvert\,.
    \end{align*}
    Since ${\{\xi_n\}}_{n\in\N}$ is bounded, say by $B$, then $\norm{\Xi_k}_{\rm op}\le B^2$, so ${\{S_n\}}_{n\in\N}$ is bounded.
    Since $S_n\to S_\infty$ in probability, the first term above tends to zero as $n\to\infty$.

    For the second term, we peel off the terms in $X_n$ by conditioning.
    Indeed,
    \begin{align*}
        \E[\exp(\mb i \,\langle \lambda, X_n \rangle) \mid \ms F_{n-1}]
        &= \E[\exp(\mb i\, \langle \lambda, X_{n-1} + n^{-1/2} \xi_n\rangle)\mid \ms F_{n-1}] \\
        &= \exp(\mb i\,\langle \lambda, X_{n-1}\rangle) \E[\exp(\mb i\, \langle \lambda, n^{-1/2} \xi_n \rangle) \mid \ms F_{n-1}]\,.
    \end{align*}
    By Taylor expansion,
    \begin{align*}
        \E[\exp(\mb i\,\langle \lambda, n^{-1/2}\xi_n \rangle) \mid \ms F_{n-1}]
        &= \E\bigl[ 1 + \frac{\mb i}{n^{1/2}}\,\langle \lambda, \xi_n \rangle - \frac{1}{2n}\,\langle \lambda, \xi_n \rangle^2 + O(n^{-3/2}) \bigm\vert \ms F_{n-1}\bigr]
    \end{align*}
    where the error term $O(n^{-3/2})$ is uniform, due to the assumption of boundedness.
    By the martingale property, this equals
    \begin{align*}
        1 - \frac{1}{2n}\,\langle \lambda, \Xi_n \,\lambda\rangle + O(n^{-3/2})
        = \exp\bigl( - \frac{1}{2n}\,\langle \lambda, \Xi_n\,\lambda \rangle + O(n^{-3/2})\bigr)\,.
    \end{align*}
    Hence,
    \begin{align*}
        \E\exp\bigl(\mb i\,\langle \lambda, X_n\rangle + \frac{1}{2}\,\langle \lambda, S_n\,\lambda \rangle\bigr)
        &= \E\exp\bigl(\mb i\,\langle \lambda, X_{n-1} \rangle + \frac{1}{2}\,\langle \lambda, S_n\,\lambda \rangle - \frac{1}{2n}\,\langle \lambda, \Xi_n\,\lambda\rangle + O(n^{-3/2})\bigr) \\[0.25em]
        &= \E\exp\bigl(\mb i\,\langle \lambda, X_{n-1} \rangle + \frac{1}{2}\,\langle \lambda\, S_{n-1}\,\lambda\rangle + O(n^{-3/2})\bigr)\,.
    \end{align*}
    Iterating,
    \begin{align*}
        \E\exp\bigl(\mb i\,\langle \lambda, X_n\rangle + \frac{1}{2}\,\langle \lambda, S_n\,\lambda \rangle\bigr)
        &= \exp(O(n^{-1/2}))\,,
    \end{align*}
    so the second error term above also tends to zero.
\end{proof}

This completes the proof of~\autoref{thm:clt_quadratic} since, if $n^{-1/2} \sum_{k=1}^n \xi_k \todist{} \normal(0, S_\infty)$, it follows that $n^{-1/2} \sum_{k=1}^n A^{-1}\xi_k \todist{} \normal(0, A^{-1} S_\infty A^{-1})$.

\paragraph{General case.}
We now return to~\ref{eq:ASGD} for a general function $f$.
In this case, the CLT still holds, where we take $A = \nabla^2 f(\theta^\star)$ to be the Hessian at the minimizer.

\begin{thm}[CLT for~\ref{eq:ASGD}, general case]\label{thm:clt}
    Assume that $f$ is a strongly convex, smooth, and has a bounded third derivative, and that the step sizes satisfy the condition~\eqref{eq:clt_step_size}.
    Assume that conditionally on $\theta_n$, each $\xi_{n+1}$ is mean zero and has covariance $\Xi_{n+1}$, such that $n^{-1} \sum_{k=1}^n \Xi_k \to S_\infty$ in probability and $\sup_{n\ge 1} \E[\norm{\xi_n}^{2+}] < \infty$.
    Then,
    \begin{align*}
        \sqrt n\,(\bar \theta_n - \theta^\star) \todist{} \normal(0, A^{-1} S_\infty A^{-1})\,, \qquad A \deq \nabla^2 f(\theta^\star)\,.
    \end{align*}
\end{thm}

Write the iterate for~\ref{eq:ASGD} as
\begin{align*}
    \theta_{n+1}
    &= \theta_n - h_{n+1}\,(\nabla f(\theta_n) + \xi_{n+1}) \\
    &= \theta_n - h_{n+1} A\,(\theta_n - \theta^\star) - h_{n+1} \xi_{n+1} - h_{n+1}\,\underbrace{(\nabla f(\theta_n) - A\,(\theta_n - \theta^\star))}_{\eqqcolon \zeta_n}\,,
\end{align*}
which leads to
\begin{align*}
    \delta_{n+1}
    &= \delta_n - h_{n+1} A \delta_n - h_{n+1}\,(\xi_{n+1}+\zeta_n)\,.
\end{align*}
By applying the derivation of~\eqref{eq:clt_iterate}, replacing $\xi_{n+1}$ with $\xi_{n+1} + \zeta_n$, we obtain
\begin{align*}
    \sqrt n\,\bar \delta_n
    &= \frac{1}{n^{1/2}}\,M_0^n \delta_0 - \frac{1}{n^{1/2}} \sum_{k=1}^{n-1} A^{-1}\,(\xi_k + \zeta_{k-1}) - \frac{1}{n^{1/2}} \sum_{k=1}^{n-1} (M_k^n - A^{-1})\,(\xi_k + \zeta_{k-1})\,.
\end{align*}
We must show that the extra terms involving the $\zeta_k$'s vanish in the limit.
Since the $M_k^n$ matrices are bounded, it suffices to show that
\begin{align*}
    \frac{1}{\sqrt n} \sum_{k=1}^n \E\norm{\zeta_k} \to 0\,.
\end{align*}
Since we assume that the third derivative of $f$ is bounded, Taylor expansion shows that
\begin{align*}
    \norm{\zeta_n}
    &= \norm{\nabla f(\theta_n) - \nabla^2 f(\theta^\star)\,(\theta_n - \theta^\star)}
    \lesssim \norm{\theta_n - \theta^\star}^2\,.
\end{align*}
By our usual argument, for $n$ large so that $h_n$ is small,
\begin{align*}
    \E[\norm{\theta_{n+1}-\theta^\star}^2]
    &= \E[\norm{\theta_n - \theta^\star}^2 - 2h_{n+1}\,\langle \nabla f(\theta_n), \theta_n - \theta^\star\rangle + h_{n+1}^2\,\norm{\hat\nabla f(\theta_n)}^2] \\
    &= \E[\norm{\theta_n - \theta^\star}^2 - 2h_{n+1}\,\langle \nabla f(\theta_n), \theta_n - \theta^\star\rangle] \\
    &\qquad{} + \E[h_{n+1}^2\,\norm{\nabla f(\theta_n)}^2 + h_{n+1}^2 \,\norm{\hat\nabla f(\theta_n) - \nabla f(\theta_n)}^2] \\
    &\le \E[(1-\alpha h_{n+1})\,\norm{\theta_n - \theta^\star}^2 + h_{n+1}^2 \tr \Xi_n] \\
    &= (1-\alpha h_{n+1})\E[\norm{\theta_n - \theta^\star}^2] + O(h_{n+1}^2)\,.
\end{align*}
Iterating,
\begin{align*}
    \E[\norm{\theta_n -\theta^\star}^2]
    &\le \exp\Bigl(-\alpha \sum_{k=1}^n h_k\Bigr)\,\norm{\theta_0 - \theta^\star}^2 + \sum_{k=1}^n O(h_k^2) \exp\Bigl(-\alpha\sum_{\ell=k+1}^n h_\ell\Bigr)\,.
\end{align*}
The estimate for the summation in the second term is similar to the computation that appears in~\autoref{lem:clt_technical}, except that it corresponds to the integral
\begin{align*}
    I'
    &\deq \int_1^n \frac{1}{k^{2\gamma}} \exp(-\alpha\,(\tau(n) - \tau(k)))\,\D k\,.
\end{align*}
A trickier calculation eventually shows that
\begin{align}\label{eq:clt_iterate_bd}
    \E[\norm{\theta_n - \theta^\star}^2]
    &\le \exp(-\Omega(n^{1-\gamma}))\,\norm{\theta_0 - \theta^\star}^2 + O(n^{-\gamma})
    = O(n^{-\gamma})\,.
\end{align}
Hence,
\begin{align*}
    \frac{1}{\sqrt n} \sum_{k=1}^n \E\norm{\zeta_k}
    &\lesssim \frac{1}{\sqrt n} \sum_{k=1}^n \E[\norm{\theta_k - \theta^\star}^2]
    \lesssim \frac{1}{\sqrt n} \sum_{k=1}^n \frac{1}{k^\gamma}
    \asymp n^{1/2-\gamma}\,.
\end{align*}
This tends to zero provided $\gamma > 1/2$, completing the proof of~\autoref{thm:clt}.

\paragraph{Application to parameter recovery.}
We now consider the example of parameter recovery in a parametric family of densities ${\{p_\theta\}}_{\theta\in \Theta}$.
Write $\ell(\theta; z) \deq \log(1/p_\theta(z))$.
In this case, the empirical risk minimizer $\widehat \theta_n$ corresponds to the maximum likelihood estimator (MLE), and if $Z_1,\dotsc,Z_n\simiid p_{\theta^\star}$, the population minimizer is indeed $\theta^\star$ (provided that the model is identifiable).

Consider one-pass averaged SGD, so that
\begin{align*}
    \xi_{k+1}
    &= \nabla_\theta\ell(\theta_k; Z_{k+1}) - \int\nabla_\theta\ell(\theta_k; z)\,p_{\theta^\star}(\D z)\,.
\end{align*}
This is conditionally unbiased, and if we define
\begin{align*}
    I(\theta)
    &\deq \cov_{p_{\theta^\star}} \nabla_\theta \ell(\theta; Z)\,,
\end{align*}
then $\Xi_{k+1} = I(\theta_k)$.
Now, adopt the following assumptions:
\begin{itemize}
    \item For each $z\in\eu Z$, the function $\theta \mapsto \ell(\theta; z)$ is strongly convex, smooth, and has a bounded third derivative.
    \item $I(\cdot)$ is Lipschitz continuous.
\end{itemize}
The second assumption, together with the fact that $\theta_n\to \theta^\star$ in probability by~\eqref{eq:clt_iterate_bd}, implies that $\Xi_n = I(\theta_{n-1}) \to I(\theta^\star)$ in probability, hence $n^{-1} \sum_{k=1}^n \Xi_k \to I(\theta^\star)$ in probability as well.
If $\sup_{n\ge 1} \E[\norm{\xi_n}^{2+}] < \infty$, then all of the assumptions of~\autoref{thm:clt} are met.

The value of $I(\cdot)$ at $\theta^\star$ is special: it is called the \emph{Fisher information matrix} and we denote it by $\ms I$:
\begin{align*}
    \ms I
    &\deq I(\theta^\star)
    = \cov_{p_{\theta^\star}} \nabla_\theta \ell(\theta^\star; Z)\,.
\end{align*}
Since
\begin{align*}
    \int \nabla_\theta \ell(\theta; z)\,p_{\theta}(\D z)
    &= -\int \nabla_\theta \log p_{\theta}(z)\,p_{\theta}(\D z)
    = -\int \nabla_\theta p_{\theta}(z) \, \D z \\
    &= -\nabla_\theta \int p_{\theta}(\D z) = 0\,,
\end{align*}
it follows that
\begin{align*}
    0
    &= \nabla_\theta \int \nabla_\theta \ell(\theta^\star; z)\,p_{\theta^\star}(\D z)
    = \int \nabla_\theta^2 \ell(\theta^\star; z)\,p_{\theta^\star}(\D z) + \int \nabla_\theta \ell(\theta^\star; z) \otimes \nabla_\theta p_{\theta^\star}(z)\,\D z \\
    &= \int \nabla_\theta^2 \ell(\theta^\star; z)\,p_{\theta^\star}(\D z) - \int \nabla_\theta \ell(\theta^\star; z) \otimes \nabla_\theta \ell(\theta^\star; z)\, p_{\theta^\star}(\D z)\,.
\end{align*}
Combined with the fact that $\int \nabla_\theta \ell(\theta^\star; z)\,p_{\theta^\star}(\D z) = 0$, we can identify the second term above as $\cov_{p_{\theta^\star}} \nabla_\theta \ell(\theta^\star; Z)$, hence
\begin{align*}
    \ms I
    &= \int \nabla_\theta^2 \ell(\theta^\star; z)\,p_{\theta^\star}(\D z)
    = \nabla^2 \eu R(\theta^\star)\,,
\end{align*}
since $\eu R(\theta) = \int \ell(\theta; z)\,p_{\theta^\star}(\D z)$ for every $\theta \in \Theta$.
Therefore,~\autoref{thm:clt} implies
\begin{align*}
    \sqrt n\,(\bar \theta_n - \theta^\star)
    &\todist{} \normal(0, \ms I^{-1})\,.
\end{align*}

On the other hand, it is classical that under these assumptions, the MLE also has an asymptotically Gaussian limit:
\begin{align*}
    \sqrt n\,(\widehat\theta_n - \theta^\star)
    \todist{} \normal(0, \ms I^{-1})\,.
\end{align*}
This is a celebrated result in statistics because the asymptotic covariance $\ms I^{-1}$ is also a lower bound on the covariance of any unbiased estimator of $\theta^\star$, by the \emph{Cram\'er{--}Rao} or \emph{information inequality}.
Moreover, via comparison of experiments, it is known that no estimator can perform better than the MLE, in the sense that the MLE is locally asymptotically minimax optimal.
We have just shown that this asymptotic optimality property also carries over to Polyak{--}Ruppert averaging of SGD\@.
Finally, we remark that when $p_\theta = \normal(\theta, A^{-1})$, this encompasses~\autoref{ex:gaussian_clt}.

\subsection{Variance reduction}

In \S\ref{ssec:sgd_generalization}, we argued that the generalization bounds for~\ref{eq:GD} and~\ref{eq:SGD} are comparable, yet the overall computational cost for~\ref{eq:GD} is roughly $n$ times larger due to the larger per-iteration cost.
In making this comparison, our assumption was that we do not aim to completely minimize the empirical risk; we simply want the optimization error to be comparable to the statistical error.
However, if our goal is indeed to fully minimize the empirical risk $\eu R_n$, then~\ref{eq:GD} can be faster than~\ref{eq:SGD} for high-accuracy solutions.

In this section, the structural assumption we impose on the objective $f$ is that it is a finite sum of $n$ functions:
\begin{align*}
    f = \frac{1}{n} \sum_{i=1}^n f_i\,.
\end{align*}
In this setting, it makes sense to measure the complexity in terms of the number of gradient evaluations of the \emph{individual} functions $f_i$.

For example, in the convex and smooth setting, assume that the cost of computing the full gradient of $\eu R_n$ is $n$ times larger than the cost of computing the gradient of a single term (corresponding to a single sample).
Then, in order to obtain an $\varepsilon$-approximate minimizer of $\eu R_n$, the overall computational cost for~\ref{eq:GD} is $O(n\beta R^2/\varepsilon)$ (\autoref{thm:gd_fn_value}).
For SGD\@, we take our stochastic gradient $\hat\nabla f(x)$ to be $\nabla f_i(x)$ for a randomly chosen index $i\sim \unif([n])$.
Then, the variance of the stochastic gradient is
\begin{align*}
    &\frac{1}{n} \sum_{i=1}^n {\norm{\nabla f_i(x) - \nabla f(x)}^2} \\
    &\qquad \lesssim \frac{1}{n} \sum_{i=1}^n {\norm{\nabla f_i(x_\star)}^2} + \frac{1}{n} \sum_{i=1}^n \norm{\nabla f_i(x) - \nabla f_i(x_\star)}^2 + \norm{\nabla f(x) - \nabla f(x_\star)}^2 \\
    &\qquad \lesssim \underbrace{\frac{1}{n} \sum_{i=1}^n {\norm{\nabla f_i(x_\star)}^2}}_{c_0} + \underbrace{\beta^2}_{c_1}\,\norm{x-x_\star}^2\,.
\end{align*}
We can apply (a variant of)~\autoref{qu:sgd_unbdd_var} to conclude that for sufficiently small $\varepsilon$, the complexity of SGD is $O(c_0 R^2/\varepsilon^2)$.
In the strongly convex and smooth setting, the rates are $\widetilde O(n\kappa \log(\alpha R^2/\varepsilon))$ and $\widetilde O(c_0/(\alpha\varepsilon))$ respectively.
In general, these rates are incomparable.

In this section, we show that we can improve upon these rates through a technique known as \emph{variance reduction}.
Namely, the method we develop runs in a number of iterations comparable to~\ref{eq:GD}, but with a per-iteration cost comparable to~\ref{eq:SGD}.

To see why there is a possibility for variance reduction, note that if we run~\ref{eq:SGD}, the variance of the stochastic gradient is bounded away from zero{---}even at the minimizer $x_\star${---}due to the presence of the $c_0$ term.
However, if the iterates of the algorithm are converging to the minimizer, $x_n \to x_\star$, then we can hope that the variance of the gradient estimator also tends to zero.

This intuition is carried out by the family of variance reduction methods, of which we pick one representative one: \emph{stochastic variance reduced gradient descent} (SVRG)~\cite{JohZha13SVRG}.
We generalize our setting to a composite objective:
\begin{align*}
    F = f + g
    = \frac{1}{n} \sum_{i=1}^n f_i + g\,.
\end{align*}
The algorithm proceeds via ``epochs'', where the $t$-th epoch runs for $N_t$ iterations.
In the $t$-th epoch, we initialize $x_0^{t} \deq x_{N_{t-1}}^{t-1}$ (that is, we start the $t$-epoch at the last iterate of the previous epoch).
The algorithm is described as follows:
\begin{align}\label{eq:SVRG}\tag{$\msf{SVRG}$}
    \begin{aligned}
        x^t_{n+1}
        &\deq \argmin_{x\in\R^d}{\bigl\{\langle \hat \nabla^t_n f, x- x^t_n \rangle + g(x) + \frac{1}{2h}\,\norm{x-x_t^n}^2\bigr\}}\,, \\
        \hat\nabla^t_n f
        &\deq \nabla f_{i^t_n}(x^t_n) - \nabla f_{i^t_n}(\bar x^t_0) + \nabla f(\bar x^t_0)\,, \qquad i^t_n \sim \unif([n])\,.
    \end{aligned}
\end{align}
Here, $\bar x_0^t$ is a certain average of iterates from the previous epoch $t-1$.
Note that in each epoch, we compute (and store) the full gradient $\nabla f(\bar x^t_0)$, which requires $n$ gradient computations, and then each subsequent iteration requires only one gradient computation.
Therefore, the $t$-th epoch requires $n+N_t$ gradient computations, and the total cost after $T$ epochs is $Tn + \sum_{t=0}^{T-1} N_t$.

The intuition here is that if we take the expectation over $i^t_n$, then
\begin{align*}
    \E\hat\nabla^t_n f
    &= \E[\nabla f_{i^t_n}(x^t_n) - \nabla f_{i^t_n}(\bar x_0^t) + \nabla f(\bar x_0^t)]
    = \nabla f(x^t_n)\,,
\end{align*}
so the gradient estimator is indeed unbiased.
But the extra centered term that we added to the gradient estimator, $-\nabla f_{i^t_n}(\bar x_0^t) + \nabla f(\bar x_0^t)$, reduces the variance: since $\bar x_0^t, x_n^t \to x_\star$, we expect that
\begin{align*}
    \hat\nabla^t_n f - \nabla f(x^t_n)
    &= \nabla f_{i^t_n}(x^t_n) - \nabla f_{i^t_n}(\bar x^t_0) + \nabla f(\bar x^t_0) - \nabla f(x^t_n) \to 0\,.
\end{align*}

\begin{thm}[convergence of~\ref{eq:SVRG}]\label{thm:SVRG}
    Assume that $f$ is $\alpha_f$-convex and $\beta$-smooth, and that $g$ is $\alpha_g$-convex.
    Then, the following assertions hold for a suitable choice of step size $h$ and averaged iterate $\bar x_0^t$.
    Let $\Delta_0 \deq F(x_0) - F_\star + \beta\,\norm{x_0 - x_\star}^2$.
    \begin{itemize}
        \item If $\alpha_f + \alpha_g = 0$, then~\ref{eq:SVRG} can achieve $\E F(\bar x_0^T) - F_\star \le \varepsilon$ with a total number of gradient evaluations at most $O(n\log(\Delta_0/\varepsilon) + \Delta_0/\varepsilon)$.
        \item If $\alpha_f + \alpha_g > 0$, then~\ref{eq:SVRG} can achieve $\E F(\bar x_0^T) - F_\star \le \varepsilon$ with a total number of gradient evaluations at most $O((n+\kappa)\log(\Delta_0/\varepsilon))$, where $\kappa \deq \beta/(\alpha_f +\alpha_g)$.
    \end{itemize}
\end{thm}
\begin{proof}
    We start by analyzing a single epoch; thus, for simplicity of notation, we drop the superscript $t$.
    The one-step inequality for SGD (see~\autoref{thm:SMPGD}) shows that
    \begin{align*}
        \E F(x_{n+1}) - F_\star
        &\le \frac{1-\alpha_f h}{2h}\E[\norm{x_n - x_\star}^2] - \frac{1+\alpha_g h}{2h}\E[\norm{x_{n+1} - x_\star}^2] \\
        &\qquad{} + h \E[\norm{\hat \nabla_n f - \nabla f(x_n)}^2]\,.
    \end{align*}
    We upper bound the variance of the stochastic gradient: by~\eqref{eq:pre_coercivity},
    \begin{align*}
        &\E[\norm{\nabla f_{i_n}(x_n) - \nabla f_{i_n}(\bar x_0) + \nabla f(\bar x_0) - \nabla f(x_n)}^2]
        \le \E[\norm{\nabla f_{i_n}(x_n) - \nabla f_{i_n}(\bar x_0)}^2] \\
        &\qquad \le 2\E[\norm{\nabla f_{i_n}(x_n) - \nabla f_{i_n}(x_\star)}^2] + 2\E[\norm{\nabla f_{i_n}(\bar x_0) - \nabla f_{i_n}(x_\star)}^2] \\
        &\qquad \le 2\beta \E[D_{f_{i_n}}(x_n, x_\star) + D_{f_{i_n}}(\bar x_0, x_\star)]
        = 2\beta \E[D_f(x_n, x_\star) + D_f(\bar x_0, x_\star)] \\
        &\qquad \le 2\beta \E[D_F(x_n, x_\star) + D_F(\bar x_0, x_\star)]
        = 2\beta\E[F(x_n) - F_\star + F(\bar x_0) - F_\star]\,.
    \end{align*}
    Note that this already captures the intuition above, namely, the variance decreases with the objective gap.
    Therefore, we end up with the recursion
    \begin{align*}
        \E F(x_{n+1}) - F_\star
        &\le \frac{1-\alpha_f h}{2h}\E[\norm{x_n - x_\star}^2] - \frac{1+\alpha_g h}{2h}\E[\norm{x_{n+1} - x_\star}^2] \\
        &\qquad{} + 2\beta h \E[F(x_n) - F_\star + F(\bar x_0) - F_\star]\,.
    \end{align*}
    We now choose $h = 1/(8\beta)$ so that $2\beta h = 1/4$.
    After dividing by $1+\alpha_g h$ and iterating using~\autoref{lem:discrete_gronwall}, it yields
    \begin{align*}
        \frac{\E[\norm{x_N - x_\star}^2]}{2h}
        &\le \frac{\lambda_h^N \E[\norm{x_0 - x_\star}^2]}{2h} \\
        &\qquad{} + \underbrace{\sum_{n=1}^N \lambda_h^{N-n}\, \Bigl( \frac{1}{4}\,\{\E F(x_{n-1}) - F_\star + \E F(\bar x_0) - F_\star\} - \{\E F(x_n) - F_\star\}\Bigr)}_{= (\star)}\,.
    \end{align*}
    Since by assumption $1/(4\lambda_h) \le 1/3$, the last summation is at most
    \begin{align*}
        (\star)
        &= -\sum_{n=1}^N \lambda_h^{N-n}\,\bigl(1 - \frac{1}{4\lambda_h}\bigr)\,(\E F(x_n) - F_\star) - \frac{1}{4\lambda_h}\,(\E F(x_N) - F_\star) \\[0.25em]
        &\qquad{} + \frac{\lambda_h^{N-1}}{4}\,(\E F(x_0) - F_\star) + \frac{S}{4}\,(\E F(\bar x_0) - F_\star) \\[0.25em]
        &\le - \frac{2}{3} \sum_{n=1}^N \lambda_h^{N-n}\,(\E F(x_n) - F_\star) - \frac{1}{4\lambda_h}\,(\E F(x_N) - F_\star) \\[0.25em]
        &\qquad{} + \frac{\lambda_h^{N-1}}{4}\,(\E F(x_0) - F_\star) + \frac{S}{3}\,(\E F(\bar x_0) - F_\star)\,,
    \end{align*}
    where $S \deq \sum_{n=0}^{N-1} \lambda_h^n$.
    Thus, the above inequality can be rearranged to yield
    \begin{align*}
        &\frac{\lambda_h^N \E[\norm{x_0 - x_\star}^2]}{2hS} + \frac{\lambda_h^{N-1}}{4S}\, (\E F(x_0) - F_\star) + \frac{1}{3} \,(\E F(\bar x_0) - F_\star) \\
        &\qquad{} \ge \frac{\E[\norm{x_N - x_\star}^2]}{2hS} + \frac{1}{4\lambda_h S}\, (\E F(x_N) - F_\star) + \frac{2}{3} \sum_{n=1}^N \frac{\lambda_h^{N-n}}{S}\,(\E F(x_n) - F_\star)\,.
    \end{align*}

    The goal now is to make this inequality telescope across the epochs.
    We recall that $x_0^{t+1} = x_{N_t}^t$, and we define $\bar x_0^{t+1} \deq S_t^{-1} \sum_{n=1}^{N_t} \lambda_h^{N_t-n}\,x^t_n$, where $S_t \deq \sum_{n=0}^{N_t} \lambda_h^n$.
    By applying convexity to the last term, the inequality can be rewritten
    \begin{align*}
        &\frac{\lambda_h^{N_t} \E[\norm{x_0^t - x_\star}^2]}{2hS_t} + \frac{\lambda_h^{N_t-1}}{4S_t}\, (\E F(x_0^t) - F_\star) + \frac{1}{3} \,(\E F(\bar x_0^t) - F_\star) \\[0.25em]
        &\qquad{} \ge \frac{\E[\norm{x_0^{t+1} - x_\star}^2]}{2hS_t} + \frac{1}{4\lambda_h S_t}\, (\E F(x_0^{t+1}) - F_\star) + \frac{2}{3} \,(\E F(\bar x_0^{t+1}) - F_\star)\,.
    \end{align*}
    We now divide the proof up into two cases.

    \textbf{Convex case.} In this case, $\lambda_h = 1$, so $S_t = N_t$.
    Here, we set $N_{t+1} = 2N_t$, which leads to
    \begin{align*}
        &\frac{\E[\norm{x_0^t - x_\star}^2]}{2hN_t} + \frac{1}{4N_t}\, (\E F(x_0^t) - F_\star) + \frac{1}{3} \,(\E F(\bar x_0^t) - F_\star) \\[0.25em]
        &\qquad{} \ge 2\,\Bigl[ \frac{\E[\norm{x_0^{t+1} - x_\star}^2]}{2hN_{t+1}} + \frac{1}{4N_{t+1}}\, (\E F(x_0^{t+1}) - F_\star) + \frac{1}{3} \,(\E F(\bar x_0^{t+1}) - F_\star)\Bigr]\,.
    \end{align*}
    The inequality clearly telescopes and shows that $\E F(\bar x_0^T) - F_\star \le \varepsilon$ after $T$ epochs, where $T \le \log_2[O(F(x_0) - F_\star + \beta\,\norm{x_0 - x_\star}^2)/\varepsilon]$ and $h \asymp 1/\beta$.
    The number of gradient evaluations is $Tn + \sum_{t=0}^{T-1} N_t = Tn + 2^T$, which yields the final result.

    \textbf{Strongly convex case.} In this case, we set $N_t = N$ for all $t$, where $N$ is chosen so that $\lambda_h^N \le 1/2$.
    With $h \asymp 1/\beta$, this leads to $N \asymp \kappa$ and
    \begin{align*}
        &\frac{\E[\norm{x_0^t - x_\star}^2]}{4hS} + \frac{1}{8\lambda_h S}\, (\E F(x_0^t) - F_\star) + \frac{1}{3} \,(\E F(\bar x_0^t) - F_\star) \\[0.25em]
        &\qquad{} \ge 2\,\Bigl[ \frac{\E[\norm{x_0^{t+1} - x_\star}^2]}{4hS} + \frac{1}{8\lambda_h S}\, (\E F(x_0^{t+1}) - F_\star) + \frac{1}{3} \,(\E F(\bar x_0^{t+1}) - F_\star)\Bigr]\,.
    \end{align*}
    Again, this telescopes, and the computational cost is $Tn + TN = O(T\,(n+\kappa))$.
\end{proof}

The result of~\autoref{thm:SVRG} indeed improves upon the rates for~\ref{eq:GD}.
Before presenting the final rate comparison, however, we note that the rates in~\autoref{thm:SVRG} are generally incomparable with the ones achieved via acceleration, i.e., for~\ref{eq:AGD}.
One can ask whether acceleration can also be combined with variance reduction, and the answer is yes; we state a representative result from~\cite{Sha+18ASVRG}.

\begin{thm}[accelerated~\ref{eq:SVRG}]\label{thm:ASVRG}
    Assume that each $f_i$ is convex and $\beta_i$-smooth, and that $g$ is $\alpha$-convex.\renewcommand{\thempfootnote}{$\dagger$}\footnote{The cited paper works under slightly different assumptions compared to~\autoref{thm:SVRG}, but they are broadly comparable.}
    Then, there is an algorithm which achieves the following guarantees.
    Let $\Delta_0 \deq F(x_0) - F_\star + \beta\,\norm{x_0 - x_\star}^2$.
    \begin{itemize}
        \item If $\alpha = 0$, then the algorithm obtains an $\varepsilon$-approximate solution with a total number of gradient evaluations at most $O(n\log(\Delta_0/\varepsilon) + \sqrt{n\Delta_0/\varepsilon})$.
        \item If $\alpha > 0$, then the algorithm obtains an $\varepsilon$-approximate solution with a total number of gradient evaluations at most $O((n+\sqrt{n\kappa})\log(\Delta_0/\varepsilon))$.
    \end{itemize}
\end{thm}

These accelerated rates are almost the best possible due to nearly matching lower bounds~\cite{WooSre16TightComp}.
Interestingly, in this setting, randomness is crucial for attaining the optimal complexity; otherwise, among the class of deterministic algorithms,~\ref{eq:AGD} is the best possible (but strictly worse than~\autoref{thm:ASVRG}).

The rates for the finite sum setting are presented in~\autoref{tab:finite_sum}.

\begin{table}[h]
    \centering
    \begin{tabular}{ccc}
        \textbf{Algorithm} & \textbf{Iterations (Convex)} & \textbf{Iterations (Strongly Convex)} \\
        {\ref{eq:SGD}${}^\dagger$} & $O(c_0 R^2/\varepsilon^2)$ & $O(c_0/(\alpha \varepsilon))$ \\
        {\ref{eq:GD}} & $O(n\Delta_0/\varepsilon)$ & $O(n\kappa\log(\Delta_0/\varepsilon))$ \\
        {\ref{eq:AGD}} & $O(n\sqrt{\Delta_0/\varepsilon})$ & $O(n\sqrt\kappa \log(\Delta_0/\varepsilon))$ \\
        {\ref{eq:SVRG}} & $O(n\log(\Delta_0/\varepsilon) + \Delta_0/\varepsilon)$ & $O((n+\kappa)\log(\Delta_0/\varepsilon))$ \\
        ASVRG & $O(n\log(\Delta_0/\varepsilon) + \sqrt{n\Delta_0/\varepsilon})$ & $O((n+\sqrt{n\kappa})\log(\Delta_0/\varepsilon))$
    \end{tabular}
    \caption{Rates for finite sum minimization.
    }\label{tab:finite_sum}
\end{table}

\subsection*{Bibliographical notes}

In the convex, Lipschitz setting, a detailed study of the role of geometry in stochastic optimization (and in particular, when it is necessary to use non-linear methods such as mirror descent) can be found in~\cite{CheLevDuc25Geometry}.

For more discussion on the statistical performance of~\ref{eq:SGD}, see~\cite{Bach24Learning}.
For an exposition to empirical process theory and statistics, see any standard reference, e.g.,~\cite{Wai19Stats}.

The CLT for~\ref{eq:ASGD} was first established in~\cite{PolJud1992Averaging}.
The treatment of the martingale CLT follows~\cite{Bil1996Probability}.
For an exposition to asymptotic statistics, see~\cite{Vaart98Asymptotic}.

The proof of~\autoref{thm:SVRG} is inspired by~\cite{AllYua16SVRG}, although care was taken to unify the convex and strongly convex proofs.

\subsection*{Exercises}

\begin{question}\label{qu:sgd_unbdd_var}
    Often, stochastic gradients do not have uniformly bounded variance.
    For example, suppose we have the objective function $f: x \mapsto \frac{1}{2n}\sum_{i=1}^n \langle a_i, x\rangle^2$, with stochastic gradient $\hat\nabla f(x) = \langle a_i, x\rangle\,a_i$ with $i \sim \unif([n])$.
    Then, the variance of the stochastic gradient is
    \begin{align*}
        \E[\norm{\hat\nabla f(x) - \nabla f(x)}^2]
        &= \frac{1}{n} \sum_{i=1}^n {\Bigl\lVert \Bigl(a_i a_i^\T - \frac{1}{n} \sum_{j=1}^n a_j a_j^\T \Bigr)\,x\Bigr\rVert^2}\,,
    \end{align*}
    which grows quadratically with $\norm x$.

    Assume therefore that $f$ is $\alpha$-strongly convex and $\beta$-smooth with respect to the Euclidean norm, and that the following variance condition holds:
    \begin{align*}
        \E[\norm{\hat\nabla f(x) - \nabla f(x)}^2]
        &\le c_0 + c_1\,\norm{x-x_\star}^2\qquad\text{for all}~x\in\R^d\,.
    \end{align*}
    Show that the iterates of stochastic gradient descent satisfy the following guarantee.
    If $\varepsilon$ is sufficiently small and the step size $h$ is chosen appropriately, then $\E f(\bar x_N) - f_\star \le \varepsilon$ for a suitably averaged iterate $\bar x_N$ and all
    \begin{align*}
        N \gtrsim \frac{c_0}{\alpha \varepsilon} \log \frac{\alpha\,\norm{x_0 - x_\star}^2}{\varepsilon}\,.
    \end{align*}
\end{question}

\begin{question}\label{qu:sgd_pl}
    Let $f : \R^d\to\R$ be $\alpha$-\ref{eq:PL} and $\beta$-smooth, $\kappa \deq \beta/\alpha$.
    Assume that we have access to a stochastic gradient $\hat\nabla f$ which is unbiased and satisfies the variance bound~\eqref{eq:sgd_var_bd} for the Euclidean norm $\norm \cdot$.
    Prove that SGD with step size $h \le 1/\beta$ achieves the bound
    \begin{align*}
        \E f(x_N) - f_\star
        &\le {(1 - \alpha h)}^N\,(f(x_0) - f_\star) + \frac{\kappa \sigma^2 dh}{2}\,.
    \end{align*}
    What rate does this imply to reach $\E f(x_N) - f_\star \le \varepsilon$?
\end{question}

\begin{question}\label{qu:sgd_regression}
    Consider linear regression with fixed design: our dataset is ${\{(X_i, Y_i)\}}_{i\in [n]}$, where the covariates $X_i$ are deterministic and fixed, and the $Y_i$ are independent with
    \begin{align*}
        Y_i = \langle \theta^\star, X_i \rangle + \xi_i\,, \qquad \xi_i \sim \normal(0, \sigma^2)\,.
    \end{align*}
    The empirical and population risks are
    \begin{align*}
        \eu R_n(\theta)
        &\deq \frac{1}{n} \sum_{i=1}^n {(Y_i - \langle \theta, X_i \rangle)}^2\,, \qquad \eu R(\theta) \deq \E \eu R_n(\theta)\,.
    \end{align*}
    \begin{enumerate}
        \item Show that the population risk is given by $\eu R(\theta) = \sigma^2 + \norm{X\,(\theta - \theta^\star)}^2/n$, where $X \in \R^{n\times d}$ is the matrix whose rows are ${\{X_i^\T\}}_{i\in [n]}$.
        \item Show that the ERM of minimal norm is the least-squares estimator $\widehat \theta_n = {(X^\T X)}^\dagger X^\T Y$, where ${}^\dagger$ denotes the Moore{--}Penrose pseudoinverse.
            Show that the excess risk of the ERM is given by
            \begin{align*}
                \E \eu R(\widehat \theta_n) - \eu R(\theta^\star) = \frac{\sigma^2 \rank X}{n}\,.
            \end{align*}
        \item Consider the iterates of GD on the empirical risk $\eu R_n$.
            Show that for a step size $h$ sufficiently small, it holds that
            \begin{align*}
                \E \eu R(\theta_N) - \eu R(\theta^\star)
                &\le \frac{\sigma^2 \rank X}{n} + O\Bigl( \frac{\norm{\theta_0 -\theta^\star}^2}{Nh}\Bigr)\,.
            \end{align*}
    \end{enumerate}
    \emph{Hints:} For all of these parts, make extensive use of the singular value decomposition of $X$.
    For the third part, write an exact recursion for $\theta_k - \theta^\star$, iterate this recursion, and then compute the excess risk.
    Use the fact that $\max_{x\in [0,1]} x\,{(1-x)}^N \lesssim 1/N$ for $N \ge 1$.
\end{question}

\section{Interior point methods}

We present polynomial-time methods for solving linear programs (LPs) and semidefinite programs (SDPs), among other structured optimization problems, through the family of interior point methods originally introduced in~\cite{Kar1984}.

\subsection{Self-concordant analysis of Newton's method}

We begin with an analysis of Newton's method: to minimize a function $f : \R^d\to\R$, we consider the iteration
\begin{align}\label{eq:NM}\tag{$\msf{NM}$}
    x_{n+1}
    &= x_n - [\nabla^2 f(x_n)]{}^{-1}\,\nabla f(x_n)\,.
\end{align}
The method is derived as follows: consider the local \emph{quadratic} approximation of $f$ around the current iterate $x_n$:
\begin{align*}
    f(x)
    &\approx f(x_n) + \langle \nabla f(x_n), x-x_n \rangle + \frac{1}{2}\,\langle x-x_n, \nabla^2 f(x_n)\,(x-x_n)\rangle\,.
\end{align*}
It is straightforward to check that, provided $\nabla^2 f(x_n) \succ 0$, the minimizer of the quadratic approximation is given by the next iterate $x_{n+1}$ of~\ref{eq:NM}.
Unlike the methods we have studied thus far, Newton's method is a second-order method in that it uses Hessian information.
Since the Hessian of $f$ may not even be invertible everywhere if $f$ is not strictly convex, Newton's method is not always well-defined, and we certainly cannot expect Newton's method to converge globally without further assumptions.
However, unlike first-order methods, Newton's method exhibits local \emph{quadratic} convergence.

\begin{thm}[local quadratic convergence of~\ref{eq:NM}]\label{thm:newton_quad_conv}
    Assume that $\nabla^2 f(x_\star) \succeq \alpha I \succ 0$ and that $\nabla^2 f$ is $\gamma$-Lipschitz in the operator norm:
    \begin{align*}
        \norm{\nabla^2 f(x) - \nabla^2 f(y)}_{\rm op}
        &\le \gamma\,\norm{x-y} \qquad\text{for all}~x,y\in\R^d\,.
    \end{align*}
    Then, provided $\norm{x_0 - x_\star} \le \alpha/(2\gamma)$,~\ref{eq:NM} satisfies
    \begin{align}\label{eq:quad_conv}
        \norm{x_{n+1} - x_\star}
        &\le \frac{\gamma}{\alpha}\,\norm{x_n - x_\star}^2
        \le \frac{1}{2}\,\norm{x_n - x_\star}\,.
    \end{align}
\end{thm}
\begin{proof}
    By Taylor expansion,
    \begin{align*}
        x_{n+1} - x_\star
        &= x_n - x_\star - [\nabla^2 f(x_n)]{}^{-1}\,\nabla f(x_n) \\[0.25em]
        &= [\nabla^2 f(x_n)]{}^{-1} \int_0^1 [\nabla^2 f(x_n) - \nabla^2 f((1-t)\,x_\star + t\,x_n)]\,(x_n - x_\star)\,\D t\,, \\[0.25em]
        \norm{x_{n+1} - x_\star}
        &\le \gamma\,\norm{[\nabla^2 f(x_n)]{}^{-1}}_{\rm op} \int_0^1 (1-t)\,\norm{x_n - x_\star}^2\,\D t \\[0.25em]
        &\le \frac{\gamma}{2}\,\norm{[\nabla^2 f(x_n)]{}^{-1}}_{\rm op} \,\norm{x_n - x_\star}^2\,.
    \end{align*}
    On the other hand,
    \begin{align*}
        \lambda_{\min}(\nabla^2 f(x_n))
        &\ge \lambda_{\min}(\nabla^2 f(x_\star)) - \gamma\,\norm{x_n - x_\star}
        \ge \alpha - \gamma\,\norm{x_n - x_\star}
        \ge \frac{\alpha}{2}\,,
    \end{align*}
    since inductively we have $\norm{x_n - x_\star} \le \alpha/(2\gamma)$.
    Thus, $\norm{[\nabla^2 f(x_n)]{}^{-1}}_{\rm op} \le 2/\alpha$.
\end{proof}

The inequality~\eqref{eq:quad_conv} implies that the error at iteration $n+1$ is proportional to the \emph{square} of the error at iteration $n$, hence ``quadratic'' convergence.
To see what rate of convergence this implies, multiply both sides by $\gamma/\alpha$, yielding
\begin{align*}
    \frac{\gamma}{\alpha}\,\norm{x_{n+1} - x_\star}
    &\le \bigl( \frac{\gamma}{\alpha}\,\norm{x_n - x_\star}\bigr){\bigsp}^2\,.
\end{align*}
Thus,
\begin{align*}
    \log \frac{\alpha}{\gamma\,\norm{x_{n+1} - x_\star}}
    &\ge 2\log \frac{\alpha}{\gamma\,\norm{x_n - x_\star}}\,.
\end{align*}
Iterating, it yields
\begin{align*}
    \norm{x_N - x_\star}
    &\le \exp\Bigl( - 2^N \log \frac{\alpha}{\gamma\,\norm{x_0 - x_\star}}\Bigr)
    \le \exp(- (\log 2)\,2^N)\,.
\end{align*}
Hence, as soon as $x_0$ lies in the region of local quadratic convergence, $\norm{x_0 - x_\star} \le \alpha/(2\gamma)$, achieving $\norm{x_N - x_\star} \le \varepsilon$ only requires a further $O(\log \log(1/\varepsilon))$ iterations.

Our main interest in Newton's method is as a subroutine for developing interior point methods.
For this purpose, local quadratic convergence is actually not as relevant as another key property: namely, affine invariance.
If $A$ is an invertible matrix, then the iterates of~\ref{eq:NM} on the transformed function $\hat x \mapsto f(A\hat x)$ are equal to $A^{-1}$ times the iterates of~\ref{eq:NM} on the original function $f$.
(This is the same notion of affine invariance that we encountered for the Frank{--}Wolfe algorithm in~\autoref{qu:fw_affine_inv}.)

From this perspective, the analysis of~\autoref{thm:newton_quad_conv} is not satisfactory, since the ratio $\gamma/\alpha$ is not affine-invariant.
Thus, instead of assuming that $\nabla^2 f$ is Lipschitz with respect to the Euclidean norm, let us instead assume that it is Lipschitz with respect to the norm generated by $\nabla^2 f$ itself.

\begin{defn}
    Let $f : \R^d\to\R\cup\{\infty\}$ be convex.
    For $x \in \interior \dom f$, the \textbf{local norm} of $v\in\R^d$ at $x$ is
    \begin{align*}
        \norm v_x
        &\deq \sqrt{\langle v, \nabla^2 f(x)\,v\rangle}\,.
    \end{align*}
    The \textbf{dual local norm} of $v^* \in \R^d$ is
    \begin{align*}
        \norm{v^*}_x^*
        &\deq \sqrt{\langle v^*, [\nabla^2 f(x)]{}^{-1}\,v^* \rangle}\,.
    \end{align*}
\end{defn}

\begin{defn}
    Let $f : \R^d\to\R\cup\{\infty\}$ be a regular convex function with an open domain.
    Then, $f$ is \textbf{self-concordant} with parameter $M > 0$ if for all $x\in\interior \dom f$ and $v\in\R^d$,
    \begin{align*}
        \abs{\nabla^3 f(x)[v, v, v]}
        &\le 2M\,\norm v_x^3\,.
    \end{align*}
    If this inequality holds with $M = 1$, we simply say that $f$ is self-concordant.
\end{defn}

The fact that $f$ is regular convex with an open domain implies that it tends to $+\infty$ at the boundary of its domain, i.e., it acts as a \emph{barrier}.
Clearly, any quadratic function is self-concordant.
The main example of a self-concordant function is $-{\log}$, which is used as a building block for further self-concordant functions in \S\ref{ssec:ipm_applications}.

\begin{ex}[self-concordance of $-{\log}$]
    Direct computation shows that for the univariate function $f : x\mapsto -\log x$ over $\R_+$,
    \begin{align*}
        f'(x) = - \frac{1}{x}\,, \qquad f''(x) = \frac{1}{x^2}\,, \qquad f''{'}(x) = - \frac{2}{x^3}\,.
    \end{align*}
    Hence,
    \begin{align*}
        \abs{f''{'}(x)}
        = \frac{2}{x^3}
        \le 2\,[f''(x)]{}^{3/2}\,,
    \end{align*}
    which shows that $-{\log}$ is self-concordant (with parameter $1$).
\end{ex}

Notice that even though $-{\log}$ blows up at the boundary of its domain, it still satisfies self-concordance.
This provides a hint as to why the notion of self-concordance is useful for constrained minimization.

Putting aside the development of further examples for now, let us describe some key properties of self-concordant functions.

\begin{defn}
    Given a convex function $f : \R^d\to\R\cup\{\infty\}$, $x \in \interior \dom f$, and $r > 0$, the \textbf{Dikin ellipsoid} of $f$ at $x$ with radius $r$ is
    \begin{align*}
        \Dikin(x, r)
        &\deq \{y\in\R^d : \norm{y-x}_x < r\}\,.
    \end{align*}
\end{defn}

Self-concordance implies that the Hessian of $f$ is stable inside the Dikin ellipsoid.

\begin{lem}[self-concordance]\label{lem:self_concordance}
    Let $f : \R^d\to\R\cup\{\infty\}$ be a self-concordant function with parameter $M$, and let $x,y\in\dom f$.
    \begin{enumerate}
        \item For any $v\in\R^d$, $-2M\,\norm v_x\,\nabla^2 f(x) \preceq \nabla^3 f(x)[v, \cdot, \cdot] \preceq 2M\,\norm v_x\,\nabla^2 f(x)$.
        \item $\Dikin(x,1/M) \subseteq \dom f$.
        \item If $y \in \Dikin(x, 1/M)$, then
            \begin{align*}
                \frac{\norm{y-x}_x}{1+M\,\norm{y-x}_x} \le \norm{y-x}_y \le \frac{\norm{y-x}_x}{1-M\,\norm{y-x}_x}\,.
            \end{align*}
        \item If $y \in \Dikin(x, 1/M)$ and $v\in\R^d$, then
            \begin{align*}
                {(1-M\,\norm{y-x}_x)}^2\,\nabla^2 f(x) \preceq \nabla^2 f(y) \preceq \frac{1}{{(1-M\,\norm{y-x}_x)}^2}\,\nabla^2 f(x)\,.
            \end{align*}
        \item It holds that
            \begin{align*}
                \langle \nabla f(y) - \nabla f(x), y-x\rangle
                &\ge \frac{\norm{y-x}_x^2}{1+M\,\norm{y-x}_x}\,.
            \end{align*}
    \end{enumerate}
\end{lem}
\begin{proof}
    The first statement follows from general theory about multilinear forms, see~\cite[Appendix 1]{NesNem1995IntPt}.

    Let $z_t \deq (1-t)\,x+t\,y$ and $\phi(t) \deq \langle y-x, \nabla^2 f(z_t)\,(y-x)\rangle^{-1/2}$.
    Then, by the definition of self-concordance,
    \begin{align*}
        \abs{\phi'(t)}
        &= \Bigl\lvert\frac{\nabla^3 f(z_t)[y-x, y-x, y-x]}{2\,\langle y-x, \nabla^2 f(z_t)\,(y-x) \rangle^{3/2}}\Bigr\rvert
        \le M\,.
    \end{align*}
    Hence, $\phi(0) - M \le \phi(1) \le \phi(0) + M$, which yields the third statement.
    The second statement follows from the third since $y \in \Dikin(x, 1/M)$ implies that $\norm{y-x}_{z_t}$ is bounded for $t\in [0,1]$, which contradicts the fact that $f$ blows up at $\partial \dom f$.

    For the fourth statement, let $\psi(t) \deq \langle v,\nabla^2 f(z_t)\,v\rangle$.
    Then, by the definition of self-concordance and the first and third statements,
    \begin{align*}
        \abs{\psi'(t)}
        &\le \abs{\nabla^3 f(z_t)[y-x, v, v]}
        \le 2M\,\norm{y-x}_{x_t}\,\norm v_{x_t}^2
        = \frac{2M}{t}\,\norm{z_t-x}_{x_t}\,\norm v_{x_t}^2 \\
        &\le \frac{2M\,\norm{y-x}_x}{1-Mt\,\norm{y-x}_x}\,\psi(t)\,.
    \end{align*}
    Letting
    \begin{align*}
        C
        &\deq \int_0^1 \frac{2M\,\norm{y-x}_x}{1-Mt\,\norm{y-x}_x}\,\D t
        = 2\log \frac{1}{1-M\,\norm{y-x}_x}\,,
    \end{align*}
    a suitable generalization of Gr\"onwall's inequality (\autoref{lem:gronwall}) implies
    \begin{align*}
        \psi(0) \exp(-C) &\le \psi(1) \le \psi(0) \exp(C)\,.
    \end{align*}

    For the last statement,
    \begin{align*}
        \langle \nabla f(y) - \nabla f(x), y- x\rangle
        &= \int_0^1 \langle \nabla^2 f(z_t)\,(y-x), y-x\rangle\,\D t
        \ge \int_0^1 \frac{\norm{y-x}_x^2}{{(1+Mt\,\norm{y-x}_x)}^2}\,\D t \\
        &= \frac{\norm{y-x}_x^2}{1+M\,\norm{y-x}_x}\,. \qedhere
    \end{align*}
\end{proof}

We are now ready to analyze the local convergence of Newton's method under self-concordance.
It is convenient to measure convergence via the following object.

\begin{defn}
    Let $f : \R^d\to\R\cup\{\infty\}$ be convex.
    The \textbf{Newton decrement} of $f$ at $x \in \interior \dom f$ is the quantity
    \begin{align*}
        \lambda_f(x)
        \deq \norm{\nabla f(x)}_x^*
        = \norm{x^+-x}_x\,,
    \end{align*}
    where $x^+ \deq x - [\nabla^2 f(x)]{}^{-1}\,\nabla f(x)$.
\end{defn}

\begin{thm}[local quadratic convergence of~\ref{eq:NM} under self-concordance]\label{thm:newton_self_conc}
    Consider a self-concordant function $f : \R^d\to\R\cup\{\infty\}$ with parameter $M$.
    Then, for any $x\in\R^d$ with $\lambda_f(x) < 1/M$,
    \begin{align*}
        \lambda_f\bigl(x-[\nabla^2 f(x)]{}^{-1}\,\nabla f(x)\bigr)
        &\le \frac{M{\lambda_f(x)}^2}{{(1-M\lambda_f(x))}^2}\,.
    \end{align*}
\end{thm}
\begin{proof}
    Let $x^+ \deq x - [\nabla^2 f(x)]{}^{-1}\,\nabla f(x)$.
    By~\autoref{lem:self_concordance},
    \begin{align*}
        \lambda_f(x^+)
        &= \norm{\nabla f(x^+)}_{x^+}^*
        \le \frac{\norm{\nabla f(x^+)}_x}{1-M\,\norm{x^+-x}_x}
        = \frac{\norm{\nabla f(x^+)}_x}{1-M\lambda_f(x)}\,.
    \end{align*}
    Then, for $z_t \deq (1-t)\,x + t\,x^+$,
    \begin{align*}
        \nabla f(x^+)
        &= \nabla f(x) + \int_0^1 \nabla^2 f(z_t)\,(x^+ - x)\,\D t
        = \underbrace{\Bigl(\int_0^1 \nabla^2 f(z_t)\,\D t - \nabla^2 f(x)\Bigr)}_{\eqqcolon \Delta}\,(x^+ - x)\,.
    \end{align*}
    Thus,
    \begin{align*}
        \norm{\nabla f(x^+)}_x^2
        &= \langle \Delta\,(x^+ - x), [\nabla^2 f(x)]{}^{-1}\,\Delta\,(x^+ - x)\rangle \\
        &\le \norm{[\nabla^2 f(x)]{}^{-1/2}\,\Delta\,[\nabla^2 f(x)]{}^{-1}\,\Delta\,[\nabla^2 f(x)]{}^{-1/2}}_{\rm op}\,\norm{x^+ - x}_x^2 \\
        &= \norm{[\nabla^2 f(x)]{}^{-1/2}\,\Delta\,[\nabla^2 f(x)]{}^{-1/2}}_{\rm op}^2\,{\lambda_f(x)}^2\,.
    \end{align*}
    To bound the operator norm, we take any unit vector $v\in\R^d$ and compute
    \begin{align*}
        &\langle v, [\nabla^2 f(x)]{}^{-1/2}\,\Delta\,[\nabla^2 f(x)]{}^{-1/2}\,v\rangle \\[0.25em]
        &\qquad = \int_0^1 \langle [\nabla^2 f(x)]{}^{-1/2}\,v, [\nabla^2 f(z_t) - \nabla^2 f(x)]\,[\nabla^2 f(x)]{}^{-1/2}\,v\rangle\,\D t \\[0.25em]
        &\qquad \le \int_0^1 \langle [\nabla^2 f(x)]{}^{-1/2}\,v, \bigl( \frac{1}{{(1-Mt\lambda_f(x))}^2} - 1\bigr)\, \nabla^2 f(x)\,[\nabla^2 f(x)]{}^{-1/2}\,v\rangle\,\D t \\[0.25em]
        &\qquad = \int_0^1 \Bigl( \frac{1}{{(1-Mt\lambda_f(x))}^2} - 1\Bigr)\,\D t
        = \frac{M\lambda_f(x)}{1-M\lambda_f(x)}\,.
    \end{align*}
    Putting everything together yields the result.
\end{proof}

\subsection{Following the central path}\label{ssec:central_path}

We now consider the following structured minimization problem:
\begin{align*}
    \minimize_{x\in\eu C}\qquad \langle a, x\rangle\,.
\end{align*}
One can also consider non-linear objective functions, but this setup is already enough to capture LPs and SDPs.

Our main assumption is that we have explicit access to a self-concordant function $\phi$ with $\overline{\dom \phi} = \eu C$.
Motivation for this assumption is provided in~\cite[\S 5.1.1]{Nes18CvxOpt}, in which Nesterov argues that a fundamental conceptual contradiction lies at the heart of black-box optimization.
Namely, black-box optimization assumes that the problem under consideration is convex, but in order to verify convexity in practice, one often needs to exploit some \emph{structure} of the problem.

The typical strategy to verify convexity is to appeal to the fact that convexity is preserved under various operations (conic combinations, composition with affine mappings, taking suprema, etc.), and thereby reduce the question to checking convexity of a few primitive building blocks.
Similarly, as we describe in \S\ref{ssec:ipm_applications}, one can develop a \emph{barrier calculus} which produces self-concordant functions for convex sets $\eu C$ which are built up out of primitive building blocks by applying basic operations.
Consequently, one can furnish a self-concordant barrier for a huge number of problems of practical interest.

Once we have such a function $\phi$, how can we use it to solve the constrained optimization problem?
As the name suggests, the family of \emph{interior point} methods maintain iterates which lie in the interior of $\eu C$ (unlike other methods, such as cutting plane methods, projected gradient methods, the simplex algorithm, etc.), by using the function $\phi$ as a barrier to exiting $\eu C$.
Although there are many types of interior point methods, here we focus on following the so-called \emph{central path}, i.e., the path
\begin{align*}
    t
    &\mapsto x_\star(t) \deq \argmin_{x\in\R^d}{\{\underbrace{t\,\langle a, x\rangle + \phi(x)}_{\eqqcolon f_t(x)}\}}\,.
\end{align*}

When $t=0$, $x_\star(0) \deq \argmin \phi$ is called the \emph{analytical center} of $\eu C$ (relative to $\phi$).
We discuss later how to obtain a point close to $x_\star(0)$ (and thereby initialize the scheme).
On the other hand, as $t\to\infty$, we expect that $x_\star(t) \to x_\star = \argmin_{x\in\eu C}{\langle a, x \rangle}$.

Suppose that the algorithm is currently on the central path at the point $x_\star(t)$.
The path following scheme proceeds by incrementing $t$ to some $t^+ > t$.
To compute the next point $x_\star(t^+)$, we must minimize the function $f_{t^+}$, and we can use the previous point $x_\star(t)$ as a warm start.
In fact, since $f_{t^+}$ is self-concordant, a natural choice is to apply a step of Newton's method, which requires $x_\star(t)$ to be in the region of local convergence for $f_{t^+}$.
This places a constraint on how fast we can increase $t$.

We remark that although the use of Newton's method is standard, it is not the only choice; one could use, e.g., a few steps of preconditioned gradient descent.

Next, let us calculate the increment for $t$, starting at $x_\star(t)$.
From~\autoref{thm:newton_self_conc}, taking $M = 1$, we want to ensure that $\lambda_{f_{t^+}}(x_\star(t)) < 1$.
However, since $ta + \nabla \phi(x_\star(t)) = 0$,
\begin{align*}
    \lambda_{f_{t^+}}(x_\star(t))
    &= \norm{t^+ a + \nabla \phi(x_\star(t))}_{x_\star(t)}^*
    = (t^+ - t)\,\norm a_{x_\star(t)}^*
    = \frac{t^+ - t}{t}\,\norm{\nabla \phi(x_\star(t))}_{x_\star(t)}^*\,.
\end{align*}
In order for this to be controlled, we want a uniform bound on $\norm{\nabla \phi(x)}_x^*$.
Note that even though the objective function is changing with time, the local norms above are unambiguous because $\nabla^2 f_t$ is independent of $t$.

\begin{defn}
    Let $f : \R^d\to\R\cup\{\infty\}$ be self-concordant.
    Then, $f$ is a \textbf{$\nu$-self-concordant barrier} if $\norm{\nabla f(x)}_x^* \le \sqrt\nu$ for all $x \in \dom f$.
\end{defn}

Assuming that $\phi$ is a $\nu$-self-concordant barrier, it follows that we can take $t^+ = (1+\Omega(1/\sqrt \nu))\,t$.
Therefore, we expect the number of iterations of the scheme to scale as $\widetilde O(\sqrt \nu)$ (up to logarithmic factors).

Before carrying out the full analysis (which takes into account the fact that we do not exactly follow the central path, as well as the issue of initialization), we first consider elements of the barrier calculus and applications.

\subsection{Barrier calculus and applications}\label{ssec:ipm_applications}

\begin{ex}[$-{\log}$ is a $1$-self-concordant barrier]\label{ex:log_barrier}
    Indeed, by explicit calculation, the self-concordant function $x \mapsto f(x) = -\log x$ satisfies
    \begin{align*}
        \frac{{f'(x)}^2}{f''(x)}
        &= 1\,.
    \end{align*}
\end{ex}

Starting from this example, we can build many more.

\begin{prop}[barrier calculus]\mbox{}\label{prop:barrier_calc}
    \begin{enumerate}
        \item Let $f_1$, $f_2$ be $\nu_1$- and $\nu_2$-self-concordant barriers respectively.
            Then, $f_1 + f_2$ is a $(\nu_1+\nu_2)$-self-concordant barrier for $\dom f_1 \cap \dom f_2$.
        \item Let $f_1$, $f_2$ be $\nu_1$- and $\nu_2$-self-concordant barriers respectively.
            Then, $(x,y)\mapsto f_1(x) + f_2(y)$ is a $(\nu_1+\nu_2)$-self-concordant barrier for $\dom f_1 \times \dom f_2$.
        \item Let $\ms A : x \mapsto Ax + b$ be an affine map and let $f$ be a $\nu$-self-concordant barrier with $\dom f \subseteq \range \ms A$.
            Then, the composition $x \mapsto f(\ms A(x))$ is a $\nu$-self-concordant barrier for the set $\ms A^{-1}(\dom f)$.
        \item Let $f$ be a $\nu$-self-concordant barrier over $\R^{d_1}\times \R^{d_2}$.
            Then, $x \mapsto \inf_{y\in\R^{d_2}} f(x, y)$ is a $\nu$-self-concordant barrier.
    \end{enumerate}
\end{prop}

We do not prove these statements since the verification can be tedious.

Since the barrier parameter $\nu$ plays a decisive role in determining the iteration complexity, it is important to know what the best possible value for $\nu$ is.

\begin{prop}[lower bound for the barrier parameter]
    Let $\eu C \subseteq \R^d$ be a convex polytope such that there exists a point in $\partial \eu C$ which belongs to exactly $k$ of the $(d-1)$-dimensional facets of $\eu C$, with the normals to these facets being linearly independent.
    Then, any $\nu$-self-concordant barrier for $\eu C$ satisfies $\nu \ge k$.

    In particular, for the cube, the non-negative orthant, and the simplex, this holds with $k = d$.
\end{prop}
\begin{proof}
    See~\cite[Proposition 2.3.6]{NesNem1995IntPt}.
\end{proof}

On the other hand, for \emph{any} convex body $\eu C$, there exists a $d$-self-concordant barrier.

\begin{thm}[optimal self-concordant barriers]\label{thm:opt_barrier}
    Let $\eu C \subseteq \R^d$ be a convex body.
    The following are $d$-self-concordant barriers for $\eu C$:
    \begin{itemize}
        \item (\cite{Chewi23EntropicBarrier}) the \textbf{entropic barrier}, defined as the convex conjugate of the map $\theta \mapsto \log \int_{\eu C} \exp{\langle \theta, x\rangle}\,\D x$;
        \item (\cite{LeeYue21UnivBarrier}) the \textbf{universal barrier}, defined as the map $x\mapsto \log \vol \eu C^\circ(x)$, where $\eu C^\circ(x) \deq \{y\in\R^d : \langle y, z-x \rangle \le 1~\text{for all}~z\in\eu C\}$ is the polar of $\eu C$ with respect to $x$.
    \end{itemize}
\end{thm}

Although these results are elegant, they are quite useless in practice since the constructed barriers do not lend themselves to easy implementation.
Instead, we present the canonical example of logarithmic barriers, although many more sophisticated barriers have been developed subsequently.

\begin{ex}[logarithmic barriers]\mbox{}
    \begin{enumerate}
        \item Let $\eu C = \{x\in\R^d :Ax \le b\}$ be a polytope, where $A \in \R^{m\times d}$ and $b\in\R^m$.
            If we let ${\{a_i\}}_{i\in [m]}$ denote the rows of $A$, then $x\mapsto -\log(b_i - \langle a_i, x\rangle)$ is a $1$-self-concordant barrier for the set $\{\langle a_i, \cdot \rangle \le b_i\}$, by~\autoref{ex:log_barrier} and~\autoref{prop:barrier_calc}.
            Hence, by~\autoref{prop:barrier_calc} again, $x \mapsto -\sum_{i=1}^m \log(b_i - \langle a_i, x \rangle)$ is an $m$-self-concordant barrier for $\eu C$.

            Since we assume that $\eu C$ is a convex body (in particular, compact), we must have $m\ge d$.
            From~\autoref{thm:opt_barrier}, we know that a $d$-self-concordant barrier for $\eu C$ exists, but the logarithmic barrier is far more tractable.

        \item Let $\eu C = \mb S_+^d$ be the cone of PSD matrices.
            Then, one can show via direct calculation that $X \mapsto -\log \det X$ is a $d$-self-concordant barrier for $\eu C$, and that this is the best possible value for the barrier parameter~\cite[Theorem 5.4.3, Lemma 5.4.7]{Nes18CvxOpt}.

            This is perhaps surprising since the dimension of $\mb S_+^d$ as a vector space is $d\,(d+1)/2$.
    \end{enumerate}
\end{ex}

We are about to show that interior point methods achieve an iteration complexity of roughly $\widetilde O(\sqrt\nu)$.
With the barriers constructed above, we obtain the following results.
\begin{itemize}
    \item \textbf{Linear programs (LPs).} An LP consists of minimizing a linear function $x \mapsto \langle a, x \rangle$ over a polytope $\eu C = \{x\in\R^d : Ax\le b\}$.
        Let $A$ be of size $m\times d$.
        Then, the arithmetic cost of taking a Newton step with the logarithmic barrier is roughly $O(d^2 m)$, so the overall computational cost is $\widetilde O(d^2 m^{3/2})$.
        For $m\asymp d$, this is $\widetilde O(d^{7/2})$.

        In fact, LPs were the original motivating application for the work of~\cite{Kar1984}.
    \item \textbf{Semidefinite programs (SDPs).} An SDP consists of minimizing a linear function $X \mapsto \langle A, X\rangle$ over the PSD cone $\eu C = \mb S_+^d$, possibly with other linear constraints.
        In the case where there are no additional linear constraints, one can show that the arithmetic cost of a Newton step is $O(d^4)$, which leads to an overall computational cost of $\widetilde O(d^{9/2})$.
\end{itemize}
As mentioned in the bibliographical notes, there are numerous improvements over these basic results and it remains an active area of research.

\subsection{Convergence analysis}

Here, we roughly analyze the iteration complexity of the path following scheme.
We recall the setup from \S\ref{ssec:central_path}.

\begin{lem}[properties of self-concordant barriers]\label{lem:barrier}
    Let $f : \R^d\to\R\cup\{\infty\}$ be a $\nu$-self-concordant barrier.
    \begin{enumerate}
        \item For any $v\in\R^d$,
            \begin{align*}
                \langle \nabla f(x), v \rangle^2 \le \nu\,\langle v,\nabla^2 f(x)\,v\rangle\,.
            \end{align*}
        \item For all $x,y\in\dom f$,
            \begin{align*}
                \langle \nabla f(x), y-x \rangle \le \nu\,.
            \end{align*}
    \end{enumerate}
\end{lem}
\begin{proof}\mbox{}
    \begin{enumerate}
        \item By Cauchy{--}Schwarz,
            \begin{align*}
                \abs{\langle \nabla f(x), v\rangle}
                &\le \norm{\nabla f(x)}_x^*\,\norm v_x
                \le \sqrt\nu\,\norm v_x\,.
            \end{align*}
        \item For $t\in [0,1]$, let $z_t \deq (1-t)\,x+t\,y$.
            Then,
                \begin{align*}
                    \partial_t \langle \nabla f(z_t), y-x\rangle
                    &= \langle \nabla^2 f(z_t)\,(y-x), y-x\rangle
                    \ge \frac{1}{\nu}\,\langle \nabla f(z_t), y-x\rangle^2\,.
                \end{align*}
                This implies that $\partial_t \langle \nabla f(z_t), y-x\rangle^{-1} \le -1/\nu$, which leads to the desired inequality.
    \end{enumerate}
\end{proof}

\paragraph{Main stage.}
Assume that at iteration $n$, we have a pair $(t_n, x_n) \in \R_+ \times \R^d$ such that
\begin{align*}
    \lambda_{f_{t_n}}(x_n) \le \frac{1}{4}\,.
\end{align*}
When we update $t_n\mapsto t_{n+1}$, we note that
\begin{align*}
    \lambda_{t_{n+1}}(x_n)
    &= \norm{t_{n+1} a + \nabla \phi(x_n)}_{x_n}^*
    = \bigl\lVert \frac{t_{n+1}}{t_n}\,(t_n a + \nabla \phi(x_n)) + \bigl(1 - \frac{t_{n+1}}{t_n}\bigr)\,\nabla \phi(x_n)\bigr\rVert_{x_n}^* \\
    &\le \frac{t_{n+1}}{t_n}\,\lambda_{f_{t_n}}(x_n) + \bigl( \frac{t_{n+1}}{t_n} - 1\bigr)\,\sqrt \nu\,.
\end{align*}
Set $t_{n+1} = (1+c_0/\sqrt\nu)\,t_n$.
This yields
\begin{align*}
    \lambda_{t_{n+1}}(x_n)
    &\le \bigl(1 + \frac{c_0}{\sqrt \nu}\bigr)\, \frac{1}{4} + c_0
    \le \frac{1+c_0}{4} + c_0\,,
\end{align*}
since we can assume $\nu \ge 1$.
The update $x_n \mapsto x_{n+1}$ is a Newton's step for $f_{t_{n+1}}$, so by~\autoref{thm:newton_self_conc} we have
\begin{align*}
    \lambda_{t_{n+1}}(x_{n+1})
    &\le \frac{{\lambda_{t_{n+1}}(x_n)}^2}{{(1-\lambda_{t_{n+1}}(x_n))}^2}
    \le \Bigl( \frac{(1+c_0)/4+c_0}{1-(1+c_0)/4-c_0}\Bigr){\Bigsp}^2
    \le \frac{1}{4}\,,
\end{align*}
provided that $c_0$ is sufficiently small: $c_0 \le 1/16$ suffices.

This implies that
\begin{align*}
    t_N
    &= \bigl( 1 + \frac{c_0}{\sqrt \nu} \bigr){\bigsp}^N\,t_0\,,
\end{align*}
so that the value of $t$ increases exponentially fast.
Once $t$ is sufficiently large, we have a nearly optimal solution to the original problem.

\begin{lem}
    For any $(t,x)\in\R_+\times \R^d$,
    \begin{align*}
        \langle a, x\rangle - \langle a, x_\star \rangle
        &\le \frac{1}{t}\,\Bigl(\nu + \frac{(\lambda_{f_t}(x) + \sqrt\nu)\,\lambda_{f_t}(x)}{1-\lambda_{f_t}(x)}\Bigr)\,.
    \end{align*}
\end{lem}
\begin{proof}
    First, by~\autoref{lem:barrier},
    \begin{align*}
        \langle a, x_\star(t) \rangle - \langle a, x_\star\rangle
        &= \frac{1}{t}\,\langle \nabla \phi(x_\star(t)), x_\star - x_\star(t)\rangle
        \le \frac{\nu}{t}\,.
    \end{align*}
    Next,
    \begin{align*}
        \langle a, x \rangle - \langle a, x_\star(t)\rangle
        &= \frac{1}{t}\,\langle \nabla f_t(x) - \nabla \phi(x), x-x_\star(t)\rangle
        \le \frac{\lambda_{f_t}(x) + \sqrt \nu}{t}\,\norm{x-x_\star(t)}_x\,.
    \end{align*}
    From~\autoref{lem:self_concordance},
    \begin{align*}
        \frac{\norm{x-x_\star(t)}_x^2}{1+\norm{x-x_\star(t)}_x}
        &\le \langle \nabla f_t(x) - \underbrace{\nabla f_t(x_\star(t))}_{=0}, x-x_\star(t)\rangle
        \le \lambda_{f_t}(x)\,\norm{x-x_\star(t)}_x\,.
    \end{align*}
    Thus, $\norm{x-x_\star(t)}_x/(1+\norm{x-x_\star(t)}_x) \le \lambda_{f_t}(x)$, or, upon rearranging,
    \begin{align*}
        \norm{x-x_\star(t)}_x
        &\le \frac{\lambda_{f_t}(x)}{1-\lambda_{f_t}(x)}\,.
    \end{align*}
    Putting everything together completes the proof.
\end{proof}

The lemma implies that in order to obtain an $\varepsilon$-approximate solution, it suffices to take $N$ such that $t_N \gtrsim \nu/\varepsilon$.
The number of iterations is therefore $O(\sqrt \nu \log(\nu/(\varepsilon t_0)))$.

\paragraph{Preliminary stage.}
The remaining missing piece is to obtain $(t_0,x_0)$ such that $\lambda_{f_{t_0}}(x_0) \le 1/4$.
The idea here is to use another path following scheme to obtain the initial point.
Namely, if we replace the vector $a$ with $-\nabla \phi(\bar x_0)$, where $\bar x_0$ is an arbitrary point in $\interior\eu C$, we obtain the central path
\begin{align*}
    t\mapsto \bar x_\star(t) = \argmin_{\bar x\in\R^d}{\{\underbrace{-t\,\langle \nabla \phi(\bar x_0), \bar x\rangle + \phi(\bar x)}_{\eqqcolon \bar f_t(\bar x)}\}}\,.
\end{align*}
Note that this central path connects $\bar x_\star(1) = \bar x_0$ to $\bar x_\star(0) = x_\star(0) = \argmin \phi$.
Therefore, we should follow the central path by \emph{decreasing} $t$.
By a similar analysis (see~\cite[\S 5.3.5]{Nes18CvxOpt}), one can show that $O(\sqrt \nu \log(\nu\,\norm{\nabla \phi(\bar x_0)}^*_{x_\star(0)}))$ iterations of the path following scheme suffices in order to initialize the main stage.
Here, the quantity $\norm{\nabla \phi(\bar x_0)}_{x_\star(0)}^*$ is a measure of how far the initial guess $\bar x_0$ is from the true analytical center $x_\star(0)$.

Actually, this still does not fully resolve the initialization issue, since it may be difficult to find \emph{any} strictly feasible point $\bar x_0$ at all.
In some situations, one can first augment the problem so that it is trivial to find a strictly feasible starting point, and then one can use yet another path following scheme to compute a strictly feasible point for the original problem.
We omit the details.

\subsection*{Bibliographical notes}

The presentation of this section is heavily inspired by~\cite{Bub15CvxOpt}.
A comprehensive guide to interior point methods can be found in~\cite{NesNem1995IntPt}.

There are many ways to speed up interior point methods beyond the basic theory covered here, e.g., by amortizing the computations cleverly across steps, or by using improved self-concordant barriers.
This remains an active area of research and we do not survey recent developments here.

The universal barrier was introduced and shown to be $O(d)$-self-concordant in~\cite{NesNem1995IntPt}.
The entropic barrier was introduced and shown to be $(1+o(1))\,d$-self-concordant in~\cite{BubEld19Entropic}.
The cited references~\cite{LeeYue21UnivBarrier, Chewi23EntropicBarrier} in~\autoref{thm:opt_barrier} obtained the optimal barrier parameter $d$ for these two barriers respectively.

\appendix

\section{Background on symmetric matrices}

A matrix $A \in \R^{d\times d}$ is called \textbf{symmetric} if $A = A^\T$.
The fundamental fact about such matrices is that their eigenvalues are real, and they can be diagonalized.

\begin{thm}[spectral theorem for symmetric matrices]\label{thm:spectral}
    Let $A\in\R^{d\times d}$ be symmetric.
    Then, $A$ admits the decomposition $A = U\Lambda U^\T$, where $U \in\R^{d\times d}$ is an orthogonal matrix ($UU^\T = U^\T U = I$) and $\Lambda$ is a diagonal matrix whose diagonal entries are the eigenvalues $\lambda_1,\dotsc,\lambda_d \in\R$ of $A$.

    Equivalently, we can write $A = \sum_{i=1}^d \lambda_i u_i u_i^\T$, where ${\{u_i\}}_{i=1}^d$ are unit eigenvectors, thus $\norm{u_i} = 1$ and $Au_i = \lambda_i u_i$ for all $i\in [d]$.
    Moreover, the eigenvectors form an orthonormal basis for $\R^d$: $\langle u_i, u_j \rangle = 0$ if $i \ne j$.
\end{thm}

A special subclass of these matrices, particularly relevant for convex optimization, is the class of positive semidefinite matrices.

\begin{defn}
    A symmetric matrix $A \in\R^{d\times d}$ is \textbf{positive semidefinite} (\textbf{PSD}), written $A \succeq 0$, if all of its eigenvalues are non-negative.
    It is \textbf{positive definite} (\textbf{PD}), written $A \succ 0$, if all of its eigenvalues are strictly positive.
\end{defn}

Note that a PD matrix has positive determinant, so it is invertible.
We denote the sets of symmetric matrices, positive semidefinite matrices, and positive definite matrices by $\mb S^d$, $\mb S_+^d$, and $\mb S_{++}^d$ respectively.

\begin{lem}\label{lem:sym_quad_form}
    A symmetric matrix $A$ has eigenvalues in the range $[\underline\lambda,\overline\lambda]$ if and only if
    \begin{align*}
        \underline \lambda\,\norm v^2 \le \langle v, A\,v\rangle \le \overline \lambda\,\norm v^2 \qquad\text{for all}~v\in\R^d\,.
    \end{align*}
\end{lem}
\begin{proof}
    By rescaling, we can restrict to unit vectors $v$, and by~\autoref{thm:spectral},
    \begin{align*}
        \langle v, A\,v \rangle = \sum_{i=1}^d \lambda_i\,\langle u_i, v\rangle^2\,.
    \end{align*}
    Since ${\{u_i\}}_{i=1}^d$ is an orthonormal basis, $\sum_{i=1}^d \langle u_i, v \rangle^2 = \norm v^2 = 1$.
    Thus, the expression above is minimized over unit vectors $v$ when $v = u_{i_{\min}}$, where $i_{\min}$ is the index of the minimum eigenvalue $\lambda_{\min}$ of $A$, in which case $\langle v, A\,v\rangle = \lambda_{\min}\,\norm v^2$.
    Similarly, the expression above is maximized over unit vectors $v$ by taking $v$ to be an eigenvector corresponding to the maximum eigenvalue $\lambda_{\max}$, in which case $\langle v, A\,v\rangle = \lambda_{\max}\,\norm v^2$.
\end{proof}

Therefore, a matrix $A$ is PSD (resp.\ PD) if
\begin{align*}
    \langle v,A\, v \rangle \ge 0\quad(\text{resp.}>0)\,, \qquad \text{for all}~v\in\R^d\,.
\end{align*}
We next extend the definition of the symbols $\succ$, $\succeq$.

\begin{defn}
    Let $A$, $B$ be two symmetric $d\times d$ matrices.
    We write $A \succeq B$ (resp.\ $A \succ B$) if $A-B\succeq 0$ (resp.\ $A-B \succ 0$).
    This is known as the \textbf{Loewner order} (or sometimes the \textbf{PSD order}).
\end{defn}

Thus, $A \succeq B$ means
\begin{align*}
    \langle v, A\,v \rangle \ge \langle v, B\,v\rangle \qquad\text{for all}~v\in\R^d\,.
\end{align*}
Note that $\succeq$ is a \emph{partial} order, meaning that not all pairs of symmetric matrices are comparable in this ordering.
For example, the matrices
\begin{align*}
    \begin{bmatrix} 1/2 & 0 \\ 0 & 2 \end{bmatrix} \qquad\text{and}\qquad \begin{bmatrix} 2 & 0 \\ 0 & 1/2 \end{bmatrix}
\end{align*}
are incomparable, because one has a larger eigenvalue in one direction, but the other has a larger eigenvalue in a different direction.
The notations $A \succeq \alpha I$ (resp.\ $A \preceq \beta I$) are convenient ways to express that all eigenvalues of $A$ are at least $\alpha$ (resp.\ at most $\beta$).

Finally, we define the operator norm for matrices.

\begin{defn}
    Let $A \in \R^{d_1\times d_2}$ be a matrix.
    The \textbf{operator norm} of $A$ is
    \begin{align*}
        \norm A_{\rm op} \deq \max\{\norm{Av} : \norm v = 1\}\,.
    \end{align*}
\end{defn}

The operator norm has an explicit description: note that $\norm{Av}^2 = \langle v\, A^\T A\,v\rangle$, and that $A^\T A$ is symmetric.
By~\autoref{lem:sym_quad_form}, $\norm{Av}^2 \le C^2$ for all unit vectors $v\in\R^{d_2}$ if and only if the eigenvalues of $A^\T A$ lie in the range $[-C^2, C^2]$, which is true if and only if the absolute values of the eigenvalues of $A^\T A$ are bounded by $C^2$.
Thus, \emph{the operator norm of $A$ is the square root of the largest absolute eigenvalue of $A^\T A$:}
\begin{align*}
    \norm A_{\rm op} = \max\{\sqrt{\abs\lambda} : \lambda~\text{an eigenvalue of}~A^\T A\}\,.
\end{align*}
It can be shown that the eigenvalues of $A^\T A$ are the same as the eigenvalues of $AA^\T$ (except possibly with a different number of zero eigenvalues), so $\norm A_{\rm op}$ is also the square root of the largest absolute eigenvalue of $AA^\T$.

This is easier to describe when $A$ is symmetric.
In this case, $A^\T A = A^2$ has eigenvalues which are the squares of the eigenvalues of $A$.
Thus, when $A$ is symmetric,
\begin{align*}
    \norm A_{\rm op} = \max\{\abs \lambda : \lambda~\text{an eigenvalue of}~A\}\,.
\end{align*}
So, $\norm A_{\rm op} \le C$ if and only if $-CI \preceq A \preceq CI$.

\clearpage
\addcontentsline{toc}{section}{References}
\printbibliography{}

\end{document}